\documentclass[12pt]{amsart}

\usepackage[shortlabels]{enumitem}
\usepackage{amsmath}
\usepackage{amsthm}
\usepackage{amsfonts}
\usepackage{amssymb}
\usepackage{mathrsfs}
\usepackage{stmaryrd}
\usepackage[all]{xy}
\usepackage[usenames, dvipsnames]{color}
\usepackage[margin=1in]{geometry} 
\usepackage[unicode, bookmarks, bookmarksdepth=2, colorlinks=true, linkcolor=blue, citecolor=blue, urlcolor=blue]{hyperref}
\usepackage{eucal}
\usepackage{contour}
\usepackage{shuffle}
\usepackage{dsfont}
\usepackage[normalem]{ulem}
\usepackage{tikz}
\usepackage{textgreek}
\usetikzlibrary{positioning}
\usetikzlibrary{arrows.meta}
\usetikzlibrary{decorations.markings}
\usepackage{booktabs}
\usepackage{caption}

\newcommand{\cA}{\mathcal{A}}
\newcommand{\fA}{\mathfrak{A}}

\newcommand{\bC}{\mathbf{C}}
\newcommand{\cC}{\mathcal{C}}

\newcommand{\bD}{\mathbf{D}}
\newcommand{\cD}{\mathcal{D}}

\newcommand{\cE}{\mathcal{E}}

\newcommand{\sE}{\mathscr{E}}
\newcommand{\bF}{\mathbf{F}}
\newcommand{\cF}{\mathcal{F}}

\newcommand{\sF}{\mathscr{F}}

\newcommand{\cI}{\mathcal{I}}

\newcommand{\sL}{\mathscr{L}}

\newcommand{\bN}{\mathbf{N}}

\newcommand{\rN}{\mathrm{N}}

\newcommand{\cO}{\mathcal{O}}

\newcommand{\bP}{\mathbf{P}}
\newcommand{\cP}{\mathcal{P}}

\newcommand{\bQ}{\mathbf{Q}}

\newcommand{\bR}{\mathbf{R}}
\newcommand{\cR}{\mathcal{R}}

\newcommand{\bS}{\mathbf{S}}
\newcommand{\cS}{\mathcal{S}}
\newcommand{\fS}{\mathfrak{S}}

\newcommand{\bT}{\mathbf{T}}
\newcommand{\cT}{\mathcal{T}}

\newcommand{\bV}{\mathbf{V}}

\newcommand{\sY}{\mathscr{Y}}
\newcommand{\bZ}{\mathbf{Z}}

\newcommand{\fa}{\mathfrak{a}}

\newcommand{\fb}{\mathfrak{b}}

\newcommand{\fc}{\mathfrak{c}}

\newcommand{\fd}{\mathfrak{d}}

\newcommand{\rf}{\mathrm{f}}

\newcommand{\fp}{\mathfrak{p}}

\newcommand{\bs}{\mathbf{s}}

\let\ol\overline
\let\ul\underline
\let\bs\backslash

\let\lbb\llbracket
\let\rbb\rrbracket

\renewcommand{\phi}{\varphi}
\renewcommand{\emptyset}{\varnothing}

\DeclareMathOperator{\Mat}{Mat}
\DeclareMathOperator{\im}{im}

\DeclareMathOperator{\coker}{coker}

\DeclareMathOperator{\avg}{avg}

\DeclareMathOperator{\tr}{tr}

\DeclareMathOperator{\End}{End}

\DeclareMathOperator{\Aut}{Aut}

\DeclareMathOperator{\Spec}{Spec}

\DeclareMathOperator{\res}{res}
\DeclareMathOperator{\sgn}{sgn}

\DeclareMathOperator{\Mod}{Mod}

\DeclareMathOperator{\Hom}{Hom}

\DeclareMathOperator{\Rep}{Rep}
\DeclareMathOperator{\Fun}{Fun}
\newcommand{\id}{\mathrm{id}}

\newcommand{\op}{\mathrm{op}}

\newcommand{\pt}{\mathrm{pt}}
\renewcommand{\Vec}{\mathrm{Vec}}

\DeclareMathOperator{\Frac}{Frac}

\DeclareMathOperator{\uHom}{\underline{Hom}}

\DeclareRobustCommand{\qbinom}{\genfrac{\lbrack}{\rbrack}{0pt}{}}

\newcommand{\GL}{\mathbf{GL}}
\newcommand{\SL}{\mathbf{SL}}

\newcommand{\Gr}{\mathbf{Gr}}

\makeatletter
\@addtoreset{equation}{section}
\makeatother

\numberwithin{equation}{section}
\newtheorem{theorem}[equation]{Theorem}

\newtheorem{proposition}[equation]{Proposition}
\newtheorem{lemma}[equation]{Lemma}
\newtheorem{corollary}[equation]{Corollary}

\theoremstyle{definition}
\newtheorem{rmk}[equation]{Remark}
\newenvironment{remark}[1][]{\begin{rmk}[#1] \pushQED{\qed}}{\popQED \end{rmk}}
\newtheorem{eg}[equation]{Example}
\newenvironment{example}[1][]{\begin{eg}[#1] \pushQED{\qed}}{\popQED \end{eg}}
\newtheorem{defnaux}[equation]{Definition}
\newenvironment{definition}[1][]{\begin{defnaux}[#1]\pushQED{\qed}}{\popQED \end{defnaux}}
\newtheorem{constraux}[equation]{Construction}

\newtheorem{convaux}[equation]{Convention}
\newenvironment{convention}[1][]{\begin{convaux}[#1]\pushQED{\qed}}{\popQED \end{convaux}}

\contourlength{1pt}

\contourlength{0.8pt}
\newcommand{\myuline}[1]{%
  \uline{\phantom{#1}}%
  \llap{\contour{white}{#1}}%
}

\newcommand{\dref}[2]{%
  (\hyperref[#1]{\protect\NoHyper\ref{#1}\protect\endNoHyper#2})%
}

\newcommand{\defn}[1]{\textit{#1}}
\newcommand{\arxiv}[1]{\href{http://arxiv.org/abs/#1}{{\tiny\tt arXiv:#1}}}
\newcommand{\DOI}[1]{\href{http://doi.org/#1}{\color{purple}{\tiny\tt DOI:#1}}}

\DeclareMathOperator{\age}{age}

\newcommand{\ev}{\mathrm{ev}}
\newcommand{\cv}{\mathrm{cv}}

\newcommand{\utr}{\ul{\tr}}

\newcommand{\bone}{\mathbf{1}}
\newcommand{\bbone}{\mathds{1}}

\newcommand{\uotimes}{\mathbin{\ul{\otimes}}}

\DeclareMathOperator{\vol}{vol}

\DeclareMathOperator{\uRep}{\text{\myuline{\rm Rep}}}
\DeclareMathOperator{\uPerm}{\ul{Perm}}

\title{Oligomorphic groups and tensor categories}

\author{Nate Harman}
\author{Andrew Snowden}

\thanks{AS was supported by NSF grants DMS-1453893 and DMS-2301871.}

\date{April 1, 2024}

\begin{document}

\begin{abstract}
Given an oligomorphic group $G$ and a measure $\mu$ for $G$ (in a sense that we introduce), we define a rigid tensor category $\uPerm(G; \mu)$ of ``permutation modules,'' and, in certain cases, an abelian envelope $\uRep(G; \mu)$ of this category. When $G$ is the infinite symmetric group, this recovers Deligne's interpolation category. Other choices for $G$ lead to fundamentally new tensor categories. For example, we construct the first known semi-simple pre-Tannakian categories in positive characteristic with super-exponential growth. One interesting aspect of our construction is that, unlike previous work in this direction, our categories are concrete: the objects are modules over a ring, and the tensor product receives a universal bi-linear map. Central to our constructions is a novel theory of integration on oligomorphic groups, which could be of more general interest. Classifying the measures on an oligomorphic group appears to be a difficult problem, which we solve in only a few cases.
\end{abstract}

\maketitle

\vspace{.5in}
\begin{figure}[!h]
\begin{tikzpicture}
\begin{scope}
\clip(-1,-1) rectangle (6.4,6.4);
\draw[orange!70,line width=14pt,line cap=round]  (2,2)--(7,7);
\filldraw[blue!50,line width=18pt,rounded corners=5pt]  (2,1)--(7,1)--(7,6)--cycle;
\filldraw[blue!50,line width=18pt,rounded corners=5pt]  (2,3)--(6,7)--(2,7)--cycle;
\filldraw[color={rgb:gray,1;green,4;white,4},line width=18pt,rounded corners=5pt]  (0,3)--(1,2)--(1,7)--(0,7)--cycle;
\draw[teal!50,line width=14pt,line cap=round]  (3,0)--(7,0);
\filldraw[red!50,line width=18pt,rounded corners=5pt]  (0,0)--(2,0)--(0,2)--cycle;
\end{scope}
\foreach \i in {0,...,6} {
\foreach \j in {0,...,6} {
\fill (\i,\j) circle (2pt);
}
}
\node at (0,-.75) {\tiny 1};
\node at (1,-.75) {\tiny 2};
\node at (2,-.75) {\tiny 3};
\node at (3,-.75) {\tiny 4};
\node at (4,-.75) {\tiny 5};
\node at (5,-.75) {\tiny 6};
\node at (6,-.75) {\tiny 7};
\node at (-.75,0) {\tiny 1};
\node at (-.75,1) {\tiny 2};
\node at (-.75,2) {\tiny 3};
\node at (-.75,3) {\tiny 4};
\node at (-.75,4) {\tiny 5};
\node at (-.75,5) {\tiny 6};
\node at (-.75,6) {\tiny 7};
\end{tikzpicture}
\caption{A decomposition of $\Omega^2$ into five disjoint $\hat{\fS}$-subsets. The two large blue regions form a single subset.} \label{fig:Shat-sets}
\end{figure}
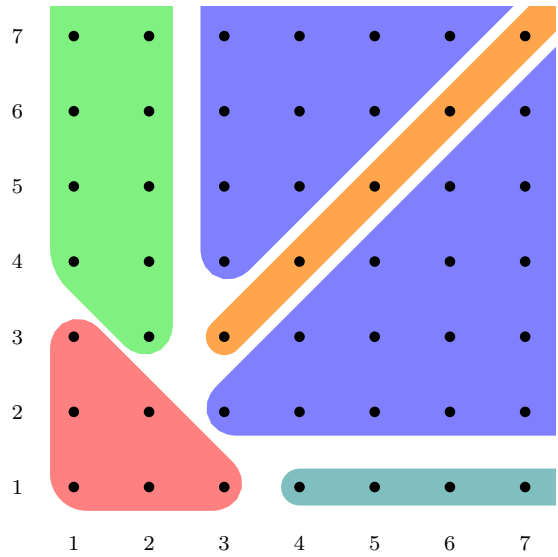

\newpage
\vspace*{0in plus 1fill}
\tableofcontents
\vspace*{0in plus 1fill}

\newpage
\section{Introduction} \label{s:intro}

\subsection{Overview}

The collection of all finite dimensional representations of an algebraic group is an abelian category that comes equipped with a symmetric tensor product. By abstracting the essential properties of this example, one arrives at the notion of a \defn{pre-Tannakian category} (see Definition~\ref{defn:pre-tan}). These categories have received much attention in the literature, yet, despite this, the general landscape of pre-Tannakian categories is still essentially unknown. In this paper, we take a step towards understanding these objects better.

We begin by recalling some relevant previous work. Deligne \cite{Deligne3}, building on earlier work \cite{Deligne1, DeligneMilne}, observed that one can construct new examples pre-Tannakian categories by ``interpolating'' known families. By interpolating the categories $\Rep(\fS_n)$ for $n \in \bN$, he constructed a new pre-Tannakian category $\uRep^{\rm ab}(\fS_t)$ for $t \in \bC$. Deligne's work has generated intense interest, and there has been a slew of ensuing papers, such as \cite{ComesOstrik, ComesOstrik1, Etingof1, Etingof2, EntovaAizenbudHeidersdorf, EAHS, Harman, Harman2, Meir}. In particular, Knop \cite{Knop,Knop2} put Deligne's method into a general framework, and used his theory to construct other examples, such as interpolation categories for finite linear groups.

Currently, the biggest roadblock to understanding general pre-Tannakian categories is the paucity of specimens: the only known examples are the classical representation categories of algebraic (super)groups, and categories obtained from them by interpolation (at least in characteristic zero\footnote{In positive characteristic, there are some very different examples, such as the Verlinde category. These examples run orthogonal to this paper, so we do not discuss them further.}). It is hard to extrapolate from these examples what a general pre-Tannakian category might look like.

The purpose of this paper is develop a new kind of representation theory for a class of groups, called oligomorphic groups, which leads to many new examples of pre-Tannakian tensor categories. To be a bit more precise, we introduce a notion of measure for an oligomorphic group. Given a measure $\mu$ for $G$, we define a non-abelian tensor category $\uPerm(G; \mu)$ of ``permutation modules.'' In certain cases, we also construct a (locally) pre-Tannakian category $\uRep(G; \mu)$.

Before describing the details of this construction, we wish to highlight what we see as the main accomplishments of this paper.
\begin{enumerate}
\item \textit{New examples.} We construct pre-Tannakian categories that are fundamentally different than previously known examples. For example, we construct a pre-Tannakian category associated to the oligomorphic group $\Aut(\bR,<)$. This category is semi-simple in all characteristics, and is the first example of a semi-simple category in positive characteristic of super-exponential growth. We prove that this category cannot be obtained by interpolating known categories. Further work on this category (discussed in \S \ref{sss:delannoy}) has revealed many other interesting properties.
\item \textit{Concrete nature.} In contrast to all previous work in this area, our categories are concrete. In other words, we do not simply formally construct a pre-Tannakian category, but we define a new class of representations that happens to form such a category. These representations have their own internal structure, which can be very useful when studying them; for example, we describe the tensor product of representations via a universal property. This concrete interpretation is new even for the previously known examples, such as Deligne's category.
\item \textit{Abelian envelopes.} To construct $\uRep^{\rm ab}(\fS_t)$, Deligne first constructed a non-abelian category $\Rep(\fS_t)$ by diagrammatics, and then produced an abelian envelope. This is typical of the subject: it is often easier to construct a preliminary non-abelian category, and then one is left with the problem of finding an abelian envelope. In the work of Deligne and Knop, the construction of the abelian envelope ultimately relies on the fact that certain endomorphism rings are semi-simple, which is proven by explicitly identifying them with group algebras.

We give a new method for attacking this problem that does not require explicitly identifying the rings. Our method recovers results of Comes--Ostrik \cite{ComesOstrik} on Deligne's category, and applies to Knop's category in cases where previous methods do not. Our method seems necessary for handling the new examples we construct. See \S \ref{sss:intro-env} for more details.
\item \textit{Theory of integration.} We develop a new theory of measures and integration for oligomorphic groups. This could very well be of interest beyond our applications to tensor categories.
\end{enumerate}
There is one additional important remark to make here. The pre-Tannakian categories we produce using oligomorphic groups are not simply a curious class of examples: in follow-up work \cite{discrete}, we show that they account for all ``discrete'' pre-Tannakian categories (see \S \ref{sss:intro-discrete}). We thus see our theory as a step towards classification of pre-Tannakian categories.

In the remainder of the introduction, we explain our constructions and results in more detail, and attempt to motivate them. We also discuss connections to previous work, and some important problems coming out of our work.

\subsection{Motivation} \label{ss:motiv}

The symmetric group $\fS_n$ acts on the set $\Omega_n=\{1,\ldots,n\}$. Linearizing this action, one obtains the permutation representation $X_n$ of $\fS_n$. In the Deligne category $\uRep(\fS_t)$, there is an analogous object $X$. The defining feature of Deligne's category is that small tensor powers of $X$ behave like small tensor powers of $X_n$ when $n$ is large. The central ideas of this paper came from trying to understand how to make sense of this property in an elegant manner from the point of view of representation theory. There are really two problems to tackle.

The first problem is to understand mapping spaces. We have
\begin{displaymath}
\Hom(X^{\otimes r}, X^{\otimes s}) = \lim_{n \to \infty} \Hom(X_n^{\otimes r}, X_n^{\otimes s})
\end{displaymath}
in the sense that the left side is the stable value of the right side. We would like to have a description of the stable value in terms of some concrete representation theoretic object. To this end, let $\fS$ be the infinite symmetric group, i.e., the group of all permutations of $\Omega=\{1,2,\ldots\}$. Then $\Hom(X_n^{\otimes r}, X_n^{\otimes s})$ is the linearization of the set of $\fS_n$-orbits on $\Omega_n^{r+s}$, and we have
\begin{displaymath}
\fS \backslash \Omega^{r+s} = \lim_{n \to \infty} \fS_n \backslash \Omega_n^{r+s}
\end{displaymath}
in the same sense as the previous equation. Thus the mapping spaces in Deligne's category can be described directly in terms of $\fS$ and its action on $\Omega$.

The second problem is to understand the origin of the parameter $t$ in terms of the group $\fS$. This parameter records the categorical dimension of the basic object $X$. The analogous object $X_n$ has dimension $n$. This reflects the fact that $\Omega_n$ has cardinality $n$, or, equivalently, that the stabilizer $\fS_n(1)$ of~1 has index $n$ in $\fS_n$. It is thus sensible to say that the parameter $t$ is recording a formal notion of index for the subgroup $\fS(1)$ of $\fS$. This led us to the notion of a \defn{generalized index} (Definition~\ref{defn:index}): such is a rule assigning a number to each containment of open subgroups of $\fS$, for which certain natural axioms hold. We prove that a complex-valued generalized index on $\fS$ is completely determined by the index of $\fS(1)$, which can take any value. This provides a clear explanation, directly in terms of $\fS$, for why Deligne's category has a single parameter.

There is a reformulation of the notion of generalized index that will actually be more convenient for us. Instead of saying that $V \subset U$ has some index $a$, we say that the quotient set $U/V$ has measure $a$. For example, $\Omega$ itself has measure $t$. We axiomatize this notion of measure in what follows, and show that measures are equivalent to generalized indices.

The above perspective on Deligne's category lends itself very well to generalization, wherein the infinite symmetric group is replaced by a permutation group $(G, \Omega)$. For this to work, two key properties must hold. First, $G$ must have finitely many orbits on $\Omega^n$ for each $n$, to ensure that mapping spaces are finite dimensional; this is the \defn{oligomorphic} property. And second, we must have a measure for $G$. This paper carries out this procedure in detail, and studies the resulting tensor categories.

\subsection{The constructions}

The main definitions and constructions in this paper come by simply following the preceding discussion to its logical ends. We now provide a summary.

\subsubsection{Oligomorphic groups}

We begin by introducing the first of two key definitions:

\begin{definition}[Cameron]
An \defn{oligomorphic group}\footnote{Cameron describes the origin of the name in a blog post \cite{Cameron10}.} is a group $G$ with a given faithful action on a set $\Omega$ such that $G$ has finitely many orbits on $\Omega^n$ for all $n \ge 0$.
\end{definition}

The most basic example of an oligomorphic group is the infinite symmetric group $\fS$ acting on the set $\Omega=\{1,2,\ldots\}$. Some more examples are given in \S \ref{ss:intro-ex}. We refer to Cameron's excellent book \cite{Cameron} for general background on the subject.

Fix an oligomorphic group $G$. The group $G$ carries a natural topology, where the basic open subgroups are stabilizer groups on powers of $\Omega$ (see \S \ref{ss:admiss} for details). We say that an action of $G$ on a set $X$ is \defn{smooth} if every stabilizer is open. Since we essentially only use smooth actions, we adopt the convention that ``$G$-set'' means ``set with a smooth action of $G$.'' We say that a $G$-set is \defn{finitary} if it has finitely many orbits. Oligomorphic groups have the important property that a product of two finitary $G$-sets is finitary.

Suppose that $X$ is a $G$-set. We often need to work with subsets of $X$ that are stable under some open subgroup, which we would prefer not to name. We therefore introduce the following concept: a \defn{$\hat{G}$-subset} of $X$ is one stable under some open subgroup. We also define a notion of $\hat{G}$-set without an ambient $G$-set; see \S \ref{ss:Ghat}. A typical example is the fiber (over some point) of the map $\Omega^{[2]} \to \Omega$ encountered in \S \ref{ss:motiv}. Figure~\ref{fig:Shat-sets} depicts some other examples of $\hat{\fS}$-subsets. Intuitively, we think of $\hat{G}$ as an infinitesimal neighborhood of the identity in $G$. The fact that every $G$-set comes equipped with a canonical algebra of subsets (the $\hat{G}$-subsets) is an important feature of oligomorphic groups.

\subsubsection{Measures and integration} \label{sss:intro-measure}

We now introduce the second key definition of this paper. Fix a commutative ring $k$.

\begin{definition}
A \defn{measure} for $G$ with values in $k$ is a rule $\mu$ that associates to each finitary $\hat{G}$-set $X$ a value $\mu(X) \in k$ such that the following conditions hold:
\begin{enumerate}
\item Isomorphism invariance: if $X$ and $Y$ are isomorphic then $\mu(X)=\mu(Y)$.
\item Normalization: $\mu(\bone)=1$, where $\bone$ is the one-point set.
\item Additivity: $\mu(X \amalg Y)=\mu(X)+\mu(Y)$.
\item Conjugation invariance: $\mu(X^g)=\mu(X)$, where $X^g$ denotes the conjugate of $X$ by $g$.
\item Multiplicativity in fibrations: given an open subgroup $U$ and a map $X \to Y$ of transitive $U$-sets with fiber $F$, we have $\mu(X)=\mu(F) \cdot \mu(Y)$. \qedhere
\end{enumerate}
\end{definition}

For additional details and comments related to this definition, see \S \ref{ss:measure}.

\begin{example} \label{ex:intro-sym}
Given $t \in \bC$, there is a unique $\bC$-valued measure $\mu_t$ for the infinite symmetric group $\fS$ satisfying $\mu_t(\Omega)=t$. This measure also satisfies $\mu_t(\Omega^{[2]})=t(t-1)$ and $\mu_t(F)=t-1$, where $F$ is any fiber of the map $\Omega^{[2]} \to \Omega$.
\end{example}

Constructing and classifying measures for oligomorphic groups seems to be a difficult problem. To this end, we introduce a ring $\Theta(G)$ as follows: for each finitary $\hat{G}$-set $X$, there is a class $[X]$, and these classes satisfy analogs of (a)--(e) above. The significance of this ring is that it receives a universal measure $\mu_{\rm univ}$, defined by $\mu_{\rm univ}(X)=[X]$. Thus giving a $k$-valued measure on $G$ is equivalent to giving a ring homomorphism $\Theta(G) \to k$. We set $\Theta(G)$ aside for the time being, but will return to it later.

Suppose now that we happen to have a $k$-valued measure $\mu$. We then obtain a theory of integration, as follows. Let $X$ be a $\hat{G}$-set. A \defn{Schwartz function} on $X$ is a $k$-valued function that is smooth (i.e., invariant under an open subgroup) and of finitary support. We let $\cC(X)$ be the set of all such functions, which we term \defn{Schwartz space}. Suppose $\phi \in \cC(X)$. Then $\phi$ assumes finitely many non-zero values $a_1, \ldots, a_n$. Let $A_i=\phi^{-1}(a_i)$, which is a $\hat{G}$-subset of $X$. We define the \defn{integral} of $\phi$ by
\begin{displaymath}
\int_X \phi(X) dx = \sum_{i=1}^n a_i \mu(A_i).
\end{displaymath}
Given a map $f \colon X \to Y$ of $\hat{G}$-sets, we define a push-forward map $f_* \colon \cC(X) \to \cC(Y)$ by integrating over the fibers of $f$. We show that this has all the usual properties of a push-forward map (see \S \ref{ss:push}).

\subsubsection{Permutation representations} \label{sss:intro-perm}

Suppose $\Gamma$ is a finite group. Given a finite $\Gamma$-set $X$, there is an associated permutation representation $k[X]$. If $Y$ is a second finite $\Gamma$-set, then the $k[\Gamma]$-maps $k[X] \to k[Y]$ are given by $Y \times X$ matrices that are $\Gamma$-invariant.

We wish to mimic the above picture to define a category $\uPerm(G;\mu)$ of ``permutation representations.'' For this, we need a theory of matrices. Let $X$ and $Y$ be finitary $G$-sets. We define a \defn{$Y \times X$ matrix} $A$ to be a Schwartz function $A \colon Y \times X \to k$. Given a $Y \times X$ matrix $A$ and a $Z \times Y$ matrix $B$, we define $BA$ to be the $Z \times X$ matrix given by
\begin{displaymath}
(BA)(z,x) = \int_Y B(z,y) A(y,x) dy.
\end{displaymath}
We show that this has the expected properties of matrix multiplication. We go on to study the trace of matrices in some detail. It is possible to extend many other concepts in linear algebra to this setting, such as Jordan decomposition, determinant, and characteristic polynomial; see \cite[\S 7]{arxiv}.

With this theory in hand, we can now define the category $\uPerm(G;\mu)$. For each finitary $G$-set $X$ there is an object $\Vec_X$, which one can think of as the space of column vectors indexed by $X$; a morphism $\Vec_X \to \Vec_Y$ is a $G$-invariant $Y \times X$ matrix; and composition is given by matrix multiplication. This category is additive and carries a tensor product, which we denote by $\uotimes$. On objects, these operations are given by
\begin{displaymath}
\Vec_X \oplus \Vec_Y = \Vec_{X \amalg Y}, \qquad
\Vec_X \uotimes \Vec_Y = \Vec_{X \times Y}.
\end{displaymath}
On morphisms, these operations are given in the usual manner, namely block matrices and Kronecker products. We show that $\uPerm(G;\mu)$ is always a rigid tensor category; see \S \ref{s:perm}.

\begin{example}
Let $\alpha,\beta \in \bC$, and consider the $n \times n$ matrix
\begin{displaymath}
\begin{pmatrix}
\alpha & \beta & \beta & \cdots & \beta \\
\beta & \alpha & \beta & \cdots & \beta \\
\beta & \beta & \alpha & \cdots & \beta \\
\vdots & \vdots & \vdots & \ddots & \vdots \\
\beta & \beta & \beta & \cdots & \alpha
\end{pmatrix}
\end{displaymath}
Such matrices give all endomorphisms of the permutation representation $\bC^n$ of $\fS_n$. Taking $n$ to $\infty$, we obtain an $\Omega \times \Omega$ matrix $A$ for $\fS$. Normally, one would not be able to form $\tr(A)$ or $A^2$ due to the presence of infinite sums. However, if we fix a measure $\mu_t$ for $\fS$, then these constructs are defined. Essentially, one uses the usual rules of algebra augmented by the identity $\sum_{x \in \Omega} 1 = t$. For example, $\tr(A)=t \alpha$. See Example~\ref{ex:matrix} for more details.
\end{example}

\subsubsection{General representations}

The category $\uPerm(G; \mu)$ is essentially never abelian. We would therefore like to construct an abelian envelope of it.

To do this, we introduce the \defn{completed group algebra} $A=A_k(G;\mu)$ of $G$. As a $k$-module, $A$ is simply the inverse limit of Schwartz spaces $\cC(G/U)$ over the open subgroups $U$. The multiplication on $A$ is defined using convolution of functions (which uses integration, hence the dependence on $\mu$). We then define $\uRep(G;\mu)$ to be the category of ``smooth'' $A$-modules; this condition simply means that the action of $A$ is continuous when the module is given the discrete topology. The Schwartz space $\cC(X)$ carries the structure of a smooth $A$-module, and these are the most important modules. The category $\uRep(G;\mu)$ is always a Grothendieck abelian category, and there is a canonical faithful functor
\begin{displaymath}
\Phi \colon \uPerm(G;\mu) \to \uRep(G;\mu)
\end{displaymath}
which takes $\Vec_X$ to $\cC(X)$. Assuming that $\mu$ is a \defn{normal} measure (Definition~\ref{defn:normal}), we show that $\Phi$ is also full.

We next define a tensor structure $\uotimes$ on $\uRep(G;\mu)$ (assuming $\mu$ is normal). There is a ``brute force'' approach to this: namely, declare $\cC(X) \uotimes \cC(Y)=\cC(X \times Y)$ and then extend to general modules by choosing presentations. We prefer to take a more organic approach. Let $M$, $N$, and $E$ be smooth $A$-modules. A \defn{strongly bilinear map} is a function
\begin{displaymath}
q \colon M \times N \to E
\end{displaymath}
that is $k$-bilinear, $G$-equivariant, and satisfies the following additional condition: given $G$-sets $X$ and $Y$ and Schwartz functions $\phi \colon X \to M$ and $\psi \colon Y \to N$, we have
\begin{displaymath}
q \left( \int_X \phi(x) dx, \int_Y \psi(y) dy \right) = \int_{X \times Y} q(\phi(x), \psi(y)) d(x,y).
\end{displaymath}
Here we are integrating module-valued functions (see \S \ref{ss:modint}). The above condition essentially means that $q$ commutes with certain kinds of infinite sums. We define a \defn{tensor product} to be an object equipped with a universal strongly bilinear map. We prove (Theorem~\ref{thm:tensor}) that a tensor product always exists, and this endows $\uRep(G; \mu)$ with the structure of a tensor category.

Assume now that $k$ is a field. One of our main results (Theorem~\ref{thm:regss}) states that if $\mu$ is \defn{quasi-regular} and Propery~(P) holds then $\uRep(G;\mu)$ is locally pre-Tannakian, and if $\mu$ is \defn{regular} then $\uRep(G;\mu)$ is additionally semi-simple. See Definition~\ref{defn:P} for Property~(P), Definition~\ref{defn:pre-tan} for locally pre-Tannakian, and \S \ref{ss:regular} for (quasi-)regular. We also show (Theorem~\ref{thm:abenv}) that, in this setting, $\uRep(G; \mu)$ is the abelian envelope of $\uPerm(G; \mu)$ in the precise sense of \cite{CEAH}. See \S \ref{sss:intro-env} for some discussion of the proofs.

\subsection{Some examples} \label{ss:intro-ex}

We now discuss some examples of oligomorphic groups, and what we can do with them. The first two examples are closely related to the categories of Deligne and Knop, while the next two are new.

\subsubsection{The symmetric group}

Recall that $\fS$ is the group of all permutations of $\Omega=\{1,2,\ldots\}$. We show that $\Theta(\fS)$ is the ring $\bZ\langle x \rangle$ of integer-valued polynomials (Theorem~\ref{thm:Theta-sym}). The $\bZ\langle x \rangle$-linear category $\uPerm(\fS; \mu_{\rm univ})$ associated to the universal measure is very closely related to the category constructed in \cite{Harman}, and seems to be the ``correct'' integral version of Deligne's interpolation category. The construction in \cite{Harman} worked with explicit bases (the Carter--Lusztig bases), and one of our initial sources of motivation was to find a basis-free construction of this category.

From the computation of $\Theta(\fS)$, we see that the measures $\mu_t$ from Example~\ref{ex:intro-sym} account for all $\bC$-valued measures for $\fS$. We show that $\uRep(\fS; \mu_t)$ is a locally pre-Tannakian category for all $t \in \bC$, and semi-simple for $t \not\in \bN$ (Theorem~\ref{thm:S-cat}). We also show that it is the abelian envelope of $\uPerm(\fS; \mu_t)$. Results of Comes--Ostirk \cite{ComesOstrik} imply that our category $\uRep(\fS; \mu_t)$ agrees with Deligne's abelian interpolation category $\uRep^{\rm ab}(\fS_t)$ constructed in \cite[Proposition~8.19]{Deligne3}.

The computation of $\Theta(\fS)$ also shows that measures valued in a field $k$ of characteristic $p$ are parametrized by $\bZ_p$. This was one of the main points of \cite{Harman}, and is consistent with the idea of ``$p$-adic dimension'' of tensor categories, as studied in \cite{EtingofHarmanOstrik}. These measures are not normal, and our $\uRep(\fS; \mu)$ categories are not well-behaved in this setting (see \S \ref{ss:sym-char-p}).

\subsubsection{Linear groups}

Let $\bF$ be a finite field with $q$ elements and let $G=\bigcup_{n \ge 1} \GL_n(\bF)$. The group $G$ acts oligomorphically on $\bV=\bigcup_{n \ge 1} \bF^n$. We show that $\Theta(G)$ is a $q$-analog of the ring of integer-valued polynomials (Theorem~\ref{thm:GL-Theta}), which is closely related to the rings studied in \cite{HarmanHopkins}. In particular, $\Theta(G) \otimes \bQ=\bQ[x]$, and so for each $t \in \bC$ there is an associated complex-valued measure $\mu_t$.

We show that $\uRep(G; \mu_t)$ is a locally pre-Tannakian category for all $t \in \bC$, and semi-simple for $t \not\in \bN_q$ (Theorem~\ref{thm:GLcat}); here $\bN_q$ denotes the set of non-negative $q$-integers. The case $t \not\in \bN_q$ had previously been studied by Knop \cite{Knop,Knop2}, and we rely on his results in this case (though this dependence could probably be excised). The case $t \in \bN_q$ appears to be new, though we have been informed that Entova-Aizenbud and Heidersdorf have results in this direction related to their forthcoming work \cite{EntovaAizenbudHeidersdorf}.

The group $G$ also acts oligomorphically on $\bV \oplus \bV_*$, where $\bV_*$ is the restricted dual of $\bV$, and this leads to a second topology on $G$. We write $G^p$ (resp.\ $G^{\ell}$) for the topological group associated to the action on $\bV$ (resp.\ $\bV \oplus \bV_*$). We show that $\Theta(G^{\ell})=\Theta(G^p)$. We suspect that $\uRep(G^{\ell}; \mu_t)$ and $\uRep(G^p; \mu_t)$ are equivalent for generic values of $t$, and differ at a countable set of values. See \S \ref{ss:levi} for more details.

\subsubsection{Homeomorphisms of the line and circle} \label{sss:intro-line}

Let $G=\Aut(\bR,<)$ be the group of order-preserving bijections of the real line, which is oligomorphic with respect to its action on $\bR$. The $\hat{G}$-subsets of $\bR$ are finite unions of points and open intervals. More generally, a finitary $\hat{G}$-set $X$ is a finite disjoint union of finite products of open intervals; in particular, $X$ is naturally a smooth manifold. We show that one obtains a $\bZ$-valued measure $\mu$ by taking $\mu(X)$ to be the compactly supported Euler characteristic of $X$. Our theory of integration in this case amounts to the Euler calculus introduced by Schapira and Viro (see, e.g., \cite{Viro}).

For a field $k$ (of any characteristic), we show that the category $\uRep_k(G; \mu)$ is locally pre-Tannakian and semi-simple (Theorem~\ref{thm:ord-rigid}). The standard permutation module in this category has dimension $-1$. This is the first example of a semi-simple rigid tensor category in positive characteristic that is not of moderate growth. We prove that $\uRep_k(G; \mu)$ cannot be obtained by interpolating representation categories of finite groups (Theorem~\ref{thm:not-interp}).

In fact, there are three other $\bZ$-valued measures for $G$, though they are not as well-behaved as $\mu$ (e.g., we do not construct pre-Tannakian categories associated to them). In all, these four measures define a ring isomorphism $\Theta(G)=\bZ^4$ (Theorem~\ref{thm:A-Theta}). Conceptually, the four measures correspond to the four partial compactifications of $\bR$.

There is a similar story associated to the group of orientation-preserving self-homeo\-morph\-isms of the unit circle $\bS$ (see \S \ref{ss:circle}). In this case, we again obtain a pre-Tannakian category, but it is no longer semi-simple.

\subsubsection{Boron trees} \label{sss:intro-boron}

\begin{figure}
\begin{displaymath}
\begin{tikzpicture}
\tikzset{leaf/.style={circle,fill=black,draw,minimum size=1mm,inner sep=0pt}}
\tikzset{boron/.style={circle,fill=white,draw,minimum size=1.3mm,inner sep=0pt}}
\node[boron] (A) at (-2,0) {};
\node[boron] (B) at (2,0) {};
\node[boron] (G) at (0,0) {};
\node[boron] (I) at (-1,0) {};
\node[boron] (J) at (1,0) {};
\node[boron] (H) at (0,1) {};
\node[leaf] (C) at (-2.866, .5) {};
\node[leaf] (D) at (-2.866, -.5) {};
\node[leaf] (E) at (2.866, .5) {};
\node[leaf] (F) at (2.866, -.5) {};
\node[leaf] (K) at (-.866, 1.5) {};
\node[leaf] (L) at (.866, 1.5) {};
\node[leaf] (M) at (-1, -1) {};
\node[leaf] (N) at (1, -1) {};
\path[draw] (A)--(I);
\path[draw] (I)--(G);
\path[draw] (G)--(J);
\path[draw] (J)--(B);
\path[draw] (A)--(C);
\path[draw] (A)--(D);
\path[draw] (B)--(E);
\path[draw] (B)--(F);
\path[draw] (G)--(H);
\path[draw] (H)--(K);
\path[draw] (H)--(L);
\path[draw] (I)--(M);
\path[draw] (J)--(N);
\end{tikzpicture}
\end{displaymath}
\caption{A boron tree with eight hydrogen atoms (black nodes) and six boron atoms (white nodes).}
\label{fig:boron8}
\end{figure}
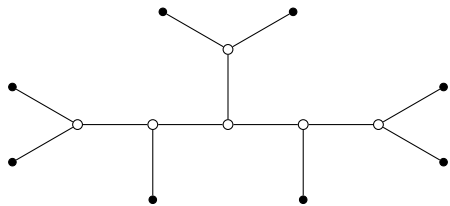

A \defn{boron tree} is a finite tree in which internal vertices have valence three. An example is depicted in Figure~\ref{fig:boron8}; see also \S \ref{s:boron}. The internal vertices are called ``boron atoms'' and the leaves are called ``hydrogen atoms.'' Given a boron tree $T$, one can define a quaternary relation $R$ on the set $T_H$ of hydrogen atoms by declaring $R(w,x; y,z)$ to be true if the geodesic joining $w$ and $x$ meets the one joining $y$ and $z$. It is really this relation that is of interest, though it is equivalent to the tree structure.

Let $\bT(\bQ_2)$ be the Bruhat--Tits tree for $\SL_2(\bQ_2)$; see \cite[Chapter~2]{Serre} for a picture of this object and more discussion. This is essentially an infinite boron tree: the boron atoms are the unimodular $\bZ_2$-lattices in $\bQ_2^2$ and the hydrogen atoms are the lines in $\bQ_2^2$ (these really constitute the boundary of the tree). The quaternary relation in this case can be described as follows: $R(w,x;y,z)$ holds if the cross-ration of the four points $w,x,y,z \in \bP^1(\bQ_2)$ reduces modulo~2 to $\infty \in \bP^1(\bF_2)$.

Let $F$ be a dense subfield of $\bigcup_{n \ge 1} \bQ_2(2^{1/n})$ of countable cardinality, and consider the analogous tree $\bT(F)$. This is like the Bruhat--Tits tree, except it branches in a dense fashion and has countable boundary. The tree $\bT(F)$ is the \defn{Fra\"iss\'e limit} of the class of finite boron trees (see \S \ref{sss:fraisse}). Let $G$ be the group of all permutations of the hydrogen atoms in $\bT(F)$ preserving the quaternary relation $R$. This is an oligomorphic group. This construction---producing an oligomorphic group as the automorphism group of a Fra\"iss\'e limit---is how most oligomorphic groups arise.

We show (Theorem~\ref{thm:boron-theta}) that a ``relative'' version of $\Theta(G)$ is the ring $\bZ[x]/(2x^2+x)$, which essentially means that there are two measures $\mu$ and $\nu$. We do not compute $\Theta(G)$ completely, but one should be able to by building on our work. For any field $k$ of characteristic $\ne 2,3$, we construct a semi-simple pre-Tannakian category over $k$ associated to $G$ and $\mu$ (Theorem~\ref{thm:boron}). The standard permutation module in this category has dimension $\tfrac{3}{2}$.

\subsection{Additional aspects}

We have now discussed the main constructions and examples of this paper. However, there are a number of additional pieces of material that merit comment. We now discuss them.

\subsubsection{Model theory}

Oligomorphic groups are closely connected to model theory. Given some class $\fA$ of finite structures, such as totally ordered sets, graphs, or boron trees, one can sometimes form the \defn{Fra\"iss\'e limit}, which is a suitably universal countable structure $\Omega$ containing each member of $\fA$. The automorphism group of the Fra\"iss\'e limit is often oligomorphic, and this is the most important construction of oligomorphic groups.

Suppose we have an oligomorphic group $G$ obtained as above, i.e., $G=\Aut(\Omega)$ where $\Omega$ is the Fra\"iss\'e limit of some class $\fA$. Since $G$ is completely determined from $\fA$, one should be able to ``see'' measures for $G$ on the $\fA$ side. We explain how this works in \S \ref{s:model}. Precisely, we define a notion of measure for the class $\fA$, and show that these measures essentially correspond to measures for $G$ (Theorem~\ref{thm:compare}); in fact, there is a slight subtlety here in that one actually gets measures for $G$ relative to a stabilizer class $\sE$ (see \S \ref{ss:relative} below).

One of the advantages of this perspective is that it can be much easier to think about measures on the $\fA$ side. For example, suppose $\fA$ is the class of boron trees. A (regular) measure for $\fA$ is a rule assigning to each boron tree a number, such that some concrete identities hold. This is much easier to work with than the corresponding notion of measure for $G$, since the open subgroups of $G$ are not obvious. We therefore take the $\fA$ approach when we analyzing measures in \S \ref{s:boron}.

Another advantage of this perspective is that one can apply general ideas from model theory to the study of measures. For example, we show (Theorem~\ref{thm:R-approx}) that if the class $\fA$ is smoothly approximable then it admits a measure. We note that the smoothly approximable condition has received significant attention in the literature \cite{CherlinHrushovski, KantorLiebeckMacpherson}. The smoothly approximable case seems to roughly coincide with previously known interpolation categories; e.g., it includes the categories of Deligne and Knop.

\subsubsection{Abelian envelopes, nilpotent matrices, and Property~(P)} \label{sss:intro-env}

We have already mentioned one of our main results, Theorem~\ref{thm:regss}, which states that if $\mu$ is quasi-regular and satisfies Property~(P) then $\uRep(G; \mu)$ is the abelian envelope of $\uPerm(G; \mu)$. We now say a little bit more about what goes into the proof of this, and what role Property~(P) has.

We can pass to an open subgroup of $G$ to assume that the measure is regular. The problem is then to show that endomorphism algebras in $\uPerm(G; \mu)$ are semi-simple. Deligne and Knop accomplish this by connecting these algebras to group algebras, but this approach does not seem viable in general. Regularity implies that the trace pairings on endomorphism algebras are non-degenerate (Proposition~\ref{prop:trace-disc}). So to obtain semi-simplicity, it thus suffices to show that nilpotent matrices have trace zero. This is really the key problem.

We show that nilpotent matrices always have vanishing trace in positive characteristic (Corollary~\ref{cor:tr-p}). This does not require any assumption at all on $\mu$, but relies on an important lifting property for measures (see \S \ref{sss:intro-binom}). From this, we deduce the same is true over general coefficient rings, provided $\mu$ satisfies Property~(P) (Theorem~\ref{thm:tr-nilp}). Property~(P) essentially means that $\mu$ can be reduced modulo enough primes, and this allows us to transport the result back from positive characteristic. We note that quasi-regularity is not needed here.

In the case of the symmetric group, we show that all measures satisfy~(P); this actually requires a non-trivial number theoretic argument (see \S \ref{ss:symP}). We thereby recover results of Comes--Ostrik \cite{ComesOstrik} that Deligne's category admits an abelian envelope. In the case of Knop's interpolation category $\uRep(\GL_t(\bF_q))$, the situation is curious. At the singular parameter values, it is easy to see that~(P) holds, and our theory produces an abelian envelope; this was not previously known. At the regular parameter values, Property~(P) is quite subtle and seems to be closely related to the Artin primitive root conjecture (see Remark~\ref{rmk:GL-P}). We expect~(P) holds, but we cannot prove it. However, Knop had already shown that the Karoubi envelope gives the abelian envelope in this case. In the case of $\Aut(\bQ,<)$ and boron trees, it is quite straightforward to see that~(P) holds.

We make two additional remarks in this direction. First, in \cite{Snowden3}, the second author constructs two regular $\bQ$-valued measures $\mu$ and $\nu$ on a particular oligomorphic group, neither of which satisfy~(P). For $\mu$ nilpotent matrices have trace zero, but for $\nu$ this is not the case. This shows that~(P) is not necessary for nilpotents to have trace zero, but also that the condition is not entirely artificial. And second, if $\mu$ is not normal then the problem of constructing an abelian envelope of $\uPerm(G; \mu)$ seems to be quite difficult; in this case, we do not even have a natural candidate.

\subsubsection{Linearizations} \label{sss:intro-linear}

Fix an oligomorphic group $G$ and let $\cS(G)$ be the category of finitary $G$-sets. We define a \defn{linearization} of $\cS(G)$ to be a $k$-linear tensor category $\cT$ having the same objects as $\cS(G)$, and where disjoint union and cartesian product in $\cS(G)$ become direct sum and tensor product in $\cT$. The precise definition (Definition~\ref{defn:linear}) is a bit more involved, but it is important to note that the definition of linearization does not refer to the concept of measure. The basic example of a linearization is the category $\uPerm(G; \mu)$.

A natural question is whether the categories $\uPerm(G; \mu)$ account for all linearizations as $\mu$ varies. Theorem~\ref{thm:linear} answers this in the affirmative. As a corollary, we obtain a moduli-theoretic (and equation-free) interpretation of $\Theta(G)$: its spectrum is the space of linearizations. This result plays an important role in \cite{discrete}.

\subsubsection{$\Theta(G)$ is a binomial ring} \label{sss:intro-binom}

Recall that a \defn{binomial ring} is a commutative ring that is torsion-free as an abelian group and closed under the operation $x \mapsto \binom{x}{n}$ for all $n \ge 0$. The deepest result we prove about $\Theta(G)$ for a general group $G$ is Theorem~\ref{thm:binom}, which states that $\Theta(G)$ is a binomial ring. This theorem has strong consequences for measures. For instance, it implies that any measure valued in a field of positive characteristic $p$ in fact takes values in the prime subfield $\bF_p$, and lifts uniquely to $\bZ_p$. To prove Theorem~\ref{thm:binom}, we show that $\Theta(G)$ admits a $\lambda$-ring structure with trivial Adam's operations, which implies the result by a theorem of Elliott \cite{Elliott}.

\subsection{Relation to previous work}

Our work connects to much previous work:
\begin{itemize}
\item The most important connection is to the theory of oligomorphic groups and homogeneous structures. Our framework enables one to import ideas from this area into the theory of tensor categories. We refer to \cite{Macpherson} for a detailed overview of homogeneous structures. We note that certain kinds of measures (or Euler characteristics) have appeared in the model theory literature, e.g., \cite{AMSW,KS,Wolf}. While these are similar to ours in some ways, they are not exactly the same.
\item Our work obviously owes a debt to Deligne's paper \cite{Deligne3}, but the most direct antecedent is Knop's paper \cite{Knop2}. Knop defines a notion of degree function on a category, and uses this to construct tensor categories. This is at least superficially similar to how we define a notion of measure on the category $\cS(G)$ of finitary $G$-sets, and use this to construct tensor categories. However, there are a number of important differences, e.g., Knop's theory does not accommodate groups like $\Aut(\bR,<)$. The two theories are compared in detail in \cite{Snowden4}, and it is shown that all of Knop's tensor categories can be constructed using the framework of this paper.
\item An important topic in interpolation categories is the construction of abelian envelopes. In this direction, we recover the results of \cite{ComesOstrik,Deligne3} for $\fS$, and establish complementary results to \cite{Knop2} for $\GL_{\infty}(\bF_q)$. Moreover, we describe the abelian envelopes in these cases (and others) in a concrete manner, as modules over the completed group algebra. We make use of \cite{CEAH} for a precise definition and characterization of abelian envelopes.
\item The paper \cite{EtingofHarmanOstrik} shows that in a nice enough tensor category over a field of characteristic $p$, objects can be assigned a $p$-adic dimension. This principle is echoed in our work in a few places, such as the lifting result for measures discussed in \S \ref{sss:intro-binom}.
\item The smooth representation theory of oligomorphic groups is studied in \cite{Nekrasov1}. Some cases are also treated in \cite{DLLX}.
\end{itemize}

\subsection{Subsequent work}

Since the first version of this paper was posted, a number of subsequent works have appeared building on it. We discuss some of these here.

\subsubsection{Discrete pre-Tannakian categories} \label{sss:intro-discrete}

Define a pre-Tannakian category $\cC$ to be \defn{discrete} if every object is a quotient of an \'etale algebra in $\cC$. Note that if $G$ is an algebraic group then $\Rep(G)$ is discrete if and only if $G$ is a finite group. In \cite{discrete}, we show that every discrete pre-Tannakian category has the form $\uRep(G;\mu)$ for some (pro-)oligomorphic group $G$ and measure $\mu$. This shows that oligomorphic groups are intrinsically connected to pre-Tannakian categories. In particular, the categories constructed in this paper account for a very natural piece of the pre-Tannakian landscape.

We say a few words about the proof of the theorem from \cite{discrete}, as it connects to some ideas in this paper. Many of the constructions in this paper really depend only on the category $\cS(G)$ of finitary $G$-sets, and not on $G$ itself. It is therefore natural to characterize this class of categories. We do this in \cite{bcat}. We define a class of categories called \defn{B-categories}, and a more restrictive class called \defn{pre-Galois categories}. We show that pre-Galois categories are exactly those of the form $\cS(G)$ for some pro-oligomorphic group $G$. More general B-categories are also of interest, as the constructions in this paper apply to them as well; for instance, they encompass the ``relative'' situation discussed in \S \ref{ss:relative}.

Suppose now that $\cC$ is any pre-Tannakian category. In \cite{discrete}, we show that the opposite category of \'etale algebras in $\cC$ is pre-Galois, and thus of the form $\cS(G)$ for some pro-oligomorphic group $G$. We define $G$ to be the \defn{oligomorphic fundamental group} of $\cC$; this should be the component group of a more comprehensive fundamental group. We show that $G$ comes with a canonical measure $\mu$; this makes crucial use of the results of \S \ref{s:linear} in this paper. If $\cC$ is discrete, we then show that it is equivalent to $\uRep(G; \mu)$.

\subsubsection{Algebraic representations}

Deligne and Milne \cite{DeligneMilne} defined a category $\uRep(\GL_t)$ that interpolates the representation categories of the algebraic groups $\GL_n$. This category is not associated to an oligomorphic group; it is not discrete. From the perspective of this paper, this category should arise from some kind of measure on the infinite dimensional algebraic group $\GL_{\infty}$. More generally, one might hope for a theory of measures and tensor categories associated to ``algebraic--oligomorphic'' groups.

With this idea in mind, we defined a notion of algebraic--oligomorphic group in \cite{homoten}. We also constructed many non-trivial examples, using a variant of Fra\"iss\'e limits. For example, for $d \ge 1$, define a $d$-space to be a vector space equipped with a symmetric $d$-multilinear form. We show that there is a universal homogeneous $d$-space (over $\bC$, say). Its automorphism group $G_d$ is algebraic-oligomorphic. The group $G_2$ is the infinite orthogonal group, while $G_3$ is a kind of cubic analog of it. We have not yet determined the right notion of measure on these groups. We see this as a very important open problem.

\subsubsection{The Delannoy category} \label{sss:delannoy}

Let $G$ be the group $\Aut(\bR,<)$. As discussed in \S \ref{sss:intro-line}, we construct a semi-simple pre-Tannakian $\cC$ associated to $G$ in this paper. In a follow-up paper \cite{line}, we study this category in great detail. We name it the \defn{Delannoy category}, since it can be described combinatorially in terms of Delannoy paths. We find that the Delannoy category has many special properties: for example, all the simple objects have categorical dimension $\pm 1$, and the Adams operations act trivially on its Grothendieck group. In another paper \cite{circle}, we study the circular analog of the Delannoy category. This case is somewhat more complicated since it is not semi-simple.

\subsubsection{Other examples}

The theory developed in this paper has been applied to a few more oligomorphic groups:
\begin{enumerate}
\item In \cite{arboreal}, we extend our work on boron trees (\S \ref{s:boron}) to trees with valence $\le n$, or trees of unbounded valence. In the bounded valence case, we find two measures, one of which leads to a semi-simple pre-Tannakian category $\cD(n)$. In the unbounded case, we find two one-parameter families of measures. One of these leads to a one-parameter family $\cC(t)$ of semi-simple pre-Tannakian categories (away from a countable set of singular parameters). The category $\cC(t)$ can be viewed as the Deligne interpolation of the categories $\cD(n)$. This is the first non-trivial example of interpolation where the initial categories are of superexponential growth.
\item In \cite{Snowden3}, the homeomorphism group of the Cantor set is studied. There are two measures $\mu$ and $\nu$. The measure $\mu$ leads to a semi-simple pre-Tannakian category of doubly exponential growth, which was the first such example. The measure $\nu$ is regular, but the category $\uPerm(G; \nu)$ has nilpotent endomorphisms with non-zero trace; thus it does not embed into a pre-Tannakian category. This was the first such example, and shows that something like our Property~(P) is necessary.
\item S.~Kriz \cite{Kriz} considers the semi-direct product of $\Aut(\bQ,<)$ with the Borel subgroup in an appropriate version of $\GL_{\infty}(\bF_q)$. She constructs a semi-simple pre-Tannakian category associated to this group, aways from finitely many bad characteristics. This is the fastest-growing known pre-Tannakian category in positive characteristic, and (in combination with \cite{Snowden2}) leads to the fastest-growing known pre-Tannakian category in characteristic~0.
\item Fix a $d$-coloring on the unit circle $\bS$ in which every color is dense, and let $G$ be the (oligomorphic) group of permutations of $\bS$ preserving the cyclic order and coloring. In \cite{Snowden1}, the second author shows that measures for $G$ correspond bijectively to trees with $n$ edges. There is also a version for the line. This generalizes results from \S \ref{s:order}.
\item In \cite{Snowden5}, the second author studies measures for permutations (finite sets equipped with two total orders), and finds that there are exactly~37 measures.
\item Deligne \cite{DeligneLetter2} studied the infinite orthogonal group over a finite field, and showed that there are two one-parameter families of measures (see \S \ref{ss:otherlinear}).
\end{enumerate}

\subsection{The first version of this paper}

The first version of this paper \cite{arxiv} is available on the arxiv. It is somewhat longer than this version: it contains some proofs we have omitted here, many additional tangential comments, and some material that we have chosen to exclude from this version (such as the discussion of Farahat--Higman algebras). We sometimes cite this first version to point to additional details. We note that what we now call ``pro-oligomorphic groups'' were called ``admissible groups'' in the first version.

\subsection{Where to go from here}

Our work presents a multitude of problems to study. However, there are two main directions we wish to highlight as particularly important.

The first is to understand the ring $\Theta(G)$ better. We have only succeeded in computing $\Theta(G)$ in a handful of cases, and each one is rather laborious. And we have been unable to even determine if $\Theta(G)$ is non-zero in some cases, such as when $G$ is the automorphism group of the Rado graph. We would therefore like some tools for working with these rings.

The second direction concerns abelian envelopes. Our results on $\uPerm(G; \mu)$ are consistent with it having an abelian envelope (see Remark~\ref{rmk:abenv}) in many cases. When the measure $\mu$ is quasi-regular and~(P) holds, we show that $\uRep(G; \mu)$ is the abelian envelope. However, we say nothing outside of this case. Can abelian envelopes be constructed more generally?

\subsection{Notation and conventions} \label{ss:notation}

There are a few important definitions and conventions we mention here:
\begin{itemize}
\item For a pro-oligomorphic group $G$, we use the term ``$G$-set'' to mean ``a set equipped with a \emph{smooth} action of $G$.'' See \S \ref{ss:smooth}.
\item A \defn{tensor category} is a $k$-linear additive symmetric monoidal category such that the monoidal product is $k$-bilinear. We denote the unit object by $\bbone$. A \defn{tensor functor} is a $k$-linear symmetric monoidal functor.
\item A \defn{rigid tensor category} is one in which all objects are rigid, i.e., possess a dual; see \S \ref{ss:perm-dual}. The notion of \defn{(locally) pre-Tannakian category} is given in Definition~\ref{defn:pre-tan}.
\end{itemize}
We also list some of the important notation:
\begin{description}[align=right,labelwidth=2.5cm,leftmargin=!]
\item[ $k$ ] the coefficient ring
\item[ $G$ ] a pro-oligomorphic group (Definition~\ref{defn:ad})
\item[ $\mu$ ] a $k$-valued measure for $G$ (Definition~\ref{defn:measure})
\item[ $\Omega$ ] the domain of an oligomorphic group (see Definition~\ref{defn:oligo})
\item[ $\uPerm(G)$ ] the category of ``permutation modules'' (Definition~\ref{defn:permcat})
\item[ $\uRep(G)$ ] the category of smooth $A(G)$-modules (Definition~\ref{defn:repcat})
\item[ $X^{[n]}$ ] the subset of $X^n$ consisting of tuples with distinct coordinates
\item[ $X^{(n)}$ ] the set of $n$-element subsets of $X$
\item[ $G(A)$ ] the subgroup of $G$ stabilizing each element of the subset $A$
\item[ $\Omega(G)$ ] the Burnside ring of $G$ (see \S \ref{ss:burnside})
\item[ $\Theta(G)$ ] the ring associated to $G$ in Definition~\ref{defn:Theta}
\item[ $\cS(G)$ ] the category of finitary $G$-sets (\S \ref{ss:smooth})
\item[ $\cC(X)$ ] the Schwartz space of $X$ (see \S \ref{ss:schwartz})
\item[ $A(G)$ ] the complete group algebra of $G$ (Definition~\ref{defn:complete-alg})
\item[ $\bZ\langle x \rangle$ ] the ring of integer-valued polynomials (see \S \ref{ss:intpoly})
\item[ $\sE$ ] A stabilizer class (see \S \ref{ss:relative})
\item[ $\delta_{X,x}$ ] the function on $X$ that is~1 at $x$ and~0 elsewhere
\item[ $1_A$ ] the indicator function of a subset $A \subset X$
\item[ $\bbone$ ] unit object of a tensor category
\item[ $\bone$ ] the one-point set
\item[ $\cA^{\rf}$ ] the finite length objects in an abelian category $\cA$
\item[ $\fS$ ] the infinite symmetric group
\end{description}

\subsection*{Acknowledgments}

We thank Bhargav Bhatt, Pierre Deligne, Pavel Etingof, Johannes Flake, Dugald Macpherson, Davesh Maulik, Ilia Nekrasov, Iian Smythe, Noah Snyder, Brian Street, David Treumann, and Jacob Tsimerman for helpful conversations.

\part{Integration on oligomorphic groups} \label{part:oligo}

\section{Oligomorphic groups} \label{s:oligo}

\subsection{Oligomorphic groups}

We begin by restating the following fundamental definition:

\begin{definition} \label{defn:oligo}
An \defn{oligomorphic group} is a group $G$ with a given faithful action on a set $\Omega$ such that $G$ has finitely many orbits on $\Omega^n$ for all $n \ge 0$.
\end{definition}

We refer to Cameron's book \cite{Cameron} for general background on oligomorphic groups. The ``finitely many orbits'' condition appears so ubiquitously in this paper that it will be convenient to give it a name: we say that an action of a group $G$ on a set $\Omega$ is \defn{finitary} if $G$ has finitely many orbits on $\Omega$, or equivalently, the quotient set $G \backslash \Omega$ is finite. Thus in the oligomorphic condition, we require that $\Omega^n$ is finitary for all $n \ge 0$.

For a set $\Omega$, let $\Omega^{[n]}$ be the subset of $\Omega^n$ consisting of points with distinct coordinates, and let $\Omega^{(n)}$ be the set of $n$-element subsets of $\Omega$. If $G$ acts on $\Omega$ then it is easy to see that the finitary conditions for $\Omega^n$, $\Omega^{[n]}$, and $\Omega^{(n)}$ are equivalent. Thus oligomorphic could equivalently be defined using $\Omega^{[n]}$ or $\Omega^{(n)}$ in place of $\Omega^n$.

Model theory provides a wealth of examples of oligomorphic groups, and, in a sense, explains the oligomorphic condition. We discuss this in \S \ref{s:model}. For now, we content ourselves with a few examples.

\begin{example} \label{ex:oligo}
We give a number of examples of oligomorphic groups.
\begin{enumerate}
\item The infinite symmetric group $\fS$ is oligomorphic (with respect to its defining action). For our purposes, it does not really matter which version of $\fS$ one uses, but to be definite we take $\fS$ to be the group of all permutations of the positive integers.
\item Let $\bF$ be a finite field and let $\bV=\bigoplus_{i \ge 1} \bF e_i$ be a vector space over $\bF$ of countable infinite dimension. Then the group $G=\bigcup_{n \ge 1} \GL_n(\bF)$ is oligomorphic, with respect to the action on $\bV$. Let $\bV_*=\bigoplus_{i \ge 1} \bF e_i^*$ be the \defn{restricted dual} of $\bV$. Then $G$ is also oligomorphic with respect to its action on $\bV \oplus \bV_*$, and this leads to a different theory (see Example~\ref{ex:gl-top}). The infinite orthogonal, symplectic, and unitary groups over $\bF$ are also oligomorphic.
\item The group $\Aut(\bR,<)$ of orientation-preserving homeomorphisms of the real line is oligomorphic (acting on $\bR$). A similar example is $\Aut(\bQ,<)$, the group of order-preserving permutations of the rational numbers. These two examples are essentially equivalent; we prefer working with the former, though the latter appears more often in the oligomorphic group literature. There are some similar examples: e.g., one can consider homeomorphisms of $\bR$ that may not preserve the orientation, or replace the real line with the circle.
\item Let $G$ be the group of self-homeomorphisms of the Cantor set. This group acts oligomorphically on the set of clopen subsets of the Cantor set. The theory developed in this paper is studied for this group in \cite{Snowden3}.
\item Let $R$ be the Rado graph. There are many ways to think about this graph. It is the Fra\"iss\'e limit of the class of finite graphs. It is also the random countable graph in the sense of Erd\H{o}s--R\'enyi \cite{Erdos}. The automorphism group $\Aut(R)$ of $R$ is oligomorphic, with respect to its action on the vertex set $\Omega$ of $R$. The orbits of $\Aut(R)$ on $\Omega^{(n)}$ are in bijection with isomorphism classes of graphs on $n$ vertices. There are several similar examples, e.g., using bipartite graphs or $K_n$-free graphs.
\item There are a number of variants of the above constructions. One can consider products (such as $\fS \times \fS$ acting on $\Omega \amalg \Omega$) and wreath products (such as $\fS \wr \fS$ acting on $\Omega \times \Omega$, or $\fS \wr \Gamma$ acting on $\Gamma^{\Omega}$, with $\Gamma$ a finite group). One can add colorings, e.g., 2-color $\bQ$ and consider permutations which preserve the order and the color. There is also a way to combine certain examples, e.g., one can combine $\Aut(\bQ,<)$ and the automorphism group of the Rado graph; see \cite{BPP}. \qedhere
\end{enumerate}
\end{example}

\subsection{Pro-oligomorphic groups} \label{ss:admiss}

Oligomorphic groups admit a natural topology which will be very useful for us. To speak about it, we introduce some terminology. Recall that a topological group $G$ is \defn{non-archimedean} if its open subgroups form a neighborhood basis of the identity, and \defn{Roelcke precompact} if whenever $U$ and $V$ are non-empty open subsets of $G$ there exists a finite subset $S$ of $G$ such that $G=USV$. We note that if $G$ is non-archimedean then $G$ is Hausdorff if and only if the intersection of all open subgroups is trivial, and $G$ is Roelcke precompact if and only if for all open subgroups $U$ and $V$ the double coset space $U \backslash G/V$ is finite.

\begin{definition} \label{defn:ad}
A \defn{pro-oligomorphic group}\footnote{In the first version of this paper, we used the term ``admissible group,'' but we now prefer the more descriptive term ``pro-oligomorphic.'' If $G$ is pro-oligomorphic then, for any open subgroup $U$, the image of $G$ in the permutation group of $G/U$ is oligomorphic, and this realizes $G$ as a dense subgroup of an inverse limit of oligomorphic groups.} is a topological group that is Hausdorff, non-archimedean, and Roelcke precompact.
\end{definition}

One easily sees that any open subgroup of a pro-oligomorphic group is pro-oligomorphic, when given the subspace topology. This will be used constantly in what follows.

Now suppose that a group $G$ acts on a set $\Omega$. For a subset $A$ of $\Omega$, let $G(A)$ be the subgroup of $G$ fixing each element of $A$. Then the $G(A)$'s, with $A$ finite, form a neighborhood basis of the identity for a topology on $G$.

\begin{proposition} \label{prop:admiss}
Let $(G, \Omega)$ be an oligomorphic group, and endow $G$ with the above topology. Then $G$ is pro-oligomorphic.
\end{proposition}

\begin{proof}
This topology is non-archimedean by definition. Since $G$ acts faithfully on $\Omega$, we have $\bigcap G(A)=1$, where the intersection is over all finite subsets $A$ of $\Omega$, and so the topology is Hausdorff. Finally, let $U$ and $V$ be open subgroups of $G$. By definition, we can find subsets $A=\{x_1, \ldots, x_r\}$ and $B=\{y_1,\ldots,y_s\}$ of $\Omega$ such that $G(A) \subset U$ and $G(B) \subset V$. Clearly, $G/G(A)$ is isomorphic the orbit of $(x_1, \ldots, x_r) \in \Omega^r$ and $G/G(B)$ is isomorphic to the orbit of $(y_1, \ldots, y_s) \in \Omega^s$. We thus find
\begin{displaymath}
G(A) \backslash G/G(B) \cong G \backslash (G/G(A) \times G/G(B)) \subset G \backslash \Omega^{r+s}.
\end{displaymath}
Since the right side is finite by assumption, so is the left side. As $U \backslash G/V$ is a quotient of $G(A) \backslash G/G(B)$, it too is finite. Thus $G$ is Roelcke precompact.
\end{proof}

\begin{example} \label{ex:sym-top}
Let $\fS$ be the infinite symmetric group acting on $\Omega=\{1,2,\ldots\}$. Define $\fS(n)$ to be the subgroup of $\fS$ consisting of permutations that fix each of the numbers $1, \ldots, n$. Then a subgroup of $\fS$ is open if and only if it contains $\fS(n)$ for some $n$. One can identify $\fS/\fS(n)$ with the set $\Omega^{[n]}$ of all injections $[n] \to \Omega$. Roelcke precompactness thus amounts to the statement that $\fS(m)$ has finitely many orbits on $\Omega^{[n]}$, for any $n$ and $m$, which is easy to see directly.
\end{example}

\begin{example} \label{ex:gl-top}
Consider the infinite general linear group $G$ acting on $\bV$ (see Example~\ref{ex:oligo}). We define the \defn{parabolic topology} on $G$ to be the one induced from its action on $\bV$. The fundamental open subgroups for this topology have the form
\begin{displaymath}
\begin{bmatrix} 1_n & \ast \\ 0 & \ast \end{bmatrix}.
\end{displaymath}
We define the \defn{Levi topology} on $G$ to be the one induced from its action on $\bV \oplus \bV_*$. The fundamental open subgroups for this topology have the form
\begin{displaymath}
\begin{bmatrix} 1_n & 0 \\ 0 & \ast \end{bmatrix}.
\end{displaymath}
We thus see that $G$ carries two distinct pro-oligomorphic topologies. These examples are discussed further in \S \ref{s:finlin}.
\end{example}

The groups of primary interest to us in this paper occur as oligomorphic groups. However, the theory we develop only depends on the topology, and not the defining permutation action. For this reason, we work with pro-oligomorphic groups throughout. We note that there are pro-oligomorphic groups which are not oligomorphic, such as infinite profinite groups.

\subsection{Smooth actions} \label{ss:smooth}

Fix a pro-oligomorphic group $G$. Let $X$ be a set on which $G$ acts. We say that an element $x$ of $X$ is \defn{smooth} if the stabilizer of $x$ is an open subgroup of $G$. We say that $X$ itself is \defn{smooth} if every element is. Nearly every action in this paper will be smooth, and so we adopt the following terminological convention:

\begin{convention} \label{conv:Gset}
The phrase ``$G$-set'' means ``set equipped with a \emph{smooth} action of $G$.''
\end{convention}

The following proposition gives some important properties of $G$-sets.

\begin{proposition} \label{prop:smooth}
Let $X$ and $Y$ be $G$-sets and let $U$ be an open subgroup of $G$.
\begin{enumerate}
\item The $G$-action on $X \times Y$ is smooth; thus $X \times Y$ is a $G$-set.
\item If $X$ and $Y$ are finitary so is $X \times Y$.
\item $X$ is finitary as a $G$-set if and only if it is finitary as a $U$-set.
\item If $X$ and $Y$ are finitary then there are only finitely many $G$-maps $X \to Y$.
\end{enumerate}
\end{proposition}

\begin{proof}
(a) If $(x,y)$ is a point of $X \times Y$ then its stabilizer is $G_x \cap G_y$. Since $X$ and $Y$ are smooth, $G_x$ and $G_y$ are open, and so $G_x \cap G_y$ is open as well. Thus $X \times Y$ is smooth.

(b) It suffices to treat the case where $X=G/U$ and $Y=G/V$ for open subgroups $U$ and $V$. Then $G \backslash (X \times Y)=U \backslash G/V$ is finite by Roelcke precompactness, and so the result follows.

(c) If $X$ is finitary as a $U$-set then it is clearly so as a $G$-set. For the converse, note that $U \backslash X \cong G \backslash (G/U \times X)$, and so $X$ is $U$-finitary since $G/U \times X$ is $G$-finitary by~(b).

(d) A $G$-map $X \to Y$ is determined by its graph, which is a $G$-stable subset of $X \times Y$. Since $X \times Y$ is finitary by~(b), it has only finitely many $G$-stable subsets.
\end{proof}

We let $\cS(G)$ denote the category of finitary smooth $G$-sets. The above proposition shows that it is closed under finite products; since it is also closed under passing to $G$-subsets, it is also closed under fiber products.  In \cite{bcat} we fully describe the additional properties the categories $\cS(G)$ enjoy, leading the notion of a pre-Galois category.

\subsection{Induction} \label{ss:induced-Gset}

Let $U$ be an open subgroup of a pro-oligomorphic group $G$ and let $X$ be a $U$-set. We define the \defn{induction} of $U$, denoted $I_U^G(X)$, to be the set $G \times^U X$; this is the quotient of the product $G \times X$ by the relations $(gu, x)=(g,ux)$ for $g \in G$, $u \in U$, and $x \in X$. The group $G$ acts on $I_U^G(X)$ through its action on the first factor: $g \cdot (g',x)=(gg',x)$. One readily verifies that this action is smooth, and so $I_U^G(X)$ is a $G$-set. We have the following additional properties, whose verification we leave to the reader:
\begin{itemize}
\item There is a well-defined map of $G$-sets $I_U^G(X) \to G/U$ given by $(g,x) \mapsto g$, the fiber of which over $1 \in G/U$ is naturally identified with $X$.
\item We have a natural bijection $G \backslash I_U^G(X) \cong U \backslash X$.
\item If $X$ is a finitary $U$-set then $I_U^G(X)$ is a finitary $G$-set.
\item Induction is functorial: if $f \colon X \to Y$ is a map of $U$-sets then there is an induced map of $G$-sets $I_U^G(f) \colon I_U^G(X) \to I_U^G(Y)$.
\end{itemize}

\subsection{Infinitesimal actions} \label{ss:Ghat}

Fix a pro-oligomorphic group $G$. It will often be convenient to work with $U$-sets, for $U$ an open subgroup of $G$, without having to specify $U$ explicitly. We now introduce a device that allows us to do so.

A \defn{$\hat{G}$-action} on a set $X$ is an action of some open subgroup $U$ of $G$. The subgroup $U$ is called a \defn{group of definition} for the action. Shrinking the group of definition does not change the $\hat{G}$-action. Formally, a $\hat{G}$-action is an element of the set $\varinjlim_{U \subset G} A(U, X)$, where the direct limit is taken over the open subgroups $U$ of $G$, and $A(U, X)$ denotes the set of all actions of $U$ on $X$. It is important to note that $\hat{G}$ is simply a piece of notation and not a mathematical object. However, at an intuitive level, we think of $\hat{G}$ as an infinitesimal neighborhood of the identity in $G$.

Consider a set $X$ equipped with a $\hat{G}$-action. We say that the action is \defn{smooth} if the action of $U$ is smooth for some group of definition $U$; one easily sees that this is independent of the choice of $U$. We only consider smooth $\hat{G}$-actions, and so, as in Convention~\ref{conv:Gset}, we define a \defn{$\hat{G}$-set} to be a set equipped with a smooth $\hat{G}$-action.

Let $X$ be a $\hat{G}$-set. We say that $X$ is \defn{finitary} if it is finitary as a $U$-set for some group of definition $U$; this is well-defined (i.e., independent of the choice of $U$) by Proposition~\ref{prop:smooth}(c). A subset $Y$ of $X$ is a $\hat{G}$-subset if there is a group of definition $U$ for $X$ that leaves $Y$ stable. Every finite subset $Y$ of $X$ is a $\hat{G}$-subset: indeed, $Y$ is stable under the open subgroup $G(Y)$ that fixes each element of $Y$. It is clear that the collection of $\hat{G}$-subsets of $X$ is closed under finite unions and intersections. Infinite unions or intersections can be problematic, however, since the groups of definition could shrink too much. For $g \in G$, we define the conjugate $X^g$ to be the $\hat{G}$-set with the same underlying set, and with action $\bullet$ defined by $h \bullet x = (g^{-1}hg) \cdot x$.

Let $X$ and $Y$ be smooth $\hat{G}$-sets. We say that a function $f \colon X \to Y$ is \defn{$\hat{G}$-equivariant}, or a \defn{map of $\hat{G}$-sets}, or \defn{smooth}, if $f$ is $U$-equivariant for some open subgroup $U$ that is a group of definition for both $X$ and $Y$. Suppose that $f$ is $\hat{G}$-equivariant. If $Y'$ is a $\hat{G}$-subset of $Y$ then $f^{-1}(Y')$ is a $\hat{G}$-subset of $X$. In particular, we see that the fibers of $f$ are $\hat{G}$-subsets of $X$; this is one of the very pleasant features of $\hat{G}$-sets.

\begin{example}
Consider the infinite symmetric group $\fS$ acting on $\Omega=\{1,2,\ldots\}$. Then the $\hat{\fS}$-subsets of $\Omega$ are exactly the finite or co-finite subsets of $\Omega$. More generally, a $\hat{\fS}$-subset of $\Omega^n$ is one that can be defined by a first-order formula using $=$ and constants. See Figure~\ref{fig:Shat-sets} for some examples.
\end{example}

\begin{example}
For $G=\Aut(\bR,<)$, the $\hat{G}$-subsets of $\bR$ are exactly the finite unions of points and open intervals. More generally, a $\hat{G}$-subset of $\bR^n$ is one that can be defined by a first-order formula using $<$, $=$, and constants.
\end{example}

\subsection{The relative setting} \label{ss:relative}

Let $G$ be a pro-oligomorphic group. A \defn{stabilizer class} in $G$ is a collection $\sE$ of open subgroups satisfying the following conditions:
\begin{enumerate}
\item $\sE$ contains $G$.
\item $\sE$ is stable under conjugation.
\item $\sE$ is stable under finite intersections.
\item $\sE$ is a neighborhood basis for the identity.
\end{enumerate}
If $G$ is oligomorphic then the collection of all stabilizer groups on $\Omega^n$, as $n$ varies, forms a stabilizer class; we call this the stabilizer class \defn{generated} by $\Omega$.

Fix a stabilizer class $\sE$. We say that a $G$-set is \defn{$\sE$-smooth} if the stabilizer of every element of $X$ belongs to $\sE$. More generally, if $U$ is an open subgroup of $G$ (not necessarily in $\sE$), we say that a $U$-set is \defn{$\sE$-smooth} if the stabilizer of every element has the form $U \cap V$ for some $V \in \sE$. This leads to a notion of $\sE$-smoothness for $\hat{G}$-sets. The class of $\sE$-smooth sets (for $G$, $U$, or $\hat{G}$) is stable under finite products and co-product, fiber products, and passing to stable subsets. Thus the category of $\sE$-smooth $G$-sets formally resembles the category of $G$-sets in many ways.

Most of our work with pro-oligomorphic groups extends to the setting of pro-oligomorphic groups equipped with a stabilizer class; we refer to this as the ``relative setting.'' To keep exposition simple, we do our main work in the absolute setting, and then sometimes comment on generalizations to the relative setting.

\begin{example} \label{ex:stab-class}
Consider the infinite symmetric group $\fS$ acting on $\Omega=\{1,2,\ldots\}$. Let $\sE$ be the stabilizer class generated by $\Omega$. Explicitly, $\sE$ consists of all conjugates of the groups $\fS(n)$ introduced in Example~\ref{ex:sym-top}. A transitive $\fS$-set is $\sE$-smooth if and only if it is isomorphic to $\Omega^{[n]}$ for some $n$. In particular, the set $\Omega^{(2)}$ of 2-element subsets of $\Omega$ is \emph{not} $\sE$-smooth.
\end{example}

\begin{example} \label{ex:stab-young}
Recall that a \defn{Young subgroup} of $\fS_n$ is a subgroup of the form $\fS_{m_1} \times \cdots \times \fS_{m_r}$, where $m_1+\cdots+m_r=n$. We define a \defn{Young subgroup} of $\fS$ to a subgroup conjugate to one of the form $H \times \fS(n)$, where $H$ is a Young subgroup of $\fS_n$. The collection of all Young subgroups forms a stabilizer class for $\fS$.
\end{example}

\section{Measures} \label{s:meas}

\subsection{Overview}

Fix a pro-oligomorphic group $G$ and a commutative ring $k$ for the duration of \S \ref{s:meas}. In this section, we introduce a notion of measure for $G$. This definition is one of the key ideas in this paper. As we explain in \S \ref{s:int}, a measure allows one to develop a theory of integration, which in turn allows us to define tensor categories; thus everything rests on measures. We also introduce the ring $\Theta(G)$ which carries the universal measure for $G$. The results in this section give various ways of describing $\Theta(G)$, each of which will be used subsequently.

\subsection{Measures} \label{ss:measure}

The following is the key definition:

\begin{definition} \label{defn:measure}
A \defn{measure} for $G$ with values in $k$ is a rule $\mu$ that assigns to every finitary $\hat{G}$-set $X$ a quantity $\mu(X)$ in $k$ such that the following conditions hold (in which $X$ and $Y$ are finitary $\hat{G}$-sets):
\begin{enumerate}
\item If $X$ and $Y$ are isomorphic then $\mu(X)=\mu(Y)$.
\item We have $\mu(\bone)=1$, where $\bone$ is the one-point set.
\item We have $\mu(X \amalg Y)=\mu(X)+\mu(Y)$.
\item For $g \in G$, we have $\mu(X)=\mu(X^g)$, where $X^g$ is the conjugate of $X$ by $g$.
\item If $X$ and $Y$ are transitive $U$-sets, for some open subgroup $U$, and $X \to Y$ is a map of $U$-sets with fiber $F$ (over some point), then $\mu(X)=\mu(F) \cdot \mu(Y)$. \qedhere
\end{enumerate} 
\end{definition}

\begin{remark}
We make a number of remarks concerning this definition:
\begin{enumerate}
\item We elaborate on condition \dref{defn:measure}{e}. Let $f \colon X \to Y$ be the map, and suppose that $F=f^{-1}(y)$ for $y \in Y$. First note that $F$ is a $\hat{G}$-subset of $X$: indeed, it is stable under the stabilizer of $y$ in $U$. Second, suppose that $F'=f^{-1}(y')$ for some other point $y' \in Y$. Since $Y$ is transitive, we have $y=gy'$ for some $g \in U$, and so $F=gF'$. We thus see that $F'$ is isomorphic to the conjugate set $F^g$, and so $\mu(F)=\mu(F')$ by \dref{defn:measure}{a} and \dref{defn:measure}{d}. Thus the condition in \dref{defn:measure}{e} is independent of the choice of fiber. The slogan for \dref{defn:measure}{e} is ``$\mu$ is multiplicative in fibrations.''
\item If $X$ and $Y$ are finitary $\hat{G}$-sets then $\mu(X \times Y)=\mu(X) \cdot \mu(Y)$. Indeed, by \dref{defn:measure}{c}, it suffices to treat the case where $X$ and $Y$ are transitive $U$-sets, for some open subgroup $U$, and then the identity follows from \dref{defn:measure}{e} by considering the projection $X \times Y \to Y$.
\item Let $U$ be an open subgroup of $G$. Then a $\hat{U}$-set is the same thing as a $\hat{G}$-set. One easily sees that a measure for $G$ is also a measure for $U$; in other words, we can restrict measures from $G$ to $U$. A measure for $U$ may not be a measure for $G$, however, as conditions (d) and (e) are potentially stronger for $G$ than for $U$.
\item There is a relative notion of measure: given a stabilizer class $\sE$, a measure for $(G, \sE)$ assigns a value $\mu(X)$ to each finitary $\sE$-smooth $\hat{G}$-set $X$ such that (a)--(e) hold when they make sense.
\item Let $\mu$ be a $k$-valued measure for $G$, let $f \colon k \to k'$ be a ring homomorphism, and define $\mu'$ by $\mu'(X)=f(\mu(X))$. Then $\mu'$ is easily seen to be a measure. We refer to this construction as \defn{extension of scalars} for measures. \qedhere
\end{enumerate}
\end{remark}

\begin{example} \label{ex:finite-measure}
Suppose $G$ is finite. Then any finitary $\hat{G}$-set is finite, and we obtain a measure $\mu$ by defining $\mu(X)=\# X$. In fact, this is the only measure: indeed, if $X$ is a finitary $\hat{G}$-set then it is a finite disjoint union of singletons, and so its measure must by $\# X$ by \dref{defn:measure}{b} and \dref{defn:measure}{c}.
\end{example}

\begin{example} \label{ex:sym-measure}
Let $\fS$ be the infinite symmetric group, i.e., the group of all permutations of $\Omega=\{1,2,\ldots\}$. Given a complex number $t$, we define a $\bC$-valued measure $\mu_t$ on $\fS$ as follows. Let $X$ be a finitary $\hat{\fS}$-set. There is a polynomial $p_X \in \bQ[x]$ such that $p_X(n)$ is the number of fixed points of $\fS(n)$ on $X$ for all sufficiently large integers $n$; recall that $\fS(n)$ denotes the subgroup of $\fS$ fixing each of $1, \ldots, n$. We define $\mu_t(X)=p_X(t)$. Some examples:
\begin{align*}
\mu_t(\Omega) &=t & 
\mu_t(\Omega^{[n]}) &= t(t-1)\cdots(t-n+1) \\
\mu_t(\Omega \setminus \{1,\ldots,n\}) &= t-n &
\mu_t(\Omega^{(n)}) &= \binom{t}{n}
\end{align*}
In fact, the $\mu_t$ account for all of the complex-valued measures on $\fS$. The above assertions are proven in \S \ref{s:sym}. From our point of view, the measure $\mu_t$ is responsible for the existence of Deligne's interpolation category $\uRep(\fS_t)$.
\end{example}

\begin{example} \label{ex:AutR-measure}
Let $G=\Aut(\bR,<)$. Any finitary $\hat{G}$-set admits a smooth manifold structure. We obtain a $\bZ$-valued measure $\mu$ for $G$ by taking $\mu(X)$ to be the compactly supported Euler characteristic of $X$. For example:
\begin{displaymath}
\mu(\bR^n) = (-1)^n, \qquad
\mu(\bR^{(n)}) = (-1)^n, \qquad
\mu(\bR^{[n]}) = (-1)^n n!.
\end{displaymath}
These assertions are proven in \S \ref{s:order}.
\end{example}

\subsection{The ring \texorpdfstring{$\Theta(G)$}{\textTheta(G)}}

Understanding the measures for $G$ is not a simple matter. To aid us in this task, we introduce the following ring:

\begin{definition} \label{defn:Theta}
We define a commutative ring $\Theta(G)$ as follows. For each finitary\ $\hat{G}$-set $X$, there is an element $[X]$ of $\Theta(G)$. These elements are subject to the following relations (in which $X$ and $Y$ denote finitary $\hat{G}$-sets):
\begin{enumerate}
\item If $X$ and $Y$ are isomorphic then $[X]=[Y]$.
\item If $X$ is a singleton set then $[X]=1$.
\item We have $[X \amalg Y]=[X]+[Y]$.
\item For $g \in G$ we have $[X]=[X^g]$, where $X^g$ is the conjugate of $X$ by $g$.
\item If $X$ and $Y$ are transitive $U$-sets, for some open subgroup $U$, and $X \to Y$ is a map of $U$-sets with fiber $F$ (over some point of $Y$), then $[X]=[F] \cdot [Y]$.
\end{enumerate}
More formally, $\Theta(G)$ can be defined by taking the polynomial ring in variables indexed by isomorphism classes of finitary $\hat{G}$-sets and forming the quotient by the ideal generated by relations corresponding to (b)--(e).
\end{definition}

If $\mu$ is a $k$-valued measure for $G$ then we obtain a homomorphism $\Theta(G) \to k$ by $[X] \mapsto \mu(X)$. Conversely, any homomorphism $\Theta(G) \to k$ yields a measure for $G$. We thus see that there is a universal measure $\mu_{\rm univ}$ for $G$ with values in $\Theta(G)$ given by $\mu_{\rm univ}(X)=[X]$. In this way, understanding measures for $G$ is equivalent to understanding the ring $\Theta(G)$.

\begin{remark}
If $X$ and $Y$ are finitary $\hat{G}$-sets then $[X \times Y]=[X] [Y]$ holds in $\Theta(G)$. This follows since $X \mapsto [X]$ is a measure.
\end{remark}

\begin{remark}
There is a relative version of this construction: given a stabilizer class $\sE$ we let $\Theta(G; \sE)$ be the ring generated by classes $[X]$, with $X$ a $\sE$-smooth finitary $\hat{G}$-set, subject to the same relations.
\end{remark}

\subsection{Basic properties of \texorpdfstring{$\Theta(G)$}{\textTheta(G)}} \label{ss:Theta-prop}

We now discuss some basic properties of $\Theta$.

(a) \textit{Functorial behavior.} If $U$ is an open subgroup of $G$ then there is a surjective ring homomorphism $\Theta(U) \to \Theta(G)$ given by $[X] \mapsto [X]$, for $X$ a finitary $\hat{U}$-set. If $f \colon G \to H$ is a continuous surjection of pro-oligomorphic groups then there is a ring homomorphism $\Theta(H) \to \Theta(G)$ given by $[X] \mapsto [f^*(X)]$, where $X$ is a $\hat{H}$-set and $f^*(X)$ denotes the $\hat{G}$-set obtained by letting $\hat{G}$ act through $f$. If $\sE \subset \sF$ are stabilizer classes, then there is a ring homomorphism $\Theta(G; \sE) \to \Theta(G; \sF)$ given by $[X] \mapsto [X]$. These statements all follow easily from the definitions.

(b) \textit{Finite sets.} If $X$ is a $\hat{G}$-set with $n<\infty$ elements then $X$ is isomorphic, as a $\hat{G}$-set, to a disjoint union of $n$ one-point sets, and so $[X]=n$ in $\Theta(G)$. As a consequence, we see that $\Theta(G)=\bZ$ if $G$ is a finite (or profinite) group. Indeed, the above discussion shows that the canonical map $\bZ \to \Theta(G)$ is surjective, while we also have a homomorphism $\Theta(G) \to \bZ$ given by $[X] \mapsto \# X$ (see Example~\ref{ex:finite-measure}).

(c) \textit{Products.} If $G$ and $H$ are pro-oligomorphic groups then there is a natural ring homomorphism
\begin{displaymath}
\Theta(G) \otimes \Theta(H) \to \Theta(G \times H), \qquad [X] \otimes [Y] \mapsto [X \times Y].
\end{displaymath}
Indeed, this is just the product of the two pull-back maps from~(a). This homomorphism may not be an isomorphism, though we know no counterexample. Let $\sE$ be the stabilizer class for $G \times H$ consisting of all subgroups of the form $U \times V$, where $U$ is an open subgroup of $G$ and $V$ is an open subgroup of $H$. Then the above homomorphism factors as
\begin{displaymath}
\Theta(G) \otimes \Theta(H) \to \Theta(G \times H; \sE) \to \Theta(G \times H),
\end{displaymath}
and one can show that the first homomorphism above is an isomorphism. This provides a simple illustration of how stabilizer classes can be useful.

(d) \textit{Large stabilizer classes.} Let $\sE$ be a stabilizer class that is ``large'' in the following sense: given any open subgroup $U$ of $G$, there exists $V \in \sE$ such that $U$ contains $V$ with finite index. Then the natural map
\begin{displaymath}
\phi \colon \Theta(G; \sE) \otimes \bQ \to \Theta(G) \otimes \bQ
\end{displaymath}
coming from (a) is an isomorphism. To show this, we construct an inverse $\psi$. Let $V \subset U$ be open subgroups of $G$ and let $W \in \sE$ be such that $V$ contains $W$ with finite index. Then $\psi$ is defined by
\begin{displaymath}
\psi([U/V]) = [V:W]^{-1} \cdot [U/W]
\end{displaymath}
One must show that this is independent of the choice of $V$ and respects the defining relations of $\Theta(G)$. We leave the details to the reader.

If $G$ is oligomorphic and satisfies the so-called ``strong small index property'' \cite[\S 4.8]{Cameron} then the stabilizer class $\sE$ generated by $\Omega$ is large. This property is known for $\fS$ and linear groups over finite fields \cite[(4.22)]{Cameron}, the automorphism group of the Rado graph \cite{Cameron5}, and several other examples \cite[Corollary~3]{PaoliniShelah}.

(e) \textit{Normal subgroups.} Let $U$ be an open normal subgroup of $G$. Then the finite group $G/U$ acts on $\Theta(U)$ by $g^{-1} \cdot [X]=[X^g]$, and the natural map $\Theta(U) \to \Theta(G)$ induces an isomorphism $\Theta(U)_{G/U} \to \Theta(G)$, where $(-)_{G/U}$ denotes the co-invariant ring.

(f) \textit{Cardinality.} It is not difficult to see that if $G$ is first-countable then it has countably many open subgroups; similarly, if $G$ is oligomorphic then it has countably many conjugacy classes of open subgroups. In either case, we see that $\Theta(G)$ has countable cardinality.

\subsection{The Burnside ring} \label{ss:burnside}

The ring $\Theta(G)$ is extremely subtle. It is closely related to a much more straightforward ring, the Burnside ring of $\hat{G}$, which we now introduce. It is often convenient to use the Burnside ring as a starting point when studying $\Theta(G)$.

The \defn{Burnside ring} of $G$, denoted $\Omega(G)$, is the free abelian group on the set of isomorphism classes of transitive $G$-sets. Given a transitive $G$-set $X$, we write $\lbb X \rbb$ for the corresponding basis vector of $\Omega(G)$. More generally, for a finitary $G$-set $X$ we define $\lbb X \rbb=\sum_{i=1}^n \lbb X_i \rbb$ where $X_1, \ldots, X_n$ are the $G$-orbits on $X$. As the name suggests, the Burnside ring admits the structure of a commutative ring by $\lbb X \rbb \cdot \lbb Y \rbb = \lbb X \times Y \rbb$. The unit element is the class $\lbb \bone \rbb$ of the one-point set $\bone$.

Suppose $f \colon H \to G$ is a continuous homomorphism of pro-oligomorphic groups such that $H$ acts with finitely many orbits on any finitary $G$-set (e.g., $f$ could be surjective). Then there is an induced ring homomorphism $f^* \colon \Omega(G) \to \Omega(H)$ defined by $f^*(\lbb X \rbb)=\lbb f^*(X) \rbb$, where $f^*(X)$ denotes the set $X$ equipped with its action of $H$.

In particular, if $V \subset U$ are opens subgroup of $G$ then there is a restriction map $\Omega(U) \to \Omega(V)$. We thus have a directed system of rings $\{\Omega(U)\}$ indexed by the open subgroups of $G$. We define $\Omega(\hat{G})$ to be the direct limit, and refer to it as the \defn{Burnside ring} of $\hat{G}$; this can also be thought of as the \defn{infinitesimal Burnside ring} of $G$. If $X$ is a finitary $\hat{G}$-set, we have an associated class $\lbb X \rbb$ in $\Omega(\hat{G})$, by first taking the class of $X$ in $\Omega(U)$ for some group of definition $U$, and then mapping this to $\Omega(\hat{G})$.

The next definition and proposition show how $\Theta(G)$ can be described using $\Omega(\hat{G})$.

\begin{definition}
We introduce the following ideals of the Burnside ring $\Omega(\hat{G})$:
\begin{itemize}
\item We let $\fa_1(G)$ be the ideal generated by the elements $\lbb X \rbb-\lbb X^g \rbb$, where $X$ is a finitary $\hat{G}$-set and $g \in G$.
\item We let $\fa_2(G)$ be the ideal generated by the elements $\lbb X \rbb-\lbb F \rbb \lbb Y \rbb$, where $X \to Y$ is a surjection of finitary $U$-sets for some open subgroup $U$, $Y$ is transitive, and $F$ is the fiber.
\item We let $\fa(G)=\fa_1(G)+\fa_2(G)$. \qedhere
\end{itemize}
\end{definition}

\begin{proposition} \label{prop:Burnside-Theta}
We have a ring isomorphism $\Omega(\hat{G})/\fa(G) \to \Theta(G)$ via $\lbb X \rbb \mapsto [X]$.
\end{proposition}

\begin{proof}
For a finitary $\hat{G}$-set $X$, let $[X]'$ be the image of $\lbb X \rbb$ in $R=\Omega(\hat{G})/\fa(G)$. We have a natural ring homomorphism $\Omega(\hat{G}) \to \Theta(G)$ defined by $\lbb X \rbb \mapsto [X]$. It is clear that this map kills $\fa(G)$, and thus induces a ring homomorphism $R \to \Theta(G)$ satisfying $[X]' \mapsto [X]$. On the other hand, the association $X \mapsto [X']$ clearly respects the defining relations of $\Theta(G)$: indeed, relations (a), (b), and (c) already hold in $\Omega(\hat{G})$, while (d) and (e) hold by fiat since we have killed $\fa(G)$. We thus have a ring homomorphism $\Theta(G) \to R$ given by $[X] \mapsto [X]'$. The result follows.
\end{proof}

\begin{remark}
In the oligomorphic case, the Burnside ring $\Omega(G)$ is closely related to the orbit algebra of $G$; see \cite[\S 3.8]{Cameron} for more on this, and \cite{FalqueThiery} for some recent results.
\end{remark}

\subsection{Another description of \texorpdfstring{$\Theta(G)$}{\textTheta(G)}} \label{ss:other-theta}

We now give a different description of the ring $\Theta(G)$. This description has some advantages described in Remark~\ref{rmk:Theta-prime} below, and is also used in a few key places below, specifically the proofs of Proposition~\ref{prop:Theta-star}, Theorems~\ref{thm:compare}, and Theorem~\ref{thm:linear}. To this end, we introduce another definition:

\begin{definition} \label{defn:Theta-prime}
We define a commutative ring $\Theta'(G)$ as follows. For each map $f \colon X \to Y$ of $G$-sets with $X$ finitary and $Y$ transitive, there is an element $\langle f \rangle$ of $\Theta'(G)$. These elements are subject to the following relations:
\begin{enumerate}
\item If $f$ is an isomorphism of transitive $G$-sets then $\langle f \rangle=1$.
\item Suppose $X=X_1 \sqcup X_2$ and let $f_i$ be the restriction of $f$ to $X_i$. Then $\langle f \rangle = \langle f_1 \rangle + \langle f_2 \rangle$.
\item For maps $f \colon X \to Y$ and $g \colon Y \to Z$ of transitive $G$-sets, we have $\langle gf \rangle=\langle g \rangle \langle f \rangle$.
\item Let $f \colon X \to Y$ and $g \colon Y' \to Y$ be maps of $G$-sets, with $Y$ and $Y'$ transitive, and let $f' \colon X' \to Y'$ be the base change of $f$. Then $\langle f \rangle = \langle f' \rangle$. \qedhere
\end{enumerate}
\end{definition}

Suppose that $f \colon X \to Y$ is a map of $G$-sets with $X$ finitary and $Y$ transitive. As we have previously seen, if $y,y' \in Y$ then the fibers $f^{-1}(y)$ and $f^{-1}(y')$ are $G$-conjugate, and thus have the same class in $\Theta(G)$. We define $[f]=[f^{-1}(y)]$ for any choice of $y$. The following proposition is our main result on $\Theta'(G)$.

\begin{proposition} \label{prop:Theta-prime}
We have an isomorphism $\Theta'(G) \to \Theta(G)$ given by $\langle f \rangle \mapsto [f]$.
\end{proposition}

We break the proof up into a few lemmas.

\begin{lemma}
There exists a well-defined ring homomorphism $\phi \colon \Theta'(G) \to \Theta(G)$ satisfying $\phi(\langle f \rangle)=[f]$.
\end{lemma}

\begin{proof}
Let $f \colon X \to Y$ be a map of $G$-sets with $X$ finitary and $Y$ transitive. Put $\tilde{\phi}(f)=[f]$. We show that $\tilde{\phi}$ respects the defining relations of $\Theta'(G)$. This is clear for (a) and (b).

We now verify (c). Let $f \colon X \to Y$ and $g \colon Y \to Z$ be given. We may as well suppose $X=G/W$, $Y=G/V$, and $Z=G/U$ with $W \subset V \subset U$. The fiber of $f$ is $V/W$, the fiber of $g$ is $U/V$, and the fiber of $gf$ is $U/W$. We thus have
\begin{displaymath}
\tilde{\phi}(gf)=[U/W], \qquad \tilde{\phi}(g) \cdot \tilde{\phi}(f) = [U/V] \cdot [V/W].
\end{displaymath}
These classes are equal in $\Theta(G)$. This follows from applying axiom \dref{defn:measure}{e} to the map $U/W \to U/V$, which has fiber $V/W$.

We finally handle (d). Let $f \colon X \to Y$ and $g \colon Y' \to Y$ be maps of $G$-sets, with $Y$ and $Y'$ transitive, and let $f' \colon X' \to Y'$ be the base change of $f$. Let $F$ be the fiber of $f$ above some point; this is also the fiber of $f'$ above some point. We thus see
\begin{displaymath}
\tilde{\phi}(f)=[F]=\tilde{\phi}(f'),
\end{displaymath}
which verifies the condition. It follows that $\tilde{\phi}$ induces the requisite $\phi$.
\end{proof}

\begin{lemma}
There exists a well-defined ring homomorphism $\psi \colon \Theta(G) \to \Theta'(G)$ satisfying the following property. Suppose $X$ is a finitary $\hat{G}$-set $X$ with group of definition $U$, and let $f \colon I_U^G(X) \to G/U$ be the natural map (see \S \ref{ss:induced-Gset}). Then $\psi([X])=\langle f \rangle$.
\end{lemma}

\begin{proof}
For $X$, $U$, and $f$ as in the statement of the lemma, we put $\tilde{\psi}(X)=\langle f \rangle$. We first claim that this is well-defined, i.e., independent of the choice of $U$. Thus suppose $V$ is a second group of definition; it suffices to treat the case $V \subset U$, so we assume this. Let $f' \colon I_V^G(X) \to G/V$ be the natural map. One then sees that $f'$ is the base change of $f$ along the natural map $G/V \to G/U$. Thus $\langle f' \rangle=\langle f \rangle$ in $\Theta'(G)$, and so $\tilde{\psi}$ is well-defined.

We now verify that $\tilde{\psi}$ respects the defining relations of $\Theta(G)$. In what follows, $X$ and $Y$ denote finitary $\hat{G}$-sets.
\begin{enumerate}
\item It is clear that if $X$ and $Y$ are isomorphic then $\tilde{\psi}(X)=\tilde{\psi}(Y)$.
\item Suppose $X$ is a singleton set, and let $U$ be a group of definition. Then $f \colon I^G_U(X) \to G/U$ is isomorphic to the identity map of $G/U$, and so $\tilde{\psi}(X)=\langle f \rangle=1$.
\item Let $U$ be a common group of definition for $X$ and $Y$. Let $f \colon I_U^G(X) \to G/U$ and $g \colon I_U^G(Y) \to G/U$ be the natural maps. The natural map $I_U^G(X \amalg Y) \to G/U$ is isomorphic to $f \amalg g$. We thus see that $\tilde{\psi}(X \amalg Y)=\langle f \amalg g \rangle=\langle f \rangle+\langle g \rangle=\tilde{\psi}(X)+\tilde{\psi}(Y)$.
\item Let $g \in G$. The maps $I_U^G(X) \to G/U$ and $I_U^G(X^g) \to G/U$ are easily seen to be isomorphic. Thus $\tilde{\psi}(X^g)=\tilde{\psi}(X)$.
\item Finally suppose that $X$ and $Y$ are transitive $U$-sets, and let $f \colon X \to Y$ be a $U$-equivariant map with fiber $F$. We then get a commutative triangle
\begin{displaymath}
\xymatrix{
I_U^G(X) \ar[rr]^{f'} \ar[rd]_h && I_U^G(Y) \ar[ld]^g \\ & G/U }
\end{displaymath}
where $f'=I_U^G(f)$. The $G$-sets $I_U^G(X)$ and $I_U^G(Y)$ are transitive, and so $\langle f' \rangle$ is an element of $\Theta'(G)$, and we have $\langle h \rangle=\langle g \rangle \langle f' \rangle$. By definition, we have $\langle h \rangle=\tilde{\psi}(X)$ and $\langle h \rangle=\tilde{\psi}(Y)$. We claim that $\langle f' \rangle=\tilde{\psi}(F)$. For this, suppose $F=f^{-1}(y)$, and let $V \subset U$ stabilize $y$. We then get a commutative diagram
\begin{displaymath}
\xymatrix{
I_V^G(F) \ar[r] \ar[d]_{f''} & I_V^G(Y) \ar[r] \ar[d] & I_U^G(Y) \ar[d]^{f'} \\
I_V^G(\bone) \ar[r] & I_V^G(X) \ar[r] & I_U^G(X) }
\end{displaymath}
Both squares are cartesian. Indeed, the left square comes from applying $I_V^G$ to a cartesian square, and $I_V^G$ commutes with fiber products. In general, $I_V^G$ is the base change of $I_U^G$ along $G/V \to G/U$, and so the right square is cartesian. We thus see that $f''$ is a base change of $f'$, and so they have the same class in $\Theta'(G)$. Since $\tilde{\psi}(F)=\langle f'' \rangle$, the claim follows. (Note: $I_V^G(\bone)=G/V$.) We have thus shown that $\tilde{\psi}(X)=\tilde{\psi}(F) \tilde{\psi}(Y)$.
\end{enumerate}
We thus see that $\tilde{\psi}$ satisfies the defining relations of $\Theta(G)$, and so $\tilde{\psi}$ induces $\psi$.
\end{proof}

\begin{lemma}
The maps $\phi$ and $\psi$ are mutually inverse.
\end{lemma}

\begin{proof}
Let $f \colon X \to Y$ be a map of $G$-sets, with $Y$ transitive, and let $F=f^{-1}(y)$. Let $U$ be the stabilizer of $y$, so that $F$ is a $U$-set. We have $\psi(\phi(\langle f \rangle))=\psi([F])=[g]$, where $g \colon I_U^G(F) \to G/U$ is the natural map. One easily sees that $g$ is isomorphic to $f$, and so $\langle f \rangle=\langle g \rangle$ in $\Theta'(G)$. Thus $\psi \circ \phi = \id$.

Now suppose that $X$ is a finitary $\hat{G}$-set, and let $U$ be the group of definition. Let $f \colon I_U^G(X) \to G/U$ be the natural map. The fiber of $f$ over $1 \in G/U$ is $X$. We thus see that $\phi(\psi([X)])=\phi(\langle f \rangle)=[X]$, and so $\phi \circ \psi=\id$. This completes the proof.
\end{proof}

\begin{remark} \label{rmk:Theta-prime}
A few remarks concerning $\Theta'$ and the above result:
\begin{enumerate}
\item It is clear that $\Theta'(G)$ only depends on the category $\cS(G)$ of finitary $G$-sets up to equivalence. Thanks to the proposition, the same is true for $\Theta(G)$.
\item There is a natural non-commutative version of $\Theta'(G)$: take the non-commutative polynomial ring in symbols $\langle f \rangle$ and form the quotient by the 2-sided ideal generated by the relations in Definition~\ref{defn:Theta-prime}. The point is that multiplication only appears in relation (c), and there is a preferred order for the multiplication there (the one we have written). Deligne pointed out to us that this ring is in fact automatically commutative \cite{DeligneLetter1}.
\item The definition of $\Theta'(G)$ is nearly identical to that of the ring $K(\cA)$ introduced in \cite[\S 8]{Knop2}. However, we apply the definition in a different context than Knop, and this leads to different results. \qedhere
\end{enumerate}
\end{remark}

\subsection{Regular measures} \label{ss:regular}

We now isolate two classes of measures that will play an important role, especially in Part~\ref{part:genrep}:

\begin{definition} \label{defn:regular}
Let $\mu$ be a $k$-valued measure for $G$. We say that $\mu$ is \defn{regular} if $\mu(X)$ is a unit of $k$ for every transitive $G$-set $X$. We say that $\mu$ is \defn{quasi-regular} if there is an open subgroup $U$ of $G$ such that $\mu \vert_U$ is regular.
\end{definition}

\begin{example} \label{ex:ordinary}
Suppose that $G$ is a finite group and $k$ is a field, and let $\mu$ be the unique measure for $G$ with values in $k$ (see Example~\ref{ex:finite-measure}). Then $\mu$ is regular if and only if $\# G$ is non-zero in $k$. Thus, in the terminology of representation theory, ``regular'' coincides with the ordinary case, and ``non-regular'' with the modular case. For a general pro-oligomorphic group, it is reasonable to also consider regular as corresponding to the ``ordinary case.'' For instance, we prove an analog of Maschke's theorem in this setting (Theorem~\ref{thm:regss}). The measure $\mu$ is always quasi-regular: its restriction to the trivial subgroup is regular.
\end{example}

It is clear that regular implies quasi-regular, and that both classes of measures are stable under extension of scalars. They are also stable under restricting to open subgroups:

\begin{proposition} \label{prop:regular-sub}
Let $\mu$ be a (quasi-)regular measure for $G$. Then the restriction of $\mu$ to any open subgroup is again (quasi-)regular.
\end{proposition}

\begin{proof}
Suppose $\mu$ is regular and let $U$ be an open subgroup of $G$. We show that $\mu \vert_U$ is regular. Thus let $X$ be a transitive $U$-set; we must show that $\mu(X)$ is a unit of $k$. Write $X=U/V$ for some open subgroup $V$ of $U$. We have a fibration $G/U \to G/V$ with fiber $U/V$, and so by \dref{defn:measure}{e}, $\mu(G/U)=\mu(G/V) \mu(U/V)$. Since $\mu(G/U)$ and $\mu(G/V)$ are units of $k$, it follows that $\mu(U/V)$ is as well.

Now suppose that $\mu$ is quasi-regular and let $U$ be an open subgroup of $G$. We show that $\mu \vert_U$ is quasi-regular. By definition, there is an open subgroup $V$ of $G$ such that $\mu \vert_V$ is regular. By the previous paragraph, $\mu \vert_{V \cap U}$ is regular. This shows that $\mu \vert_U$ is quasi-regular, as required.
\end{proof}

Define $\Theta^*(G)$ to be the localization of $\Theta(G)$ obtained by inverting the classes $[X]$ for all transitive $G$-sets $X$. This ring classifies regular measures: a $k$-valued regular measure is exactly a homomorphism $\Theta^*(G) \to k$. If $U$ is an open subgroup of $G$ then a regular measure for $G$ restricts to a regular measure for $U$ (Proposition~\ref{prop:regular-sub}), and so the natural map $\Theta(U) \to \Theta(G)$ induces a map $\Theta^*(U) \to \Theta^*(G)$.

It turns out that $\Theta^*(G)$ has a substantially simpler presentation than $\Theta(G)$. To describe it, let $\Omega^*(G)$ be the localizations of $\Omega(G)$ obtained by inverting the classes $\lbb X \rbb$ for all transitive $G$-sets $X$, and let $\fb(G) \subset \Omega^*(G)$ be the ideal generated by all elements of the form
\begin{displaymath}
\lbb X \times_Z Y \rbb - \frac{\lbb X \rbb \cdot \lbb Y \rbb}{\lbb Z \rbb}
\end{displaymath}
where $X \to Z$ and $Y \to Z$ are surjections of finitary $G$-sets, with $Z$ transitive. Note that here we are using the Burnside ring of $G$ and not $\hat{G}$. We then have:

\begin{proposition} \label{prop:Theta-star}
The natural map $\Omega^*(G) \to \Theta^*(G)$ is surjective with kernel $\fb(G)$.
\end{proposition}

\begin{proof}
Put $R=\Omega^*(G)/\fb(G)$. For a finitary $G$-set $X$, we write $\langle X \rangle$ for its class in $R$. We claim that the natural map $\Omega^*(G) \to \Theta^*(G)$ kills $\fb(G)$. To see this, let $X \to Z$ and $Y \to Z$ be as in the definition of $\fb(G)$. The maps $X \to Z$ and $X \times_Z Y \to Y$ have isomorphic fibers $F$. Thus $[X][Z]^{-1}=[X \times_Z Y][Y]^{-1}=[F]$ in $\Theta^*(G)$, which proves the claim. It follows that we have a ring homomorphism $\phi \colon R \to \Theta^*(G)$ satisfying $\phi(\langle X \rangle)=[X]$ for any finitary $G$-set $X$. To show that this map is an isomorphism, we construct an inverse.

Let $f \colon X \to Y$ be a map of finitary $G$-sets, with $Y$ transitive. Define $\psi_1(f)=\langle X \rangle \langle Y \rangle^{-1}$. We now verify that $\psi_1$ respects the defining relations of $\Theta'(G)$:
\begin{enumerate}
\item If $f$ and $g$ are isomorphic morphisms then it is clear that $\psi_1(f)=\psi_1(g)$.
\item If $f$ is the identity then $\psi_1(f)=\langle X \rangle \langle X \rangle^{-1}=1$.
\item Suppose that $f \colon X \to Y$ and $g \colon Y \to Z$ are maps of transitive $G$-sets. Then
\begin{displaymath}
\psi_1(gf)=\langle X \rangle \langle Z \rangle^{-1} = \langle X \rangle \langle Y \rangle^{-1} \langle Y \rangle \langle Z \rangle^{-1} = \psi_1(g) \psi_1(f).
\end{displaymath}
\item Let $f \colon X \to Y$ and $g \colon Y' \to Y$ be maps of transitive $G$-sets, let $X'=X \times_Y Y'$ be the fiber product, and let $f' \colon X' \to Y'$ be the natural map. Then
\begin{displaymath}
\psi_1(f')=\langle X' \rangle \langle Y' \rangle^{-1} = (\langle X \rangle \langle Y' \rangle \langle Y \rangle^{-1}) \cdot \langle Y' \rangle^{-1} = \psi_1(f),
\end{displaymath}
where in the second step we used the defining relation of $\fb(G)$.
\end{enumerate}
We thus see that $\psi_1$ respects the defining relations of $\Theta'(G)$. It follows that there is a ring homomorphism $\psi_2 \colon \Theta'(G) \to R$ satisfying $\psi_2(\langle f \rangle)=\psi_1(f)$.

Let $\psi_3 \colon \Theta(G) \to R$ be the transfer of $\psi_2$ along the isomorphism $\Theta(G)=\Theta'(G)$. Recall that if $X$ is a finitary $U$-set, for some open subgroup $U$ of $G$, then $[X] \in \Theta(G)$ corresponds to $\langle f \rangle \in \Theta'(G)$, where $f \colon I_U^G(X) \to G/U$ is the natural map. We thus see that $\psi_3([X]) = \langle I_U^G(X) \rangle \langle G/U \rangle^{-1}$. If $X$ is a transitive $G$-set then $\psi_3([X])=\langle X \rangle$ is a unit of $R$. Thus $\psi_3$ extends to a homomorphism $\psi \colon \Theta^*(G) \to R$.

We finally claim that $\phi$ and $\psi$ are inverse to each other. It is clear that $\psi \circ \phi$ is the identity. We now show that $\phi(\psi([X]))=[X]$ for any finitary $\hat{G}$-set $X$. If $X$ is a $G$-set then this holds by the noted properties of the maps. The general case now follows from the observation that $\Theta^*(G)$ is generated, as a ring, by the elements $[X]$ and $[X]^{-1}$ as $X$ varies over transitive $G$-sets. Indeed, it suffices to show that we can generate any class of the form $[X]$ where $X$ is a transitive $U$-set. Write $X \cong U/V$. As $U/V$ is the fiber of the map $G/V \to G/U$, we have $[X]=[G/V][G/U]^{-1}$. This completes the proof.
\end{proof}

The above presentation for $\Theta^*(G)$ translates to a nice description of regular measures. To state it, we introduce the following notion:

\begin{definition}
An \defn{R-measure} for $G$ is a rule $\mu$ assigning to each finitary $G$-set $X$ a quantity $\mu(X) \in k$ such that the following conditions hold:
\begin{enumerate}
\item If $X$ and $Y$ are isomorphic then $\mu(X)=\mu(Y)$.
\item We have $\mu(\bone)=1$.
\item We have $\mu(X \amalg Y)=\mu(X) + \mu(Y)$.
\item If $Z$ is a transitive $G$-set then $\mu(Z)$ is a unit of $k$.
\item Given maps of $G$-sets $X \to Z$ and $Y \to Z$, with $Z$ transitive, we have
\begin{displaymath}
\mu(X \times_Z Y) = \frac{\mu(X) \cdot \mu(Y)}{\mu(Z)}. \qedhere
\end{displaymath}
\end{enumerate}
\end{definition}

With this language, we have:

\begin{proposition}
Regular measures correspond to R-measures. Precisely, any regular measure defines an R-measure, and any R-measure uniquely extends to a regular measure
\end{proposition}

\subsection{Generalized index}

To close \S \ref{s:meas}, we give  a more group-theoretic interpretation of measures. This perspective is not used in the remainder of the paper, but we have found it to be helpful to our own understanding. The following is the key concept:

\begin{definition} \label{defn:index}
A \emph{generalized index} on $G$ with values in a commutative ring $k$ is a rule that assigns to every containment of open subgroups $V \subset U$ an element $\lbb U : V \rbb$ of $k$ such that the following conditions hold:
\begin{enumerate}
\item We have $\lbb U:U \rbb=1$ for all $U$.
\item We have $\lbb U:V \rbb = \lbb gUg^{-1}:gVg^{-1} \rbb$ for all $g \in G$.
\item Given open subgroups $W \subset V \subset U$, we have $\lbb U:W \rbb=\lbb U:V \rbb \lbb V:W \rbb$.
\item Let $V,W \subset U$ be open subgroups of $G$. Write $U=\bigsqcup_{i=1}^n Vx_iW$. Then
\begin{displaymath}
\lbb U:W \rbb = \sum_{i=1}^n \lbb V: V \cap x_iWx_i^{-1} \rbb.
\end{displaymath}
See Example~\ref{ex:finite-index} for an explanation of this identity. \qedhere
\end{enumerate}
\end{definition}

The following proposition relates generalized indices to measures. We omit the proof since it is not used in the remainder of the paper; see \cite[\S 3.6]{arxiv} for details.

\begin{proposition}
There is a bijection between generalized indices on $G$ with values in $k$ and measures on $G$ with values in $k$, where the generalized index $\lbb-:-\rbb$ corresponds to the measure $\mu$ if $\lbb U:V \rbb = \mu(U/V)$ for all open subgroups $V \subset U$.
\end{proposition}

\begin{example} \label{ex:finite-index}
Suppose $G$ is finite. For subgroups $V \subset U$, define $\lbb U:V \rbb$ to be the usual index $[U:V]$. We claim that this is a generalized index. The first three conditions are clear, while the fourth follows from the decomposition
\begin{displaymath}
U/W \cong \coprod_{i=1}^n V/(V \cap x_iWx_i^{-1}).
\end{displaymath}
Of course, this generalized index corresponds to the measure from Example~\ref{ex:finite-measure}.
\end{example}

\begin{example}
Recall that for $t \in \bC$ there is an associated $\bC$-valued measure $\mu_t$ for the infinite symmetric group $\fS$ (Example~\ref{ex:sym-measure}). By the above proposition, we can translate $\mu_t$ into a $\bC$-valued generalized index $\lbb -:- \rbb_t$ on $\fS$. This generalized index satisfies
\begin{displaymath}
\lbb \fS:\fS(n) \rbb_t = t(t-1) \cdots (t-n+1), \qquad
\lbb \fS:\fS_n \times \fS(n) \rbb_t = \binom{t}{n}.
\end{displaymath}
Here $\fS(n)$ is the subgroup of $\fS$ fixing each of $1, \ldots, n$. Since the $\mu_t$ account for all $\bC$-valued measures on $\fS$, the proposition implies that generalized indices $\lbb-:-\rbb_t$ account for all $\bC$-valued generalized indices on $\fS$.
\end{example}

\section{Integration} \label{s:int}

\subsection{Overview}

Fix a pro-oligomorphic group $G$, a ring $k$, and a $k$-valued measure $\mu$ for the duration of \S \ref{s:int}. The purpose of \S \ref{s:int} is to introduce a theory of integration for oligomorphic groups. Precisely, for a $\hat{G}$-set $X$, we introduce the class of Schwartz functions on $X$, and define an integral for them. We also define a push-forward operation on Schwartz functions with respect to a map $f \colon X \to Y$ of $\hat{G}$-sets, and show that it has the expected properties. This theory of integration underlies the tensor category constructions later in the paper.

\subsection{Schwartz functions} \label{ss:schwartz}

Let $X$ be a $\hat{G}$-set and let $\phi \colon X \to k$ be a function. We say that $\phi$ is \defn{smooth} if there is some group of definition $U$ for $X$ such that $\phi$ is left $U$-invariant, meaning $\phi(gx)=\phi(x)$ for all $g \in U$ and $x \in X$. We define the \defn{support} of $\phi$ to be the subset of $X$ where $\phi$ is non-zero. If $\phi$ is smooth then the support of $\phi$ is a $\hat{G}$-subset. In this case, we say that $\phi$ has \defn{finitary support} if its support is a finitary $\hat{G}$-set. We say that $\phi$ is a \defn{Schwartz function} if it is smooth and has finitary support. We let $\cC(X)$ be the set of Schwartz functions, which is naturally a $k$-module. It is an analog of Schwartz space.

Let $A$ be a finitary $\hat{G}$-subset of $X$, and let $1_A$ be its indicator function. One easily sees that $1_A$ is a Schwartz function. We refer to a function of this form as \defn{elementary}.

\begin{proposition} \label{prop:Schwartz-gen}
The elementary functions generate $\cC(X)$ as a $k$-module.
\end{proposition}

\begin{proof}
Let $\phi \in \cC(X)$ be given. Let $A$ be the support of $\phi$ and let $U$ be a group of definition for $X$ such that $\phi$ is $U$-invariant; thus $A$ is $U$-stable. Since $A$ is finitary, $U$ has finitely many orbits on $Y$, say $A_1, \ldots, A_n$. Since $\phi$ is $U$-invariant, it is constant on each orbit; let $c_i$ be the common value of $\phi$ on $A_i$. Then $\phi=\sum_{i=1}^n c_i \cdot 1_{A_i}$, which completes the proof.
\end{proof}

\begin{remark} \label{rmk:Schwartz-gen}
In fact, a more precise statement is true. Given any group of definition $U$ of $X$, the functions $1_A$ with $A$ an orbit of an open subgroup of $U$ generate $\cC(X)$. This can be deduced easily from the proposition.
\end{remark}

\begin{remark}
The terminology ``Schwartz function'' is motivated by Bruhat's generalization of the concept of Schwartz function to general locally compact abelian groups \cite{Bruhat}: in the case of $\bQ_p$, a (Bruhat-)Schwartz function is one that is smooth (locally constant) and compactly supported.
\end{remark}

\subsection{Integration} \label{ss:int}

Let $X$ be a $\hat{G}$-set and let $\phi \colon X \to k$ be a Schwartz function. Let $A$ be the support of $\phi$, write $A=A_1 \sqcup \cdots \sqcup A_n$ where the $A_i$ are $\hat{G}$-stable and $\phi$ is constant on $A_i$, and let $c_i$ be the common value of $f$ on $A_i$. We define the \defn{integral} of $\phi$ by
\begin{displaymath}
\int_X \phi(x) dx = \sum_{i=1}^n c_i \mu(A_i).
\end{displaymath}
It is a standard fact that this is well-defined: one can integrate finite sums of step functions against a finitely additive measure. A detailed proof is given in \cite[\S 3.3]{arxiv}.

It is clear that integration is $k$-linear in the integrand $\phi$. Moreover, it is $G$-equivariant, that is, we have
\begin{displaymath}
\int_{X^g} \phi^g(x) dx = \int_X \phi(x) dx,
\end{displaymath}
where $g \in G$ and $X^g$ and $\phi^g$ denote the conjugates by $g$. Concretely, this allows us to make familiar changes of variables, e.g., if $X$ is a $G$-set then we have
\begin{displaymath}
\int_X \phi(x) dx = \int_X \phi(gx) dx
\end{displaymath}
for any $g \in G$.

\begin{example}
Let $G$ be a finite group and let $X$ be a finitary $\hat{G}$-set. Then $X$ is finite, and any function on $X$ is a Schwartz function. For $\phi \in \cC(X)$ we have
\begin{displaymath}
\int_X \phi(x) dx = \sum_{x \in X} \phi(x).
\end{displaymath}
Here we are integrating with respect to the unique measure for $G$; see Example~\ref{ex:finite-measure}.
\end{example}

\begin{example}
Let $\mu$ be the measure $\mu_t$ for the symmetric group $\fS$ described in Example~\ref{ex:sym-measure}. Suppose that $\phi \in \cC(\Omega)$. The smoothness condition exactly means that there is some integer $n \ge 0$ such that $\phi(m)$ is constant for $m>n$; let $c$ be this common value. Then
\begin{displaymath}
\int_{\Omega} \phi(x) dx = c \cdot (t-n) + \sum_{m=1}^n \phi(m).
\end{displaymath}
Note that the right side is indeed independent of the choice of $n$.

Now suppose $\phi \in \cC(\Omega^2)$. We describe the basic procedure for computing its integral. First, $\Omega^2$ decomposes as the disjoint union of $\Omega^{[2]}$ and the diagonal, and the integral similarly decomposes. The diagonal is isomorphic to $\Omega$ and thus handled like the previous case, so we concentrate on the $\Omega^{[2]}$ integral. The smoothness condition implies that there is some integer $n \ge 0$ such that $\phi(i,j)$ is constant, say equal to $c$, for all distinct $i, j >n$. The integral over the $i,j>n$ region is equal to $2c \cdot \binom{t-n}{2}$. The remaining region decomposes into finitely many pieces which are either isomorphic to a cofinite subset of $\Omega$ (and can be handled as above) or singletons. Figure~\ref{fig:Shat-sets} depicts an example of what the level sets of $\phi$ could look like.
\end{example}

\begin{example}
Let $G=\Aut(\bR,<)$ and let $X$ be a finitary $\hat{G}$-set. As we said in Example~\ref{ex:AutR-measure}, $X$ is canonically a smooth manifold. A Schwartz functions on $X$ is constructible, i.e., its level sets are constructible sets. Our notion of integration (with respect to the measure from Example~\ref{ex:AutR-measure}) is exactly integration with respect to Euler characteristic, as defined by Schapira and Viro (see, e.g., \cite{Viro}).
\end{example}

\subsection{Push-forwards} \label{ss:push}

We now show that there is a good theory of push-forward for Schwartz functions. We note that so far we have not made use of axiom \dref{defn:measure}{e}, but it will play a critical role in establishing basic properties of push-forward.

Let $f \colon X \to Y$ be a map of $\hat{G}$-sets and let $\phi \in \cC(X)$. We define the \defn{push-forward} $f_*(\phi)$ to be the function on $Y$ given by
\begin{displaymath}
(f_* \phi)(y)=\int_{f^{-1}(y)} \phi(x) dx.
\end{displaymath}
Note that $f^{-1}(y)$ is a $\hat{G}$-subset of $X$ and the restriction of $\phi$ to this set is a Schwartz function, and so the integral is defined. We now establish the basic properties of this construction.

\begin{proposition} \label{prop:push-schwartz}
Push-forward preserves Schwartz space.
\end{proposition}

\begin{proof}
Let $f$ and $\phi$ be as above. Let $U$ be an open subgroup of $G$ such that (i) $U$ is a group of definition of $X$ and $Y$; (ii) $f$ is $U$-equivariant; and (iii) $\phi$ is left $U$-invariant. Then for $g \in U$ and $y \in Y$, we have
\begin{displaymath}
(f_* \phi)(gy) = \int_{f^{-1}(gy)} \phi(x) dx = \int_{f^{-1}(y)} \phi(g^{-1} x) dx = (f_* \phi)(x).
\end{displaymath}
where in the second step we used $f^{-1}(gy)=gf^{-1}(y)$ and made the change of variables $x \to gx$, and in the third step we used the left $U$-invariance of $\phi$. We thus see that $f_* \phi$ is left $U$-invariant, and thus smooth. It is clear that the support of $f_* \phi$ is contained in the image of the support of $\phi$, and is therefore finitary. Thus $f_* \phi$ is a Schwartz function.
\end{proof}

We thus see that push-forward defines a $k$-linear map $f_* \colon \cC(X) \to \cC(Y)$. Note that if $Y=\bone$ is a single point then we can identify $\cC(Y)$ with $k$, and $f_*$ just becomes integration over $X$.

\begin{proposition}[Transitivity of push-forward] \label{prop:push-trans}
Let $f \colon X \to Y$ and $g \colon Y \to Z$ be maps of $\hat{G}$-sets. Then $(gf)_*=g_* f_*$.
\end{proposition}

\begin{proof}
Let $U$ be a group of definition for $X$, $Y$, and $Z$ such that $f$ and $g$ are $U$-equivariant. Let $A$ be a $U$-orbit $X$, let $B=f(A)$, and let $C=g(B)$. Since $B$ is $U$-transitive, all fibers of $f \colon A \to B$ are $U$-conjugate, and thus have the same measure; call it $a$. Similarly define $b$ for $g \colon B \to C$ and $c$ for $gf \colon A \to C$. Computing directly, we find
\begin{displaymath}
(gf)_*(1_A)=c \cdot 1_C, \qquad g_*(f_*(1_A))=g_*(a \cdot 1_B) = ab \cdot 1_C.
\end{displaymath}
We claim that $c=ab$. Let $x \in A$, let $V_1$ be the stabilizer of $x$, $V_2$ the stabilizer of $f(x)$, and $V_3$ the stabilizer of $g(f(x))$. Thus $A\cong U/V_1$, $B \cong U/V_2$, and $C \cong U/V_3$. The fiber of $f \colon A \to B$ is $V_2/V_1$, and so $a=\mu(V_2/V_1)$. The fiber of $g \colon B \to C$ is $V_3/V_2$, and so $b=\mu(V_3/V_2)$. The fiber of $gf \colon A \to C$ is $V_3/V_1$, and so $c=\mu(V_3/V_1)$. We must therefore show $\mu(V_3/V_1)=\mu(V_2/V_1) \mu(V_3/V_2)$. This follows from \dref{defn:measure}{e}, as the natural map $V_3/V_1 \to V_3/V_2$ has fiber $V_2/V_1$. This establishes the claim. We thus have $(gf)_*(1_A)=g_*(f_*(1_A))$. Since the functions $1_A$ span $\cC(X)$ by Remark~\ref{rmk:Schwartz-gen}, the result follows.
\end{proof}

\begin{corollary}[Fubini's theorem] \label{cor:fubini}
Let $X$ and $Y$ be $\hat{G}$-sets and let $\phi \in \cC(X \times Y)$. Then
\begin{displaymath}
\int_{X \times Y} \phi(x,y) d(x,y) = \int_X \int_Y \phi(x,y) dy dx.
\end{displaymath}
\end{corollary}

\begin{proof}
Let $p \colon X \times Y \to Y$ be the projection map and let $f \colon X \times Y \to \bone$ and $g \colon Y \to \bone$ be the unique maps. Then $f=g \circ p$, and so $f_*=g_* \circ p_*$ by Proposition~\ref{prop:push-trans}. This exactly gives the stated formula.
\end{proof}

\begin{corollary} \label{cor:Theta-d-to-1}
Let $f \colon X \to Y$ be a map of finitary $\hat{G}$-sets such that each fiber is finite of cardinality $n$. Then $\mu(X)=n \cdot \mu(Y)$, and $[X]=n \cdot [Y]$ in $\Theta(G)$.
\end{corollary}

\begin{proof}
We have $f_*(1_X)=n \cdot 1_Y$. Applying push-forward along $Y \to \bone$ yields the first formula. The second follows by considering the universal measure valued in $\Theta(G)$.
\end{proof}

Given a map $f \colon X \to Y$ of $\hat{G}$-sets and a function $\psi \colon Y \to k$ on $Y$, we let $f^*(\psi)$ be the function on $X$ given by $(f^* \psi)(x)=\psi(f(x))$. If $\psi$ is smooth then so is $f^*(\psi)$. However, $f^*$ does not preserve Schwartz space in general. It does if $X$ is finitary, as then any smooth function is a Schwartz function.

\begin{proposition}[Projection formula] \label{prop:projection}
Let $f \colon X \to Y$ be a map of $\hat{G}$-sets, let $\phi \in \cC(X)$, and let $\psi \in \cC(Y)$. Then $f^*(\psi) \phi$ is a Schwartz function on $X$ and $f_*(f^*(\psi) \phi)=\psi f_*(\phi)$.
\end{proposition}

\begin{proof}
The function $f^*(\psi) \phi$ is smooth, and its support is contained in the support of $\phi$, and thus finitary; hence it is Schwartz. We have
\begin{displaymath}
f_*(f^*(\psi) \phi)(y) = \int_{f^{-1}(y)} (f^*\psi)(x) \phi(x) dx = \psi(y) \int_{f^{-1}(y)} \phi(x) dx = \psi(y) (f_* \phi)(y),
\end{displaymath}
where in the second step we used that $(f^* \psi)(x)=\psi(y)$ is constant and thus pulls out of the integral.
\end{proof}

\begin{proposition}[Base change] \label{prop:push-bc}
Consider a cartesian square of finitary $\hat{G}$-sets
\begin{displaymath}
\xymatrix{
X' \ar[r]^{g'} \ar[d]_{f'} & X \ar[d]^f \\
Y' \ar[r]^g & Y }
\end{displaymath}
Then for $\phi \in \cC(X)$ we have
\begin{displaymath}
g^*(f_* \phi) = f'_*((g')^* \phi).
\end{displaymath}
\end{proposition}

\begin{proof}
Put $\psi_1=g^*(f_* \phi)$ and $\psi_2=f'_*((g')^* \phi)$. Let $y' \in Y'$. Then
\begin{align*}
\psi_1(y') &= \int_{f^{-1}(g(y'))} \phi(x) dx \\
\psi_2(y') &= \int_{(f')^{-1}(y')} \phi(g'(x')) dx'
\end{align*}
The map $g' \colon (f')^{-1}(y') \to f^{-1}(g(y'))$ is an isomorphism of $\hat{G}$-sets. Thus the two integrals above are equal, and so $\psi_1=\psi_2$.
\end{proof}

\section{The \texorpdfstring{$\Theta$}{\textTheta} ring is binomial} \label{s:binom}

\subsection{The main theorem}

Fix a pro-oligomorphic group $G$ for the duration of \S \ref{s:binom}. Recall that a commutative ring $R$ is called a \defn{binomial ring} if it is torsion free as a $\bZ$-module and for every $x \in R$ and $n \in \bN$, the element
\begin{displaymath}
\binom{x}{n} = \frac{x(x-1) \cdots (x-n+1)}{n!}
\end{displaymath}
of $R \otimes \bQ$ belongs to $R$. The purpose of \S \ref{s:binom} is to prove the following theorem:

\begin{theorem} \label{thm:binom}
The ring $\Theta(G)$ is a binomial ring.
\end{theorem}

The theorem has some strong ramifications concerning measures:

\begin{corollary} \label{cor:meas-lift}
Let $\mu$ be a measure for $G$ with values in a field $k$ of characteristic $p>0$. Then $\mu$ takes values in the prime subfield $\bF_p \subset k$, and lifts uniquely to a $\bZ_p$-valued measure.
\end{corollary}

\begin{proof}
Regard $\mu$ as a ring homomorphism $\Theta(G) \to k$. For any binomial ring $R$, one has $x^p=x$ for $x \in R/pR$ \cite[Theorem~4.1]{Elliott}, and so any homomorphism $R \to k$ lands in $\bF_p$. The lifting statement follows from the discussion in \cite[\S 9]{Elliott}.
\end{proof}

\begin{corollary}
If an oligomorphic group admits any measure then it admits a $\bC$-valued measure.
\end{corollary}

\begin{proof}
Let $G$ be oligomorphic, and suppose that $G$ admits some measure. Then $\Theta(G)$ is non-zero. Since $\Theta(G)$ is torsion-free (Theorem~\ref{thm:binom}) it follows that $\Theta(G) \otimes \bQ$ is non-zero. Since $\Theta(G)$ is countable (\S \ref{ss:Theta-prop}(f)) the result follows from the fact that any non-zero countable $\bQ$-algebra admits a homomorphism to $\bC$.
\end{proof}

\begin{corollary}
If $\Theta(G)$ is finitely generated as a ring then there is a ring isomorphism
\begin{displaymath}
\Theta(G) \cong \bZ[1/m_1] \times \cdots \times \bZ[1/m_r]
\end{displaymath}
for some $r \ge 0$ and integers $m_1, \ldots, m_r$.
\end{corollary}

\begin{proof}
This is a general result for finitely generated binomial rings \cite[Theorem~9]{Xantcha}.
\end{proof}

\begin{remark}
There are many examples where $\Theta(G) \otimes \bQ$ is finitely generated as a $\bQ$-algebra and of positive Krull dimension. In these cases, the above corollary implies that $\Theta(G)$ cannot be finitely generated. For instance, this happens when $G$ is the infinite symmetric group: here, $\Theta(G)$ is the ring of integer-valued polynomials, which is not finitely generated, but $\Theta(G) \otimes \bQ \cong \bQ[t]$. The situation in the corollary occurs for the group $\Aut(\bR,<)$ and the group attached to boron trees.
\end{remark}

To prove the theorem, we show that $\Theta(G)$ carries the structure of a $\lambda$-ring, with $\lambda^n([X])=[X^{(n)}]$. We then show that the Adam's operations are trivial. A theorem of Elliott states that any such $\lambda$-ring is binomial, and so this yields the theorem.

\begin{remark}
Let $X$ be a finitary $\hat{G}$-set and let $x=[X]$ be its class in $\Theta(G)$. It is easy to see that $[X^{[n]}]=x(x-1)\cdots (x-n+1)$. Moreover, the natural map $X^{[n]} \to X^{(n)}$ is $n!$-to-1, and so $[X^{[n]}]=n! \cdot [X^{(n)}]$ by Corollary~\ref{cor:Theta-d-to-1}. Thus if we knew that $\Theta(G)$ were torsion-free then we would have $\binom{x}{n}=[X^{(n)}]$, and so $\Theta(G)$ would be a binomial ring. Thus, in a sense, the real content of Theorem~\ref{thm:binom} is the torsion-freeness. However, we do not know how to prove this directly. Instead, as stated above, we construct operations like $x \mapsto \binom{x}{n}$ on $\Theta(G)$, and then deduce the theorem from properties of these operations.
\end{remark}

\begin{remark} \label{rmk:rel-binom}
There is a relative version of Theorem~\ref{thm:binom}. Let $\sE$ be a stabilizer class satisfying the following condition:
\begin{itemize}
\item[($\ast$)] If $X$ is an $\sE$-smooth $G$-set then so is $X^{(n)}$, for any $n \ge 0$.
\end{itemize}
Then $\Theta(G; \sE)$ is a binomial ring. This follows from the proof of Theorem~\ref{thm:binom}. If $\sE$ does not satisfy ($\ast$) then $\Theta(G; \sE)$ need not be a binomial ring: for example, when $G$ is the infinite symmetric group and $\sE$ is the stabilizer class from Example~\ref{ex:stab-class}, we find that $\Theta(G; \sE)=\bZ[x]$ (Proposition~\ref{prop:sym-theta-rel}), which is not a binomial ring. We do not know if $\Theta(G; \sE)$ is always torsion-free.
\end{remark}

\subsection{\texorpdfstring{$\lambda$}{\textlambda}-rings}

We now recall the basic definitions surrounding $\lambda$-rings. We refer to \cite[\S 3]{Dieck} or \cite[\S I]{Knutson} for additional details. We note that some sources (such as \cite{Knutson}) use the terminology ``pre-$\lambda$-ring'' and ``$\lambda$-ring'' in place of our ``$\lambda$-ring'' and ``special $\lambda$-ring.''

\begin{definition}
A \defn{$\lambda$-ring} is a commutative ring $R$ equipped with (set) maps $\lambda^n \colon R \to R$ for all non-negative integers $n$ such that the following conditions hold (for $x,y \in R$):
\begin{enumerate}
\item $\lambda^0(x)=1$.
\item $\lambda^1(x)=x$.
\item $\lambda^n(x+y)=\sum_{i+j=n} \lambda^i(x) \lambda^j(y)$ for all $n \ge 0$.
\end{enumerate}
A $\lambda$-ring is called \defn{special} if the following additional conditions hold:
\begin{enumerate}[resume]
\item $\lambda^n(1)=0$ for $n \ge 2$.
\item $\lambda^n(xy)=P_n(\lambda^1(x), \ldots, \lambda^n(x), \lambda^1(y), \ldots, \lambda^n(y))$ for all $n \ge 0$.
\item $\lambda^n(\lambda^m(x))=P_{n,m}(\lambda^1(x), \ldots, \lambda^{nm}(x))$ for all $n,m \ge 0$.
\end{enumerate}
Here $P_n$ and $P_{n,m}$ are certain universal polynomials; see \cite[\S 3.1]{Dieck} or \cite[\S I.1]{Knutson} for the precise definitions.
\end{definition}

\begin{definition}
Let $R$ and $S$ be $\lambda$-rings. A \defn{homomorphism of $\lambda$-rings} is a ring homomorphism $f \colon R \to S$ such that $f(\lambda^i(x))=\lambda^i(f(x))$ for all $x \in R$ and all $i \ge 0$.
\end{definition}

\begin{example}
The ring $\bZ$ of integers is a special $\lambda$-ring via $\lambda^n(x) = \binom{x}{n}$. See \cite[\S I.2]{Knutson}.
\end{example}

\begin{example} \label{ex:lambda-burnside}
The Burnside ring $\Omega(G)$ carries a $\lambda$-ring structure given by $\lambda^n(\lbb X \rbb)=\lbb X^{(n)} \rbb$. Conditions (a) and (b) are clear, while (c) follows from the natural isomorphism
\begin{displaymath}
(X \amalg Y)^{(n)} = \coprod_{i+j=n} \big( X^{(i)} \times Y^{(j)} \big).
\end{displaymath}
This $\lambda$-structure is typically not special, even when $G$ is finite; see \cite{Siebeneicher}.
\end{example}

Let $R$ be a commutative ring. Let $\Lambda(R)$ be the set $1+tR\lbb t \rbb$ of all formal power series with coefficients in $R$ having constant term~1. Let $a,b \in \Lambda(R)$, and write $a=\sum_{i \ge 0} a_i t^i$ and $b=\sum_{i \ge 0} b_i t^i$ with $a_i,b_i \in R$ and $a_0=b_0=1$. We make the following definitions:
\begin{itemize}
\item $a \oplus b$ is the ordinary power series product $ab$.
\item $a \otimes b = \sum_{n \ge 0} P_n(a_1, \ldots, a_n, b_1, \ldots, b_n) t^n$.
\item $\lambda^i(a) = \sum_{n \ge 0} P_{n,i}(a_1, \ldots, a_{ni}) t^n$.
\end{itemize}
Grothendieck proved that $\oplus$ and $\otimes$ give $\Lambda(R)$ the structure of a commutative, associative, unital ring, and that the $\lambda^i$ define on it the structure of a special $\lambda$-ring; see \cite[\S I.2]{Knutson} or \cite[Proposition~3.1.8]{Dieck}. (The multiplicative identity in $\Lambda(R)$ is $1+t$.) This ring is sometimes called the ring of ``big Witt vectors.'' If $\phi \colon R \to S$ is a ring homomorphism, then there is an induced homomorphism of $\lambda$-rings $\phi \colon \Lambda(R) \to \Lambda(S)$, by $\phi(\sum_{i \ge 0} a_i t^i)=\sum_{i \ge 0} \phi(a_i) t^i$.

Suppose now that we have functions $\lambda^n \colon R \to R$ for all $n \ge 0$, with $\lambda^0(x)=1$ and $\lambda^1(x)=x$ for all $x$. Define a function $\lambda_t \colon R \to \Lambda(R)$ by $\lambda_t(x)=\sum_{n \ge 0} \lambda^n(x) t^n$. Then the $\lambda^n$ define a $\lambda$-ring structure on $R$ if and only if $\lambda_t$ is an additive group homomorphism. Assuming this, the $\lambda^n$ define a special $\lambda$-ring structure on $R$ if and only if $\lambda_t$ is a homomorphism of $\lambda$-rings.

Suppose that $R$ is a $\lambda$-ring, and let $\lambda_t \colon R \to \Lambda(R)$ be as above. For $x \in R$ and $n \ge 1$, define $\psi^n(x) \in R$ by the identity
\begin{displaymath}
\frac{d}{dt} \log{\lambda_t(x)} = \sum_{n=0}^{\infty} (-1)^n \psi^{n+1}(x) t^n.
\end{displaymath}
Then $\psi^n \colon R \to R$ is called the $n$th \defn{Adam's operation} associated to the $\lambda$-ring structure. These operations satisfy a number of nice properties, especially when the $\lambda$-ring structure is special; see \cite[\S I.4]{Knutson} or \cite[\S 3.4]{Dieck}. We say that the Adam's operations are \defn{trivial} if each $\psi^n$ is the identity. The relevance of these operations to us is the following result:

\begin{theorem} \label{thm:elliott}
If $R$ is a $\lambda$-ring with trivial Adam's operations then $R$ is a binomial ring. Conversely, every binomial ring admits a unique $\lambda$-ring structure with trivial Adam's operations, given by $\lambda^n(x)=\binom{x}{n}$, and this $\lambda$-ring structure is special. 
\end{theorem}

\begin{proof}
The first statement is due to Elliott \cite[Proposition~8.3]{Elliott}. The second is an older result of Wilkerson \cite{Wilkerson}.
\end{proof}

\subsection{The \texorpdfstring{$\lambda$}{\textlambda}-ring structure on \texorpdfstring{$\Theta(G)$}{\textTheta(G)}}

The remainder of \S \ref{s:binom} is devoted to the proof of the following theorem:

\begin{theorem} \label{thm:lambda}
The ring $\Theta(G)$ admits a $\lambda$-ring structure characterized uniquely by the following condition: if $X$ is a finitary $\hat{G}$-set then $\lambda^n([X])=[X^{(n)}]$. The Adam's operations for this structure are trivial.
\end{theorem}

This theorem implies Theorem~\ref{thm:binom} (by appealing to Theorem~\ref{thm:elliott}). The proof of Theorem~\ref{thm:lambda} proceeds as follows. We see from Example~\ref{ex:lambda-burnside} that $\Omega(\hat{G})$ admits a $\lambda$-ring structure (which is typically not special). We show that the composition
\begin{displaymath}
\xymatrix@C=4em{
\Omega(\hat{G}) \ar[r]^-{\lambda_t} & \Lambda(\Omega(\hat{G})) \ar[r] & \Lambda(\Theta(G)) }
\end{displaymath}
is a homomorphism of $\lambda$-rings. This involves establishing some identities in $\Theta(G)$ for classes of the form $[X^{(n)}]$, which we do in the next few subsections. We then show that the above map kills the ideal $\fa(G)$, and thus factors through $\Theta(G)$; this shows that $\Theta(G)$ carries a $\lambda$-ring structure as in the theorem. Finally, we show that the Adam's operations are trivial by relating them to the Adam's operations on the ring of integer-valued polynomials.

\begin{remark} \label{rmk:lambda-not-obvious}
Theorem~\ref{thm:lambda} implies that if $X$ and $Y$ are finitary $\hat{G}$-sets such that $[X]=[Y]$ then $[X^{(n)}]=[Y^{(n)}]$ for all $n$. This does not seem obvious even for $n=2$. (It is easy to see that $[X^{(2)}]-[Y^{(2)}]$ is 2-torsion though, since $2[X^{(2)}]=[X]^2-[X]$.)
\end{remark}

\subsection{Integer-valued polynomials} \label{ss:intpoly}

Recall that an \defn{integer-valued polynomial} is a polynomial $p \in \bQ[x]$ such that $p(m)$ is an integer for all integers $m$. The collection of all integer-valued polynomials forms a subring of $\bQ[x]$ that we denote by $\bZ\langle x \rangle$. This ring is a binomial ring. Indeed, if $p \in \bZ\langle x \rangle$ then $p(m)$ is an integer for all integers $m$, and so $\binom{p(m)}{n}$ is also an integer for all $m$; thus $\binom{p(x)}{n}$ still belongs to $\bZ\langle x \rangle$. For notational ease, we write $\lambda^n(p)$ for the polynomial $\binom{p(x)}{n}$; this defines a special $\lambda$-ring structure on $\bZ\langle x \rangle$ by Theorem~\ref{thm:elliott}. It is well-known that the elements $\lambda_n(x)=\binom{x}{n}$ form a basis for $\bZ\langle x \rangle$ as a $\bZ$-module. This, and other basic facts about $\bZ\langle x \rangle$, can be found in the survey article \cite{CahenChabert}. 

\begin{proposition} \label{prop:R-Theta}
Let $X$ be a finitary $\hat{G}$-set. Then there exists a unique ring homomorphism $\phi \colon \bZ\langle x \rangle\to \Theta(G)$ satisfying $\phi(\lambda_n(x))=[X^{(n)}]$ for all $n \ge 0$.
\end{proposition}

\begin{proof}
Put $a_n=\lambda_n(x)$. Since the $a_n$'s form a $\bZ$-basis of $\bZ\langle x \rangle$, there is a unique $\bZ$-linear map $\phi$ given by the stated formula. We must show that it is a ring homomorphism. It is clear that $\phi(1)=1$. We have
\begin{displaymath}
a_na_m = \sum_{k=\max(n,m)}^{n+m} N_{n,m}^k \cdot a_k, \qquad
N_{n,m}^k = \binom{k}{n} \binom{n}{n+m-k}.
\end{displaymath}
(See, e.g., \cite[Theorem~3.1]{HarmanHopkins} for a proof.) Write
\begin{displaymath}
X^{(n)} \times X^{(m)} = \coprod_{k=\max(n,m)}^{n+m} Y_k,
\end{displaymath}
where $Y_k$ is the set of pairs $(A,B) \in X^{(n)} \times X^{(m)}$ such that $\#(A \cup B)=k$. We have a map $h_k \colon Y_k \to X^{(k)}$ of $\hat{G}$-sets given by $(A,B) \mapsto A \cup B$. Now, the fiber of $h_k$ over $C \in X^{(k)}$ consists of all ways of writing $C$ as a union $A \cup B$ where $A$ and $B$ are subsets of $C$ of cardinality $n$ and $m$. One easily sees that there are $N_{n,m}^k$ such choices for $(A,B)$, and so the fibers of $h_k$ all have cardinality $N_{n,m}^k$. Thus, appealing to Corollary~\ref{cor:Theta-d-to-1}, we find $[Y_k]=N_{n,m}^k [X^{(k)}]$ in $\Theta(G)$. We therefore find
\begin{displaymath}
[X^{(n)}] [X^{(m)}] = \sum_{k=\max(n,m)}^{n+m} N_{n,m}^k \cdot [X^{(k)}].
\end{displaymath}
We thus see that $\phi(a_na_m)=\phi(a_n) \phi(a_m)$, from which it follows that $\phi$ is a ring homomorphism.
\end{proof}

\begin{remark}
Given $X$ as above, there is a unique $\bZ$-linear map $\tilde{\phi} \colon \bZ\langle x \rangle \to \Omega(\hat{G})$ satisfying $\tilde{\phi}(\lambda_n(x))=\lbb X^{(n)} \rbb$. This lifts $\phi$, but is typically not a ring homomorphism.
\end{remark}

We will also need multi-variate integer-valued polynomials. Let $I$ be an index set. We define $\bZ\langle x_i \rangle_{i \in I}$ to be the subring of $\bQ[x_i]_{i \in I}$ consisting of those polynomials that take integer values on $\bZ^I$. As with the univariate version, this is a binomial ring. If $i_1, \ldots, i_r \in I$ are distinct elements and $n_1, \ldots, n_r \in \bZ$ then the polynomial
\begin{displaymath}
\binom{x_{i_1}}{n_1} \cdots \binom{x_{i_r}}{n_r}
\end{displaymath}
belongs to $\bZ\langle x_i \rangle$. These polynomials form a $\bZ$-basis \cite[Lemma~2.2]{Elliott}. It follows that $\bZ\langle x_i \rangle_{i \in I}$ is naturally isomorphic to the tensor product $\bigotimes_{i \in I} \bZ\langle x_i \rangle$. We thus obtain the following corollary to the above proposition:

\begin{corollary} \label{cor:R-Theta}
Let $\{X_i\}_{i \in I}$ be finitary $\hat{G}$-sets. Then there exists a unique ring homomorphism $\phi \colon \bZ\langle x_i \rangle_{i \in I} \to \Theta(G)$ satisfying $\phi(\lambda_n(x_i))=[X_i^{(n)}]$ for all $i \in I$ and $n \in \bN$.
\end{corollary}

\subsection{Subsets of products}

We now examine how $X^{(n)}$ behaves when $X$ is a product, or, more generally, fibers over some base set.

\begin{lemma} \label{lem:lambda-5}
There exists an absolute polynomial $Q_n \in \bZ[X_1, \ldots, X_n, Y_1, \ldots, Y_n]$ with the following property. Let $f \colon X \to Y$ be a map of finitary $\hat{G}$-sets, and let $F_y=f^{-1}(y)$ be the fiber over $y$. Assume that $[F_{y_1}^{(n)}]=[F_{y_2}^{(n)}]$ for all $y_1,y_2 \in Y$ and $n \ge 0$. Then
\begin{displaymath}
[X^{(n)}] = Q_n([F^{(1)}], \ldots, [F^{(n)}], [Y^{(1)}], \ldots, [Y^{(n)}]).
\end{displaymath}
where $F=F_y$ for some $y$.
\end{lemma}

\begin{proof}
If $A \in X^{(n)}$ is an $n$-element subset of $X$ then $f(A)$ is a subset of $Y$ of cardinality at most $n$. We can therefore write $X^{(n)}=W^1 \sqcup \cdots \sqcup W^n$, where $W^r$ is the subset of $X^{(n)}$ where the image has size $r$. We have a natural map of $\hat{G}$-sets $f^r \colon W^r \to Y^{(r)}$ given by $A \mapsto f(A)$.

Let $B=\{y_1,\ldots,y_r\}$ be an $r$-element subset of $Y$, and consider the fiber $W^r_B$ of $f^r$ over $B$. An element of $W^r_B$ is an $n$-element subset $\{x_1,\ldots,x_n\}$ of $X$ such that $f(x_i) \in B$ for each $i$, and every element of $B$ occurs in this form. We can decompose this set by counting how many times each element of $B$ occurs. Specifically, let $n=\lambda_1+\cdots+\lambda_r$ be a composition of $n$ (into non-zero parts), and let $W^r_{B,\lambda}$ be subset of $W^r_B$ consisting of subsets $\{x_1,\ldots,x_n\}$ such that for each $1 \le j \le r$ we have $f(x_i)=y_j$ for exactly $\lambda_j$ values of $i$. Clearly, $W^r_B$ is the disjoint union of the $W^r_{B,\lambda}$ over all choices of $\lambda$. Furthermore, we have an isomorphism of $\hat{G}$-sets
\begin{displaymath}
W^r_{B,\lambda} \to F_{y_1}^{(\lambda_1)} \times \cdots \times F_{y_r}^{(\lambda_r)}.
\end{displaymath}
Indeed, suppose $\{x_1,\ldots,x_n\}$ is an element of $W^r_{B,\lambda}$, and index so that $x_1, \ldots, x_{\lambda_1}$ are the elements above $y_1$. Then each of these $x$'s belongs to $F_{y_1}$, and so $\{x_1,\ldots,x_{\lambda_1}\}$ is an element of $F_{y_1}^{(\lambda_1)}$. Treating the other $y$'s in the same way, we obtain a map as above, which is easily seen to be bijective.

By the previous paragraph, we see that
\begin{displaymath}
[W^r_B] = \sum_{\lambda} [W^r_{B,\lambda}] = \sum_{\lambda} [F^{(\lambda_1)}] \cdots [F^{(\lambda_r)}],
\end{displaymath}
where the sum is over all compositions of $n$ into $r$ parts. Note that the above is independent of the choice of $B$. We thus see that each fiber of $f^r \colon W^r \to Y^{(r)}$ defines the same class in $\Theta(G)$, and so $[W^r]=[W^r_B][Y^{(r)}]$ for any $B \in Y^{(r)}$ by Proposition~\ref{prop:push-trans} (using an argument similar to that in Corollary~\ref{cor:Theta-d-to-1}). Finally, we have
\begin{displaymath}
[X^{(n)}]=\sum_{r=1}^n [W^r] = \sum_{r=1}^n \sum_{\lambda} [F^{(\lambda_1)}] \cdots [F^{(\lambda_r)}] [Y^{(r)}].
\end{displaymath}
We can thus take
\begin{displaymath}
Q_n = \sum_{r=1}^n \sum_{\lambda} X_{\lambda_1} \cdots X_{\lambda_r} Y_r.
\end{displaymath}
This completes the proof.
\end{proof}

\begin{lemma} \label{lem:lambda-6}
Let $X$ and $Y$ be transitive $U$-sets, for some open subgroup $U$ of $G$, and let $f \colon X \to Y$ be a map of $U$-sets with fiber $F$. Put $\ol{X}=Y \times F$. Then $[X^{(n)}]=[\ol{X}{}^{(n)}]$ in $\Theta(G)$ for all $n \ge 0$.
\end{lemma}

\begin{proof}
Let $F_y=f^{-1}(y)$ be the fiber of $f$ over $y \in Y$. Since $U$ acts transitively on $Y$, for each $y_1,y_2 \in U$, we have $F_{y_1} \cong F_{y_2}^g$ for some $g \in G$. It follows that $F_{y_1}^{(n)} \cong (F_{y_2}^{(n)})^g$ and so $[F_{y_1}^{(n)}]=[F_{y_2}^{(n)}]$. Thus $f$ satisfies the condition of Lemma~\ref{lem:lambda-5}. The projection map $\ol{X} \to Y$ clearly also satisfies this condition, since all of its fibers are $F$. The formulas for $[X^{(n)}]$ and $[\ol{X}{}^{(n)}]$ provided by the lemma are the same, and so these classes are equal.
\end{proof}

\begin{lemma} \label{lem:lambda-7}
The identity
\begin{displaymath}
\lambda^n(xy) = Q_n(\lambda^1(x), \ldots, \lambda^n(x), \lambda^1(y), \ldots, \lambda^m(y))
\end{displaymath}
holds in $\bZ\langle x,y \rangle$.
\end{lemma}

\begin{proof}
Apply Lemma~\ref{lem:lambda-5} with $G$ being the trivial group and $X=Y \times F$ where $Y$ and $F$ finite sets of cardinalities $a$ and $b$. Using the identification $\Theta(G)=\bZ$ (see \S \ref{ss:Theta-prop}(b)), we find
\begin{displaymath}
\binom{ab}{n} = Q_n(\binom{a}{1}, \ldots, \binom{a}{n}, \binom{b}{1}, \ldots, \binom{b}{n}).
\end{displaymath}
Since this holds for all non-negative integers $a$ and $b$, it holds as an identity in the ring $\bZ\langle x, y \rangle$.
\end{proof}

\begin{lemma} \label{lem:lambda-8}
Let $X$ and $Y$ be finitary $\hat{G}$-sets. Then
\begin{displaymath}
[(X \times Y)^{(n)}] = P_n([X^{(1)}], \ldots, [X^{(n)}], [Y^{(1)}], \ldots, [Y^{(n)}])
\end{displaymath}
holds in $\Theta(G)$.
\end{lemma}

\begin{proof}
Let $\phi \colon \bZ\langle x, y \rangle \to \Theta(G)$ be the ring homomorphism afforded by Corollary~\ref{cor:R-Theta} satisfying $\phi(\lambda^n(x)) = [X^{(n)}]$ and $\phi(\lambda^n(y)) = [Y^{(n)}]$. We have
\begin{displaymath}
Q_n(\lambda^1(x), \ldots, \lambda^n(x), \lambda^1(y), \ldots, \lambda^n(y))
=\lambda^n(xy)
=P_n(\lambda^1(x), \ldots, \lambda^n(x), \lambda^1(y), \ldots, \lambda^n(y)),
\end{displaymath}
where the first equality comes from Lemma~\ref{lem:lambda-7}, and the second from the fact that $\bZ\langle x,y \rangle$ is a special $\lambda$-ring (Theorem~\ref{thm:elliott}). Applying $\phi$, we obtain
\begin{displaymath}
Q_n([X^{(1)}], \ldots, [X^{(n)}], [Y^{(1)}], \ldots, [Y^{(n)}])
=P_n([X^{(1)}], \ldots, [X^{(n)}], [Y^{(1)}], \ldots, [Y^{(n)}]).
\end{displaymath}
Since the left side is equal to $[(X \times Y)^{(n)}]$ by Lemma~\ref{lem:lambda-5}, the result follows.
\end{proof}

\begin{remark}
We have
\begin{displaymath}
P_2(X_1,X_2,Y_1,Y_2) = X_1^2 Y_2 + Y_1^2 X_2 - 2 X_2 Y_2.
\end{displaymath}
On the other hand, the polynomial $Q_2$ found in the proof of Lemma~\ref{lem:lambda-5} is
\begin{displaymath}
Q_2(X_1,X_2,Y_1,Y_2) = X_2Y_1+X_1^2 Y_2.
\end{displaymath}
The two identities for $\lambda^2(xy)$ in the proof of Lemma~\ref{lem:lambda-8} are
\begin{displaymath}
\binom{xy}{2} = \binom{x}{2} y + x^2 \binom{y}{2} = x^2 \binom{y}{2} + y^2 \binom{x}{2} - 2\binom{x}{2} \binom{y}{2}.
\end{displaymath}
One can check directly that these hold.
\end{remark}

\subsection{Subsets of subsets}

We now examine the construction $(X^{(m)})^{(n)}$.

\begin{lemma} \label{lem:lambda-9}
For $n,m \ge 0$ there exists an absolute polynomial $Q_{n,m} \in \bZ[X_1,\ldots,X_{nm}]$ with the following property: given an finitary $\hat{G}$-set $X$, we have
\begin{displaymath}
[(X^{(m)})^{(n)}]=Q_{n,m}([X^{(1)}], \ldots, [X^{(nm)}])
\end{displaymath}
in the ring $\Theta(G)$.
\end{lemma}

\begin{proof}
Put $W=(X^{(m)})^{(n)}$. An element of $W$ is a set $\{S_1, \ldots, S_n\}$ the $S_i$'s are distinct $m$-element subsets of $X$. The union $S_1 \cup \cdots \cup S_n$ is a subset of $X$ of size at most $nm$. For $0 \le r \le nm$, let $W_r$ be the subset of $W$ where this set has cardinality $r$. Let $f_r \colon W_r \to X^{(r)}$ be the map given by $f_r(\{S_1,\ldots,S_n\})=S_1 \cup \cdots \cup S_r$. The fiber of $f_r$ over some $T \in X^{(r)}$ consists of all sets $\{S_1, \ldots, S_n\}$ where the $S_i$'s are distinct $m$-element subsets of $T$ that union to $T$. This fiber is finite, and its cardinality depends only on $n$, $m$, and $r$; denote it by $c_{n,m}^r$. We thus see that $[W_r]=c_{n,m}^r [X^{(r)}]$ in $\Theta(G)$ (Corollary~\ref{cor:Theta-d-to-1}), and so
\begin{displaymath}
[W] = \sum_{r=1}^{nm} [W_r] = \sum_{r=0}^{nm} c_{n,m}^r [X^{(r)}].
\end{displaymath}
We can therefore take
\begin{displaymath}
Q_{n,m} = \sum_{r=0}^{nm} c_{n,m}^r X_r,
\end{displaymath}
which completes the proof.
\end{proof}

\begin{lemma} \label{lem:lambda-10}
We have
\begin{displaymath}
\lambda^n(\lambda^m(x)) = Q_{n,m}(\lambda^1(x), \ldots, \lambda^{nm}(x))
\end{displaymath}
in the ring $\bZ\langle x \rangle$.
\end{lemma}

\begin{proof}
Apply Lemma~\ref{lem:lambda-9} with $G$ being the trivial group and $X$ a finite set of cardinality $a$. Using the identification $\Theta(G)=\bZ$ (see \S \ref{ss:Theta-prop}(b)), we find
\begin{displaymath}
\binom{\binom{a}{m}}{n} = Q_{n,m}(\binom{a}{1}, \ldots, \binom{a}{nm}).
\end{displaymath}
Since this holds for all non-negative integers $a$, it holds as an identity in $\bZ\langle x \rangle$.
\end{proof}

\begin{lemma} \label{lem:lambda-11}
Let $X$ be a finitary $\hat{G}$-set. Then we have
\begin{displaymath}
[(X^{(m)})^{(n)}]=P_{n,m}([X^{(1)}], \ldots, X^{(nm)}).
\end{displaymath}
in the ring $\Theta(G)$.
\end{lemma}

\begin{proof}
Let $\phi \colon \bZ\langle x \rangle \to \Theta(G)$ be the ring homomorphism provided by Proposition~\ref{prop:R-Theta} satisfying $\phi(\lambda^n(x))=[X^{(n)}]$. We have
\begin{displaymath}
Q_{n,m}(\lambda^1(x), \ldots, \lambda^{nm}(x))
= \lambda^n(\lambda^m(x))
= P_{n,m}(\lambda^1(x), \ldots, \lambda^{nm}(x)),
\end{displaymath}
where the first equality comes from Lemma~\ref{lem:lambda-10}, and the second from the fact that $\bZ\langle x \rangle$ is a special $\lambda$-ring (Theorem~\ref{thm:elliott}). Applying $\phi$, we obtain
\begin{displaymath}
Q_{n,m}([X^{(1)}], \ldots, [X^{(nm)}])=P_{n,m}([X^{(1)}], \ldots, [X^{(nm)}]).
\end{displaymath}
Since the left side is equal to $[(X^{(m)})^{(n)}]$ (Lemma~\ref{lem:lambda-9}), the result follows.
\end{proof}

\begin{remark}
Just as in the previous subsection, the $Q_{n,m}$ polynomials found in the proof of Lemma~\ref{lem:lambda-9} are not the same as the $P_{n,m}$ polynomials; indeed, our $Q_{n,m}$ are linear forms, which is not true for the $P_{n,m}$. Nonetheless, the two identities for $\lambda^n(\lambda^m(x))$ in the proof of Lemma~\ref{lem:lambda-11} do hold.
\end{remark}

\subsection{Proof of Theorem~\ref{thm:lambda}}

We now prove the theorem, in a few steps. Recall from Example~\ref{ex:lambda-burnside} that $\Omega(G)$ admits the structure of a $\lambda$-ring via $\lambda(\lbb X \rbb)=\lbb X^{(n)} \rbb$. Of course, this applies to open subgroup of $G$ as well. It is clear that if $V \subset U$ are open subgroups then the restriction map $\Omega(U) \to \Omega(V)$ is a homomorphism of $\lambda$-rings. It follows that the direct limit $\Omega(\hat{G})$ carries a $\lambda$-ring structure as well. We thus have an \emph{additive} map $\lambda_t \colon \Omega(\hat{G}) \to \Lambda(\Omega(\hat{G}))$. Let $\pi \colon \Omega(\hat{G}) \to \Theta(G)$ be the natural map, which is a ring homomorphism, and write $\pi' \colon \Lambda(\Omega(\hat{G})) \to \Lambda(\Theta(G))$ for the induced map, which is a $\lambda$-ring homomorphism.

\begin{proposition}
Consider the diagram
\begin{displaymath}
\xymatrix@C=4em{
\Omega(\hat{G}) \ar[d]_{\pi} \ar[r]^-{\lambda_t} \ar[rd]^q & \Lambda(\Omega(\hat{G})) \ar[d]^{\pi'} \\
\Theta(G) \ar@{..>}[r]^-{\lambda_t} & \Lambda(\Theta(G)) }
\end{displaymath}
where $q$ is defined to be the composition $\pi' \circ \lambda_t$. Then:
\begin{enumerate}
\item $q$ is a $\lambda$-ring homomorphism.
\item $q$ annihilates the ideal $\fa(G) \subset \Omega(\hat{G})$.
\item $q$ induces the bottom $\lambda_t$ map, which gives $\Theta(G)$ the structure of a special $\lambda$-ring.
\end{enumerate}
\end{proposition}

\begin{proof}
(a) Let $X$ and $Y$ be finitary $\hat{G}$-sets. We have
\begin{align*}
q(\lbb X \rbb \cdot \lbb Y \rbb)
&= q(\lbb X \times Y \rbb)
= \sum_{n \ge 0} [(X \times Y)^{(n)}] t^n \\
&= \sum_{n \ge 0} P_n([X^{(1)}], \ldots, [X^{(n)}], [Y^{(1)}], \ldots, [Y^{(n)}]) t^n
= q(\lbb X \rbb) \otimes q(\lbb Y \rbb),
\end{align*}
where we have used Lemma~\ref{lem:lambda-8} in the second line, and the definition of $\otimes$ in the third. Thus $q$ is compatible with multiplication on ``effective'' elements, i.e., those of the form $\lbb X \rbb$. Since $q$ is additive and the effective elements generate $\Omega(\hat{G})$ under addition, it follows that $q$ is compatible with multiplication. Since $\lambda_t(1)=1$ and $\pi'$ is a ring homomorphism, we also have $q(1)=1$, and so $q$ is a ring homomorphism.

We also have
\begin{align*}
q(\lambda_m(\lbb X \rbb))
&= q(\lbb X^{(m)} \rbb)
= \sum_{n \ge 0} [(X^{(m)})^{(n)}] t^n \\
&= \sum_{n \ge 0} P_{n,m}([X^{(1)}], \ldots, [X^{(nm)}]) t^n
= \lambda^m(q(\lbb X \rbb))
\end{align*}
where we have used Lemma~\ref{lem:lambda-11} in the second line and the definition of $\lambda^m$ on $\Lambda(\Theta(G))$ in the third. Thus $q$ is compatible with the $\lambda$-operations on effective elements. Again, this implies compatibility on all elements, and so $q$ is a $\lambda$-ring homomorphism.

(b) We must show that $q$ kills the ideals $\fa_1(G)$ and $\fa_2(G)$ of $\Omega(\hat{G})$. For $\fa_1(G)$, this amounts to showing $q(\lbb X^g \rbb) = q(\lbb X \rbb)$ when $X$ is a finitary $\hat{G}$-set and $g \in G$. This follows from the isomorphism $(X^g)^{(n)} \cong (X^{(n)})^g$. We now treat $\fa_2(G)$. Let $f \colon X \to Y$ be a map of transitive $U$-sets with fiber $F$. We must show $q(\lbb X \rbb)=q(\lbb Y \rbb \lbb F \rbb)$. This is exactly Lemma~\ref{lem:lambda-6}.

(c) Since $q$ kills $\fa(G)$, it factors as $\lambda_t \circ \pi$ for some ring homomorphism $\lambda_t \colon \Theta(G) \to \Lambda(\Theta(G))$. It follows from the definitions that $\lambda^n([X])=[X^{(n)}]$. Since $\lambda_t$ is additive, the $\lambda^n$ induce a $\lambda$-ring structure on $\Theta(G)$. Since $q$ and $\pi$ are $\lambda$-ring homomorphisms and $\pi$ is surjective, it follows that $\lambda_t$ is also a $\lambda$-ring homomorphism, and so the $\lambda$-ring structure on $\Theta(G)$ is special. (This also follows directly from Lemmas~\ref{lem:lambda-8} and~\ref{lem:lambda-11}.)
\end{proof}

To complete the proof of Theorem~\ref{thm:lambda}, we must show that the Adam's operations on $\Theta(G)$ are trivial. We need a bit of preparation for this.

\begin{lemma} \label{lem:lambda-14}
Let $R$ and $S$ be special $\lambda$-rings, let $\phi \colon R \to S$ be a ring homomorphism, and let $R'$ be the set of elements $x \in R$ such that $\phi(\lambda^n(x))=\lambda^n(\phi(x))$ for all $n$. Then $R'$ is a subring of $R$ and is stable by the $\lambda^n$-operations.
\end{lemma}

\begin{proof}
We have $x \in R'$ if and only if $\phi(\lambda_t(x))=\lambda_t(\phi(x))$. Suppose $x,y \in R'$. Then
\begin{align*}
\phi(\lambda_t(x+y))
&=\phi(\lambda_t(x) \oplus \lambda_t(y))=\phi(\lambda_t(x)) \oplus \phi(\lambda_t(y)) \\
&=\lambda_t(\phi(x)) \oplus \lambda_t(\phi(y))=\lambda_t(\phi(x)+\phi(y))=\lambda_t(\phi(x+y)).
\end{align*}
Thus $x+y \in R'$. A similar computation shows that $x-y$, $xy$, and $\lambda^n(x)$ (for any $n \ge 0$) belong to $R'$. This completes the proof.
\end{proof}

The following result gives an improvement of Proposition~\ref{prop:R-Theta}. (We note that this result is actually implied by Theorem~\ref{thm:lambda} and the mapping property of $\bZ\langle x \rangle$ in the category of binomial rings; see \cite[Proposition~2.1]{Elliott}.)

\begin{lemma} \label{lem:R-Theta-2}
Let $X$ be a finitary $\hat{G}$-set. Then there exists a unique $\lambda$-ring homomorphism $\phi \colon \bZ\langle x \rangle \to \Theta(G)$ satisfying $\phi(x)=[X]$.
\end{lemma}

\begin{proof}
Let $\phi$ be the ring homomorphism provided by Proposition~\ref{prop:R-Theta} satisfying $\phi(\lambda^n(x))=[X^{(n)}]$. By definition, we have $\phi(\lambda^n(x))=\lambda^n(\phi(x))$ for all $n \ge 0$. Since $x$ generates $\bZ\langle x \rangle$ as a $\lambda$-ring and both $\bZ\langle x \rangle$ and $\Theta(G)$ are special $\lambda$-rings, it follows from Lemma~\ref{lem:lambda-14} that $\phi$ is a map of $\lambda$-rings. Thus we have found a $\lambda$-ring homomorphism satisfying the required condition. It is unique since the $\lambda^n(x)$ are a $\bZ$-basis for $\bZ\langle x \rangle$.
\end{proof}

We now complete the proof of Theorem~\ref{thm:lambda}, and thus of Theorem~\ref{thm:binom}:

\begin{proposition}
The Adam's operations on $\Theta(G)$ are trivial.
\end{proposition}

\begin{proof}
Let $X$ be a finitary $\hat{G}$-set and let $\phi \colon \bZ\langle x \rangle \to \Theta(G)$ be the $\lambda$-ring homomorphism provided by Lemma~\ref{lem:R-Theta-2} satisfying $\phi(x)=[X]$. The Adam's operations on $\bZ\langle x \rangle$ are trivial (Theorem~\ref{thm:elliott}, or direct computation). Since $\phi$ is a $\lambda$-ring homomorphism, it is compatible with the Adam's operations. We thus find
\begin{displaymath}
\psi^n([X])=\psi^n(\phi(x))=\phi(\psi^n(x))=\phi(x)=[X].
\end{displaymath}
We have thus shown that the Adam's operations are trivial on the ``effective'' classes of $\Theta(G)$, i.e., those of the form $[X]$. Since the Adam's operations are ring homomorphisms (\cite[\S I.4]{Knutson} or \cite[\S 3.4]{Dieck}) and the effective classes generate $\Theta(G)$ (even as an abelian group), it follows that the Adam's operations are trivial.
\end{proof}

\section{The model-theoretic perspective} \label{s:model}

\subsection{Overview}

There is an intimate connection between oligomorphic groups and model theory. We review this in some detail in \S \ref{ss:fraisse}, but provide an abridged rendition here. Suppose $\fA$ is some class of finite structures, e.g., total orders, partial orders, graphs, etc. Under certain conditions, one can form the \defn{Fra\"iss\'e limit} of the class $\fA$; this is a countable structure $\Omega$ that contains a copy of every member of $\fA$ and satisfies an important homogeneity property. In many cases, the automorphism group $G$ of $\Omega$ acts oligomorphically on $\Omega$. In this way, Fra\"iss\'e limits provide a powerful way of constructing oligomorphic groups; in fact, most examples of oligomorphic groups are constructed via this procedure.

Suppose we have $\fA$, $\Omega$, and $G$ as above, and $G$ acts oligomorphically. Since $G$ is completely determined by $\fA$, one can translate anything about $G$ to the model-theoretic side, at least in principle. The purpose of \S \ref{s:model} is to perform this translation for the concept of measure. In Definition~\ref{defn:model-meas}, we introduce the notion of measure for $\fA$. The main result of this section, Theorem~\ref{thm:compare}, establishes an equivalence between measures for $\fA$ and measures for $G$.

In some cases it is extremely difficult to understand anything about $G$ directly, e.g., for the automorphism group of the universal boron tree or Rago graph. In these cases, the only way to effectively work with measures is through the dictionary established here. Even in cases where one can understand something about $G$, the model-theoretic perspective can still provide substantial insight. We will see some examples of this as we go; for a more substantive example, see \S \ref{s:boron}.

There is one other significant result in this section we wish to highlight: Theorem~\ref{thm:R-approx} gives a model-theoretic criterion for $\Theta$ to be non-zero. While we believe this is a significant step in understanding $\Theta$, it does not apply very broadly. Expanding the scope of this theorem is an important problem.

\subsection{Fra\"iss\'e's theorem} \label{ss:fraisse}

We now provide some background on model theory and Fra\"iss\'e's theorem. We refer to the book \cite{Cameron} and the survery article \cite{Macpherson} for more details.

\subsubsection{Structures}

A \defn{signature} (of a first order relational language) is a collection $\Sigma=\{(R_i,n_i)\}_{i \in I}$ where $R_i$ is a symbol and $n_i$ is a positive integer, called the arity of $R_i$. Fix a signature $\Sigma$ for the remainder of this section. A \defn{structure} (for $\Sigma$) is a set $X$ equipped with, for each $i \in I$, an $n_i$-arity relation $R_i$ on $X$ (i.e., a subset of $X^{n_i}$). There is an obvious notion of isomorphism of structures. If $X$ is a structure and $Y$ is a subset of $X$, there is an induced structure on $Y$: simply restrict each relation from $X$ to $Y$. The structures obtained in this way are the \defn{substructures} of $X$. An \defn{embedding} $i \colon Y \to X$ of structures is an injective function that induces an isomorphism between $Y$ and the substructure $i(Y)$.

\subsubsection{Amalgamations}

Let $i \colon Y \to X$ and $j \colon Y \to Y'$ be embeddings of structures. An \defn{amalgamation} of $(i,j)$ is a triple $(X',i',j')$ consisting of a structure $X'$ and embeddings $i' \colon Y' \to X'$ and $j' \colon X \to X'$ such that two conditions hold:
\begin{enumerate}
\item $i'$ and $j'$ agree on $Y$, that is, $i' \circ j = i \circ j'$
\item $i'$ and $j'$ are jointly surjective, that is, $X'=\im(i') \cup \im(j')$.
\end{enumerate}
Thus, intuitively, an amalgamation is obtained by glueing $X$ and $Y'$ together along the common copy of $Y$. Note that amalgamations can identify more than required, i.e., the intersections of (the images of) $X$ and $Y'$ in $X'$ can be larger than (the image of) $Y$.

If $(X'',i'',j'')$ is a second amalgamation, an \defn{isomorphism of amalgamations} is an isomorphism $k \colon X' \to X''$ of structures such that $i''=k \circ i'$ and $j''=k \circ j'$. Amalgamations have no non-trivial automorphisms: in effect, $X'$ is labeled by $i'$ and $j'$, and an automorphism must preserve the labeling. We often list ``the'' amalgamations of $(i,j)$; such lists are always intended to contain exactly one member from each isomorphism class.

\subsubsection{Homogeneous structures} \label{sss:homo}

A structure $\Omega$ is \defn{homogeneous} if the following condition holds: given finite substructures $X$ and $Y$ of $\Omega$ and an isomorphism $i \colon Y \to X$, there exists an automorphism $\sigma$ of $\Omega$ extending $i$ (i.e., $\sigma(y)=i(y)$ for all $y \in Y$). Suppose $\Omega$ is homogeneous and let $G$ be its automorphism group. Then two finite subsets of $\Omega$, equipped with the induced substructures, are abstractly isomorphic if and only if they belong to the same $G$-orbit. We thus see that the $G$-orbits on $\Omega^{(n)}$ are in bijection with the isomorphism classes of $n$-element substructures of $\Omega$. In particular, if for each $n$ there are only finitely many isomorphism types of $n$-element substructures of $\Omega$ then $G$ acts oligomorphically on $\Omega$ (and conversely). In this way, homogeneous structures provide a bridge between model theory and oligomorphic groups.

\subsubsection{Ages} \label{sss:age}

The \defn{age} of a structure $\Omega$, denoted $\age(\Omega)$, is the class of all finite structures that embed into $\Omega$. Suppose $\Omega$ is a countable homogeneous structure, and put $\fA=\age(\Omega)$. Then $\fA$ has the following properties:
\begin{enumerate}
\item If $X \in \fA$ and $Y$ is isomorphic to $X$ then $Y \in \fA$.
\item If $X \in \fA$ and $Y$ is a substructure of $X$ then $Y \in \fA$.
\item Up to isomorphism, $\fA$ has only countably many members.
\item Given embeddings $i \colon Y \to X$ and $j \colon Y \to Y'$, with $X,Y,Y' \in \fA$, there is an amalgamation $(X',i',j')$ of $(i,j)$ with $X' \in \fA$.
\end{enumerate}
Properties (a), (b), and (c) are clear, and do not rely on homogeneity. We explain (d). We may as well suppose that $X$ and $Y'$ are substructures of $\Omega$. Then $i$ and $j$ can be regarded as two embeddings of $Y$ into $\Omega$. By homogeneity, there is an automorphism $\sigma$ of $\Omega$ such that $j=\sigma \circ i$. Thus, replacing $X$ with $\sigma(X)$, we can assume that $i$ and $j$ are the same embedding of $Y$ into $\Omega$. It thus follows that $X \cup Y$, with the induced substructure, is naturally an amalgamation of $X$ and $Y$, and is clearly in $\fA$.

We name one additional condition on $\fA$:
\begin{enumerate}[resume]
\item For each $n$, there are only finitely many $n$-element structures in $\fA$ up to isomorphism.
\end{enumerate}
Of course, (e) implies (c). Condition (e) does not have to hold for a general countable homogeneous structure. However, it does automatically hold if the signature is finite. From the discussion in \S \ref{sss:homo}, we see that $\Aut(\Omega)$ acts oligomorphically on $\Omega$ if and only if (e) holds.

\subsubsection{Fra\"iss\'e's theorem} \label{sss:fraisse}

This extremely influential theorem was proven by Roland Fra\"iss\'e in the early 1950's \cite{Fraisse}, and provides a converse to the discussion in \S \ref{sss:age}:

\begin{theorem}
Let $\fA$ be a class of finite structures satisfying (a)--(d) from \S \ref{sss:age}. Then there is a countable homogeneous structure $\Omega$ such that $\fA=\age(\Omega)$. Moreover, $\Omega$ is unique up to isomorphism.
\end{theorem}

A proof of Fra\"iss\'e's theorem can be found in \cite[\S 2.6]{Cameron}. A class $\fA$ satisfying (a)--(d) of \S \ref{sss:age} is called a \defn{Fra\"iss\'e class}, and the structure $\Omega$ in the theorem is called the \defn{Fra\"iss\'e limit} of $\fA$. Fra\"iss\'e limits provide a powerful and systematic way of producing homogeneous structures and oligomorphic groups.

\subsubsection{Example: total orders}

Let $\fA$ be the class of all finite totally ordered sets. (The signature in this case consists of a single binary relation.) This is easily seen to be a Fra\"iss\'e class. Thus $\fA$ is the age of a countable homogeneous structure $\Omega$.

In fact, we can take $\Omega$ to be the set $\bQ$ of rational numbers equipped with its standard order. It is clear that $\fA=\age(\bQ)$. To see that $\bQ$ is homogeneous, suppose that $X$ and $Y$ are finite subsets of $\bQ$ and $i \colon Y \to X$ is an order-preserving bijection. Since $X$ and $Y$ have the same cardinality, we can thus find an order-preserving bijection $\sigma \colon \bQ \to \bQ$ such that $\sigma(Y)=X$; in fact, we can take $\sigma$ to be a piecewise linear function. The restriction of $\sigma$ to $Y$ necessarily induces $i$, as there is only one order-preserving bijection $Y \to X$. The uniqueness aspect of the Fra\"iss\'e limit, in this case, amounts to the fact that any two countable dense total orders are isomorphic, a statement first proved by Cantor.

Since \S \ref{sss:age}(e) holds, we see that $G=\Aut(\bQ,<)$ acts oligomorphically on $\bQ$. Of course, this can be seen directly in this case too (e.g., using piecewise linear maps).

\subsection{Measures} \label{ss:model-meas}

Let $\fA$ be a class of finite structures (for our fixed signature). We impose the following mild finiteness condition on $\fA$:
\begin{itemize}
\item[(C)] Given embeddings $i \colon Y \to X$ and $j \colon Y \to Y'$ in $\fA$, there are only finitely many amalgamations of $(i,j)$ in $\fA$ (up to isomorphism).
\end{itemize}
We do not require amalgamations to exist. We note that (C) is a consequence of \S \ref{sss:age}(e). The following is the main concept studied in \S \ref{s:model}:

\begin{definition} \label{defn:model-meas}
A \defn{measure} for $\fA$ with values in $k$ is a rule that assigns to each embedding $i \colon Y \to X$ of structures in $\fA$ a quantity $\mu(i) \in k$ such that the following conditions hold:
\begin{enumerate}
\item We have $\mu(i)=1$ if $i$ is an isomorphism.
\item Given embeddings $i \colon Y \to X$ and $j \colon Z \to Y$, we have $\mu(i \circ j)=\mu(i) \cdot \mu(j)$.
\item Let $i \colon Y \to X$ and $j \colon Y \to Y'$ be embeddings, and let $(X'_{\alpha}, i'_{\alpha}, j'_{\alpha})$, for $1 \le \alpha \le r$, be the amalgamations of $(i,j)$ in $\fA$. Then $\mu(i)=\sum_{\alpha=1}^r \mu(i'_{\alpha})$. \qedhere
\end{enumerate}
\end{definition}

The main result of \S \ref{s:model}, Theorem~\ref{thm:compare} below, connects the above definition to our previous definition of measure. As before, there is an important class of regular measures:

\begin{definition}
A measure $\mu$ for $\fA$ is \defn{regular} if $\mu(i)$ is a unit of $k$ for all embeddings $i$.
\end{definition}

General measures can be a bit cumbersome since they are defined for embeddings of structures. For regular measures, the situation simplifies considerably. The following definition and proposition make this precise. We assume in what follows that $\fA$ contains $\emptyset$.

\begin{definition} \label{defn:R-meas}
An \defn{R-measure} for $\fA$ is a rule $\nu$ that assigns to each structure $X$ of $\fA$ a unit $\nu(X)$ of the ring $k$, such that the following conditions hold:
\begin{enumerate}
\item We have $\nu(X)=\nu(Y)$ if $X$ and $Y$ are isomorphic.
\item We have $\nu(\emptyset)=1$.
\item Suppose that $i \colon Y \to X$ and $j \colon Y \to Y'$ are embeddings in $\fA$, and let $(X'_{\alpha}, i'_{\alpha}, j'_{\alpha})$, for $1 \le \alpha \le r$, be the amalgamations of $(i,j)$ in $\fA$. Then
\begin{displaymath}
\nu(X) \cdot \nu(Y') = \nu(Y) \cdot \sum_{\alpha=1}^r \nu(X'_{\alpha}). \qedhere
\end{displaymath}
\end{enumerate}
\end{definition}

\begin{proposition} \label{prop:R-meas}
There is a bijection between regular measures and R-measures, with $\mu$ corresponding to $\nu$ if $\nu(X)=\mu(\emptyset \subset X)$ for all $X \in \fA$.
\end{proposition}

\begin{proof}
Let $\mu$ be a regular measure. Define $\nu(X)=\mu(\emptyset \subset X)$, which is a unit of $k$ by assumption. We note that if $Y \to X$ is any embedding then by \dref{defn:model-meas}{b} we have
\begin{displaymath}
\mu(\emptyset \subset X) = \mu(Y \subset X) \mu(\emptyset \subset Y),
\end{displaymath}
and so $\mu(Y \subset X)=\nu(X)/\nu(Y)$. We now see that $\nu$ is an R-measure: indeed, \dref{defn:R-meas}{a} and \dref{defn:R-meas}{b} are clear, while \dref{defn:R-meas}{c} follows from \dref{defn:model-meas}{c} and the expression for $\mu$ in terms of $\nu$. Conversely, if we start with an R-measure $\nu$ then similar reasoning shows that we obtain a measure $\mu$ by defining $\mu(Y \subset X)=\nu(X)/\nu(Y)$. As the two constructions are clearly inverse, the result follows.
\end{proof}

We close \S \ref{ss:model-meas} with two examples of measures.

\begin{example} \label{ex:set-meas}
Let $\fA$ be the class of all finite sets and let $t \in \bC \setminus \bN$. For a finite set $X$ of cardinality $n$, put
\begin{displaymath}
\nu_t(X)=(t)_n=t(t-1)\cdots(t-n+1).
\end{displaymath}
Note that this is non-zero since $t \not\in \bN$. We claim that $\nu_t$ is an R-measure for $\fA$. Properties \dref{defn:R-meas}{a} and \dref{defn:R-meas}{b} are clear. We now verify \dref{defn:R-meas}{c}. It suffices to treat the case where $i \colon [\ell] \to [\ell] \amalg [m]$ and $j \colon [\ell] \to [\ell] \amalg [n]$ are the standard inclusions and $m \le n$. The amalgamations are the sets of the form $[\ell] \amalg ([n+m]/\sim)$ where $\sim$ is an equivalence relation on $[n+m]$ that induces the discrete relation on $\{1,\ldots,n\}$ and $\{n+1,\ldots,m\}$. To give such a relation, we must give subsets of $[n]$ and $[m]$ and a bijection between them. It follows that the number of amalgamations of size $\ell+n+m-s$, for $0 \le s \le m$, is $\binom{n}{s} \binom{m}{s} s!$. Thus \dref{defn:R-meas}{c} becomes
\begin{displaymath}
(t)_{\ell+m} \cdot (t)_{\ell+n} = (t)_{\ell} \cdot \sum_{s=0}^m \left[ \binom{n}{s} \binom{m}{s} s! \cdot (t)_{\ell+n+m-s} \right].
\end{displaymath}
This is a standard identity for the Pochhammer symbol. We briefly explain a proof. Regard the equation as an identity in the polynomial ring $\bC[t]$. Divide by $(t)_{\ell}$ and make the change of variables $t \to t+\ell$ to reduce to the case $\ell=0$. The identity then becomes equivalent to the one used in the proof of Proposition~\ref{prop:R-Theta}; see \cite[Theorem~3.1]{HarmanHopkins}. We thus see that $\nu_t$ is an R-measure, as claimed.

The measure $\mu_t$ corresponding to $\nu_t$ is given as follows: if $i \colon Y \to X$ is an injection of sets of cardinalities $m \le n$ then
\begin{displaymath}
\mu_t(i)=(t-m)(t-m-1) \cdots (t-n+1).
\end{displaymath}
As we have seen, this defines a regular measure for $\fA$ when $t \in \bC \setminus \bN$. In fact, it defines a measure for all $t \in \bC$. This can be shown directly, but also follows from Theorem~\ref{thm:compare} and the results of \S \ref{s:sym}. Via Corollary~\ref{cor:compare}, the measure $\mu_t$ defined here corresponds to the measure $\mu_t$ for $\fS$ defined in Example~\ref{ex:sym-measure}.
\end{example}

\begin{example}
Let $\fA$ be the class of finite totally ordered sets. This admits an R-measure $\nu$ defined by $\nu(X)=(-1)^{\# X}$. That this is indeed an R-measure follows from Theorem~\ref{thm:compare} and the results of \S \ref{s:order}. It is possible to give a direct self-contained proof of this, but as it is a bit involved we simply give an example of \dref{defn:R-meas}{c}. Let
\begin{displaymath}
Y=\{1\}, \qquad X=\{1<2\}, \qquad Y'=\{1<3\}
\end{displaymath}
and let $i \colon Y \to X$ and $j \colon Y \to Y'$ be the inclusions. An amalgamation of $(i,j)$ is a totally ordered set where the elements are labeled by 1, 2, and 3 (with 2 and 3 possibly labeling the same element), such that $1<2$ and $1<3$. There are three possibilities:
\begin{displaymath}
\{1 < (2=3)\}, \qquad \{1<2<3\}, \qquad \{1<3<2\}.
\end{displaymath}
Thus \dref{defn:R-meas}{c} becomes the following identity:
\begin{displaymath}
\nu([2]) \cdot \nu([2]) = \nu([1]) \cdot (\nu([2])+2\nu([3])),
\end{displaymath}
where here we have identified a totally ordered set of cardinality $n$ with the standard such set $[n]$. As $\nu([n])=(-1)^n$, the above identity becomes $1=(-1)(1-2)$, which does indeed hold. Let $\mu$ be the regular measure corresponding to $\nu$. Via Theorem~\ref{thm:compare}, $\mu$ corresponds to the measure for $\Aut(\bR,<)$ defined in Example~\ref{ex:AutR-measure}.
\end{example}

\subsection{The \texorpdfstring{$\Theta$}{\textTheta} ring}

Let $\fA$ be as in \S \ref{ss:model-meas}. As in the group setting, it is useful to introduce a ring that governs measures:

\begin{definition} \label{defn:model-theta}
We define a commutative ring $\Theta(\fA)$ as follows. For each embedding $i \colon Y \to X$ in $\fA$, there is a class $[i]$ in $\Theta(\fA)$. These classes satisfy identities parallel to the conditions in Definition~\ref{defn:model-meas}.
\end{definition}

We thus see that giving a measure for $\fA$ with values in $k$ is equivalent to giving a ring homomorphism $\Theta(\fA) \to k$. Therefore, understanding measures for $\fA$ amounts to computing the ring $\Theta(\fA)$. We let $\Theta^*(\fA)$ be the localization of $\Theta(\fA)$ where the classes $[i]$ are inverted, for all embeddings $i$. Thus giving a regular measure for $\fA$ is equivalent to giving a homomorphism $\Theta^*(\fA) \to k$. Proposition~\ref{prop:R-meas} allows one to give a presentation for $\Theta^*(\fA)$ where the generators are classes $[X]$, with $X \in \fA$, and the relations correspond to the conditions in Definition~\ref{defn:R-meas}.

\subsection{Comparison} \label{ss:compare}

Let $\Omega$ be a homogeneous structure such that $\fA=\age(\Omega)$ satisfies the condition \S \ref{sss:age}(e). Then $\fA$ satisfies condition~(C) from \S \ref{ss:model-meas}, and $G=\Aut(\Omega)$ acts oligomorphically on $\Omega$. Let $\sE$ be the stabilizer class generated by $\Omega$ (see \S \ref{ss:relative}). The following is our main theorem connecting the model-theoretic and group-theoretic versions of $\Theta$:

\begin{theorem} \label{thm:compare}
We have a natural ring isomorphism $\Theta(\fA) = \Theta(G; \sE)$.
\end{theorem}

It follows easily from the proof that this isomorphism induces as isomorphism $\Theta^*(\fA) = \Theta^*(G; \sE)$. In particular, we see that (regular) measures for $\fA$ correspond bijectively to (regular) measures for $(G; \sE)$. In some cases, the theorem allows us to get at the absolute ring $\Theta(G)$:

\begin{corollary} \label{cor:compare}
Suppose that $\sE$ is large in the sense of \S \ref{ss:Theta-prop}(d). Then there is a natural ring isomorphism $\Theta(\fA) \otimes \bQ \cong \Theta(G) \otimes \bQ$.
\end{corollary}

We now begin the proof of Theorem~\ref{thm:compare}, which will take the remainder of \S \ref{ss:compare}. We first set-up some notation. We use the symbols $A$, $B$, and $C$ to denote $G$-sets, and the symbols $X$, $Y$, and $Z$ to denote structures in $\fA$. Given a structure $X$ in $\fA$, we let $\Omega^{[X]}$ denote the set of all embeddings $X \to \Omega$. Since $\Omega$ is homogeneous, the group $G$ acts transitively on $\Omega^{[X]}$. It is clear that $\Omega^{[n]}$ decomposes as a disjoint union of the sets $\Omega^{[X]}$, as $X$ varies over all structures in $\fA$ on the set $[n]$. It follows that the $\Omega^{[X]}$ account for all transitive $\sE$-smooth $G$-sets, as any such set is an orbit on some $\Omega^n$.

For an embedding $i \colon Y \to X$ of structures in $\fA$, we let $\omega_i \colon \Omega^{[X]} \to \Omega^{[Y]}$ be the map given by $\omega_i(k)=k \circ i$. This is a map of $G$-sets, and thus surjective since the sets are transitive. We note that $\omega$ is faithful in the sense that if $j \colon Y \to X$ is an embedding with $\omega_i=\omega_j$ then $i=j$; indeed, if $k \colon X \to \Omega$ is any embedding then we have $k \circ i = k \circ j$, and so $i=j$ since $k$ is injective. We say that the embedding $i$ is a \defn{pre-isomorphism} if $\omega_i$ is an isomorphism. Non-trivial pre-isomorphisms can exist: see Example~\ref{ex:pre-isom}.

The following lemma is really the key computation in the proof of the theorem, as it relates amalgamations on the model-theoretic side to fiber products on the group-theoretic side.

\begin{lemma} \label{lem:compare1}
Let $i \colon Y \to X$ and $j \colon Y \to Y'$ be embeddings in $\fA$, and let $(X'_{\alpha}, i'_{\alpha}, j'_{\alpha})$, for $1 \le \alpha \le r$, be the amalgamations of $(i,j)$. Then we have a natural isomorphism of $G$-sets
\begin{displaymath}
\Omega^{[X]} \times_{\Omega^{[Y]}} \Omega^{[Y']} = \coprod_{\alpha=1}^r \Omega^{[X'_{\alpha}]},
\end{displaymath}
where the fiber product is taken with respect to $\omega_i$ and $\omega_j$.
\end{lemma}

\begin{proof}
An element of the fiber product is a pair $(a,b)$ where $a \colon X \to \Omega$ and $b \colon Y' \to \Omega$ are embeddings that agree on $Y$, i.e., $a \circ i = b \circ j$. We thus have a natural map
\begin{displaymath}
\phi \colon \coprod_{\alpha=1}^r \Omega^{[X'_{\alpha}]} \to \Omega^{[X]} \times_{\Omega^{[Y]}} \Omega^{[Y']}
\end{displaymath}
defined as follows: for an embedding $c \colon X'_{\alpha} \to \Omega$, we let $\phi(c)=(a,b)$ where $a=c \circ i'_{\alpha}$ and $b=c \circ j'_{\alpha}$. It is clear that $\phi$ is $G$-equivariant. We also have a map
\begin{displaymath}
\psi \colon \Omega^{[X]} \times_{\Omega^{[Y]}} \Omega^{[Y']} \to  \coprod_{\alpha=1}^r \Omega^{[X'_{\alpha}]}
\end{displaymath}
defined as follows. Let $(a,b)$ in the domain be given. Then $\im(a) \cup \im(b)$, with the induced substructure from $\Omega$, is an amalgamation of $(i,j)$, and therefore (uniquely) isomorphic to $(X'_{\alpha}, i'_{\alpha}, j'_{\alpha})$ for a unique $\alpha$. This means that there is an embedding $c \colon X'_{\alpha} \to \Omega$ (necessarily unique) such that $c \circ j'_{\alpha}=a$ and $c \circ i'_{\alpha}=b$. We define $\psi(a,b)=c$. One readily verifies that $\phi$ and $\psi$ are inverse, which completes the proof.
\end{proof}

Recall the relative ring $\Theta'(G;\sE)$ from Remark~\ref{rmk:Theta-prime}(a).

\begin{lemma} \label{lem:compare2}
There is a ring homomorphism $\phi \colon \Theta(\fA) \to \Theta'(G; \cE)$ satisfying $\phi([i])=\langle \omega_i \rangle$ for any embedding $i \colon Y \to X$ in $\fA$.
\end{lemma}

\begin{proof}
For an embedding $i$, put $\tilde{\phi}([i])=\langle \omega_i \rangle$. We verify that $\tilde{\phi}$ respects the defining relations of $\Theta(\fA)$, which are labeled as in Definition~\ref{defn:model-meas}.
\begin{enumerate}
\item If $i \colon Y \to X$ is an isomorphism of structures then $\omega_i$ is an isomorphism of transitive $G$-sets, and so $\langle \omega_i \rangle=1$ by \dref{defn:Theta-prime}{a}. Thus $\tilde{\phi}(i)=1$.
\item Let $i \colon Y \to X$ and $j \colon Z \to Y$ be embeddings. Then $\omega_{i \circ j}=\omega_j \circ \omega_i$, and so $\langle \omega_{i \circ j} \rangle = \langle \omega_i \rangle \cdot \langle \omega_j \rangle$ \dref{defn:Theta-prime}{c}. Thus $\tilde{\phi}(i \circ j)=\tilde{\phi}(i) \cdot \tilde{\phi}(j)$.
\item Let $i \colon Y \to X$ and $j \colon Y \to Y'$ be embeddings, and let $(X'_{\alpha}, i'_{\alpha}, j'_{\alpha})$, for $1 \le \alpha \le r$, be the amalgamations of $(i,j)$. By Lemma~\ref{lem:compare1}, we have a cartesian square of $G$-sets
\begin{displaymath}
\xymatrix@C=4em{
\coprod \Omega^{[X'_{\alpha}]} \ar[r]^{\coprod \omega_{j'_{\alpha}}} \ar[d]_{\coprod \omega_{i'_{\alpha}}}  & \Omega^{[X]} \ar[d]^{\omega_i} \\
\Omega^{[Y']} \ar[r]^{\omega_j} & \Omega^{[Y]} }
\end{displaymath}
We therefore have $\langle \omega_i \rangle = \sum_{\alpha=1}^r \langle \omega_{i'_{\alpha}} \rangle$ by \dref{defn:Theta-prime}{b} and \dref{defn:Theta-prime}{d}, and so $\tilde{\phi}(i)=\sum_{i=1}^r \tilde{\phi}(i'_{\alpha})$.
\end{enumerate}
Thus $\tilde{\phi}$ respects the defining relations, and so induces a ring homomorphism $\phi$ as in the statement of the lemma.
\end{proof}

We now wish to define a ring homomorphism in the opposite direction. This is a bit harder since we cannot associate a structure to a $G$-set in a canonical manner. The next few lemmas address this issue.

\begin{lemma} \label{lem:compare3}
Let $X$ and $Y$ be structures in $\fA$ and let $f \colon \Omega^{[X]} \to \Omega^{[Y]}$ be a map of $G$-sets. Then there exists $X' \in \fA$ and embeddings $i \colon X \to X'$ and $j \colon Y \to X'$, with $i$ a pre-isomorphism, such that $f=\omega_j \circ \omega_i^{-1}$.
\end{lemma}

\begin{proof}
For simplicity, identify the sets underlying $X$ and $Y$ with $[n]$ and $[m]$, so that $\Omega^{[X]} \subset \Omega^{[n]}$ and $\Omega^{[Y]} \subset \Omega^{[m]}$. Let $\Gamma_f \subset \Omega^{n+m}$ be the graph of $f$. The projection map $\Gamma_f \to \Omega^{[X]}$ is an isomorphism, and so $\Gamma_f$ is a transitive $G$-set. We obtain an equivalence relation $\sim$ on $[n+m]$ by $i \sim j$ if $x_i=x_j$ for all $x \in \Gamma_f$. Indentifying $[n+m]/\sim$ with $[\ell]$, for some $\ell$, we see that $\Gamma_f$ is a $G$-orbit on $\Omega^{[\ell]}$, and thus of the form $\Omega^{[X']}$ for some structure $X'$ on the set $[\ell]$. The natural map $[n] \to ([n+m]/\sim) \cong [\ell]$ gives an embedding $i \colon X \to X'$; we similarly get an embedding $j \colon Y \to X'$. The maps $\omega_i$ and $\omega_j$ are identified with the projections $\Gamma_f \to \Omega^{[X]}$ and $\Gamma_f \to \Omega^{[Y]}$. It follows that $\omega_i$ is an isomorphism (i.e., $i$ is a pre-isomorphism) and $f=\omega_j \circ \omega_i^{-1}$, which completes the proof.
\end{proof}

In the setting of the above lemma, $f$ is isomorphic to $\omega_j$ as morphisms of $G$-sets, and so $\langle f \rangle=\langle \omega_j \rangle$ in $\Theta'(G; \sE)$. Since every transitive $\sE$-smooth $G$-set is isomorphic to some $\Omega^{[X]}$, it follows that the elements $\langle \omega_i \rangle$, with $i$ an embedding in $\fA$, generate the ring $\Theta'(G;\sE)$.

\begin{lemma} \label{lem:compare4-1}
Let $i \colon Y \to X$ be a pre-isomorphism. Then $[i]=1$ in $\Theta(\fA)$.
\end{lemma}

\begin{proof}
Since $\omega_i$ is an isomorphism, it follows that the natural map $\Omega^{[X]} \times_{\Omega^{[Y]}} \Omega^{[X]} \to \Omega^{[X]}$ is an isomorphism, and so the fiber product is a transitive $G$-set. Thus by Lemma~\ref{lem:compare1}, there is a unique amalgamation of $(i,i)$. But we know an amalgamation: namely, $(X,i',j')$ where $i'$ and $j'$ are the identity maps of $X$. By \dref{defn:model-meas}{c}, we have $[i]=[i']$, and by \dref{defn:model-meas}{a} we have $[i']=1$; thus $[i]=1$ as required.
\end{proof}

\begin{lemma} \label{lem:compare4}
Let $i \colon Y \to X$ and $i' \colon Y' \to X'$ be embeddings in $\fA$ such that $\omega_i$ and $\omega_{i'}$ are isomorphic as morphisms of $G$-sets. Then $[i]=[i']$ in $\Theta(\fA)$.
\end{lemma}

\begin{proof}
By assumption, we have a commutative diagram
\begin{displaymath}
\xymatrix@C=3em{
\Omega^{[X']} \ar[d]_{\omega_{i'}} \ar[r]^{\phi} & \Omega^{[X]} \ar[d]^{\omega_{i}} \\
\Omega^{[Y']} \ar[r]^{\psi} & \Omega^{[Y]} }
\end{displaymath}
where $\phi$ and $\psi$ are isomorphisms of $G$-sets. We proceed in three steps.

\textit{Step 1: $\psi$ is the identity.} Suppose that $Y=Y'$ and $\psi$ is the identity map. Applying Lemma~\ref{lem:compare3}, write $\phi=\omega_b \circ \omega_a^{-1}$ where $a \colon X' \to X''$ and $b \colon X \to X''$ are embeddings with $a$ a pre-isomorphism; note that since $\phi$ and $\omega_a$ are isomorphisms, so is $\omega_b$, that is, $b$ is a pre-isomorphism. We thus have a commutative diagram
\begin{displaymath}
\xymatrix@C=4em{
\Omega^{[X'']} \ar[r]^{\omega_b} \ar[d]_{\omega_a} & \Omega^{[X]} \ar[d]^{\omega_i} \\
\Omega^{[X']} \ar[r]^{\omega_{i'}} & \Omega^{[Y]} }
\end{displaymath}
Since $\omega$ is faithful, we have $a \circ i'=b \circ i$. Thus, by \dref{defn:model-meas}{b}, we have $[a][i']=[b][i]$. Since $a$ and $b$ are pre-isomorphisms, we have $[a]=[b]=1$ by Lemma~\ref{lem:compare4-1}. Thus $[i]=[i']$, as required.

\textit{Step 2: $\psi=\omega_j$.} Now suppose that $\psi=\omega_j$ for some embedding $j \colon Y \to Y'$. Since $\phi$ and $\psi$ are isomorphisms, our original diagram is cartesian; in particular, we see that the fiber product $\Omega^{[X]} \times_{\Omega^{[Y]}} \Omega^{[Y']}$ is a transitive $G$-set. Thus, by Lemma~\ref{lem:compare1}, there is a unique amalgamation of $(i,j)$, call it $(X'', i'', j')$, and $\Omega^{[X'']}$ is also the fiber product. By uniqueness of fiber products, we thus have a diagram
\begin{displaymath}
\xymatrix{
\Omega^{[X']} \ar[rr] \ar[rd]_{\omega_{i'}} && \Omega^{[X'']} \ar[ld]^{\omega_{i''}} \\
& \Omega^{[Y']} }
\end{displaymath}
where the top map is an isomorphism. By Step~1, we have $[i']=[i'']$. By \dref{defn:model-meas}{c}, we have $[i'']=[i]$. Thus $[i]=[i']$, as required.

\textit{Step 3: general case.} Applying Lemma~\ref{lem:compare3}, write $\psi=\omega_j \circ \omega_k^{-1}$, where $k \colon Y' \to Y''$ and $j \colon Y \to Y''$ are embeddings, with $k$ a pre-isomorphism; note that since $\psi$ is an isomorphism, $j$ is a pre-isomorphism too. Let $(X'',i'',j')$ be the unique amalgamation $(i',j)$; uniqueness follows from Lemma~\ref{lem:compare1} since $\omega_j$ is an isomorphism. We now have the following commutative diagram
\begin{displaymath}
\xymatrix@C=3em{
\Omega^{[X']} \ar[d]_{\omega_{i'}} &
\Omega^{[X'']} \ar[d]_{\omega_{i''}} \ar[r]^{\omega_{j'}} \ar[l] &
\Omega^{[X]} \ar[d]^{\omega_i} \\
\Omega^{[Y']} &
\Omega^{[Y'']} \ar[r]^{\omega_j} \ar[l]_{\omega_k} &
\Omega^{[Y]} }
\end{displaymath}
The top left arrow is the isomorphism $\phi^{-1} \circ \omega_{j'}$. Since the right square is cartesian (by Lemma~\ref{lem:compare1}), we see that $\omega_{j'}$ is an isomorphism, and so the top left arrow is an isomorphism too. Thus, by Step~2 applied to each square separately, we find $[i']=[i'']$ and $[i'']=[i]$, and so $[i]=[i']$.
\end{proof}

\begin{lemma} \label{lem:compare5}
There is a ring homomorphism $\psi \colon \Theta'(G; \sE) \to \Theta(\fA)$ satisfying $\psi(\langle \omega_i \rangle) = [i]$ whenever $i \colon Y \to X$ is an embedding in $\fA$.
\end{lemma}

\begin{proof}
Let $f$ be a morphism of transitive $\sE$-smooth $G$-sets. We define $\tilde{\psi}(f)=[i]$ where $i \colon Y \to X$ is an embedding in $\fA$ such that $f$ is isomorphic to $\omega_i$. This is well-defined since $f$ is isomorphic to $\omega_i$ for some $i$ (by Lemma~\ref{lem:compare3}), and if $\omega_i \cong \omega_j$ then $[i]=[j]$ (by Lemma~\ref{lem:compare4}). More generally, suppose $f \colon X \to Y$ is a morphisms of $\sE$-smooth $G$-sets with $Y$ transitive. Let $X_1, \ldots, X_r$ be the orbits on $X$ and let $f_i$ be the restriction of $f$ to $X_i$. We define $\tilde{\psi}(f)=\sum_{i=1}^r \tilde{\psi}(f_i)$.

We now show that $\tilde{\psi}$ respects the defining relations of $\Theta'(G; \sE)$ (see Definition~\ref{defn:Theta-prime}).
\begin{enumerate}
\item If $f$ is an isomorphism of transitive $\sE$-smooth $G$-sets then $f$ is isomorphic to $\omega_i$, where $i$ is the identity map on some structure, and so $\tilde{\psi}(f)=[i]=1$ by \dref{defn:model-meas}{a}.
\item This follows from the definition of $\tilde{\psi}$.
\item Let $f \colon A \to B$ and $g \colon B \to C$ be maps of transitive $\sE$-smooth $G$-sets. Choose an isomorphism $g \cong \omega_j$ where $j \colon Z \to Y$ is an embedding. Now choose an isomorphism $a \colon A \to \Omega^{[X']}$ for some $X'$ in $\fA$. Thus $f$ corresponds to a map $f' \colon \Omega^{[X']} \to \Omega^{[Y]}$. Applying Lemma~\ref{lem:compare3}, write $f'=\omega_i \circ \omega_k^{-1}$, where $i \colon Y \to X$ and $k \colon X' \to X$ are embeddings, with $k$ a pre-isomorphism. We thus have the following commutative diagram
\begin{displaymath}
\xymatrix{
A \ar[r]^f \ar[d] & B \ar[r]^g \ar[d] & C \ar[d] \\
\Omega^{[X]} \ar[r]^{\omega_i} & \Omega^{[Y]} \ar[r]^{\omega_j} & \Omega^{[Z]} }
\end{displaymath}
where the vertical maps are isomorphisms; the left vertical map is $\omega_k^{-1} \circ a$. We thus find
\begin{displaymath}
\tilde{\psi}(g \circ f)=[i \circ j]=[i][j]=\tilde{\psi}(f) \tilde{\psi}(g),
\end{displaymath}
where in the second step we used \dref{defn:model-meas}{b}.
\item Let $f \colon A \to B$ and $g \colon B' \to B$ be maps of $\sE$-smooth $G$-sets, with $B$ and $B'$ transitive, and let $f' \colon A' \to B'$ be the base change of $f$. We are free to replace $f$ and $g$ with isomorphic morphisms, so we may as well suppose that $f=\omega_i$ and $g=\omega_j$, where $i \colon Y \to X$ and $j \colon Y \to Y'$ are embeddings in $\fA$; note that $A=\Omega^{[X]}$, $B=\Omega^{[Y]}$, and $B'=\Omega^{[Y']}$. Let $(X'_{\alpha}, i'_{\alpha}, j'_{\alpha})$, for $1 \le \alpha \le r$, be the amalgamations of $(i,j)$. By Lemma~\ref{lem:compare1}, we have $A'=\coprod_{\alpha=1}^r \Omega^{[X'_{\alpha}]}$, and $f'=\coprod \omega_{i'_{\alpha}}$. We thus have
\begin{displaymath}
\tilde{\psi}(f')=\sum_{\alpha=1}^r [i'_{\alpha}] = [i] = \tilde{\psi}(f),
\end{displaymath}
where in the second step we used \dref{defn:model-meas}{c}.
\end{enumerate}
We thus see that $\tilde{\psi}$ respects the defining relations for $\Theta'(G; \sE)$, and thus induces a homomorphism $\psi$ as in the statement of the lemma.
\end{proof}

\begin{proof}[Proof of Theorem~\ref{thm:compare}]
In Lemmas~\ref{lem:compare2} and~\ref{lem:compare5}, we have constructed ring homomorphisms
\begin{displaymath}
\phi \colon \Theta(\fA) \to \Theta'(G; \sE), \qquad
\psi \colon \Theta'(G; \sE) \to \Theta(\fA).
\end{displaymath}
By their definitions, we have $\phi(\psi(\langle \omega_i \rangle))=\langle \omega_i \rangle$ and $\psi(\phi([i]))=[i]$ whenever $i \colon Y \to X$ is an embedding in $\fA$. Since the elements $[i]$ and $\langle \omega_i \rangle$ generate $\Theta(\fA)$ and $\Theta'(G; \sE)$, it follows that $\phi \circ \psi=\id$ and $\psi \circ \phi=\id$. Finally, we have an isomorphism $\Theta'(G; \sE) \cong \Theta(G; \sE)$ by the relative version of Proposition~\ref{prop:Theta-prime}.
\end{proof}

\begin{example} \label{ex:pre-isom}
We give an example of a  non-trivial pre-isomorphism. Let $\fA$ be the class of all finite graphs in which each vertex belongs to at most one edge. This is a Fra\"iss\'e class, and the limiting object $\Omega$ is a perfect matching on a countably infinite set. The automorphism group of $\Omega$ is the wreath product $(\bZ/2 \bZ) \wr \fS$. Let $Y$ be the graph with one vertex, let $X$ be the graph with one edge, and choose an embedding $i \colon Y \to X$. Then $i$ is not an isomorphism, but $\omega_i$ is, as any embedding $X \to \Omega$ is determined by what it does to any single vertex. Thus $i$ is a non-trivial pre-isomorphism.
\end{example}

\begin{remark}
Treating $\fA$ as a category (with morphisms being embeddings), we have a functor $\fA^{\op} \to \cS(G;\sE)$, taking a structure $X$ to $\Omega^{[X]}$ and an embedding $i$ to $\omega_i$. In fact, this realizes the category of atoms in $\cS(G;\sE)$ as the localization of $\fA^{\op}$ at the class of pre-isomorphisms; the key result here is Lemma~\ref{lem:compare3}.
\end{remark}

\subsection{Ultralimits of measures} \label{ss:ultra}

Let $\fA$ be a class of finite structures satisfying condition~(C). Write $\fA$ as a directed union $\bigcup_{\alpha \in I} \fA_{\alpha}$ of subclasses, i.e., $I$ is a directed set and $\alpha \le \beta$ implies $\fA_{\alpha} \subset \fA_{\beta}$. For $\alpha \in I$, let $I_{\ge \alpha}$ be the set of elements $\beta \in I$ with $\beta \ge \alpha$. Fix a non-principal ultrafilter $\cF$ on $I$ that contains $I_{\ge \alpha}$ for all $\alpha \in I$.

Suppose that for each $\alpha \in I$ we have a measure $\mu_{\alpha}$ for $\fA_{\alpha}$ valued in some ring $k_{\alpha}$. We now show that the ultralimit of the $\mu_{\alpha}$'s is a measure on $\fA$. Let $k$ be the ultraproduct of the $k_{\alpha}$'s. Suppose that $i \colon Y \to X$ in $\fA$ is an embedding in $\fA$. There is some $\beta \in I$ such that $X$ and $Y$ both belong to $\fA_{\beta}$, and thus to $\fA_{\alpha}$ for all $\alpha \ge \beta$. Thus $\mu_{\alpha}(i) \in k_{\alpha}$ is defined for such $\alpha$. We define $\mu(i)$ to be the element $(\mu_{\alpha}(i))_{\alpha \in I}$ of $k$.

\begin{proposition}
$\mu$ is a measure for $\fA$, and regular if each $\mu_{\alpha}$ is.
\end{proposition}

\begin{proof}
To show that $\mu$ is a measure, we must check the three conditions of Definition~\ref{defn:model-meas}. Consider an instance of one of these conditions. Let $\beta \in I$ be such that all relevant structures belong to $\fA_{\beta}$. Then $\mu_{\alpha}$ satisfies this instance of the condition for all $\alpha \ge \beta$. It follows that $\mu$ also satisfies this instance of the condition. Now suppose that each $\mu_{\alpha}$ is regular. Given $X \in \fA$, let $\beta$ be such that $X \in \fA_{\beta}$. Then $\mu_{\alpha}(X)$ is a unit of $k_{\alpha}$ for all $\alpha \ge \beta$, and so $\mu(X)$ is a unit of $k$. Thus $\mu$ is regular.
\end{proof}

\begin{corollary}
We have a ring map from $\Theta(\fA)$ to the ultraproduct of the $\Theta(\fA_{\alpha})$'s.
\end{corollary}

\begin{remark} \label{rmk:not-ultra}
The use of ultralimits above is not really necessary. Let $\tilde{k} = \prod_{\alpha \in I} k_{\alpha}$, and define $k'$ to be the quotient of $\tilde{k}$ where two elements $x$ and $y$ are identified if there is some $\beta$ such that $x_{\alpha}=y_{\alpha}$ for all $\alpha \ge \beta$; note that this does not depend on any choice of ultrafilter. Then the $\mu_{\alpha}$'s actually give a well-defined measure $\mu'$ in $k'$. The choice of an ultraproduct is equivalent to choosing a maximal ideal in $k'$, and the ultralimit measure $\mu$ defined above comes from $\mu'$ via extension of scalars along the quotient map $k' \to k$.
\end{remark}

\begin{remark}\label{rmk:ultra-categories}
Suppose that each $\fA_{\alpha}$ is a Fra\"iss\'e class, with corresponding oligomorphic group $G_{\alpha}$, and let $G$ be the group for $\fA$. One can then show that the category $\uPerm(G; \mu)$ we construct in Part~\ref{part:perm} is the ultraproduct of the categories $\uPerm(G_{\alpha}; \mu_{\alpha})$. This is closely related to an observation of Deligne \cite{Deligne3} (which was elaborated upon by the first author \cite{Harman2}) that the category $\uRep(\fS_t)$ at a transcendental value of $t$ can be constructed as an ultraproduct of the ordinary representation categories $\Rep(\fS_n)$.
\end{remark}

\subsection{Existence of measures} \label{ss:Happrox}

We now give a model-theoretic criterion that ensures $\Theta$ is non-zero, and therefore that measures exist. Fix a Fra\"iss\'e class $\fA$ satisfying condition~(C) from \S \ref{ss:model-meas}. For $X,Y \in \fA$, let $h_X(Y)$ denote the set of all embeddings $Y \to X$. We say that $\Gamma \in \fA$ is \defn{homogeneous} if $\Aut(\Gamma)$ acts transitively on $h_{\Gamma}(X)$, whenever the latter is non-empty. Such objects figure into the following definition:

\begin{definition}
The class $\fA$ is \defn{smoothly approximable} if every object of $\fA$ embeds into a homogeneous object in $\fA$.
\end{definition}

Smoothly approximable classes have received much attention in the literature, e.g., \cite{CherlinHrushovski, KantorLiebeckMacpherson}. We establish the following result about them:

\begin{theorem} \label{thm:R-approx}
If $\fA$ is smoothly approximable then $\Theta^*(\fA) \otimes \bQ$ is non-zero.
\end{theorem}

In other words, a smoothly approximable class always admits a regular measure valued in a field of characteristic~0. In fact, smoothly approximable is a stronger hypothesis than what is needed to prove the theorem; see \cite[\S 6.6]{arxiv} for details. For a homogeneous structure $\Gamma$, let $\fA_{\Gamma}$ be the age of $\Gamma$ (\S  \ref{sss:age}). The following is the key lemma we require:

\begin{lemma} \label{lem:S-reg}
Let $\Gamma \in \fA$ be homogeneous. Then $X \mapsto \# h_{\Gamma}(X)$ is an R-measure on $\fA_{\Gamma}$, if we regard it as valued in $\bQ$.
\end{lemma}

\begin{proof}
Conditions \dref{defn:R-meas}{a} and \dref{defn:R-meas}{b} are clear. We now verify \dref{defn:R-meas}{c}. Thus let $i \colon Y \to X$ and $j \colon Y \to Y'$ be embeddings in $\fA_{\Gamma}$, and let $(X'_{\alpha}, i'_{\alpha}, j'_{\alpha})$, for $1 \le \alpha \le r$, be the amalgamations of $(i,j)$ in $\fA_{\Gamma}$. We must show
\begin{displaymath}
\# h_{\Gamma}(X) \cdot \# h_{\Gamma}(Y') = \# h_{\Gamma}(Y) \cdot \sum_{\alpha=1}^r \# h_{\Gamma}(X'_{\alpha}).
\end{displaymath}
Since $\Gamma$ is homogeneous, the fibers of the map map $i^* \colon h_{\Gamma}(X) \to h_{\Gamma}(Y)$, all have the same cardinality, say $n$; note that $\# h_{\Gamma}(X) = n \cdot \# h_{\Gamma}(Y)$. Let $m$ be defined similarly for $j$, so that $\# h_{\Gamma}(Y') = m \cdot \# h_{\Gamma}(Y)$. The fibers of the map $h_{\Gamma}(Y') \times_{h_{\Gamma}(Y)} h_{\Gamma}(X) \to h_{\Gamma}(Y)$ all have size $nm$, and so the cardinality of the fiber product is $nm \cdot \# h_{\Gamma}(Y)$. Putting this all together, we find
\begin{displaymath}
\# h_{\Gamma}(X) \cdot \# h_{\Gamma}(Y') = \# h_{\Gamma}(Y) \cdot \#( h_{\Gamma}(X) \times_{h_{\Gamma}(Y)} h_{\Gamma}(Y') ).
\end{displaymath}
Thus to complete the proof, we must show
\begin{displaymath}
\#( h_{\Gamma}(Y') \times_{h_{\Gamma}(Y)} h_{\Gamma}(X) ) = \sum_{\alpha=1}^r \# h_{\Gamma}(X'_{\alpha}).
\end{displaymath}
In fact, we have a natural bijection
\begin{displaymath}
\phi \colon \coprod_{\alpha=1}^r h_{\Gamma}(X'_{\alpha}) \to h_{\Gamma}(Y') \times_{h_{\Gamma}(Y)} h_{\Gamma}(X)
\end{displaymath}
as follows. (This argument is similar to the proof Lemma~\ref{lem:compare1}.) Let $c \in h_{\Gamma}(X'_{\alpha})$ be given. Let $a=c \circ i'_{\alpha}$ and $b=c \circ j'_{\alpha}$. Then $a \colon Y' \to \Gamma$ and $b \colon X \to \Gamma$ are embeddings that agree on $Y$, i.e., $a \circ j=b \circ i$. Thus $(a,b)$ is an element of the fiber product, and we put $\phi(c)=(a,b)$.

To see that $\phi$ is a bijection, we construct its inverse $\psi$. Suppose $(a,b)$ belongs to the fiber product. Thus $a \colon Y' \to \Gamma$ and $b \colon X \to \Gamma$ are embeddings that agree on $Y$, i.e., $a \circ j = b \circ i$. Let $X'=\im(a) \cup \im(b)$ with the induced substructure. Then the restrictions of $a$ and $b$ define embeddings $i' \colon Y' \to X'$ and $j' \colon X \to X'$, which realizes $(X',i',j')$ an amalgamation of $(i,j)$. We thus have $(X',i',j') \cong (X'_{\alpha}, i'_{\alpha}, j'_{\alpha})$ for some $\alpha$, and the inclusion $X' \subset \Gamma$ translates to an embedding $c \colon X'_{\alpha} \to \Gamma$. We put $\psi(a,b)=c$. One readily checks that $\phi$ and $\psi$ are indeed mutually inverse, which completes the proof.
\end{proof}

\begin{proof}[Proof of Theorem~\ref{thm:R-approx}]
Since $\fA$ is smoothly approximable, we have $\fA=\bigcup_{\Gamma} \fA_{\Gamma}$, where the union is over the homogeneous structures $\Gamma \in \fA$. This union is directed: if $\Gamma$ and $\Gamma'$ are two homogeneous structures then they embed into a common structure (since $\fA$ is a Fra\"iss\'e class), which embeds into some homogeneous structure $\Gamma''$ (since $\fA$ is smoothly approximable), and so $\fA_{\Gamma}$ and $\fA_{\Gamma'}$ are each contained in $\fA_{\Gamma''}$. Since each $\fA_{\Gamma}$ admits a regular measure (Lemma~\ref{lem:S-reg}), $\fA$ admits an ultralimit measure that is regular (\S \ref{ss:ultra}).
\end{proof}

In fact, we can even say a bit more. Say that $X \in \fA$ is \defn{mobile} if for every integer $N$ there exists some $Y \in \fA$ such that $\# h_Y(X)>N$.

\begin{proposition} \label{prop:mobile}
Suppose that $\fA$ is smoothly approximable and $X \in \fA$ is mobile. Then the ring homomorphism $f \colon \bQ[x] \to \Theta^*(\fA) \otimes \bQ$ given by $x \mapsto [X]$ is injective.
\end{proposition}

\begin{proof}
By assumption, for an integer $N$ there is some $Y \in \fA$ such that $\# h_Y(X)>N$. Since $Y$ embeds into some homogeneous structure $\Gamma$, we also have $\# h_{\Gamma}(X)>N$. Regard $h_{\Gamma}$ as $\bQ$-valued measure on $\fA_{\Gamma}$. Let $\mu$ be the ultralimit measure, which is a regular measure valued in $\bQ^*$ (the ultrapower of $\bQ$). From the above, we see that $\mu(X) \in \bQ^*$ is transcendental. Thus $\mu \circ f$ is injective, and so $f$ is injective.
\end{proof}

\begin{example}
Let $\fA$ be the class of finite sets. Then $\fA$ is smoothly approximable: indeed, every finite set is homogeneous. Theorem~\ref{thm:R-approx} thus shows that $\Theta^*(\fA)$ is non-zero. The set $X=\{1\}$ is mobile since it has $>N$ embeddings into the set $\{1, \ldots, N+1\}$. Thus by Proposition~\ref{prop:mobile}, we see that the map $\bQ[x] \to \Theta^*(\fA) \otimes \bQ$ given by $x \mapsto [X]$ is injective. Of course, we already knew this from Example~\ref{ex:set-meas}, but the proof here is much easier.
\end{example}

\part{Permutation representations} \label{part:perm}

\section{Matrix algebra} \label{s:matrix}

\subsection{Overview}

An $n \times m$ matrix with entries in a ring $k$ can be thought of as a function $[n] \times [m] \to k$. If $Y$ and $X$ are finitary $\hat{G}$-sets then, by analogy, we define a $Y \times X$ matrix to be a Schwartz function $Y \times X \to k$. It turns out that many familiar ideas from linear algebra---including matrix multiplication, trace, determinant, and eigenvalues---can be adapted to this setting, if we have a measure.

The purpose of \S \ref{s:matrix} is to develop these ideas. This theory will play a significant role in our work on the representation categories $\uPerm(G)$ and $\uRep(G)$. The most important result in \S \ref{s:matrix} is Theorem~\ref{thm:tr-nilp}, which shows that nilpotent matrices have vanishing trace (under a certain hypothesis).  This result is used to show that certain endomorphism algebras in $\uPerm(G)$ are semi-simple, which is an important step in the proof of Theorem~\ref{thm:regss}.

We fix a pro-oligomorphic group $G$ and a $k$-valued measure $\mu$ throughout \S \ref{s:matrix}.

\subsection{Matrices and vectors}

The following is the main concept studied in \S \ref{s:matrix}:

\begin{definition}
Let $X$ and $Y$ be finitary $\hat{G}$-sets. A \defn{$Y \times X$ matrix} $A$ with entries in $k$ is a Schwartz function $A \colon Y \times X \to k$.
\end{definition}

Of course, in a sense this definition does not introduce anything new: matrices are the same mathematical objects as Schwartz functions. However, having a different name and notation will be helpful psychologically and in providing context.

We now introduce a few more notations, conventions, and definitions:
\begin{itemize}
\item We write $\Mat_{Y,X}(k)$, or simply $\Mat_{Y,X}$, for the set of all $Y \times X$ matrices. By convention, whenever we write $\Mat_{Y,X}$ we mean that $X$ and $Y$ are finitary $\hat{G}$-sets. We write $\Mat_X$ in place of $\Mat_{X,X}$.
\item Suppose $U$ is a group of definition for $X$ and $Y$. We say that a matrix $A \in \Mat_{Y,X}$ is \defn{$U$-invariant} if $A(gx,gy)=A(x,y)$ for all $g \in U$, $x \in X$, $y \in Y$. We write $\Mat_{Y,X}^U$ for the set of $U$-invariant matrices. We note that $\Mat^U_{Y,X}$ can be identified with the space of functions $U \backslash (Y \times X) \to k$, and is thus a free $k$-module of finite rank.
\item We define the \defn{identity matrix} $I_X \in \Mat_X$ by $I_X(x,y)=1_{\Delta}(x,y)$, where $\Delta \subset X \times X$ is the diagonal. We sometimes write~1 in place of $I_X$.
\item For $A \in \Mat_{Y,X}$, we define the \defn{tranpose} of $A$, denoted $A^t$, to be the element of $\Mat_{X,Y}$ given by $A^t(x,y)=A(y,x)$.
\item We write $\Vec_X$ in place of $\Mat_{X,\bone}$. We think of $\Vec_X$ as the space of column vectors of height $X$. It is sometimes also helpful to think of elements of $\Vec_X$ as infinite linear combinations of elements of $X$. We note that $\Vec_X$ is just another name for $\cC(X)$.
\end{itemize}

One of the most interesting features about matrices in classical linear algebra is matrix multiplication. We now extend this to our setting. Let $B \in \Mat_{Y,X}$ and $A \in \Mat_{Z,Y}$. We define the \defn{product matrix} $AB \in \Mat_{Z,X}$ by
\begin{displaymath}
(AB)(z,x) = \int_Y A(z,y) B(y,x) dy.
\end{displaymath}
One can view $AB$ as the push-forward of the integrand appearing above along the projection map $Z \times Y \times X \to Z \times X$, and so by general properties of push-forwards (specifically Proposition~\ref{prop:push-schwartz}) it is indeed a Schwartz function on $Z \times X$.

\begin{proposition} \label{prop:matrix-mult}
Matrix multiplication is associative (when defined), and the identity matrix is the identity for multiplication (when defined).
\end{proposition}

\begin{proof}
Let $A$ and $B$ be as above, and let $C \in \Mat_{X,W}$. Then
\begin{align*}
(A(BC))(z,w)
&= \int_Y A(z,y) \cdot \left( \int_X B(y,x) C(x,w) dx \right) dy \\
&= \int_Y \int_X A(z,y) B(y,x) C(x,w) dx dy \\
&= \int_{X \times Y} A(z,y) B(y,x) C(x,w) d(x,y),
\end{align*}
where in the final step we used Fubini's theorem (Corollary~\ref{cor:fubini}). By a similar computation, one sees that $((AB)C)(z,w)$ is equal to the last line above as well, which proves associativity.

Let $C \in \Mat_{X,W}$ be as above. Then
\begin{displaymath}
(I_X \cdot C)(x,w) = \int_X 1_{\Delta}(x,x') C(x',w) dx' = C(x,w).
\end{displaymath}
Thus $I_X \cdot C=C$. A similar computation shows that $B \cdot I_X=B$ for $B$ as above.
\end{proof}

\begin{corollary}
$\Mat_X$ is an associative unital algebra under matrix multiplication.
\end{corollary}

\begin{example} \label{ex:matrix}
Consider the permutation representation $\bC^n$ of the symmetric group $\fS_n$. Let $A_n \colon \bC^n \to \bC^n$ be the map sending every basis vector to the invariant vector $\sum_{i=1}^n e_i$. The matrix for $A_n$ has every entry equal to~1. We have $A_n^2 = nA_n$. For $n \ge 2$, the space $\End_{\fS_n}(\bC^n)$ is two-dimensional, with basis~1 and $A_n$.

Now consider the infinite symmetric group $\fS$ acting on $\Omega=\{1,2,\ldots\}$, and let $\mu_t$ be the measure from Example~\ref{ex:sym-measure}. Let $A \in \Mat_{\Omega}$ be the constant function $A(x,y)=1$. Then, as above, $A^2=tA$, and the matrices~1 and $A$ form a basis for $\Mat_{\Omega}^{\fS}$. One can view $A$ as a sort of limit of the $A_n$'s.
\end{example}

For $A \in \Mat_{Y,X}$ and $v \in \Vec_X$, we can consider the matrix-vector product $Av \in \Vec_Y$; this is defined simply by treating $v$ as an $X \times \bone$ matrix. This product defines a map
\begin{displaymath}
A \colon \Vec_X \to \Vec_Y, \qquad v \mapsto Av.
\end{displaymath}
As usual, this function determines the matrix $A$:

\begin{proposition} \label{prop:matrix-faithful}
Let $A,B \in \Mat_{Y,X}$ be two matrices such that $Av=Bv$ for all $v \in \Vec_X$. Then $A=B$.
\end{proposition}

\begin{proof}
Suppose $v$ is the point mass at $x$. Then
\begin{displaymath}
(Av)(y) = \int_X A(y,x') v(x') dx' = A(y,x).
\end{displaymath}
We thus see that $Av$ is the ``$x$th column'' of $A$. Thus the function $A \colon \Vec_X \to \Vec_Y$ determines all the columns of $A$, and thus $A$ itself.
\end{proof}

\subsection{Traces} \label{ss:matrix-trace}

Let $A \in \Mat_X$. We define the \defn{trace} of $A$ by
\begin{displaymath}
\tr(A) = \int_X A(x,x) dx.
\end{displaymath}
Trace has the usual symmetry property:

\begin{proposition} \label{prop:trace-product}
Let $A \in \Mat_{Y,X}$ and $B \in \Mat_{X,Y}$. Then $\tr(AB)=\tr(BA)$.
\end{proposition}

\begin{proof}
By definition,
\begin{displaymath}
(AB)(y,y')=\int_X A(y,x) B(x,y') dx,
\end{displaymath}
and so
\begin{displaymath}
\tr(AB) = \int_Y \int_X A(y,x) B(x,y) dx dy = \int_{X \times Y} A(y,x) B(x,y) d(x,y),
\end{displaymath}
where we used Fubini's theorem (Corollary~\ref{cor:fubini}) in the second step. A similar manipulation leads to an identical expression for $\tr(BA)$.
\end{proof}

We now define some higher analogs of trace. For $n \ge 1$, define $A^{(n)} \in \cC(X^{(n)})$ by
\begin{displaymath}
A^{(n)}(x) = \det(A(x_i,x_j))_{1 \le i,j \le n}
\end{displaymath}
where $x=\{x_1,\ldots,x_n\}$. One easily sees that the right side above is independent of the enumeration of $x$. We define the $n$th \defn{higher trace} of $A$ by
\begin{displaymath}
T_n(A) = \int_{X^{(n)}} A^{(n)}(x) dx.
\end{displaymath}
We have $T_0(A)=1$ (essentially by convention) and $T_1(A)=\tr(A)$. Intuitively, $T_n(A)$ is the trace of the $n$th exterior power of $A$. In characteristic~0, we have the expression for $T_n(A)$ in terms of traces of powers of $A$ as this intuition would predict:

\begin{proposition} \label{prop:T}
Suppose $\bQ \subset k$. For a partition $\lambda=(\lambda_1,\ldots,\lambda_r)$, let
\begin{displaymath}
T_{\lambda}(A) = \prod_{j \ge 1} \left[ \frac{1}{m_j!} \left( \frac{(-1)^{j+1} \tr(A^j)}{j} \right)^{m_j} \right],
\end{displaymath}
where $m_j$ is the multiplicity of $j$ in $\lambda$. Then
\begin{displaymath}
T_n(A) = \sum_{\lambda \vdash n} T_{\lambda}(A),
\end{displaymath}
where the sum is over all partitions of $n$.
\end{proposition}

\begin{proof}
Define $\phi \in \cC(X^n)$ by
\begin{displaymath}
\phi(x) = \det(A(x_i,x_j))_{1 \le i,j \le n}.
\end{displaymath}
On $X^{[n]}$, we see that $\phi$ coincides with the pullback of $A^{(n)}$, while off of $X^{[n]}$, we see that $\phi$ vanishes, as the determinant has a repeated column. For a permutation $\sigma \in \fS_n$, let $\phi_{\sigma} \in \cC(X^n)$ be the function given by
\begin{displaymath}
\phi_{\sigma}(x)= A(x_1,x_{\sigma(1)}) \cdots A(x_n, x_{\sigma(n)}).
\end{displaymath}
Expanding the determinant in $\phi$, we find
\begin{displaymath}
\phi(x) = \sum_{\sigma \in \fS_n} \sgn(\sigma) \phi_{\sigma}(x),
\end{displaymath}
and so integrating over $X^n$, we obtain
\begin{displaymath}
T_n(A) = \frac{1}{n!} \int_{X^n} \phi(x) dx = \sum_{\sigma \in \fS_n} \sgn(\sigma) \int_{X^n} \phi_{\sigma}(x) dx.
\end{displaymath}
Now, for $m \ge 0$, we have
\begin{displaymath}
\tr(A^m) = \int_{X^m} A(x_1, x_2) A(x_2, x_3) \cdots A(x_{m-1}, x_m) A(x_m, x_1) dx
\end{displaymath}
(for $m=2$, see the proof of Proposition~\ref{prop:trace-product}), from which one finds
\begin{displaymath}
\int_{X^n} \phi_{\sigma}(x) dx = \tr(A^{\lambda_1}) \cdots \tr(A^{\lambda_r})
\end{displaymath}
where $n=\lambda_1+\cdots+\lambda_r$ is the cycle type of $\sigma$. Putting this into our previous formula, we find
\begin{displaymath}
T_n(A) = \frac{1}{n!} \sum_{\lambda \vdash n} \left[ (-1)^{(1+\lambda_1)+\cdots+(1+\lambda_r)} \cdot \# C_{\lambda} \cdot \tr(A^{\lambda_1}) \cdots \tr(A^{\lambda_r}) \right],
\end{displaymath}
where the sum is over all partitions of $n$ and $C_{\lambda}$ is the conjugacy class of $\fS_n$ corresponding to $\lambda$. We have $\# C_{\lambda}=n!/\prod_{j \ge 1} (j^{m_j} m_j!)$, where $m_j$ is the multiplicity of $j$ in $\lambda$, which yields the stated formula.
\end{proof}

\subsection{Trace of nilpotents}

A fundamental property of nilpotent matrices, in classical linear algebra, is that they have vanishing trace. We now establish this for our matrices, under a hypothesis. This result will play an important role in proving semi-simplicity results later on (see Proposition~\ref{prop:semisimple-alg}). The following is the key observation:

\begin{proposition} \label{prop:tr-p}
If $k$ is a field of positive characteristic $p$ and $A \in \Mat_X$ then
\begin{displaymath}
\tr(A)^p = \tr(A^p).
\end{displaymath}
\end{proposition}

\begin{proof}
By Corollary~\ref{cor:meas-lift}, the measure $\mu$ is valued in $\bF_p$, and lifts uniquely to a measure $\tilde{\mu}$ valued in $\bZ_p$. Suppose $A$ takes on the values $a_1, \ldots, a_n \in k$. Put $R=\bZ_p[x_1, \ldots, x_n]$, and let $R \to k$ be the ring homomorphism satisfying $x_i \mapsto a_i$. Then $A$ lifts to a matrix $\tilde{A}$ with coefficients in $R$. For $n \ge 0$, we have the higher trace $T_n(\tilde{A}) \in R$, defined with respect to $\tilde{\mu}$.

Let $\lambda$ be a partition of $p$, and consider the quantity $T_{\lambda}(\tilde{A}) \in R \otimes \bQ$ defined in Proposition~\ref{prop:T}. One easily sees that the rational numbers appearing in this expression have denominators dividing $p!$, and so $p! \cdot T_{\lambda}(\tilde{A})$ belongs to $R$. In fact, if $\lambda$ is not $(1^p)$ or $(p)$ then there are no $p$'s in the denominators, and so $p! \cdot T_{\lambda}(\tilde{A}) \in pR$. Of course, $p! \cdot T_p(\tilde{A})$ also belongs to $pR$. Thus Proposition~\ref{prop:T} shows
\begin{displaymath}
p! \cdot T_{(1^p)}(\tilde{A}) + p! \cdot T_{(p)}(\tilde{A}) \in pR.
\end{displaymath}
Now, we have
\begin{displaymath}
p! \cdot T_{(1^p)}(\tilde{A}) = \tr(\tilde{A})^p, \qquad p! \cdot T_{(p)}(\tilde{A}) = (-1)^{p+1} (p-1)! \cdot  \tr(\tilde{A}^p).
\end{displaymath}
Applying the homomorphism $R \to k$ yields the stated result.
\end{proof}

\begin{corollary} \label{cor:tr-p}
If $k$ is as above and $A \in \Mat_X$ is nilpotent then $\tr(A)=0$.
\end{corollary}

\begin{proof}
Let $n$ be such that $A^{p^n}=0$. By Proposition~\ref{prop:tr-p}, we have
\begin{displaymath}
\tr(A)^{p^n} = \tr(A^{p^n}) = 0.
\end{displaymath}
Since $k$ is a field, it follows that $\tr(A)=0$
\end{proof}

We thus see that nilpotent matrices have vanishing trace in positive characteristic. We can extend this result to characteristic~0, provided that certain rings have enough maps to fields of positive characteristic. The precise condition we require is the following:

\begin{definition} \label{defn:P}
A commutative ring $R$ satisfies \defn{property~(P)} if for every finitely generated reduced $R$-algebra $R'$ we have $\bigcap \ker(\phi)=0$, where the intersection is over all homomorphisms $\phi$ from $R'$ to fields of positive characteristic. We say that a measure $\mu$ for $G$ satisfies (P) if $\mu(\Theta(G))$ is contained in a subring of $k$ that satisfies~(P).
\end{definition}

We make a few notes concerning this condition. First, if $\mu$ satisfies~(P) then any measure obtained from $\mu$ by extension of scalars satisfies~(P). In particular, if $\Theta(G)$ satisfies (P) then any measure satisfies (P). Second, any measure $\mu$ valued in a reduced ring of positive characteristic satisfies (P). And third, if $\mu$ satisfies~(P) then so does $\mu\vert_U$ for any open subgroup $U$, as $\mu(\Theta(U))$ is a subring of $\mu(\Theta(G))$.

The following is the most general result we are able to prove on traces of nilpotents:

\begin{theorem} \label{thm:tr-nilp}
Suppose $k$ is reduced, $\mu$ satisfies (P), and $A \in \Mat_X$ is nilpotent. Then $\tr(A)=0$.
\end{theorem}

\begin{proof}
Let $k_0$ be a subring of $k$ that contains $\mu(\Theta(G))$ and satisfies~(P), which exists by assumption. The matrix $A$ has only finitely many distinct entries. Let $k_1 \subset k$ be the $k_0$-subalgebra generated by these entries. We can then regard $A$ as living in $\Mat_X(k_1)$. For any homomorphism $\phi \colon k_1 \to k'$ with $k'$ a field of positive characteristic, we have $\phi(\tr(A))=0$ by Corollary~\ref{cor:tr-p}. Thus $\tr(A)=0$ by definition of~(P).
\end{proof}

\begin{remark}
The second author has given an example of a regular measure that does not satisfy~(P) for which there exist nilpotent matrices of non-zero trace \cite[\S 1.2(c)]{Snowden3}.
\end{remark}

\begin{remark}
One can prove a version of the theorem for the higher traces $T_n$ too. More generally, one can define an analog of the characteristic polynomial of a matrix (using a version of determinant), and prove that it has a particular form. See \cite[\S 7]{arxiv} for details.
\end{remark}

\subsection{Trace of idempotents} \label{ss:tr-idemp}

We now examine the higher traces of idempotent matrices:

\begin{proposition} \label{prop:idemp-0}
Suppose $k$ is $\bZ$-flat, let $E \in \Mat_X$ be idempotent, and put $a=\tr(E)$. Then $T_n(E)=\binom{a}{n}$ holds for all $n \ge 0$. In particular, $\binom{a}{n}$ belongs to $k$ for all $n \ge 0$.
\end{proposition}

\begin{proof}
Since $k$ is $\bZ$-flat, we can check the identity in $\bQ \otimes k$. Since $E$ is idempotent, we have $\tr(E^n)=a$ for any $n \ge 1$. The result thus follows from Proposition~\ref{prop:T}, and the polynomial identity
\begin{displaymath}
\binom{x}{n} = \sum_{\lambda \vdash n}   \prod_{j \ge 1} \left[ \frac{1}{m_j!} \left( \frac{(-1)^{j+1} x}{j} \right)^{m_j} \right].
\end{displaymath}
To prove this identity, it suffices to check that the two sides agree when $x=d$ is a positive integer. This in fact follows from Proposition~\ref{prop:T} by taking $A$ to be the $d \times d$ identity matrix.
\end{proof}

\begin{proposition} \label{prop:idemp-p}
Suppose $k$ is a field of positive characteristic $p$, and let $E \in \Mat_X$ be idempotent. Then there exists a unique $a \in \bZ_p$ such that $T_n(E)=\binom{a}{n}$ for all $n \ge 0$.
\end{proposition}

\begin{proof}
We can verify this over an extension field of $k$, so we assume $k$ is algebraically closed. It follows that the ring $W(k)$ of Witt vectors is a complete DVR with uniformizer $p$. Let $U$ be a group of definition for $X$ such that $E$ is $U$-invariant. We have a surjective algebra homomorphism
\begin{displaymath}
\Mat_X^U(W(k)) \to \Mat_X^U(k).
\end{displaymath}
(Note: to work with $\Mat_X(W(k))$ we use the canonical lift of $\mu$ to $W(k)$, provided by Corollary~\ref{cor:meas-lift}.) Since $\Mat_X^U(W(k))$ is a free $W(k)$-module of finite rank, it is $p$-adically complete. Thus we can lift $E$ to an idempotent matrix $\tilde{E}$ in $\Mat_X^U(W(k))$ (see \cite[Theorem~21.28]{Lam} for lifting modulo $p^n$, then take a limit).

Let $a = \tr(\tilde{E})$. By Proposition~\ref{prop:idemp-0}, we have $\binom{a}{n} \in W(k)$ for all $n$. This implies $a \in \bZ_p$. For instance, letting $\ol{a}$ be the image of $a$ in $k=W(k)/(p)$, the identity $p\binom{a}{p}=a(a-1) \cdots (a-p+1)$ shows that $0=\ol{a} (\ol{a}-1) \cdots (\ol{a}-p+1)$, and so $\ol{a} \in \bF_p$. By Proposition~\ref{prop:idemp-0}, we also have $T_n(\tilde{E})=\binom{a}{n}$. Reducing modulo $p$ shows that $T_n(E)=\binom{a}{n}$.
\end{proof}

\begin{definition}
Let $E \in \Mat_X$ be idempotent with $k$ a field of characteristic $p$. We define the \defn{$p$-adic trace} of $E$ to be the value $a \in \bZ_p$ given by Proposition~\ref{prop:idemp-p}.
\end{definition}

\begin{remark}
The $p$-adic trace of an idempotent is very closely related to $p$-adic dimension in tensor categories, as defined in \cite{EtingofHarmanOstrik}. See Remark~\ref{rmk:abenv} for more discussion.
\end{remark}

\subsection{Invariant algebras}

We define the \defn{trace pairing} on the algebra $\Mat_X$ by
\begin{displaymath}
\langle A, B \rangle = \tr(AB).
\end{displaymath}
A simple computation (see the proof of Proposition~\ref{prop:trace-product}) yields the following explicit expression for the pairing:
\begin{displaymath}
\langle A, B \rangle = \int_{X \times X} A(x,y) B(y,x) d(x,y).
\end{displaymath}
The trace pairing is a symmetric bilinear form on $\Mat_X$. Assuming $X$ is a $G$-set, it induces a symmetric bilinear form on the invariant subalgebra $\Mat_X^G$. We now examine this.

\begin{proposition} \label{prop:trace-disc}
Let $X$ be a finitary $G$-set and let $Z_1, \ldots, Z_n$ be the $G$-orbits on $X \times X$. Then the discriminant of the trace pairing on $\Mat_X^G$ is
\begin{displaymath}
(-1)^r \cdot \prod_{i=1}^n \mu(Z_i)
\end{displaymath}
where $r$ is the number of transpose-conjugate pairs among the $Z_i$ (see proof). In particular, the trace pairing is perfect if and only if $\mu(Z_i)$ is a unit for all $i$.
\end{proposition}

\begin{proof}
Let $B_i \in \Mat_X^G$ be the matrix given by $B_i(x,y)=1_{Z_i}(x,y)$. Since the $1_{Z_i}$ clearly form a basis for the $G$-invariant functions on $X \times X$, we see that the $B_i$ form a basis for $\Mat_X^G$. Let $Z_i^t$ be the transpose of $Z_i$, i.e., the set of pairs $(y,x)$ with $(x,y) \in Z_i$. Clearly, $Z_i^t$ is a $G$-orbit on $X \times X$, and therefore equal to $Z_j$ for some $j$. Let $\tau$ be the involution of $\{1,\ldots,n\}$ defined by $Z_i^t=Z_{\tau(i)}$. We have
\begin{displaymath}
\langle B_i, B_j \rangle = \int_{X \times X} 1_{Z_i}(x,y) 1_{Z_j}(y,x) d(x,y) = \vol(Z_i \cap Z_{\tau(j)}) = \begin{cases}
\mu(Z_i) & \text{if $i=\tau(j)$} \\
0 & \text{otherwise.} \end{cases}
\end{displaymath}
We thus see that the matrix for the trace pairing is block diagonal: each fixed point $i$ of $\tau$ contributes a $1 \times 1$ block with entry $\mu(Z_i)$, while each pair $i \ne \tau(i)$ contributes the $2 \times 2$ block
\begin{displaymath}
\begin{pmatrix} 0 & \mu(Z_i) \\ \mu(Z_i) & 0 \end{pmatrix}.
\end{displaymath}
(Note that $Z_i$ and $Z_i^t$ have the same measure.) The determinant of this matrix is the sign of $\tau$ times the product of the $\mu(Z_i)$'s, which yields the claim.
\end{proof}

\begin{proposition} \label{prop:semisimple-alg}
Suppose $k$ is a field and let $X$ be a finitary $G$-set. Suppose that
\begin{enumerate}
\item Every nilpotent of $\Mat_X^G$ has trace~0.
\item For every $G$-orbit $Z$ on $X \times X$ we have $\mu(Z) \ne 0$.
\end{enumerate}
Then $\Mat_X^G$ is a semi-simple algebra.
\end{proposition}

\begin{proof}
Let $J$ be the Jacobson radical of $\Mat_X^G$. Every element of $J$ is nilpotent, and $\Mat_X^G$ is semi-simple if and only if $J=0$. If $A \in J$ and $B \in \Mat_X^G$ then $AB$ belongs to $J$ and is thus nilpotent, and so $\tr(AB)=0$ by Theorem~\ref{thm:tr-nilp}. Thus $\langle A, B \rangle=0$ for all $B$, and so $A=0$ since the pairing is non-degenerate (Proposition~\ref{prop:trace-disc}). Hence $J=0$, and so $\Mat_X^G$ is semi-simple.
\end{proof}

\subsection{Matrices associated to functions} \label{ss:alpha-beta}

Let $X$ and $Y$ be finitary $\hat{G}$-sets and let $f \colon X \to Y$ be a $\hat{G}$-equivariant function. We then have maps
\begin{displaymath}
f_* \colon \Vec_X \to \Vec_Y, \qquad f^* \colon \Vec_Y \to \Vec_X
\end{displaymath}
by identifying vectors with Schwartz functions. We now define matrices that induce these maps. These matrices will play a particularly important role in \S \ref{s:linear}.

Let $\Gamma(f) \subset X \times Y$ be the graph of $f$ and let $1_{\Gamma(f)} \in \cC(X \times Y)$ be the indicator function of $\Gamma(f)$. We define matrices
\begin{displaymath}
A_f \in \Mat_{Y,X}, \qquad B_f \in \Mat_{X,Y}
\end{displaymath}
by $A_f(y,x)=1_{\Gamma(f)}(x,y)$ and $B_f(x,y)=1_{\Gamma(f)}(x,y)$. Note that $B_f$ is the transpose of $A_f$. We now verify that these matrices have the desired properties:

\begin{proposition} \label{prop:C-alpha}
For $v \in \Vec_X$ and $w \in \Vec_Y$, we have
\begin{displaymath}
A_fv=f_*(v), \qquad B_fw=f^*(w).
\end{displaymath}
\end{proposition}

\begin{proof}
We have
\begin{displaymath}
(A_f v)(y) = \int_X 1_{\Gamma(f)}(x,y) v(x) dx = \int_{f^{-1}(y)} v(x) dx=(f_* v)(y),
\end{displaymath}
which proves the first statement. We have
\begin{displaymath}
(B_f w)(x) = \int_Y 1_{\Gamma(f)}(x,y) w(y) dy = w(f(x)),
\end{displaymath}
which proves the second.
\end{proof}

The following proposition gives the most important properties of these matrices:

\begin{proposition} \label{prop:alpha-prop}
We have the following:
\begin{enumerate}
\item Let $f \colon X \to Y$ and $g \colon Y \to Z$ be maps of finitary $\hat{G}$-sets. Then
\begin{displaymath}
A_g A_f = A_{gf}, \qquad
B_f B_g = B_{gf}.
\end{displaymath}
\item Consider a cartesian square of finitary $\hat{G}$-sets
\begin{displaymath}
\xymatrix{
X' \ar[r]^{g'} \ar[d]_{f'} & X \ar[d]^f \\
Y' \ar[r]^g & Y }
\end{displaymath}
Then we have the base change formula $B_g A_f=A_{f'} B_{g'}$.
\item Let $f \colon X \to Y$ be a map of transitive $G$-sets, and let $c$ be the common measure of a fiber. Then $A_f B_f=c \cdot I_Y$.
\item Let $X$ and $Y$ be finitary $\hat{G}$-set and let $i \colon X \to X \amalg Y$ and $j \colon Y \to X \amalg Y$ be the natural maps. Then
\begin{displaymath}
B_i A_i = I_X, \quad
B_j A_i = 0, \quad
B_i A_j = 0, \quad
B_j A_j = I_Y, \quad
A_iB_i+A_jB_j = I_{X \amalg Y}.
\end{displaymath}
\end{enumerate}
\end{proposition}

\begin{proof}
For each identity, it suffices to show that each side acts the same on vectors. Proposition~\ref{prop:C-alpha} then reduces the problem to an identity involving push-forwards and pull-backs. We now go through the details.

(a) We must show $g_* f_*=(gf)_*$ and $f^*g^*=(gf)^*$. The first is Proposition~\ref{prop:push-trans} and the second is obvious.

(b) We must show $g^*f_*=f'_* (g')^*$, which is exactly Proposition~\ref{prop:push-bc}.

(c) We must show $f_*(f^*(\phi))=c \cdot \psi$ for any $\psi \in \cC(Y)$. By the projection formula (Proposition~\ref{prop:projection}), we have $f_*(f^*(\psi))=\psi \cdot f_*(1_X)$. As $f_*(1_X)$ is clearly equal to $c \cdot 1_Y$, the result follows.

(d) We must show
\begin{displaymath}
i^*i_* = \id_{\cC(X)}, \quad
j^*i_*=0, \quad
i^*j_*=0, \quad
j^*j_* = \id_{\cC(Y)}, \quad
i_*i^*+j_*j^* = \id_{\cC(X \amalg Y)}.
\end{displaymath}
These are clear.
\end{proof}

Finally, we show that the $A$ and $B$ matrices generate all matrices.

\begin{proposition} \label{prop:hom-basis}
Let $X$ and $Y$ be finitary $G$-sets, let $Z_1, \ldots, Z_n$ be the $G$-orbits on $X \times Y$, let $C_i \in \Mat_{Y,X}^G$ be the matrix given by $C_i(y,z)=1_{Z_i}(x,y)$, and let $p_i \colon Z_i \to X_i$ and $q_i \colon Z_i \to Y_i$ be the projection maps. Then:
\begin{enumerate}
\item The matrices $C_i$ form a $k$-basis for $\Mat_{Y,X}^G$.
\item We have $C_i=A_{q_i} B_{p_i}$.
\end{enumerate}
\end{proposition}

\begin{proof}
It is clear that the functions $1_{Z_i}$ for a $k$-basis for $\cC(X \times Y)^G$, and so (a) follows. One easily verifies that the the two matrices in (b) act the same on vectors, and so the equality follows.
\end{proof}

\subsection{The relative setting} \label{ss:rel-matrix}

We make a few comments concerning the relative setting. The definitions of matrix, matrix multiplication, and trace go through unchanged. There is an issue, however, with the definition of the higher trace $T_n$ since it uses $X^{(n)}$, which need not exist in the relative setting. If the $X^{(n)}$ always exist---which is exactly the condition $(\ast)$ from Remark~\ref{rmk:rel-binom}---then everything goes through. Without this, some results are no longer true: e.g., we give an example of a nilpotent matrix with non-zero trace in \S \ref{ss:sym-rel}.

In characteristic~0, even without $(\ast)$, one can \emph{define} $T_n$ using the formula in Proposition~\ref{prop:T}. This allows one to define determinants and characteristic series. However, since our proof of Theorem~\ref{thm:tr-nilp} goes through positive characteristic, it does not apply in this setting in general.

\section{The category of permutation representations} \label{s:perm}

\subsection{Overview}

Fix a pro-oligomorphic group $G$ and a $k$-valued measure $\mu$ for $G$. In this section, we construct a $k$-linear rigid tensor category $\uPerm_k(G;\mu)$ of ``permutation modules'' of $G$. The motivation for our construction comes from the case of finite groups, as discussed in \S \ref{sss:intro-perm}. This construction is one of the main achievements of this paper.

We remind the reader to consult \S \ref{ss:notation} for our conventions on tensor categories. For general background on tensor categories, we refer to \cite{DeligneMilne} and \cite{Etingof}.

\subsection{The category}

The following definition introduces the main object of study in \S \ref{s:perm}:

\begin{definition} \label{defn:permcat}
We define a $k$-linear category $\uPerm_k(G; \mu)$ as follows:
\begin{itemize}
\item For each finitary $G$-set $X$ there is an object $\Vec_X$.
\item We let $\Hom(\Vec_X,\Vec_Y)=\Mat_{Y,X}^G$ be the space of $G$-invariant $Y \times X$ matrices.
\item Composition is given by matrix multiplication.
\end{itemize}
The properties of matrix multiplication established in Proposition~\ref{prop:matrix-mult} show that this is a well-defined category. When there is no danger of ambiguity, we write $\uPerm(G;\mu)$ or just $\uPerm(G)$ in place of $\uPerm_k(G;\mu)$.
\end{definition}

In $\uPerm(G)$, one can treat $\Vec_X$ either as a formal symbol or as the $k$-module $\cC(X)$. Similarly, one can treat a morphism $\Vec_X \to \Vec_Y$ simply as a matrix $A$, or as the $k$-linear map $\Vec_X \to \Vec_Y$ defined by $A$; the two perspectives are equivalent by Corollary~\ref{prop:matrix-faithful}. In any case, there is a faithful $k$-linear functor
\begin{displaymath}
\Phi \colon \uPerm(G) \to \Mod_k
\end{displaymath}
taking $\Vec_X$ to the Schwartz space $\cC(X)$ and a matrix to the linear map it defines.

\begin{proposition}
The category $\uPerm(G)$ is additive, with $\Vec_X \oplus \Vec_Y=\Vec_{X \amalg Y}$.
\end{proposition}

\begin{proof}
One easily sees that the usual matrices defining the inclusion $\Vec_X \to \Vec_{X \amalg Y}$ and projection $\Vec_{X \amalg Y} \to \Vec_X$ exist in our setting and satisfy the requisite properties. We note that the direct sum of two morphisms is given by the familiar block matrix.
\end{proof}

\subsection{Tensor structure} \label{ss:perm-tensor}

We now define a tensor product $\uotimes$ on $\uPerm(G)$. Let $X_i$, $Y_i$, and $Z_i$, for $i=1,2$, be finite $G$-sets. On objects, we define $\uotimes$ by
\begin{displaymath}
\Vec_{X_1} \uotimes \Vec_{X_2} = \Vec_{X_1 \times X_2}.
\end{displaymath}
Suppose now we have matrices $A_i \in \Mat_{Y_i,X_i}^G$ for $i=1,2$. We define
\begin{displaymath}
A_1 \uotimes A_2 \in \Mat_{Y_1 \times Y_2, X_1 \times X_2}
\end{displaymath}
to be the matrix given by
\begin{displaymath}
(A_1 \uotimes A_2)(y_1,y_2,x_1,x_2) = A_1(y_1,x_1) A_2(y_2,x_2).
\end{displaymath}
Note that this is a direct analog of the Kronecker product for ordinary matrices. We now show that this construction has the expected properties:

\begin{proposition}
With the above definitions, $\uotimes$ gives $\uPerm(G)$ the structure of a tensor category. The unit object is $\bbone=\Vec_{\bone}$.
\end{proposition}

\begin{proof}
We first show that $\uotimes$ is a bi-functor. It is clear from the definition that $I_{X_1} \uotimes I_{X_2}=I_{X_1 \times X_2}$, and so $\uotimes$ is compatible with identity morphisms. Now let $B_i \in \Mat_{Y_i \times X_i}^G$ and $A_i \in \Mat_{Z_i \times Y_i}^G$ for $i=1,2$ be given. We then have
\begin{align*}
&((A_1 \uotimes A_2)(B_1 \uotimes B_2))(z_1,z_2,x_1,x_2) \\
=& \int_{Y_1 \times Y_2} (A_1 \uotimes A_2)(z_1,z_2,y_1,y_2) (B_1 \uotimes B_2)(y_1,y_2,x_1,x_2) d(y_1, y_2) \\
=& \int_{Y_1 \times Y_2} A_1(z_1,y_1) A_2(z_2,y_2) B_1(y_1,x_1) B_2(y_2,x_2) d(y_1, y_2) \\
=& \bigg( \int_{Y_1} A_1(z_1,y_1) B_1(y_1,x_1) dy_1 \bigg) \cdot \bigg( \int_{Y_2} A_2(z_2,y_2) B_2(y_2,x_2) dy_2 \bigg) \\
=& ((A_1B_1) \uotimes (A_2B_2))(x_1,x_2,z_1,z_2),
\end{align*}
where in the fourth line we used Fubini's theorem (Corollary~\ref{cor:fubini}). We thus find
\begin{displaymath}
(A_1 \uotimes A_2)(B_1 \uotimes B_2) = (A_1B_1) \uotimes (A_2B_2),
\end{displaymath}
whic shows that $\uotimes$ is compatible with composition. We have thus shown that $\uotimes$ is a bi-functor. It is clear from the definition that $\uotimes$ is $k$-bilinear.

To complete the proof, we must construct the associativity, commutativity, and unital constraints on $\uotimes$. These are induced from the corresponding structure on the cartesian product $\times$ on the category $\cS(G)$ of finitary $G$-sets. We omit the details.
\end{proof}

\begin{remark}
There is a natural injective map of $k$-modules
\begin{displaymath}
\cC(X) \otimes_k \cC(Y) \to \cC(X \times Y) = \cC(X) \uotimes \cC(Y).
\end{displaymath}
This map is usually not an isomorphism. Thus the forgetful functor $\uPerm(G) \to \Mod_k$ is lax monoidal, but not monoidal.
\end{remark}

\begin{remark} \label{rmk:Bell}
Suppose $X$ is an infinite $G$-set. Let $A=\{A_1, \ldots, A_r\}$ be a partition of the set $[n]$, i.e., the $A_i$ are disjoint non-empty subsets of $[n]$ whose union is $[n]$. Consider the subset $X^n(A)$ of $X^n$ consisting of those tuples $(x_1, \ldots, x_n)$ such that $x_i=x_j$ if and only if $i$ and $j$ belong to the same part of $A$. This is non-empty since $X$ is infinite. We have a decomposition
\begin{displaymath}
X^n = \coprod_A X^n(A).
\end{displaymath}
Thus the number of orbits on $X^n$ is at least the number of set-partitions of $[n]$, which the Bell number $B_n$. It follows that $\Vec^{\uotimes n}_X=\Vec_{X^n}$ decomposes into a direct sum of at least $B_n$ non-zero subobjects. Since $B_n$ grows super-exponentially, we see that the tensor powers of $\Vec_X$ also grow at this rate (in the sense of direct sum decompositions just discussed).
\end{remark}

\subsection{Duality} \label{ss:perm-dual}

Recall that if $X$ is an object in a tensor category, a \defn{dual} of $X$ is an object $Y$ together with maps
\begin{displaymath}
\ev \colon X \otimes Y \to \bbone, \qquad
\cv \colon \bbone \to X \otimes Y,
\end{displaymath}
called \defn{evaluation} and \defn{co-evaluation}, which satisfy certain conditions (see \cite[Definition~2.10.1]{Etingof}). If $X$ admits a dual then it is unique (up to canonical isomorphism), and denoted $X^{\vee}$; in this case, $X$ is called \defn{rigid}. The tensor category itself is called \defn{rigid} if every object admits a dual.

We now investigate duality in $\uPerm(G)$. For a finitary $G$-set $X$, define matrices
\begin{displaymath}
\ev_X \in \Mat_{X \times X, \bone}, \qquad \cv_X \in \Mat_{\bone, X \times X}
\end{displaymath}
by
\begin{displaymath}
\ev_X((x,y), \ast) = \cv_X(\ast, (x,y))=1_{\Delta}(x,y),
\end{displaymath}
where $\Delta \subset X \times X$ is the diagonal.

\begin{proposition} \label{prop:perm-rigid}
Let $X$ be a finitary $G$-set. Then $\Vec_X$ is self-dual, with evaluation and co-evaluation maps given by $\ev_X$ and $\cv_X$.
\end{proposition}

\begin{proof}
We must show that the composition
\begin{displaymath}
\xymatrix@C=3em{
\Vec_X \ar[r]^-{\cv \uotimes \id} &
\Vec_{X \times X \times X} \ar[r]^-{\id \uotimes \ev} &
\Vec_X }
\end{displaymath}
is the identity. The matrices $\id \uotimes \ev$ and $\cv \uotimes \id$ are the functions on $X^4$ which map $(x_1,x_2,x_3,x_4)$ to $1_{\Delta}(x_2,x_3) 1_{\Delta}(x_1,x_4)$ and $1_{\Delta}(x_1,x_4) 1_{\Delta}(x_2,x_3)$. The product matrix $A \in \Mat_{X,X}$ is given by
\begin{displaymath}
A(y,z) = \int_{X \times X \times X} 1_{\Delta}(y,x_4) 1_{\Delta}(x_2,x_3) 1_{\Delta}(x_3, x_4) 1_{\Delta}(x_2, z)  dx_2 dx_3 dx_4.
\end{displaymath}
One easily sees that $A=I_X$, as required.
\end{proof}

\begin{corollary}
$\uPerm(G)$ is a rigid tensor category.
\end{corollary}

For a morphism $f$ between rigid objects in a tensor category, there is a canonical dual morphism $f^{\vee}$. We now explicitly  determine how this works in $\uPerm(G)$.

\begin{proposition} \label{prop:perm-dual}
Let $X$ and $Y$ be finitary $G$-sets and let $A \in \Mat_{Y,X}$. The $A^{\vee}=A^t$, where $A^{\vee} \colon \Vec_Y \to \Vec_X$ denotes the morphism dual to $A$, and $A^t \in \Mat_{X,Y}$ denotes the transpose matrix of $A$.
\end{proposition}

\begin{proof}
By definition (see the discussion following \cite[Proposition~2.10.5]{Etingof}), $A^{\vee}$ is the following composition
\begin{displaymath}
\xymatrix@C=5em{
\Vec_Y \ar[r]^-{\id \uotimes \cv} &
\Vec_{Y \times X \times X} \ar[r]^-{\id \uotimes A \uotimes \id} &
\Vec_{Y \times Y \times X} \ar[r]^-{\ev \uotimes \id} & \Vec_X }
\end{displaymath}
where we have omitted subscripts for readability. Call the above maps $B_1$, $B_2$, and $B_3$. Thus $A^{\vee}=B_3B_2B_1$. Letting $\Delta$ be the diagonal of $X$ and $\Gamma$ the diagonal of $Y$, we have
\begin{align*}
B_1(y_1,x_1,x_2;y_2) &= 1_{\Gamma}(y_1,y_2) 1_{\Delta}(x_1,x_2) \\
B_2(y_1,y_2,x_1;y_3,x_2,x_3) &= 1_{\Gamma}(y_1,y_3) A(y_2,x_2) 1_{\Delta}(x_1,x_3) \\
B_3(x_1;y_1,y_2,x_2) &= 1_{\Gamma}(y_1,y_2) 1_{\Delta}(x_1,x_2).
\end{align*}
From the expression
\begin{displaymath}
A^{\vee}(x,y) = \int_{Y^3 \times X^3} B_3(x;y_2,y_3,x_3) B_2(y_2,y_3,x_3;y_1,x_1,x_2),B_1(y_1,x_1,x_2;y) dy_i dx_i
\end{displaymath}
one easily sees that $A^{\vee}=A^t$, which completes the proof.
\end{proof}

\begin{remark}
In a rigid tensor category, one can define internal Hom, denoted $\uHom$, via the usual tensor-Hom adjunction (see \cite[Definition~1.6]{DeligneMilne}). In $\uPerm(G)$, this is given by
\begin{displaymath}
\uHom(\Vec_X, \Vec_Y) = \Vec_X^{\vee} \uotimes \Vec_Y = \Vec_{X \times Y}. \qedhere
\end{displaymath}
\end{remark}

\subsection{Trace}

Recall that there is a notion of trace for an endomorphism of a rigid object in a tensor category (see \cite[(1.7.3)]{DeligneMilne} or \cite[\S 4.7]{Etingof}); we refer to this as the \defn{categorical trace}, and denote it $\utr$. Categorical trace plays an important role in the study of a tensor category, and so it is important to understand well. The following proposition computes the categorical trace for $\uPerm(G)$.

\begin{proposition} \label{prop:trace-formula}
Let $X$ be a finitary $G$-set and let $A \in \Mat_X^G$. Then $\utr(A)=\tr(A)$. That is, the categorical trace of $A$ as an endomorphism of $\Vec_X$ is equal to the trace of the matrix $A$ as defined in \S \ref{ss:matrix-trace}.
\end{proposition}

\begin{proof}
By definition, the categorical trace of $A$ is the unique element $a$ of $k$ such that the composition
\begin{displaymath}
\xymatrix@C=4em{
\bbone \ar[r]^-{\cv} &
\Vec_{X \times X} \ar[r]^-{A \uotimes \id} &
\Vec_{X \times X} \ar[r]^-{\ev} &
\bbone }
\end{displaymath}
is multiplication by $a$. Identifying $\bone \times \bone$ matrices with scalars, we therefore have
\begin{align*}
\utr(A)
&= \ev_X \cdot (A \uotimes \id) \cdot \cv_X \\
&=\int_{X^4} 1_{\Delta}(x_1,x_2) A(x_1,x_3) 1_{\Delta}(x_2,x_4) 1_{\Delta}(x_3,x_4) dx_1 dx_2 dx_3 dx_4 \\
&= \int_X A(x,x) dx = \tr(A),
\end{align*}
which completes the proof.
\end{proof}

\begin{corollary} \label{cor:cat-tr-nilp}
Suppose $\mu$ satisfies (P) (see Definition~\ref{defn:P}) and $k$ is reduced. Then any nilpotent endomorphism of $\Vec_X$ has categorical trace equal to zero.
\end{corollary}

\begin{proof}
Since $\utr(A)=\tr(A)$ for $A \in \Mat_X^G$, this follows from Theorem~\ref{thm:tr-nilp}.
\end{proof}

Recall that the \defn{categorical dimension} of a rigid object is the categorical trace of its identity map; it is an element of the coefficient ring $k$. We have:

\begin{corollary} \label{cor:cat-dim}
The categorical dimension of $\Vec_X$ is $\mu(X)$.
\end{corollary}

\begin{remark} \label{rmk:abenv}
We now explain the significance of the above results. Let $k$ be a field. Suppose that $\cC$ is a $k$-linear abelian rigid tensor category in which all objects have finite length. Then $\cC$ has the following two features:
\begin{enumerate}
\item The categorical trace of any nilpotent endomorphism is~0.
\item If $k$ has positive characteristic $p$ then the categorical dimension of any object belongs to $\bF_p$ \cite[Exercise~9.9.9(ii)]{Etingof}, and admits a canonical lift to $\bZ_p$ \cite[Definition~2.7]{EtingofHarmanOstrik} (specifically, we have in mind the exterior $p$-adic dimension).
\end{enumerate}
We thus see that if $\cC$ is a $k$-linear additive rigid tensor category then (a) and (b) are necessary conditions for $\cC$ to admit a ``good'' abelian envelope. Corollary~\ref{cor:cat-tr-nilp} shows that $\uPerm(G)$ satisfies (a), at least assuming~(P). Corollary~\ref{cor:cat-dim} and Theorem~\ref{thm:binom} show that it satisfies (b); better, Proposition~\ref{prop:idemp-p} shows that the Karoubian envelope of $\uPerm(G)$ satisfies (b). We therefore know no obstruction to $\uPerm(G)$ admitting a good abelian envelope. In Part~\ref{part:genrep}, we construct such an envelope under some conditions on $\mu$. Note that the comments here do not apply in the relative case, in general.
\end{remark}

\section{Classification of linearizations} \label{s:linear}

\subsection{Overview}

Let $G$ be a pro-oligomorphic group. The category $\cS(G)$ of finitary smooth $G$-sets can be thought of as ``non-linear tensor category,'' and our category $\uPerm(G; \mu)$ can be seen as a linear version of it. The purpose of \S \ref{s:linear} is to make this idea precise.

We introduce a rigorous notion of a ``linearization'' of $\cS(G)$ (Definition~\ref{defn:linear}). Roughly speaking, a linearization is a tensor category $\cT$ with the same objects as $\cS(G)$, and where disjoint union and cartesian product in $\cS(G)$ are transformed into direct sum and tensor product in $\cT$. The precise definition is a little more technical, but it is important to underscore that the definition is rather weak (in the sense that it demands little) and, in particular, does not make any reference to the notion of measure or the ring $\Theta(G)$.

Our main theorem (Theorem~\ref{thm:linear}) asserts that every linearization of $\cS(G)$ is of the form $\uPerm(G; \mu)$ for some measure $\mu$. The theorem can be reformulated as the following moduli-theoretic description of $\Theta(G)$ (see Corollary~\ref{cor:linear}):
\begin{displaymath}
\textit{$\Spec{\Theta(G)}$ is the space of linearizations of $\cS(G)$.}
\end{displaymath}
We therefore see that the concept of measure and the ring $\Theta(G)$ are, in a sense, inevitable. For this reason, we view Theorem~\ref{thm:linear} as philosophically important. It is also practically important: it is central to our classification of discrete pre-Tannakian categories in \cite{discrete}.

The proofs of the results in \S \ref{s:linear} are rather lengthy. We therefore state the main definitions and results in \S \ref{ss:balanced}--\ref{ss:linear}, and defer the proofs to the end of the section.

\subsection{Balanced functors} \label{ss:balanced}

Our first goal is to describe the $k$-linear functors out of $\uPerm(G; \mu)$ in terms of the category $\cS(G)$. We have two natural functors from $\cS(G)$ to $\uPerm(G; \mu)$: on objects, each takes $X$ to $\Vec_X$, but on morphisms one takes $f$ to $A_f$ and the other (which is contravariant) takes $f$ to $B_f$ (see \S \ref{ss:alpha-beta}). Thus if we have a $k$-linear functor $\uPerm(G; \mu) \to \cT$, for some $k$-linear category $\cT$, we obtain, by composition, two functors from $\cS(G)$ to $\cT$. This is the motivation for Definition~\ref{defn:balanced} below.

Before giving the definition, we introduce some notation. For a category $\cC$, we let $\cC^{\circ}$ be the maximal groupoid in $\cC$, i.e., the category with the same objects but with morphisms being isomorphisms in $\cC$. We note that there is a canonical isomorphism of categories $\cC^{\circ}=(\cC^{\circ})^{\op}$.

\begin{definition} \label{defn:balanced}
A \defn{balanced functor} $\cS(G) \to \cT$ is a pair $(\Phi, \Phi')$ of functors $\Phi \colon \cS(G) \to \cT$ and $\Phi' \colon \cS(G)^{\op} \to \cT$ that have \emph{equal} restriction to $\cS(G)^{\circ}=(\cS(G)^{\circ})^{\op}$.
\end{definition}

Explicitly, to give a balanced functor we must give for each finitary $G$-set $X$ an object $\Phi(X)$ of $\cT$, and for each $G$-morphism $f \colon X \to Y$ two morphisms
\begin{displaymath}
\alpha_f \colon \Phi(X) \to \Phi(Y), \qquad
\beta_f \colon \Phi(Y) \to \Phi(X).
\end{displaymath}
Here we write $\alpha_f$ in place of $\Phi(f)$ and $\beta_f$ in place of $\Phi'(f)$; we will prefer this notation in much of what follows. Formation of $\alpha$ and $\beta$ must be compatible with composition, and we must have $\alpha_f=\beta_f^{-1}$ whenever $f$ is an isomorphism. The notion of isomorphism of balanced functor is evident.

\begin{example} \label{ex:balanced}
Let $\mu$ be a $k$-valued measure for $G$. We then have a balanced functor
\begin{displaymath}
\Phi_{\mu} \colon \cS(G) \to \uPerm(G; \mu)
\end{displaymath}
given on objects by $\Phi_{\mu}(X)=\Vec_X$ and on morphisms by $\alpha_f=A_f$ and $\beta_f=B_f$, where $A$ and $B$ are as in \S \ref{ss:alpha-beta}. This is the motivating example.
\end{example}

We now introduce a few conditions on a balanced functor $\Phi$:
\begin{itemize}
\item We say that $\Phi$ is \defn{additive} if the following holds. Let $X$ and $Y$ be finitary $G$-sets and let $i \colon X \to X \amalg Y$ and $j \colon Y \to X \amalg Y$ be the natural maps. Then we demand
\begin{displaymath}
\beta_i \alpha_i = \id_{\Phi(X)}, \quad \beta_i \alpha_j = 0, \quad
\beta_j \alpha_i = 0, \quad \beta_j \alpha_j = \id_{\Phi(Y)}
\end{displaymath}
\begin{displaymath}
\alpha_i \beta_i + \alpha_j \beta_j = \id_{\Phi(X \amalg Y)}.
\end{displaymath}
If this condition holds then $\Phi(X \amalg Y)$ is the direct sum of $\Phi(X)$ and $\Phi(Y)$, with $\alpha_i$ and $\alpha_j$ giving the inclusions and $\beta_i$ and $\beta_j$ the projections. This condition also implies $\Phi(\emptyset)=0$.
\item We say that $\Phi$ satisfies \defn{base change} if whenever
\begin{displaymath}
\xymatrix{
X' \ar[r]^{g'} \ar[d]_{f'} & X \ar[d]^f \\
Y' \ar[r]^g & Y }
\end{displaymath}
is a cartesian square in $\cS(G)$, we have $\beta_g \alpha_f = \alpha_{f'} \beta_{g'}$.
\item Let $\mu$ be a $k$-valued measure for $G$. We say that $\Phi$ is \defn{$\mu$-adapted} if whenever $f \colon X \to Y$ is a map of transitive $G$-sets we have $\alpha_f \beta_f = c \cdot \id_Y$, where $c$ is the common measure of a fiber of $f$.
\end{itemize}
The functor $\Phi_{\mu}$ introduced in Example~\ref{ex:balanced} is is additive, satisfies base change, and is $\mu$-adapted (Proposition~\ref{prop:alpha-prop}). In fact, it is the universal such functor, as the following proposition spells out. This proposition can also be viewed as a mapping property for $\uPerm(G; \mu)$ (as a $k$-linear category). The proof is given in \S \ref{ss:perm-maps}.

\begin{proposition} \label{prop:perm-maps}
Let $\Phi \colon \cS(G) \to \cT$ be a balanced functor that is additive, satisfies base change, and is $\mu$-adapted. Then there is a unique $k$-linear functor $\Psi \colon \uPerm(G; \mu) \to \cT$ such that $\Phi=\Psi \circ \Phi_{\mu}$ (actual equality).
\end{proposition}

\subsection{Monoidal balanced functors}

The category $\cS(G)$ is symmetric monoidal under cartesian product. Supposing $\cT$ is a tensor category, we can then consider balanced functors that respect the monoidal structures. Here is the precise definition:

\begin{definition}
A \defn{symmetric monoidal balanced functor} $\cS(G) \to \cT$ is a pair $(\Phi, \Phi')$ of symmetric monoidal functors $\Phi \colon \cS(G) \to \cT$ and $\Phi' \colon \cS(G)^{\op} \to \cT$ that have equal restriction to $\cS(G)^{\circ}=(\cS(G)^{\circ})^{\op}$, as monoidal functors.
\end{definition}

\begin{example}
The balanced functor $\Phi_{\mu} \colon \cS(G) \to \uPerm(G; \mu)$ carries a natural symmetric monoidal structure.
\end{example}

Fix a symmetric monoidal balanced functor $\Phi \colon \cS(G) \to \cT$ that is also additive. For a finitary $G$-set $X$, we define morphisms
\begin{displaymath}
\alpha_X \colon \Phi(X) \to \bbone, \qquad \beta_X \colon \bbone \to \Phi(X)
\end{displaymath}
in $\cT$ by $\alpha_X=\alpha_p$ and $\beta_X=\beta_p$, where $p \colon X \to \bone$ is the canonical map. The next two propositions show that $\Phi$ interacts well with duality. The proofs (and definition of Frobenius algebra) are given in \S \ref{ss:bifunc-frob}.

\begin{proposition} \label{prop:bifunc-frob}
Let $X$ be a finitary $G$-set, let $\Delta \colon X \to X \times X$ be the diagonal map, and put $X'=\Phi(X)$. Then the maps
\begin{displaymath}
\beta_X \colon \bbone \to X', \quad \beta_{\Delta} \colon X' \otimes X' \to X', \quad
\alpha_X \colon X' \to \bbone, \quad \alpha_{\Delta} \colon X' \to X' \otimes X'
\end{displaymath}
define the structure of a Frobenius algebra on $X'$. In particular, $X'$ is rigid and self-dual.
\end{proposition}

\begin{proposition} \label{prop:alpha-dual}
Let $f \colon X \to Y$ be a map of finitary $G$-sets, and put $X'=\Phi(X)$ and $Y'=\Phi(Y)$. Then $\alpha_f \colon X' \to Y'$ and $\beta_f \colon Y' \to X'$ are dual to one another, with respect to the self-dualities of $X'$ and $Y'$ coming from Proposition~\ref{prop:bifunc-frob}.
\end{proposition}

\subsection{The main theorem} \label{ss:linear}

We introduce one final condition on balanced functors:
\begin{itemize}
\item A symmetric monoidal balanced functor $\Phi \colon \cS(G) \to \cT$ is \defn{plenary} if for every transitive $G$-set $X$ the $k$-module $\Hom_{\cT}(\Phi(X), \bbone)$ is free of rank one with basis $\alpha_X$.
\end{itemize}
We now come to the key definition of \S \ref{s:linear}:

\begin{definition} \label{defn:linear}
A \defn{linearization} of $\cS(G)$ is a pair $(\cT, \Phi)$ consisting of a $k$-linear tensor category $\cT$ and a symmetric monoidal balanced functor $\Phi \colon \cS(G) \to \cT$ that is additive, plenary, and essentially surjective.
\end{definition}

We note that $\Phi_{\mu} \colon \cS(G) \to \uPerm(G;\mu)$ is a linearization of $\cS(G)$. The following is the main theorem of \S \ref{s:linear}, and proven in \S \ref{ss:linear-pf} below.

\begin{theorem} \label{thm:linear}
The linearizations of $\cS(G)$ are exactly the categories $\uPerm(G; \mu)$. More precisely, if $\Phi \colon \cS(G) \to \cT$ is a linearization then there exists a $k$-valued measure $\mu$ for $G$ and an equivalence of $k$-linear tensor categories
\begin{displaymath}
\Psi \colon \uPerm(G; \mu) \to \cT,
\end{displaymath}
such that we have $\Phi = \Psi \circ \Phi_{\mu}$ (actual equality). Both $\mu$ and $\Psi$ are unique.
\end{theorem}

As a consequence of the theorem, we see that the notion of linearization actually determines the ring $\Theta(G)$:

\begin{corollary} \label{cor:linear}
Consider the functor
\begin{displaymath}
\sL \colon \{ \text{commutative rings} \} \to \{ \text{sets} \}
\end{displaymath}
attaching to a commutative ring $k$ the set of equivalence classes of linearizations of $\cS(G)$ over $k$. Then $\sL$ is represented by $\Theta(G)$.
\end{corollary}

\begin{proof}
Given a linearization $(\Phi, \cT)$ of $\cS(G)$ over $k$ and a ring homomorphism $k \to k'$, there is a base change $(\Phi', \cT')$ which is a linearization over $k'$; to form $\cT'$, one simply applies $-\otimes_k k'$ to the $\Hom$ spaces in $\cT$. This base change operation induces map $\sL(k) \to \sL(k')$, which is how $\sL$ is a functor. Next, there is a natural map
\begin{displaymath}
\Hom(\Theta(G), k) \to \sL(k)
\end{displaymath}
taking a measure $\mu$ to the equivalence class of $\uPerm(G;\mu)$. This is easily seen to define a natural transformation. The theorem asserts that it is an isomorphism, which completes the proof.
\end{proof}

\subsection{Proof of Proposition~\ref{prop:perm-maps}} \label{ss:perm-maps}

Let $\cT$ be a $k$-linear category and let $\Phi \colon \cS(G) \to \cT$ be a balanced functor that is additive, satisfies base change, and is $\mu$-adapted. We must produce a $k$-linear functor $\Psi \colon \uPerm(G; \mu) \to \cT$ such that $\Phi=\Psi \circ \Phi_{\mu}$, and show that $\Psi$ is unique.

We begin by defining the functor $\Psi$. On objects, we put $\Psi(\Vec_X)=\Phi(X)$. Let $X$ and $Y$ be finitary $G$-sets, let $Z_1, \ldots, Z_n$ be the $G$-orbits on $X \times Y$, let $p_i \colon Z_i \to X$ and $q_i \colon Z_i \to Y$ be the projections, and let $C_i=A_{q_i}B_{p_i}$. Recall (Proposition~\ref{prop:hom-basis}) that the $C_i$ form a $k$-basis for $\Hom(\Vec_X, \Vec_Y)$. We define
\begin{displaymath}
\Psi \colon \Hom(\Vec_X, \Vec_Y) \to \Hom_{\cT}(\Phi(X), \Phi(Y))
\end{displaymath}
to be the unique $k$-linear map taking $C_i$ to $\alpha_{q_i} \circ \beta_{p_i}$. We now verify that $\Psi$ is a functor. We proceed in four steps. We note that since $\Phi \colon \cS(G) \to \cT$ is arbitrary, whatever we prove about it (and the $\alpha$ and $\beta$ morphisms) also applies to $\Phi_{\mu}$ (and the $A$ and $B$ matrices).

\textit{Step 1.} We first show that $\Psi$ is compatible with identity morphisms. Let $X$ be a finitary $G$-set. The identity map of $\Vec_X$ in $\uPerm(G; \mu)$ is the identity matrix $I_X$. Let $X_1, \ldots, X_n$ be the $G$-orbits on $X$ and let $\Delta_i$ be the diagonal of $X_i$; note that the $\Delta_i$ are the $G$-orbits on $\Delta$. Let $p_i,q_i \colon \Delta_i \to X$ be the two projections. Then, by definition, we have $\Psi(I_X)=\sum_{i=1}^n \alpha_{q_i} \beta_{p_i}$. Let $j_i \colon X_i \to X$ be the inclusion. Write $p_i=j_i \ol{p}_i$, where $\ol{p}_i \colon \Delta_i \to X_i$ is the projection, and similarly write $q_i=j_i \ol{q}_i$. We have
\begin{displaymath}
\alpha_{q_i} \beta_{p_i} = \alpha_{j_i} \alpha_{\ol{q}_i} \beta_{\ol{p}_i} \beta_{j_i} = \alpha_{j_i} \beta_{j_i}.
\end{displaymath}
In the first step we used compatibility of $\alpha$ and $\beta$ with composition. In the second step, we used the identity $\alpha_f \beta_f=\id$ when $f$ is an isomorphism, applied to $f=\ol{p}_i=\ol{q}_i$. We thus see that $\Psi(I_X)=\sum_{i=1}^n \alpha_{j_i} \beta_{j_i}$. This is the identity of $\Phi(X)$ since $\Phi$ is additive.

\textit{Step 2.} Let $p \colon D \to X$ and $q \colon D \to Z$ be maps of finitary $G$-sets. Let $D_1, \ldots, D_n$ be the $G$-orbits on $D$, and let $p_i$ and $q_i$ be the restrictions of $p$ and $q$ to $D_i$. We claim that
\begin{displaymath}
\alpha_q \beta_p = \sum_{i=1}^n \alpha_{q_i} \beta_{p_i}.
\end{displaymath}
Let $j_i \colon D_i \to D$ be the inclusion. Since $\Phi$ is additive, we have $\id_{\Phi(D)}=\sum_{i=1}^n \alpha_{j_i} \beta_{j_i}$. We thus have
\begin{displaymath}
\alpha_q \beta_p = \alpha_q \cdot \id_{\Phi(D)} \cdot \beta_p = \sum_{i=1}^n \alpha_q \alpha_{j_i} \beta_{j_i} \beta_{p_i} = \sum_{i=1}^n \alpha_{q_i} \beta_{p_i},
\end{displaymath}
which proves the claim.

\textit{Step 3.} Suppose $X$ and $Z$ are finitary $G$-sets and we have $G$-equivariant maps $p \colon D \to X$ and $q \colon D \to Z$, with $D$ a transitive $G$-set. Let $\ol{D}$ be the image of $D$ in $X \times Z$, let $\ol{p} \colon D \to X$ and $\ol{q} \colon D \to Z$ be the projections, and let $f \colon D \to \ol{D}$ be the natural map. We thus have a commutative diagram
\begin{displaymath}
\xymatrix@C=4em{
& D \ar[ld]_p \ar[d]^f \ar[rd]^q \\
X & \ol{D} \ar[l]_-{\ol{p}} \ar[r]^-{\ol{q}} & Z }
\end{displaymath}
Since $D$ is transitive, so is $\ol{D}$, and so the fibers of $f$ all have the same measure under $\mu$; call this $c$. We have
\begin{displaymath}
\alpha_q \beta_p = \alpha_{\ol{q}} \alpha_f \beta_f \beta_{\ol{p}} = c \cdot \alpha_{\ol{q}} \beta_{\ol{p}}.
\end{displaymath}
In the first step, we used the compatibility of $\alpha$ and $\beta$ with composition, and in the second step we used that $\Phi$ is $\mu$-adapted.

\textit{Step 4.} We now show that $\Psi$ respects composition. Thus let $A \colon \Vec_X \to \Vec_Y$ and $B \colon \Vec_Y \to \Vec_Z$ be morphisms in $\uPerm(G;\mu)$. We show
\begin{displaymath}
\Psi(BA)=\Psi(B) \circ \Psi(A).
\end{displaymath}
Since $\Psi$ is $k$-linear on $\Hom$ spaces, it suffices to consider the case where $A$ and $B$ are basis vectors as in Proposition~\ref{prop:hom-basis}. We thus suppose $A(y,x)=1_E(x,y)$ and $B(z,y)=1_F(y,z)$, where $E \subset X \times Y$ and $F \subset Y \times Z$ are $G$-orbits. Let $D=E \times_Y F$ be the fiber product, and consider the following commutative diagram:
\begin{displaymath}
\xymatrix{
&& D \ar[ld]_{p_3} \ar[rd]^{q_3} \\
& E \ar[ld]_{p_1} \ar[rd]^{q_1} && F \ar[ld]_{p_2} \ar[rd]^{q_2} \\
X && Y && Z }
\end{displaymath}
Here the $p$'s and $q$'s are the standard projection maps. Put $p_4=p_1 p_3$ and $q_4=q_2q_3$. Let $D_1, \ldots, D_n$ be the $G$-orbits on $D$, let $p_{4,i} \colon D_i \to X$ be the restriction of $p_4$, and let $q_{4,i} \colon D_i \to Z$ be the restriction of $q_4$. Let $\ol{D}_i$ be the image of $D_i$ in $X \times Z$ (under $p_{4,i} \times q_{4,i}$), let $\ol{p}_{4,i} \colon \ol{D}_i \to X$ and $\ol{q}_{4,i} \colon \ol{D}_i \to Z$ be the projection maps, and let $c_i$ be the common measure of the fibers of $D_i \to \ol{D}_i$. We have
\begin{displaymath}
(\alpha_{q_2} \beta_{p_2}) (\alpha_{q_1} \beta_{p_1})
= \alpha_{q_2} \alpha_{q_3} \beta_{p_3}  \beta_{p_1}
= \alpha_{q_4} \beta_{p_4}
= \sum_{i=1}^n \alpha_{q_{4,i}} \beta_{p_{4,i}} = \sum_{i=1}^n c_i \alpha_{\ol{q}_{4,i}} \beta_{\ol{p}_{4,i}}.
\end{displaymath}
In the first step we used base change, in the second compatibility with composition, in the third the identity from Step~2, and in the fourth the identity from Step~3. The same statement holds for the $A$ and $B$ matrices, and shows that $BA = \sum_{i=1}^n c_i \cdot C_i$ where $C_i(z,x)=1_{\ol{D}_i}(x,z)$; note that this is the expression for $BA$ in the basis for $\Hom(\Vec_X, \Vec_Z)$. Compatibility with composition now follows. Indeed
\begin{align*}
\Psi(BA) &= \sum_{i=1}^n c_i \Psi(C_i) = \sum_{i=1}^n c_i \alpha_{\ol{q}_{4,i}} \beta_{\ol{p}_{4,i}} \\
\Psi(B) \circ \Psi(A) &= (\alpha_{q_2} \beta_{p_2})(\alpha_{q_1} \beta_{p_1}) = \sum_{i=1}^n c_i \alpha_{\ol{q}_{4,i}} \beta_{\ol{p}_{4,i}}.
\end{align*}
We have thus shown that $\Psi$ is a functor.

We now show that $\Phi=\Psi \circ \Phi_{\mu}$. On objects this is clear. On morphisms, we must show that $\alpha_f=\Psi(A_f)$ and $\beta_f=\Psi(B_f)$ for any morphism $f \colon X \to Y$ in $\cS(G)$. Since everything is additive, we can reduce to the case where $X$ and $Y$ are transitive. Let $\Gamma(f) \subset X \times Y$ be the graph of $f$ and let $p \colon \Gamma(f) \to X$ and $q \colon \Gamma(f) \to Y$ be the projections. Then $A_f(y,x)=1_{\Gamma(f)}(x,y)$ and so, by definition, we have $\Psi(A_f)=\alpha_q \beta_p$. As $f=qp^{-1}$, we have $\alpha_f=\alpha_q \alpha_{p^{-1}}=\alpha_q \beta_p$. Thus $\Psi(A_f)=\alpha_f$. The proof for $\beta$ is similar.

It is clear that $\Psi$ is the unique $k$-linear functor satisfying $\Phi=\Psi \circ \Phi_{\mu}$. Indeed, this condition determines $\Psi$ on objects and morphisms of the form $A_f$ and $B_f$. Since these morphisms generate all morphisms in $\uPerm(G; \mu)$, it follows that $\Psi$ is determined on all morphisms.

\subsection{Proofs of Propositions~\ref{prop:bifunc-frob} and~\ref{prop:alpha-dual}} \label{ss:bifunc-frob}

Fix a $k$-linear tensor category $\cT$ and a symmetric monoidal balanced functor $\Phi \colon \cS(G) \to \cT$ that is also additive. Before giving the proofs, we recall the notion of Frobenius algebra:

\begin{definition} \label{defn:frob}
A \defn{Frobenius algebra} in $\cT$ is an object $X$ equipped with maps
\begin{displaymath}
\eta \colon \bbone \to X, \quad \mu \colon X \otimes X \to X, \quad \epsilon \colon X \to \bbone, \quad \delta \colon X \to X \otimes X
\end{displaymath}
such that the following conditions hold:
\begin{enumerate}
\item $(X, \mu, \eta)$ is a commutative, associative, unital algebra object.
\item $(X, \delta, \epsilon)$ is a co-commutative, co-associative, co-unital co-algebra object.
\item We have
\begin{displaymath}
(\id_X \otimes \mu) \circ (\delta \otimes \id_X) = \delta \circ \mu = (\mu \otimes \id_X) \circ (\id_X \otimes \delta)
\end{displaymath}
as maps $X \otimes X \to X \otimes X$.
\item We have $\mu \circ \delta=\id_X$. \qedhere
\end{enumerate}
\end{definition}

We note that the above structure is more commonly called a ``special commutative Frobenius algebra.'' We have no need for more general Frobenius algebras, so we stick to this terminology. A Frobenius algebra is naturally self-dual: in the above notation, the evaluation map is given by $\epsilon \circ \mu$, and the co-evaluation map by $\delta \circ \eta$.

\begin{proof}[Proof of Proposition~\ref{prop:bifunc-frob}]
Let $X$ be a finitary $G$-set and put $X'=\Phi(X)$. We must show that the maps
\begin{displaymath}
\beta_X \colon \bbone \to X', \quad \beta_{\Delta} \colon X' \otimes X' \to X', \quad
\alpha_X \colon X' \to \bbone, \quad \alpha_{\Delta} \colon X' \to X' \otimes X'
\end{displaymath}
define the structure of a Frobenius algebra on $X'$. Here $\Delta$ is the diagonal in $X \times X$.

The object $X$ is a co-commutative, co-associative, co-unital co-algebra in $\cS(G)$, with co-multiplication $\Delta$ and co-unit $X \to \bbone$. Applying $\Phi$, we see that $\alpha_X$ and $\alpha_{\Delta}$ endow $X'$ with a similar structure. Applying $\Phi'$, we see that $\beta_X$ and $\beta_{\Delta}$ endow $X'$ with the dual structure.

We have a decomposition $X \times X = X \amalg X^{[2]}$; note that $\Delta$ identifies $X$ with the first factor on the left side. Similarly, we have a decomposition
\begin{displaymath}
X \times X \times X = X \amalg X^{[2]} \amalg X^{[2]} \amalg X^{[2]} \amalg X^{[3]},
\end{displaymath}
where the first $X^{[2]}$ on the right side is the subobject of the left side where the first and second coordinates are equal and the third is different; the second and third $X^{[2]}$'s are defined similarly, but with respect to coordinates $(1,3)$ and $(2,3)$. We thus have a commutative diagram
\begin{displaymath}
\xymatrix@C=4em{
\Phi(X \times X) \ar[r]^-{\alpha_{\Delta} \otimes 1} \ar@{=}[d] & \Phi(X \times X \times X) \ar[r]^-{1 \otimes \beta_{\Delta}} \ar@{=}[d] & \Phi(X \times X) \ar@{=}[d] \\
\Phi(X) \oplus \Phi(X^{[2]}) \ar[r] & \Phi(X) \oplus \Phi(X^{[2]})^{\oplus 3} \oplus \Phi(X^{[3]}) \ar[r] & \Phi(X) \oplus \Phi(X^{[2]}) }
\end{displaymath}
The lower left map is the natural inclusion corresponding to the first copy of $\Phi(X^{[2]})$ in the middle, while the lower right map is the natural projection corresponding to the third copy of $\Phi(X^{[2]})$. It follows that the composition on the bottom is the identity on $\Phi(X)$ and~0 on $\Phi(X^{[2]})$.

Now consider the following commutative diagram:
\begin{displaymath}
\xymatrix@C=4em{
\Phi(X \times X) \ar[r]^-{\beta_{\Delta}} \ar@{=}[d] & \Phi(X) \ar@{=}[d] \ar[r]^-{\alpha_{\Delta}} & \Phi(X \times X) \ar@{=}[d] \\
\Phi(X) \oplus \Phi(X^{[2]}) \ar[r] & \Phi(X) \ar[r] & \Phi(X) \oplus \Phi(X^{[2]}) }
\end{displaymath}
The bottom left map is the projection onto $\Phi(X)$, while the bottom right map is the inclusion of $\Phi(X)$. It follows that the composition on the bottom is the identity on $\Phi(X)$ and~0 on $\Phi(X^{[2]})$.

The above analysis establishes the identity
\begin{displaymath}
(\id_{X'} \otimes \beta_{\Delta}) \circ (\alpha_{\Delta} \otimes \id_{X'}) = \alpha_{\Delta} \circ \beta_{\Delta}
\end{displaymath}
as maps $X' \otimes X' \to X' \otimes X'$. A similar analysis establishes the analogous identity where the tensor factors on the left are switched. The analysis of the second diagram above also shows that $\beta_{\Delta} \circ \alpha_{\Delta}=\id_{X'}$, as the bottom row is the composition of the inclusion and projection of $\Phi(X)$. This completes the proof.
\end{proof}

\begin{proof}[Proof of Proposition~\ref{prop:alpha-dual}]
Let $f \colon X \to Y$ be a map of finitary $G$-sets, and put $X'=\Phi(X)$ and $Y'=\Phi(Y)$. We must show that $\alpha_f \colon X' \to Y'$ and $\beta_f \colon Y' \to X'$ are dual to one another, with respect to the self-dualities of $X'$ and $Y'$ coming from the Frobenius algebra structures.

Write $Y \times X = \Gamma \sqcup W$, where $\Gamma$ is the (transpose of the) graph of $f$, and $W$ is its complement. Let $p$ and $q$ (resp.\ $p'$ and $q'$) be the projection maps from $\Gamma$ (resp.\ $W$) to $X$ and $Y$. By definition, the dual of $\alpha_f$ is the following composition
\begin{equation} \label{eq:alpha-dual}
\resizebox{.9\hsize}{!}{
\xymatrix@C=4em{
Y' \ar[r]^-{1 \otimes \beta_X} &
Y' \otimes X' \ar[r]^-{1 \otimes \alpha_{\Delta}} &
Y' \otimes X' \otimes X' \ar[r]^-{1 \otimes \alpha_f \otimes 1} &
Y' \otimes Y' \otimes X' \ar[r]^-{\beta_{\Delta} \otimes 1} &
Y' \otimes X' \ar[r]^-{\alpha_Y \otimes 1} &
X' }}
\end{equation}
Consider the following diagram of $G$-sets:
\begin{displaymath}
\xymatrix@C=4em{
Y \times X \ar[r]^-{1 \times \Delta} \ar@{=}[d] & Y \times X \times X \ar[r]^-{1 \times f \times 1} & Y \times Y \times X \ar@{=}[d] \\
\Gamma \amalg W \ar[rr]^{\id \amalg h} && \Gamma \amalg W \amalg (Y^{[2]} \times X) }
\end{displaymath}
On the right side, we have decomposed $Y \times Y$ as $Y \amalg Y^{[2]}$, and then further decomposed $Y \times X$. The bottom map is the identity on the $\Gamma$ pieces, and $h$ is some map $W \to Y^{[2]} \times X$. The key point here is that $\Gamma$ on the left maps into $\Delta(Y) \times X \subset Y \times Y \times X$, and $W$ does not. Now, $\beta_{\Delta} \colon Y' \otimes Y' \to Y'$ is the identity on the diagonal $Y'$, and~0 on $\Phi(Y^{[2]})$. We thus see that the following diagram commutes:
\begin{displaymath}
\xymatrix@C=4em{
Y' \otimes X' \ar[r]^-{1 \otimes \alpha_{\Delta}} \ar@{=}[d] &
Y' \otimes X' \otimes X' \ar[r]^-{1 \otimes \alpha_f \otimes 1} &
Y' \otimes Y' \otimes X' \ar[r]^-{\beta_{\Delta} \otimes 1} &
Y' \otimes X' \ar@{=}[d] \\
\Phi(\Gamma) \oplus \Phi(W) \ar[rrr]^{1 \oplus 0} &&& \Phi(\Gamma) \oplus \Phi(W) }
\end{displaymath}
Now, the first map in \eqref{eq:alpha-dual} is simply the $\beta$ map for the projection $Y \times X \to Y$. Thus, decomposing $Y' \otimes X'$ as $\Phi(\Gamma) \oplus \Phi(W)$, this map becomes $\beta_q \oplus \beta_{q'}$. Similarly, with this decomposition, the final map in \eqref{eq:alpha-dual} becomes $\alpha_p \oplus \alpha_{p'}$.

Putting all this together, we see that $\alpha_f^{\vee}$, i.e., the composition of \eqref{eq:alpha-dual}, is equal to the following composition
\begin{displaymath}
\xymatrix@C=4em{
Y' \ar[r]^-{\beta_q \oplus \beta_{q'}} &
\Phi(\Gamma) \oplus \Phi(W) \ar[r]^-{1 \oplus 0} &
\Phi(\Gamma) \oplus \Phi(W) \ar[r]^-{\alpha_p \oplus \alpha_{p'}} &
X' }
\end{displaymath}
Thus $\alpha_f^{\vee}=\alpha_p \beta_q$. Now, $q=fp$, and so $\beta_q=\beta_p \beta_f$. Since $p$ is an isomorphism, we thus have $\beta_f=\alpha_p \beta_q$. We have thus shown $\alpha_f^{\vee}=\beta_f$, which completes the proof.
\end{proof}

\subsection{Proof of Theorem~\ref{thm:linear}} \label{ss:linear-pf}

Fix a linearization $\Phi \colon \cS(G) \to \cT$. For finitary $G$-sets $X$ and $Y$, we put
\begin{displaymath}
[X,Y]=\Hom_{\cT}(\Phi(X), \Phi(Y)).
\end{displaymath}
We must show that $\cT$ is equivalent to $\uPerm(G;\mu)$ for some measure $\mu$. The plan is as follows:
\begin{enumerate}
\item We define a pairing on $[Z,X]$, when $Z$ is transitive, and show that the $\alpha$ and $\beta$ maps are adjoint (Lemma~\ref{lem:linear1}). This is a very useful computational tool that is employed several times in the subsequent steps.
\item We show that $[X,Y]$ has a natural basis indexed by $G$-orbits on $X \times Y$, and give a useful formula for expressing an arbitrary map in the basis (Lemma~\ref{lem:linear2}).
\item We show that $\Phi$ satisfies base change (Lemma~\ref{lem:linear3}). The basic idea is to express the two maps in bases using Lemma~\ref{lem:linear2} and compare.
\item We construct a measure $\mu$ (Lemma~\ref{lem:linear4}). The basic idea is as follows. If $f \colon X \to Y$ is a map of $G$-sets, with $Y$ transitive, then $\alpha_f \beta_X=c \beta_Y$ for some scalar $c$; this follows since $\Phi$ is plenary. We define $\mu(f)=c$. It is relatively straightforward to verify that $\mu$ respects the defining relations of $\Theta'(G)$ and therefore defines a measure.
\item We show that $\Phi$ is $\mu$-adapted (Lemma~\ref{lem:linear5}). The idea, once again, is to compute in bases using Lemma~\ref{lem:linear2}.
\end{enumerate}
Having completed these steps, Proposition~\ref{prop:perm-maps} yields a $k$-linear functor $\Psi \colon \uPerm(G;\mu) \to \cT$. It is then a simple matter to complete the proof of the theorem.

We now start on our plan. Let $X$, $Y$, and $Z$ be finitary $G$-sets, with $Z$ transitive. Let $\phi,\psi \colon \Phi(Z) \to \Phi(Y)$ be two maps in $\cT$. Consider the composition
\begin{displaymath}
\xymatrix@C=4em{
\Phi(Z) \ar[r]^-{\alpha_{\Delta}} &
\Phi(Z \times Z) \ar[r]^-{\phi \otimes \psi} &
\Phi(Y \times Y) \ar[r]^-{\ev_{\Phi(Y)}} &
\bbone }
\end{displaymath}
where $\ev_{\Phi(Y)} = \alpha_Y \beta_{\Delta}$ is the pairing on $\Phi(Y)$. Since $[Z, \bbone]$ is free of rank one and spanned by $\alpha_Z$, the above composition is equal to $c \cdot \alpha_Z$ for a unique $c \in k$. We define $\langle \phi, \psi \rangle=c$. In this way, we have a pairing $\langle, \rangle$ on the $k$-module $[Z,Y]$. It is easily seen to be symmetric and $k$-bilinear.

\begin{lemma} \label{lem:linear1}
Let $f \colon X \to Y$ be a map of $G$-sets. Then for $\phi \in [Z,X]$ and $\psi \in [Z,Y]$, we have $\langle \phi, \beta_f \psi \rangle=\langle \alpha_f \phi, \psi \rangle$; in other words, $\alpha_f$ and $\beta_f$ induced adjoint maps between $[Z,X]$ and $[Z,Y]$.
\end{lemma}

\begin{proof}
Consider the following diagram
\begin{displaymath}
\xymatrix@C=3em{
&&& \Phi(Y \times Y) \ar[rd]^{\ev_{\Phi(Y)}} \\
\Phi(Z) \ar[r]^-{\alpha_{\Delta}} &
\Phi(Z \times Z) \ar[r]^{\phi \otimes \psi} &
\Phi(X \times Y) \ar[ru]^{\alpha_f \otimes 1} \ar[rd]_{1 \otimes \beta_f} &&
\bbone \\
&&& \Phi(X \times X) \ar[ru]_{\ev_{\Phi(X)}} }
\end{displaymath}
Since $\alpha_f$ and $\beta_f$ are dual (Proposition~\ref{prop:alpha-dual}), the square commutes. As the top path computes $\langle \alpha_f \phi, \psi \rangle$ and the bottom $\langle \phi, \beta_f \psi \rangle$, the result follows.
\end{proof}

\begin{lemma} \label{lem:linear2}
Let $Z_1, \ldots, Z_n$ be the $G$-orbits on $X \times Y$ and let $p_i \colon Z_i \to X$ and $q_i \colon Z_i\to Y$ be the projections.
\begin{enumerate}
\item The maps $\alpha_{q_i} \beta_{p_i}$ for $1 \le i \le n$ form a $k$-basis of $[X,Y]$.
\item Given $\phi \in [X,Y]$, we have $\phi=\sum_{i=1}^n c_i \alpha_{q_i} \beta_{p_i}$ where $c_i=\langle \phi \circ \alpha_{p_i}, \alpha_{q_i} \rangle$. 
\end{enumerate}
\end{lemma}

\begin{proof}
(a) Since $\Phi$ is additive, we have $\Phi(X \times Y) = \bigoplus_{i=1}^n \Phi(Z_i)$. Since $\Phi$ is plenary, we see that $[Z_i, \bbone]$ is a free $k$-module of rank~1, spanned by $\alpha_{Z_i}$. It follows that $[X \times Y, \bbone]$ is a free $k$-module of rank~$n$, with basis given by $\lambda_i=\alpha_{Z_i} \beta_{j_i}$ for $1 \le i \le n$, where $j_i \colon Z_i \to X \times Y$ is the inclusion. Now, since $\Phi(Y)$ is rigid, we have a natural isomorphism $[X,Y]=[X \times Y, \bbone]$. Let $\phi_i \in [X,Y]$ correspond to $\lambda_i \in [X \times Y, \bbone]$. Since the $\lambda_i$ form a basis, so do the $\phi_i$.

We now compute $\phi_i$ explicitly. Write $X'=\Phi(X)$, $Y'=\Phi(Y)$, $Z'_i=\Phi(Z_i)$. By definition, $\phi_i$ is the composition
\begin{displaymath}
\xymatrix@C=4em{
X' \ar[r]^-{1 \otimes \beta_Y} &
X' \otimes Y' \ar[r]^-{1 \otimes \alpha_{\Delta}} &
X' \otimes Y' \otimes Y' \ar[r]^-{\beta_{j_i} \otimes 1} &
Z'_i \otimes Y' \ar[r]^-{\alpha_{Z_i} \otimes 1} &
Y' }
\end{displaymath}
Now, $X' \otimes Y'$ decomposes as $\bigoplus_{i=1}^n Z_i'$; similarly, $X' \otimes Y' \otimes Y'$ decomposes as $\bigoplus_{i=1}^n Z_i' \otimes Y'$. With respect to these decompositions, the first map above is $\bigoplus_{i=1}^n \beta_{p_i}$, the second is $\bigoplus_{i=1}^n \alpha_{(\id,q_i)}$, where $(\id,q_i) \colon Z_i \to Z_i \times Y$, and the third map projects onto the $i$th piece. We thus see that $\phi_i$ is equal to the following composition
\begin{displaymath}
\xymatrix@C=4em{
X' \ar[r]^{\beta_{p_i}} &
Z_i' \ar[r]^-{\alpha_{(\id,q_i)}} &
Z_i' \otimes Y' \ar[r]^{\alpha_{Z_i} \otimes 1} &
Y' }
\end{displaymath}
The latter two maps above compose to $\alpha_{q_i}$. We thus see that $\phi_i=\alpha_{q_i} \beta_{p_i}$. As we know that the $\phi_i$ form a basis, this completes the proof.

(b) We first claim that $\langle \beta_{p_i} \alpha_{p_j}, \beta_{q_i} \alpha_{q_j} \rangle=\delta_{i,j}$ for all $1 \le i,j \le n$. This pairing is computed by the composition
\begin{displaymath}
\xymatrix@C=6em{
Z'_j \ar[r]^-{\alpha_{\Delta}} &
Z'_j \otimes Z'_j \ar[r]^-{\beta_{p_i} \alpha_{p_j} \otimes \beta_{q_i} \alpha_{q_j}} &
Z'_i \otimes Z'_i \ar[r]^-{\beta_{\Delta}} &
Z'_i \ar[r]^-{\alpha_{Z_i}} & \bbone }
\end{displaymath}
Now, we have
\begin{displaymath}
\beta_{p_i} \alpha_{p_j} \otimes \beta_{q_i} \alpha_{q_j} =
(\beta_{p_i} \otimes \beta_{q_i}) \circ (\alpha_{p_j} \otimes \alpha_{q_j}).
\end{displaymath}
Also, $(\alpha_{p_j} \otimes \alpha_{q_j}) \circ \alpha_{\Delta}=\alpha_{(p_j,q_j)}$, where $(p_j,q_j) \colon Z_j \to X \times Y$, and similarly with $\beta$'s. We thus see that the above composition is equal to the following composition
\begin{displaymath}
\xymatrix@C=4em{
Z'_j \ar[r]^-{\alpha_{(p_j,q_j)}} &
X' \otimes Y' \ar[r]^-{\beta_{(p_i,q_i)}} &
Z'_i \ar[r]^-{\alpha_{Z_i}} & \bbone }
\end{displaymath}
Since $\Phi$ is additive, the first two maps above compose to~0 if $i \ne j$ and the identity if $i=j$. Thus the whole composition is $\delta_{i,j} \cdot \alpha_{Z_j}$, which proves the claim.

Now, let $\phi \in [X,Y]$ be given. By (a), we have $\phi=\sum_{i=1}^n c_i \alpha_{q_i} \beta_{p_i}$ for some $c_i \in k$. We have
\begin{displaymath}
\langle \phi \alpha_{p_j}, \alpha_{q_j} \rangle
= \sum_{i=1}^n c_i \langle \alpha_{q_i} \beta_{p_i} \alpha_{p_j}, \alpha_{q_j} \rangle
= \sum_{i=1}^n c_i \langle \beta_{p_i} \alpha_{p_j}, \beta_{q_i} \alpha_{q_j} \rangle
= c_j
\end{displaymath}
where in the second step we used adjunction (Lemma~\ref{lem:linear1}), and in the third step the computation from the previous paragraph. This completes the proof.
\end{proof}

\begin{lemma} \label{lem:linear3}
$\Phi$ satisfies base change.
\end{lemma}

\begin{proof}
Consider a cartesian square of finitary $G$-sets
\begin{displaymath}
\xymatrix{
X' \ar[r]^{g'} \ar[d]_{f'} & X \ar[d]^f \\
Y' \ar[r]^g & Y }
\end{displaymath}
Notice that $X'=X \times_Y Y'$ is a $G$-subset of $X \times Y'$. Let $Z_1, \ldots, Z_n$ be the $G$-orbits on $X \times Y'$, labeled so that $Z_1, \ldots, Z_m$ are the orbits on $X'$, and let $p_i \colon Z_i \to X$ and $q_i \colon Z_i \to Y'$ be the projections.

By Lemma~\ref{lem:linear2}, we have $\beta_g \alpha_f = \sum_{i=1}^n c_i \alpha_{q_i} \beta_{p_i}$ where
\begin{displaymath}
c_i = \langle \beta_g \alpha_f \alpha_{p_i}, \alpha_{q_i} \rangle = \langle \alpha_{fp_i}, \alpha_{gq_i} \rangle,
\end{displaymath}
and in the second step we used adjunction (Lemma~\ref{lem:linear1}). The above pairing is computed by the composition
\begin{displaymath}
\xymatrix@C=5em{
\Phi(Z_i) \ar[r]^-{\alpha_{(fp_i,gq_i)}} & \Phi(Y \times Y) \ar[r]^-{\ev_{\Phi(Y)}} & \bbone }
\end{displaymath}
where $(fp_i,gq_i) \colon Z_i \to Y \times Y$. Since $Z_i$ is transitive, it either maps into $\Delta(Y)$ or $Y^{[2]}$. The former happens if and only if $fp_i=gq_i$, meaning $Z_i$ belongs to the fiber product (and so $1 \le i \le m$). In the latter case, the above composition is~0, since $\ev_{\Phi(Y)}$ is zero on $Y^{[2]}$. Thus assume the former. Then the above composition coincides with the following one
\begin{displaymath}
\xymatrix@C=4em{
\Phi(Z_i) \ar[r]^-{\alpha_{fp_i}} & \Phi(Y) \ar[r]^-{\alpha_{\Delta}} & \Phi(Y \times Y) \ar[r]^-{\ev_{\Phi(Y)}} & \bbone }
\end{displaymath}
which simplifies to
\begin{displaymath}
\xymatrix@C=4em{
\Phi(Z_i) \ar[r]^-{\alpha_{fp_i}} & \Phi(Y) \ar[r]^-{\alpha_Y} & \bbone }
\end{displaymath}
which is clearly $\alpha_{Z_i}$. We thus see that the pairing is~1 in this case. In conclusion, we have $\beta_g \alpha_f = \sum_{i=1}^m \alpha_{q_i} \beta_{p_i}$.

Now, let $j_i \colon Z_i \to X \times Y$ be the inclusion. Then $\sum_{i=1}^m \alpha_{j_i} \beta_{j_i}$ is the identity map on $\Phi(X')$, and so
\begin{displaymath}
\alpha_{f'} \beta_{g'} = \sum_{i=1}^m \alpha_{f'} \alpha_{j_i} \beta_{j_i} \beta_{g'} = \sum_{i=1}^m \alpha_{q_i} \beta_{p_i}.
\end{displaymath}
This agrees with our formula for $\beta_g \alpha_f$, and so the proof is complete.
\end{proof}

Given a morphism $f \colon X \to Y$ of $G$-sets, with $Y$ transitive, define $\mu(f) \in k$ to be the unique scalar such that $\alpha_f \beta_X=\mu(f) \beta_Y$. Note that such a unique scalar exists since $[\bbone, Y]$ is free of rank one over $k$ by the normalization condition.

\begin{lemma} \label{lem:linear4}
$\mu$ defines a measure for $G$.
\end{lemma}

\begin{proof}
We verify that $\mu$ respects the defining relations of $\Theta'(G)$ (see Definition~\ref{defn:Theta-prime}):
\begin{enumerate}
\item If $f \colon X \to Y$ is an isomorphism then $\alpha_f \beta_X=\beta_Y$, so $\mu(f)=1$.
\item Let $f \colon X \to Y$ be a map of $G$-sets with $Y$ transitive. Suppose that $X=X_1 \sqcup X_2$, and let $f_i$ be the restriction of $f$ to $X_i$. We have $\Phi(X)=\Phi(X_1) \oplus \Phi(X_2)$ and (with slight abuse of notation), $\alpha_f=\alpha_{f_1}+\alpha_{f_2}$ and $\beta_X=\beta_{X_1}+\beta_{X_2}$. Thus
\begin{align*}
\mu(f) \beta_Y
&= \alpha_f \beta_X = (\alpha_{f_1} +\alpha_{f_2})(\beta_{X_1}+\beta_{X_2}) \\
&= \alpha_{f_1} \beta_{X_1}+\alpha_{f_2} \beta_{X_2} = (\mu(f_1)+\mu(f_2)) \beta_Y
\end{align*}
Note that $\alpha_{f_1} \beta_{X_2}=0$. We thus find that $\mu(f)=\mu(f_1)+\mu(f_2)$.
\item Consider morphisms $f \colon X \to Y$ and $g \colon Y \to Z$ of transitive $G$-sets. Then
\begin{displaymath}
\mu(gf) \beta_Z = \alpha_{gf} \beta_X=\alpha_g \alpha_f \beta_X=\mu(f) \alpha_g \beta_Y=\mu(f) \mu(g) \beta_Z.
\end{displaymath}
We thus find $\mu(gf)=\mu(g)\mu(f)$.
\item Finally, consider a cartesian square of $G$-sets
\begin{displaymath}
\xymatrix{
X' \ar[r]^{g'} \ar[d]_{f'} & X \ar[d]^f \\
Y' \ar[r]^g & Y }
\end{displaymath}
with $Y'$ and $Y$ transitive. We then have
\begin{displaymath}
\mu(f) \beta_{Y'} = \mu(f) \beta_g \beta_Y = \beta_g \alpha_f \beta_X = \alpha_{f'} \beta_{g'} \beta_X = \alpha_{f'} \beta_{X'} = \mu(f') \beta_Z.
\end{displaymath}
In the third step above, we used base change (Lemma~\ref{lem:linear3}). We thus find that $\mu(f')=\mu(f)$.
\end{enumerate}
We have thus verified the relations, and so we have a homomorphism $\mu \colon \Theta'(G) \to k$. This yields a measure on $G$.
\end{proof}

\begin{lemma} \label{lem:linear5}
$\Phi$ is $\mu$-adapted.
\end{lemma}

\begin{proof}
Let $f \colon X \to Y$ be a map of transitive $G$-sets. We must compute the composition $\alpha_f \beta_f$. Let $Z_1, \ldots, Z_n$ be the $G$-orbits on $Y \times Y$, labeled so that $Z_1$ is the diagonal copy of $Y$, and let $p_i,q_i \colon Z_i \to X$ be the projections. By Lemma~\ref{lem:linear2}, we have $\alpha_f \beta_f = \sum_{i=1}^n c_i \alpha_{q_i} \beta_{p_i}$, where
\begin{displaymath}
c_i=\langle \alpha_f \beta_f \alpha_{p_i}, \alpha_{q_i} \rangle
=\langle \beta_f \alpha_{p_i}, \beta_f \alpha_{q_i} \rangle,
\end{displaymath}
where in the second step we used adjunction (Lemma~\ref{lem:linear1}). Consider the following commutative diagram
\begin{displaymath}
\xymatrix@C=4em{
\Phi(Z_i) \ar[r]^-{\alpha_{(p_i,q_i)}} &
\Phi(Y \times Y) \ar[r]^-{\beta_f \otimes \beta_f} \ar[d]_{\beta_{\Delta}} &
\Phi(X \times X) \ar[d]^{\beta_{\Delta}} \\
& \Phi(Y) \ar[r]^-{\beta_f} & \Phi(X) \ar[r]^-{\alpha_X} & \bbone }
\end{displaymath}
The composition is $c_i \cdot \alpha_{Z_i}$, as one sees by following the top path. Now, $\beta_{\Delta}$ vanishes on $\Phi(Y^{[2]})$. We thus see that $c_i=0$ for $i>1$, since $(p_i,q_i)$ maps into $Y^{[2]}$ in this case. We can identify $Z_1$ with $Y$ and $(p_1,q_1)$ with the diagonal map $Y \to Y \times Y$. It follows that $\beta_{\Delta} \circ \alpha_{(p_1,q_1)}$ is the identity. Thus the above composition is $\beta_f \alpha_X$. We have therefore obtained the equality $\alpha_X \beta_f=c_1 \alpha_Y$. Taking duals (and appealing to Proposition~\ref{prop:alpha-dual}), we have $\alpha_f \beta_X = c_1 \beta_Y$, and so $c_1=\mu(f)$ by definition of $\mu$. Putting all of the above together, we have $\alpha_f \beta_f = \mu(f) \alpha_{q_1} \beta_{p_1}$. Finally, since $p_1=q_1$ is an isomorphism $\alpha_{q_1} \beta_{p_1} = \id_{\Phi(Y)}$. This completes the proof.
\end{proof}

We thus see that $\Phi$ is additive (by assumption), satisfies base change (Lemma~\ref{lem:linear3}), and is $\mu$-adapted (Lemma~\ref{lem:linear5}). Applying Proposition~\ref{prop:perm-maps}, we find that there is a unique $k$-linear functor
\begin{displaymath}
\Psi \colon \uPerm(G; \mu) \to \cT
\end{displaymath}
such that $\Phi = \Psi \circ \Phi_{\mu}$. Lemma~\ref{lem:linear2} shows that the $\Hom$ spaces in $\cT$ admit the same bases as those in $\uPerm(G; \mu)$ (see Proposition~\ref{prop:hom-basis}), and so $\Psi$ is fully faithful. Since $\Phi$ is essentially surjective by assumption, so is $\Psi$. Thus $\Psi$ is an equivalence. Finally, since the symmetric monoidal structure on $\uPerm(G; \mu)$ is derived from that on $\cS(G)$, we see that $\Psi$ can be given a symmetric monoidal structure.

\part{General representations} \label{part:genrep}

\section{The completed group algebra and its modules} \label{s:gpalg}

\subsection{Overview}

In Part~\ref{part:genrep}, we define and study an abelian tensor category $\uRep_k(G; \mu)$ associated to a pro-oligomorphic group $G$ and a measure $\mu$. The main theorem (Theorem~\ref{thm:regss}) asserts that this construction results in a rigid tensor category in certain circumstances. The purpose of \S \ref{s:gpalg} is simply to construct the abelian category $\uRep(G)$.

We first define a $k$-algebra $A$ called the completed group algebra of $G$. We would like to define this as a function space on $G$ under convolution. This is not actually possible, since we do not have a theory of integration for functions on $G$ in general. However, if we imagined that we had such a theory, a function on $G$ would push-forward to one on $G/U$ for each open subgroup $U$, and we would thereby obtain a compatible sequence of elements of $\cC(G/U)$ as $U$ varies. We therefore \emph{define} $A$ to be the inverse limit of the Schwartz spaces $\cC(G/U)$, and we show that it carries an algebra structure under convolution. We then define $\uRep(G)$ as a certain category of $A$-modules.

We do not prove any substantial results about $\uRep(G)$ in \S \ref{s:gpalg}. In fact, we do not know any interesting results about $\uRep(G)$ that hold in general: to say anything, we must impose some conditions on the measure $\mu$. Such results are established in subsequent sections.

We fix a pro-oligomorphic group $G$ and a $k$-valued measure $\mu$ for the duration \S \ref{s:gpalg}. Throughout \S\S \ref{s:gpalg}--\ref{s:regular} we work in the absolute case; see \S \ref{ss:abrel} for comments about the relative case.

\subsection{Convolutions} \label{ss:convol}

To define the completed group algebra, we need properties of certain convolution operations. We establish them now.

Let $V$ be an open subgroup of $G$. We let $\cC(V \backslash G)$ be the set of $k$-valued functions on $V \backslash G$ that are smooth, i.e., right invariant by an open subgroup of $G$. Suppose that $\phi \in \cC(G/V)$ and $\psi \in \cC(V \bs G)$. We define their \emph{convolution} $\phi \ast_V \psi$ to be the function on $G$ given by
\begin{displaymath}
(\phi \ast_V \psi)(g) = \int_{G/V} \phi(x) \psi(x^{-1} g) dx.
\end{displaymath}
The $V$ in the notation $\ast_V$ is necessary, as the result depends on $V$. That is, if $W$ is an open subgroup of $V$ then $\phi \ast_W \psi$ is defined, but we have $\phi \ast_W \psi = \mu(V/W) \cdot \phi \ast_V \psi$. One easily sees that if $\phi$ is left invariant under $U$ then so is $\phi \ast_V \psi$; similarly, if $\psi$ is right invariant under $W$ then so is $\phi \ast_V \psi$. In other words, writing $\cC(V \bs G/U)$ for the set of all $k$-valued functions on the finite set $V \bs G/U$, convolution defines a $k$-bilinear function
\begin{displaymath}
\ast_V \colon \cC(W \bs G/V) \times \cC(V \bs G/U) \to \cC(W \bs G/U).
\end{displaymath}
As in other situations, convolution is associative:

\begin{proposition} \label{prop:convol1}
Let $\phi \in \cC(G/V)$, $\psi \in \cC(V \bs G/U)$, and $\theta \in \cC(U \bs G)$. Then
\begin{displaymath}
(\phi \ast_V \psi) \ast_U \theta = \phi \ast_V (\psi \ast_U \theta)
\end{displaymath}
\end{proposition}

\begin{proof}
We have
\begin{align*}
((\phi \ast_V \psi) \ast_U \theta)(g)
&= \int_{G/U} \int_{G/V} \phi(y) \psi(y^{-1}x) \theta(x^{-1} g) dy dx. \\
&= \int_{G/U \times G/V} \phi(y) \psi(y^{-1}x) \theta(x^{-1} g) d(y,x),
\end{align*}
where in the second step we used Fubini's theorem (Corollary~\ref{cor:fubini}). The other parenthesization leads to the same integral, which proves the result.
\end{proof}

Given open subgroups $V \subset U$, let $\pi_{V,U} \colon \cC(G/V) \to \cC(G/U)$ be the push-forward operation. Thus, explicitly,
\begin{displaymath}
(\pi_{V,U} \phi)(g) = \int_{U/V} \phi(gh) dh
\end{displaymath}
We require the following two compatibility results between push-forward and convolution:

\begin{proposition} \label{prop:convol2}
Let $V_2 \subset V_1$ be open subgroups of $G$, let $\phi \in \cC(G/V_2)$ and let $\psi \in \cC(V_1 \bs G)$. Then
\begin{displaymath}
\phi \ast_{V_2} \psi = \pi_{V_2,V_2}(\phi) \ast_{V_1} \psi.
\end{displaymath}
\end{proposition}

\begin{proof}
We have
\begin{displaymath}
(\phi \ast_{V_2} \psi)(g)
= \int_{G/V_2} \phi(x) \psi(x^{-1} g) dx
= \int_{G/V_1} \int_{V_1/V_2} \phi(xy) \psi(y^{-1} x^{-1} g) dy dx
\end{displaymath}
In the second step, we have used the transitivity of push-forward (Proposition~\ref{prop:push-trans}). Since $\psi$ is left invariant under $V_1$, we have $\psi(y^{-1} x^{-1} g)=\psi(x^{-1} g)$. Thus the $\psi$ factor is constant on the inner integral and can be moved outside of it. Doing so, the inner integral becomes exactly $(\pi_{V_1,V_2} \phi)(x)$, and so the result follows.
\end{proof}

\begin{proposition} \label{prop:convol3}
Let $V$ and $U_2 \subset U_1$ be open subgroups of $G$, let $\phi \in \cC(G/V)$ and let $\psi \in \cC(V \bs G/U_2)$. Then
\begin{displaymath}
\pi_{U_2,U_1}(\phi \ast_V \psi) = \phi \ast_V \pi_{U_2,U_1}(\psi).
\end{displaymath}
\end{proposition}

\begin{proof}
We have
\begin{align*}
\pi_{U_2,U_1}(\phi \ast_V \psi)(g)
&= \int_{U_1/U_2} \int_{G/V} \phi(x) \psi(x^{-1} gy) dx dy \\
&=  \int_{G/V}\phi(x) \left[ \int_{U_1/U_2} \psi(x^{-1} gy) dy \right] dx
= (\phi \ast_V \pi_{U_2,U_1}(g)
\end{align*}
In the second step, we switched the order of integration (by Fubini's theorem, Corollary~\ref{cor:fubini}) and pulled the $\phi(x)$ factor out of the $y$ integral.
\end{proof}

\subsection{The completed group algebra}

Let $A$ be the $k$-module defined by
\begin{displaymath}
A = \varprojlim_U \cC(G/U) = \varprojlim_U \varinjlim_V \cC(V \bs G/U).
\end{displaymath}
The transition maps in the inverse limit are the push-forward maps $\pi_{V,U}$. Explicitly, an element $a$ of $A$ is a tuple $a=(a_U)$ consisting of an element $a_U \in \cC(G/U)$ for each open subgroup $U$ of $G$ such that $\pi_{V,U}(a_V)=a_U$ whenever $V \subset U$.

Suppose $a,b \in A$. We define their product $ab \in A$ as follows. Let $U$ be an open subgroup of $G$ and let $V$ be such that $b_U$ is left $V$-invariant. We put
\begin{displaymath}
(ab)_U=a_V \ast_V b_U
\end{displaymath}
We now verify that this does indeed endow $A$ with an algebra structure.

\begin{proposition} \label{prop:cga}
The above multiplication law is well-defined, and gives $A$ the structure of an associative unital $k$-algebra.
\end{proposition}

\begin{proof}
We first verify that $(ab)_U$ is well-defined, i.e., independent of the choice of $V$. Thus let $V'$ be a second open subgroup such that $b_U$ is left $V'$-invariant. We must show $a_{V'} \ast_{V'} b_U=a_V \ast_V b_U$. It suffices to treat the case $V' \subset V$. We then have
\begin{displaymath}
a_{V'} \ast_{V'} b_U = \pi_{V',V}(a_{V'}) \ast_V b_U = a_V \ast_V b_U,
\end{displaymath}
where in the second step we used Proposition~\ref{prop:convol2}. We thus see that $(ab)_U$ is well-defined.

We next show that the family $((ab)_U)$ defines an element of $A$. Let $U' \subset U$ be open subgroups. We must show that $\pi_{U',U}((ab)_{U'})=(ab)_U$. Let $V$ be such that $b_{U'}$ is left $V$-invariant. Then
\begin{displaymath}
\pi_{U',U}((ab)_{U'}) = \pi_{U',U}(a_V \ast_V b_{U'}) = a_V \ast_V \pi_{U',U}(b_{U'}) = a_V \ast_V b_U = (ab)_U,
\end{displaymath}
where in the second step we used Proposition~\ref{prop:convol3}. We thus have a well-defined element $ab \in A$.

It is clear that the multiplication law is $k$-bilinear. We now show that it is associative. Thus let $a,b,c \in A$. Let $U$ be an open subgroup, let $V$ be such that $c_U$ is left $V$-invariant, and let $W$ be such that $b_V$ is left $W$-invariant. Then $(bc)_U=b_V \ast_V c_U$ is left $W$-invariant, and so
\begin{displaymath}
(a(bc))_U=a_W \ast_W (bc)_U = a_W \ast_W (b_V \ast_V c_U).
\end{displaymath}
Similarly, $(ab)_V=a_W \ast_W b_V$, and so
\begin{displaymath}
(ab)c)_U=(ab)_V \ast_V c_U = (a_W \ast_W b_V) \ast_V c_U.
\end{displaymath}
Thus associativity follows from Proposition~\ref{prop:convol1}. Finally, we have an element $1 \in A$ given by $1_U=\delta_{1,G/U}$, which is easily seen to be the identity element.
\end{proof}

\begin{definition} \label{defn:complete-alg}
The \defn{completed group algebra} of $G$, denoted $A_k(G; \mu)$, is the $k$-algebra $A$ constructed above. We typically write $A$ or $A(G)$ instead of $A_k(G; \mu)$.
\end{definition}

\subsection{Basic properties} \label{ss:A-basic}

We now give a few simple properties of the algebra $A$.

\textit{(a) Relation to the group algebra.} For an element $g \in G$, let $c_{g,U} \in \cC(G/U)$ be the point-mass at $g \in G/U$. It is clear that $c_g=(c_{g,U})$ is an element of $A$. Furthermore, a simple computation shows that $c_g c_h=c_{gh}$. We thus obtain a $k$-algebra homomorphism $k[G] \to A$ by $g \mapsto c_g$, which is easily seen to be injective. In what follows, we regard $k[G]$ as a subalgebra of $A$ via this embedding.

\textit{(b) The augmentation map.} We define the \defn{augmentation map}
\begin{displaymath}
\epsilon \colon A \to k, \qquad \epsilon(a)=a_G.
\end{displaymath}
Here $a_G$ is an element of $\cC(G/G)$, which we identify with $k$. One easily verifies that $\epsilon$ is a map of $k$-algebras. In particular, we can define an $A$-module structure on $k$ by $a \cdot 1=\epsilon(a)$. We call this the \defn{trivial $A$-module}, and denote it $\bbone$.

\textit{(c) Open subgroups.} Let $U$ be an open subgroup of $G$. Given an open subgroup $V$ of $U$, we have an injective map $\cC(U/V) \to \cC(G/V)$ via extension by~0. One readily verifies that this is compatible with transition maps, and thus induces an inclusion of $k$-algebras $A(U) \to A(G)$. (Note that $A(G)$ can be identified with the inverse limit of $\cC(G/V)$ over open subgroups $V \subset U$, as such subgroups constitute a cofinial system.)

\subsection{Smooth modules}

Let $M$ be a $A$-module. Let $x \in M$ and let $U$ be an open subgroup of $G$. We say that $x$ is \emph{strictly $U$-invariant} if the action of $A$ on $x$ factors through $\cC(G/U)$, i.e., if $a \in A$ has $a_U=0$ then $ax=0$. Intuitively, this means that if $V \subset U$ is an open subgroup and $\phi \in \cC(G/V)$ then
\begin{displaymath}
\int_{G/V} \phi(g) gx \, dg = \int_{G/U} \psi(g) gx\, dg,
\end{displaymath}
where $\psi \in \cC(G/U)$ is the push-forward of $\phi$. We have not actually defined integrals of the above sort yet however, though we will in \S \ref{ss:modint}. We write $M^{sU}$ for the set of all strictly $U$-invariant elements of $M$.

We say that $x$ is \defn{smooth} if it is strictly $U$-invariant for some $U$, and we say that $M$ is \defn{smooth} if every element of $M$ is. We note that $M$ is smooth if and only if the multiplication map $A \times M \to M$ is continuous when $A$ is given the inverse limit topology and $M$ the discrete topology. Essentially every $A$-module we consider will be smooth. In particular, we note that the trivial module $\bbone$ is smooth.

We now introduce our abelian representation category:

\begin{definition} \label{defn:repcat}
We define $\uRep_k(G; \mu)$ to be the full subcategory of $\Mod_A$ spanned by smooth $A$-modules. As usual, we omit $\mu$ or $k$ from the notation when possible.
\end{definition}

\begin{proposition}
The category $\uRep(G)$ is a Grothendieck abelian category.
\end{proposition}

\begin{proof}
It is clear that if $M$ is a smooth $A$-module then any sub or quotient of $M$ (as an $A$-module) is again smooth. Thus $\uRep(G)$ is an abelian subcategory of $\Mod_A$. Moreover, a direct limit of smooth $A$-modules is again smooth. It follows that $\uRep(G)$ is co-complete, and that direct limits in $\uRep(G)$ can be computed in $\Mod_A$. In particular, filtered colimits in $\uRep(G)$ are exact, and so Grothendieck's axiom (AB5) holds. Finally, $\uRep(G)$ has a generator, namely, the direct sum of all smooth modules of the form $A/I$ with $I$ a left ideal.
\end{proof}

\begin{remark}
If $M$ is an arbitrary $A$-module then the collection of all smooth elements of $M$ forms an $A$-submodule, which is smooth; we denote it by $M^{\rm sm}$. One can show that $M \mapsto M^{\rm sm}$ defines a functor $\Mod_A \to \uRep(G)$ that is right adjoint to the inclusion $\uRep(G) \to \Mod_A$. We thus see that $\uRep(G)$ is a co-reflective subcategory of $\Mod_A$. In particular, limits in $\uRep(G)$ can be computed by first taking the limit in $\Mod_A$ and then applying $(-)^{\rm sm}$. For example, if $\{M_i\}_{i \in I}$ is a family of smooth modules then the product in $\uRep(G)$ is given by
\begin{displaymath}
\big( \prod_{i \in I} M_i \big)^{\rm sm} = \bigcup_{U \subset G} \big( \prod_{i \in I} M_i^{sU_i} \big),
\end{displaymath}
where the directed union is over all open subgroups $U$ of $G$.
\end{remark}

Suppose $U$ is an open subgroup of $G$. Then we have an inclusion of completed group algebras $A(U) \to A(G)$, see \S \ref{ss:A-basic}(c). If $M$ is a smooth $A(G)$-module then one readily verifies that it is smooth as an $A(U)$-module. We thus have a \defn{restriction functor}
\begin{displaymath}
\res \colon \uRep(G) \to \uRep(U).
\end{displaymath}
Since this functor is the identity on the underlying $k$-modules, it is exact and co-continuous.

Since the group algebra $k[G]$ is naturally a subalgebra of $A$ (see \S \ref{ss:A-basic}(a)), any $A$-module can be regarded as a $k[G]$-module. In particular, for a subgroup $U$ of $G$, we can consider the $U$-invariant elements of $M$, which we denote by $M^U$. The following proposition describes how invariants and strict invariants relate; a partial converse is given in Proposition~\ref{prop:sinv}.

\begin{proposition} \label{prop:strict-smooth}
Let $M$ be an $A$-module and let $U$ be an open subgroup. Then $M^{sU} \subset M^U$, that is, any element that is strictly $U$-invariant is $U$-invariant.
\end{proposition}

\begin{proof}
Suppose $x \in M$ is strictly $U$-invariant. For $g \in U$, consider $1-g \in k[G]$ as an element of $A$. Its push-forward to $\cC(G/U)$ vanishes. Since $x$ is strictly $U$-invariant, we therefore have $(1-g)x=0$, and so $gx=x$. This shows that $x$ is $U$-invariant.
\end{proof}

\begin{remark}
Define a $k[G]$-module to be \defn{smooth} if every element has open stabilizer. The above proposition shows that any smooth $A$-module restricts to a smooth $k[G]$-module. We warn the reader that an $A$-module that is smooth as a $k[G]$-module may not be smooth as an $A$-module. The category of smooth $k[G]$-modules is currently being examined by Ilia Nekrasov \cite{Nekrasov1}; see also \cite{DLLX}.
\end{remark}

\subsection{Schwartz spaces}

The following proposition constructs a natural $A$-module structure on Schwartz spaces. These will be the most important $A$-modules in our discussion.

\begin{proposition} \label{prop:A-schwartz}
Let $X$ be a $G$-set. Then $\cC(X)$ carries the structure of a smooth $A$-module in the following manner: given $a \in A$ and $\phi \in \cC(X)$ that is left $U$-invariant, $a\phi \in \cC(X)$ is defined by
\begin{displaymath}
(a\phi)(x) = \int_{G/U} a_U(g) \phi(g^{-1} x) dg
\end{displaymath}
\end{proposition}

\begin{proof}
It suffices to treat the case where $X=G/W$ for some $W$. Then $a\phi$ is just the convolution $a_U \ast_U \phi$ defined in \S \ref{ss:convol}. The proof that the stated formula defines an $A$-module structure on $\cC(X)$ is now similar to the proof that the algebra structure on $A$ is associative and unital (Proposition~\ref{prop:cga}). The fact that the $A$-module $\cC(X)$ is smooth follows from the definition of the module structure.
\end{proof}

If $a,b \in A$ then $(ab)_U=ab_U$, where on the left side we are using the multiplication law in $A$, and on the right side we are letting $a \in A$ act on $b_U \in \cC(G/U)$ via the $A$-module structure from the above proposition. This equality follows immediately from the definitions. In particular, we see that the projection map $A \to \cC(G/U)$ is one of left $A$-modules.

\begin{proposition} \label{prop:A-push-pull}
Let $f \colon X \to Y$ be a map of $G$-sets.
\begin{enumerate}
\item The map $f_* \colon \cC(X) \to \cC(Y)$ is $A$-linear.
\item If $X$ and $Y$ are finitary then the map $f^* \colon \cC(Y) \to \cC(X)$ is $A$-linear.
\end{enumerate}
\end{proposition}

\begin{proof}
(a) Suppose $\phi \in \cC(X)$ is left $U$-invariant. Then $f_*(\phi)$ is also left $U$-invariant. Thus
\begin{align*}
(a f_*(\phi))(y)
&= \int_{G/U} a_U(g) (f_*\phi)(g^{-1} y) dg
= \int_{G/U} \int_{f^{-1}(y)} a_U(g) \phi(g^{-1} x) dx dg \\
&= \int_{f^{-1}(y)} \int_{G/U} a_U(g) \phi(g^{-1} x) dg dx = (f_*(a\phi))(y),
\end{align*}
and so $f_*(a\phi)=a f_*(\phi)$ as required.

(b) Suppose $\psi \in \cC(Y)$ is left $U$-invariant. Then $f^*(\psi)$ is also left $U$-invariant. Thus
\begin{displaymath}
(a f^*(\psi))(x) =\int_{G/U} a_U(g) f(g^{-1} \psi(x)) dg = f^*(a\psi)(\psi(x)),
\end{displaymath}
and so $f^*(a\psi)=a f^*(\psi)$. This completes the proof.
\end{proof}

\begin{proposition} \label{prop:perm-rep-func}
We have a natural faithful $k$-linear functor
\begin{displaymath}
\Phi \colon \uPerm(G) \to \uRep(G), \qquad \Vec_X \mapsto \cC(X).
\end{displaymath}
\end{proposition}

\begin{proof}
Let $\Phi_0 \colon \uPerm(G) \to \Mod_k$ be the forgetful functor. This is $k$-linear and faithful, and takes $\Vec_X$ to $\cC(X)$. Since $\cC(X)$ is an $A$-module, we can regard $\Phi_0$ as taking values in $A$-modules on objects. To show that $\Phi_0$ defines a functor to $\uRep(G)$, it suffices to show that it carries morphisms in $\uPerm(G)$ to $A$-linear morphisms. Proposition~\ref{prop:A-push-pull} verifies this for the morphisms $A_f$ and $B_f$ constructed in \S \ref{ss:alpha-beta}. Since these generate all morphisms in $\uPerm(G)$ (Proposition~\ref{prop:hom-basis}), the result follows.
\end{proof}

\section{Normal measures} \label{s:normal}

\subsection{Overview}

In the previous section, we defined the representation category $\uRep(G)$. In this section, we introduce the notion of a normal measure (which is weaker than the quasi-regular condition introduced in \S \ref{ss:regular}), and prove a number of fundamental results about $\uRep(G)$ assuming normality. Here are some of these results:
\begin{itemize}
\item Schwartz space $\cC(X)$ has a mapping property.
\item The functor $\Phi \colon \uPerm(G) \to \uRep(G)$ is fully faithful.
\item We establish some fundamental exact sequences involving Schwartz spaces.
\item We show that invariants and strict invariants coincide for smooth modules.
\item We define some important averaging operators on smooth modules.
\item We define a version of integration for module-valued functions.
\end{itemize}
These foundational results are needed to establish the more exciting results about $\uRep(G)$ in subsequent sections. In \S \ref{ss:sym-char-p} we will see counterexamples to many of these results in the non-normal case.

We fix a pro-oligomorphic group $G$ and a $k$-valued measure $\mu$ throughout \S \ref{s:normal}. We assume for simplicity that $G$ is first-countable, though this is not needed for every result. We note that if $G$ is oligomorphic and $\Omega$ is countable then $G$ is first-countable. From \S \ref{ss:normal-schwartz} onwards, we assume that $\mu$ is normal.

\subsection{Normal measures}

We now introduce the key concept of \S \ref{s:normal}.

\begin{definition} \label{defn:normal}
We say that the measure $\mu$ is \defn{normal} if whenever $f \colon X \to Y$ is a surjection of finitary $G$-sets, the map $f_* \colon \cC(X) \to \cC(Y)$ is surjective.
\end{definition}

To verify $f_*$ is surjective, it is enough to realize $1_Y$ in its image: indeed, if $1_Y=f_*(\epsilon)$, for some $\epsilon \in \cC(X)$, then for any $\phi \in \cC(Y)$ we have $\phi=f_*(f^*(\phi) \cdot \epsilon)$ by the projection formula (Proposition~\ref{prop:projection}). From this, we see that normality is preserved under extension of scalars. (Given $k \to k'$ and a surjection $f \colon X \to Y$, choose a lift of $1_Y$ to a $k$-valued Schwartz function $\phi$ on $X$. Then the extension of scalars of $\phi$ to $k'$ is a lift of $1_Y$ as a $k'$-valued function.)

While the condition defining normality appears to be very plausible, there are cases where it does not hold. In particular, it often seems to fail in positive characteristic. See \S \ref{ss:sym-char-p} for an explicit example in the case of the symmetric group.

We now establish two basic results concerning normality.

\begin{proposition} \label{prop:normal-sub}
For $U \subset G$ an open subgroup, $\mu$ is normal if and only if $\mu \vert_U$ is.
\end{proposition}

\begin{proof}
Suppose $\mu \vert_U$ is normal and let $f \colon X \to Y$ be a surjection of finitary $G$-sets. Then $f$ is also a surjection of finitary $U$-sets, and so $f_*$ is surjective. Thus $\mu$ is normal.

Now suppose that $\mu$ is normal. Let $f \colon X \to Y$ be a surjection of finitary $U$-sets and let $\phi \in \cC(Y)$ be given. Let $\tilde{X}=I_U^G(X)$ and $\tilde{Y}=I_U^G(Y)$ be the inductions of $X$ and $Y$ to $G$ (see \S \ref{ss:induced-Gset}), and let $\tilde{f} \colon \tilde{X} \to \tilde{Y}$ be the induced map. We have a commutative diagram of $G$-sets
\begin{displaymath}
\xymatrix{
\tilde{X} \ar[rr]^{\tilde{f}} \ar[rd] && \tilde{Y} \ar[ld] \\ & G/U }
\end{displaymath}
and $X$ and $Y$ are identified with the fibers over $1 \in G/U$; in particular, $\tilde{f}^{-1}(Y)=X$. Now, regard $\phi$ as a function on $\tilde{Y}$ by extending it to~0 outside of $Y$. As such, it is still a smooth function. Since $\mu$ is normal, there is some $\psi \in \cC(\tilde{X})$ such that $\tilde{f}_*(\psi)=\phi$. Since the only points above $Y$ belong to $X$, it follows that $f_*(\psi \vert_X)=\phi$. Thus $\mu \vert_U$ is normal.
\end{proof}

Recall the notion of (quasi-)regular measure from Definition~\ref{defn:regular}.

\begin{proposition}
If $\mu$ is quasi-regular then it is normal.
\end{proposition}

\begin{proof}
Let $U$ be an open subgroup such that $\mu \vert_U$ is regular. Let $f \colon X \to Y$ be a surjection of $U$-sets. We claim that $f_*$ is surjective. We can consider each orbit on $Y$ separately, and thereby reduce to the case where $Y$ is transitive. And we can discard all but one orbit on $X$, and thus assume $X$ is transitive. Since $\mu \vert_U$ is regular, we have $f_*(c^{-1} 1_X)=1_Y$, where $c=\mu(X)/\mu(Y)$ is a unit of $k$, and so $f_*$ is surjective (see comments following Definition~\ref{defn:normal}). We thus see that $\mu \vert_U$ is normal, and so $\mu$ is normal by Proposition~\ref{prop:normal-sub}.
\end{proof}

In \S \ref{ss:sym-rel}, we give an example in the relative case of a measure that is normal and not quasi-regular. We know of no such example in the absolute case.

\subsection{Characterizations of normality}

We now give two alternate characterizations of normality (see \cite[\S 14.2]{arxiv} for a third). The first one is the main reason normality is so useful in studying $A$-modules.

\begin{proposition} \label{prop:A-surj}
The following conditions are equivalent:
\begin{enumerate}
\item The measure $\mu$ is normal.
\item For every open subgroup $U$ of $G$, the natural map $A \to \cC(G/U)$ is surjective.
\end{enumerate}
\end{proposition}

\begin{proof}
First suppose (a) holds. Then the transition maps $\pi_{U,V}$ in the inverse limit defining $A$ are surjective. Since $G$ is first-countable, the inverse limit can be taken over a countable cofinal family of open subgroups. Condition~(b) now follows from the following general fact: if $X_{\bullet}$ is a countably indexed inverse system of sets with surjective transition maps then the map $\varprojlim X_{\bullet} \to X_i$ is surjective for each $i$.

Now suppose (b) holds. Let $V \subset U$ be open subgroups of $G$. Then the surjective map $A \to \cC(G/U)$ factors through $\cC(G/V)$, and so wee see that $\pi_{V,U} \colon \cC(G/V) \to \cC(G/U)$ is surjective. It follows that if $f \colon Y \to X$ is any map of transitive $G$-sets then $f_*$ is surjective. Normality follows easily from this.
\end{proof}

\begin{remark} \label{rmk:not-1st-count}
Suppose $\mu$ is regular. Let $V \subset U$ be open subgroups, and let $f \colon G/V \to G/U$ be the natural map. Then $\mu(U/V)^{-1} f^*$ provides a canonical section to $f_*$ (by the projection formula Proposition~\ref{prop:projection}). From this, one can show Proposition~\ref{prop:A-surj}(b) holds even if $G$ is not first-countable. This argument can be adapted to the quasi-regular case as well.
\end{remark}

Our next result is a matrix-theoretic interpretation of normality.

\begin{proposition} \label{prop:A-right-inv}
The following conditions are equivalent:
\begin{enumerate}
\item The measure $\mu$ is normal.
\item Given a surjection $f \colon Y \to X$ of finitary $\hat{G}$-sets, the matrix $A_f \in \Mat_{X,Y}$ has a right inverse, i.e., there is a matrix $C \in \Mat_{Y,X}$ such that $A_f \cdot C=I_X$.
\end{enumerate}
\end{proposition}

\begin{proof}
First suppose that (b) holds. Let $f \colon Y \to X$ of finitary $\hat{G}$-sets. We claim that $f_* \colon \cC(Y) \to \cC(X)$ is surjective. To see this, let $w \in \Vec_X$ be given. Putting $v=Cw$, where $C$ is the right-inverse to $A_f$, we have $A_fv=w$. But $A_fv=f_*(v)$ by Proposition~\ref{prop:C-alpha}, and so the claim follows. Thus $\mu$ is normal.

Now suppose (a) holds. Let $f \colon Y \to X$ be a surjection of finitary $\hat{G}$-sets. Let $U$ be a group of definition for $X$ and $Y$ such that $f$ is $U$-equivariant. Since $\mu \vert_U$ is normal (Proposition~\ref{prop:normal-sub}), there is a function $\epsilon \in \cC(Y)$ such that $f_*(\epsilon)=1_X$. Let $C$ be the matrix given by $C(y,x)=\epsilon(y) B_f(y,x)$. For $v \in \Vec_X$, we have
\begin{displaymath}
(C v)(y) = \int_X C(y,x) v(x) dx = \epsilon(y) (B_fv)(y).
\end{displaymath}
We thus have $Cv=\epsilon \cdot f^*(v)$ by Proposition~\ref{prop:C-alpha}. Therefore,
\begin{displaymath}
A_fCv = f_*(\epsilon \cdot f^*(v)) = v \cdot f_*(\epsilon) = v,
\end{displaymath}
where in the first step we used Proposition~\ref{prop:C-alpha}, in the second the projection formula (Proposition~\ref{prop:projection}), and in the third the defining property of $\epsilon$. We thus have $A_fC=I_X$ by Proposition~\ref{prop:matrix-faithful}, and so (b) holds.
\end{proof}

\begin{remark} \label{rmk:normal-matrix}
We make three remarks concerning the above proposition:
\begin{enumerate}
\item If $f$ is a map of $G$-sets then the matrix $A_f$ is $G$-invariant. However, we cannot necessarily take $C$ to be $G$-invariant.
\item Taking transpose, condition~(b) is equivalent to the matrix $B_f$ having a left inverse. Note that if $f \colon X \to Y$ is a surjection of $G$-sets then $f^*$ is always injective. However, we need normality to conclude $B_f$ has a left inverse matrix (in general).
\item Suppose $\mu$ is regular, let $f \colon X \to Y$ be a surjection of transitive $G$-sets, and let $c$ be the common measure of fibers of $f$. Similar to Remark~\ref{rmk:not-1st-count}, we see that $C=c^{-1} B_f$ is a right inverse to $A_f$. We thus have an explicit $G$-invariant right inverse to $A_f$. \qedhere
\end{enumerate}
\end{remark}

\subsection{Schwartz spaces} \label{ss:normal-schwartz}

We assume for the remainder of \S \ref{s:normal} that the measure $\mu$ is normal. We now establish some basic properties of Schwartz spaces.

\begin{proposition} \label{prop:schwartz-generators}
Let $X$ be a $G$-set and let $\{x_i\}_{i \in I}$ be representatives for the $G$-orbits on $X$. Then the point masses $\delta_{X,x_i}$ generate $\cC(X)$ as an $A$-module.
\end{proposition}

\begin{proof}
It suffices to treat the case where $X=G/U$ for some open subgroup $U$, and show that the point-mass $\delta_1$ at $1 \in G/U$ generates $\cC(G/U)$. Since $\delta_1$ is left $U$-invariant, for $a \in A$ we have
\begin{displaymath}
(a \delta_1)(y)
=\int_{G/U} a_U(g) \delta_1(g^{-1}y) dg
=a_U(y).
\end{displaymath}
We therefore find $a\delta_1=a_U$. Thus given any function $\phi \in \cC(X)$, we have $\phi=a \delta_1$ where $a$ is any element of $A$ with $a_U=\phi$; such an $a$ exists by Proposition~\ref{prop:A-surj}. We thus see that $\delta_1$ generates, as required.
\end{proof}

\begin{corollary} \label{cor:schwartz-fg}
If $X$ is a finitary $G$-set then $\cC(X)$ is a finitely generated $A$-module.
\end{corollary}

\begin{proposition} \label{prop:schwartz-map-prop}
Let $M$ be an arbitrary $A$-module and let $U$ be an open subgroup of $G$. We then have an isomorphism
\begin{displaymath}
\Phi \colon \Hom_A(\cC(G/U), M) \to M^{sU}, \qquad \Phi(f)=f(\delta_1).
\end{displaymath}
\end{proposition}

\begin{proof}
The function $\delta_1$ is a strictly $U$-invariant element of $\cC(G/U)$. It follows that if $f \colon \cC(G/U) \to M$ is a map of $A$-modules then $f(\delta_1)$ is also strictly $U$-invariant. Thus $\Phi$ is a well-defined function. The element $\delta_1$ generates $\cC(G/U)$ by Proposition~\ref{prop:schwartz-generators}, and so $\Phi$ is injective. It remains to prove that $\Phi$ is surjective.

Let $x \in M^{sU}$ be given. We define a function $f \colon \cC(G/U) \to M$ by $f(\ol{a})=ax$ where $a$ is any lift of $\ol{a}$ to $A$. We note that a lift exists by Proposition~\ref{prop:A-surj}, and the element $ax$ is independent of the choice of $a$ since $x$ is strictly $U$-invariant. If $\ol{a} \in \cC(G/U)$ and $b \in A$ then $ba$ is a lift of $b \ol{a}$, since the map $A \to \cC(G/U)$ is $A$-linear (see the discussion following Proposition~\ref{prop:A-schwartz}). It follows that $f(b\ol{a})=bax=bf(\ol{a})$, and so $f$ is a map of $A$-modules. As $\Phi(f)=x$, we see that $\Phi$ is surjective, which completes the proof.
\end{proof}

\begin{proposition} \label{prop:schwartz-quot}
Let $M$ be a smooth $A$-module. Then there is a $G$-set $X$ and a surjection $\cC(X) \to M$ of $A$-modules. One can take $X$ finitary if and only if $M$ is finitely generated.
\end{proposition}

\begin{proof}
Let $\{x_i\}_{i \in I}$ be a generating set for $M$ as an $A$-module. Let $U_i$ be an open subgroup such that $x_i$ is strictly $U_i$-invariant. By Proposition~\ref{prop:schwartz-map-prop}, we have a map $\cC(G/U_i) \to M$ of $A$-modules that takes $\delta_1$ to $x_i$. We thus obtain a surjection $\cC(X) \to M$ with $X=\coprod_{i \in I} G/U_i$. If $M$ is finitely generated, we can take $I$ to be finite, and then $X$ is finitary. Conversely, if we have a surjection $\cC(X) \to M$ with $X$ finitary then $M$ is finitely generated since $\cC(X)$ is (Corollary~\ref{cor:schwartz-fg}).
\end{proof}

\subsection{Towards abelian envelopes} \label{ss:abenv}

Recall (Proposition~\ref{prop:perm-rep-func}) we have a faithful functor
\begin{displaymath}
\Phi \colon \uPerm(G) \to \uRep(G).
\end{displaymath}
We now establish some necessary properties for $\Phi$ to be an abelian envelope. See Theorem~\ref{thm:abenv} for a more definitive result.

\begin{proposition} \label{prop:normal-full}
The functor $\Phi$ is full.
\end{proposition}

\begin{proof}
Let $X$ and $Y$ be finitary $G$-sets. We must show that the map
\begin{equation} \label{eq:normal-full}
\Phi \colon \Hom(\Vec_X, \Vec_Y) \to \Hom_A(\cC(X), \cC(Y))
\end{equation}
is surjective. It suffices to treat the case $X=G/U$, so we assume this in what follows. By definition of $\uPerm(G)$, we have
\begin{displaymath}
\Hom(\Vec_X, \Vec_Y)=\Mat_{Y,X}^G = \Fun(G \backslash (X \times Y), k).
\end{displaymath}
We also have
\begin{displaymath}
\Hom_A(\cC(X), \cC(Y)) = \cC(Y)^{sU} = \cC(Y)^U = \Fun(G \backslash (X \times Y), k).
\end{displaymath}
The first identification above is Proposition~\ref{prop:schwartz-map-prop}, the second comes from the definition of the $A$-module structure on $\cC(Y)$, and the third comes from the isomorphism $G \backslash (X \times Y)=U \backslash Y$. We thus see that the domain and target in \eqref{eq:normal-full} are canonically identified. One can now argue that, under these identifications, $\Phi$ is the identity. Alternatively, one can argue as follows. First suppose $k$ is a field. Then \eqref{eq:normal-full} is an injection of vector spaces of the same finite dimension, and so it is necessarily surjective. The general case follows from the field case, since the source and target in \eqref{eq:normal-full} are free modules of the same finite rank (look at the determinant).
\end{proof}

\subsection{Some exact sequences}

We now establish some fundamental exact sequences involving Schwartz spaces. These results crucially rely upon the normality of the measure. We use these sequences in \S \ref{ss:strict-invar} to compare invariants and strict invariants.

\begin{proposition}
Consider a cartesian square
\begin{displaymath}
\xymatrix{
X' \ar[r]^{g'} \ar[d]_{f'} & X \ar[d]^f \\
Y' \ar[r]^g & Y }
\end{displaymath}
of finitary $G$-sets, with $f$ and $g$ surjective. Then the sequence
\begin{displaymath}
\xymatrix@C=4em{
\cC(X') \ar[r]^-{g'_*+f'_*} & \cC(X) \oplus \cC(Y') \ar[r]^-{f_*-g_*} & \cC(Y) \ar[r] & 0 }
\end{displaymath}
is exact.
\end{proposition}

\begin{proof}
Since $f$ is surjective, so is $f_*$ (by normality of $\mu$), and so $f_*-g_*$ is also surjective. By transitivity of push-forward (Proposition~\ref{prop:push-trans}), we have $f_*g'_*=g_*f'_*$, and so the first two maps in the sequence compose to zero.

Now suppose we have $\phi \in \cC(X)$ and $\psi \in \cC(Y')$ satisfying $f_*(\phi)=g_*(\psi)$, so that $(\phi, \psi)$ is a typical element of $\ker(f_*-g_*)$. We must show that $(\phi,\psi)$ belongs to the image of $g'_*+f'_*$. Since $g'$ is surjective, so is $g'_*$ (by normality), and so we can find $\tilde{\phi} \in \cC(X')$ with $g'_*(\tilde{\phi})=\phi$. Without loss of generality, we may as well replace $(\phi, \psi)$ with $(\phi,\psi)-(g'_*(\tilde{\phi}),f'_*(\tilde{\phi}))$, and thereby assume $\phi=0$ in what follows.

Appealing to normality again, let $\epsilon \in \cC(X)$ satisfy $f_*(\epsilon)=1_Y$. Let $\eta=(f')^*(\psi) \cdot (g')^*(\epsilon)$. We have
\begin{displaymath}
f'_*(\eta)=\psi \cdot f'_*((g')^*(\epsilon))=\psi \cdot g^*(f_*(\epsilon))=\psi \cdot 1_{Y'}=\psi
\end{displaymath}
where in the first step we used the projection formula (Proposition~\ref{prop:projection}) and in the second step base change (Proposition~\ref{prop:push-bc}). We also have
\begin{displaymath}
g'_*(\eta)=\epsilon \cdot g'_*((f')^*(\psi)) = \epsilon \cdot f^*(g_*(\psi))=0,
\end{displaymath}
again by the projection formula and base change. Note that $g_*(\psi)=0$ since $(0,\psi)$ belongs to the kernel of $f_*-g_*$. Thus $(0,\psi)=(g'_*(\eta), f'_*(\eta))$, and so $(0,\psi)$ is in the image of $f'_*+g'_*$. This completes the proof.
\end{proof}

\begin{corollary} \label{cor:exact}
Let $f \colon X \to Y$ be a surjection of finitary $G$-sets, and let $p,q \colon X \times_Y X \to X$ be the two projections. Then the sequence
\begin{displaymath}
\xymatrix@C=4em{
\cC(X \times_Y X) \ar[r]^-{p_*-q_*} & \cC(X) \ar[r]^-{f_*} & \cC(Y) \ar[r] & 0 }
\end{displaymath}
is exact. In particular, $\ker(f_*)$ is generated, as an $A$-module, by the elements $\delta_w-\delta_x$ where $w,x \in X$ satisfy $f(w)=f(x)$.
\end{corollary}

\begin{proof}
As in the proposition, $f_*$ is surjective and the two maps compose to zero. Suppose now that $\phi \in \cC(X)$ satisfies $f_*(\phi)=0$. Then $(\phi,0)$ belongs to the kernel of the map
\begin{displaymath}
(f_*,-f_*) \colon \cC(X) \oplus \cC(X) \to \cC(Y).
\end{displaymath}
Thus, by the proposition, there is $\psi \in \cC(X \times_Y X)$ such that $(\phi,0)=(p_*(\psi),q_*(\psi))$, and so $\phi=p_*(\psi)-q_*(\psi)$ belongs to the image of $p_*-q_*$. This shows that the sequence is exact. The statement about generators follows since the elements $\delta_{(w,x)}$ generate $\cC(X \times_Y X)$ by Proposition~\ref{prop:schwartz-generators}.
\end{proof}

\subsection{Invariants and strict invariants} \label{ss:strict-invar}

We now show that invariants and strict invariants coincide for a smooth module. This allows us to improve the mapping property for Schwartz spaces, and is also a crucial ingredient for the version of integration defined in \S \ref{ss:modint} (see Remark~\ref{rmk:strict-invar}).

\begin{proposition} \label{prop:sinv}
Let $M$ be a smooth $A$-module and let $U$ be an open subgroup of $G$. Then
\begin{displaymath}
M^{sU}=M^U.
\end{displaymath}
That is, an element of $M$ is strictly $U$-invariant if and only if it is $U$-invariant.
\end{proposition}

\begin{proof}
Let $x \in M$ be given. We have already seen that if $x$ is strictly $U$-invariant then it is $U$-invariant (Proposition~\ref{prop:strict-smooth}). Suppose now that $x$ is $U$-invariant. Let $V \subset U$ be an open subgroup such that $x$ is strictly $V$-invariant, and let $f \colon G/V \to G/U$ be the natural map. By Proposition~\ref{prop:schwartz-map-prop}, we have a map $h \colon \cC(G/V) \to M$ of $A$-modules given by $h(\ol{a})=ax$, where $a$ is any lift of $\ol{a}$ to $A$. For $g \in G$ and $u \in U$, we have
\begin{displaymath}
h(\delta_{G/V,g}-\delta_{G/V,gu}) = gx - gux = 0
\end{displaymath}
since $x$ is $U$-invariant. Since $\ker(h)$ is a $A$-submodule of $\cC(G/V)$, it contains the $A$-submodule generated by the elements $\delta_{G/V,g}-\delta_{G/V,gu}$. By Corollary~\ref{cor:exact}, these elements generate $\ker(f_*)$. Thus $\ker(f_*) \subset \ker(h)$. In particular, if $a \in A$ has $a_U=0$ then $ax=h(a_V)=0$ since $a_V \in \ker(f_*)$. Thus $x$ is strictly $U$-invariant, which completes the proof.
\end{proof}

\begin{corollary} \label{cor:sinv}
Let $M$ be a smooth $A$-module and let $U$ be an open subgroup of $G$. We then have an isomorphism
\begin{displaymath}
\Phi \colon \Hom_A(\cC(G/U), M) \to M^U, \qquad \Phi(f)=f(\delta_{G/U,1}).
\end{displaymath}
\end{corollary}

\begin{proof}
We have already seen (Proposition~\ref{prop:schwartz-map-prop}) that the map $\Phi$ gives an isomorphism to $M^{sU}$, and since $M$ is smooth we have $M^{sU}=M^U$ (Proposition~\ref{prop:sinv}).
\end{proof}

This corollary can be reformulated as follows:

\begin{corollary} \label{cor:schwartz-adj}
Let $M$ be a smooth $A$-module and let $X$ be a $G$-set. Then we have a natural isomorphism
\begin{displaymath}
\Hom_A(\cC(X), M) = \Hom_G(X, M),
\end{displaymath}
where the right side consists of all $G$-equivariant functions $X \to M$. In particular, we see that $\cC$ is the left adjoint of the forgetful functor $\uRep(G) \to \{ \text{$G$-sets} \}$.
\end{corollary}

\subsection{Averaging operators}

We now define certain averaging operators on smooth modules. These operators give an alternate way of working with the $A$-module structure which can sometimes be more convenient.

Let $M$ be a smooth $A$-module and let $V \subset U$ be open subgroups. We define a map
\begin{displaymath}
\avg_{U/V} \colon M^V \to M^U
\end{displaymath}
by $\avg_{U/V}(x)=ax$, where $a \in A$ is any element with $a_V=1_{U/V}$. We note that such an element $a$ exists by Proposition~\ref{prop:A-surj}, and $ax$ is independent of the choice of $a$ since $x$ is strictly $V$-invariant by Proposition~\ref{prop:sinv}. If $b$ is a second element of $A$ then $(ba)_V=b_U \ast_U a_V$ since $a_V$ is left $U$-invariant. In particular, if $b_U=0$ then $(ba)_V=0$, and so $bax=0$ since $x$ is strictly $V$-invariant. This shows that $ax$ is strictly $U$-invariant, and so $\avg_{U/V}$ does indeed take values in $M^U$. We refer to $\avg_{U/V}$ as an \defn{averaging operator}.

\begin{example}
If $G$ is finite then $\avg_{U/V}(x)=\sum_{g \in U/V} gx$.
\end{example}

The following proposition shows that the averaging operators and the ordinary group algebra $k[G]$ generate $A$, in a sense.

\begin{proposition} \label{prop:A-avg}
Let $x \in M^V$ and let $a \in A$. Suppose that $a_V$ is left $U$-invariant, write $G=\bigsqcup_{i=1}^r U g_i V$, and put $W_i=U \cap g_iVg_i^{-1}$. Then
\begin{displaymath}
ax = \sum_{i=1}^r a_V(g_i) \avg_{U/W_i}(g_ix).
\end{displaymath}
\end{proposition}

\begin{proof}
Let $b^i,c^i \in A$ satisfy $b^i_{W_i}=1_{U/W_i}$ and $c^i_V=\delta_{g_i}$. Then $c^i_V$ is left $W_i$-invariant, and so $(b^ic^i)_V=b^i_{W_i} \ast_{W_i} c^i$, which is easily seen to be the characteristic function of $Ug_i \subset G/V$. We thus have
\begin{displaymath}
a_V = \sum_{i=1}^r a_V(g_i) (b^ic^i)_V,
\end{displaymath}
and so
\begin{displaymath}
ax = \sum_{i=1}^r a_V(g_i) b^ic^i x.
\end{displaymath}
As $b^ic^ix=\avg_{U/W_i}(g_ix)$, the result now follows.
\end{proof}

As a consequence of the proposition, we find that one can simply work with averaging operators and the group algebra in place of the completed group algebra in many circumstances. In fact, one can give a precise description of smooth $A$-modules as smooth $k[G]$-modules equipped with averaging operators satisfying certain relations, but we do not do this. Here is one sample application of this point of view:

\begin{corollary} \label{cor:avg-homo}
Let $N$ be a smooth $A$-module and let $f \colon M \to N$ be a $k$-linear map. Then $f$ is $A$-linear if and only if the following two conditions hold:
\begin{enumerate}
\item $f$ is $G$-equivariant (i.e., $k[G]$-linear).
\item If $x \in M^V$ then $f(\avg_{U/V}{x})=\avg_{U/V}(f(x))$ for any $U$ containing $V$.
\end{enumerate}
\end{corollary}

We will, in fact, need one relation that the averaging operators satisfy. This is given in the following proposition.

\begin{proposition} \label{prop:avg-decomp}
Let $V$ be an open subgroup of $G$, and let $W$ and $V'$ be open subgroups of $V$. Write $V=\bigsqcup_{i=1}^r V'g_iW$, and let $W'_i=V' \cap g_iWg_i^{-1}$. Then we have
\begin{displaymath}
\avg_{V/W}(m) = \sum_{i=1}^r \avg_{V'/W'_i}(g_i m)
\end{displaymath}
whenever $m$ is a $V$-invariant element of a smooth $A$-module.
\end{proposition}

\begin{proof}
We have $V/W = \bigsqcup_{i=1}^r (V'g_iW)/W$ as subsets of $G/W$, and so $1_{V/W} = \sum_{i=1}^r 1_{(V'g_iW)/W}$ in $\cC(G/W)$. Let $a \in A$ satisfy $a_W=1_{V/W}$, and let $b_i \in A$ satisfy $(b_i)_{W'_i}=1_{V'/W'_i}$. By definition, we have
\begin{displaymath}
\avg_{V/W}(m)=am, \qquad \avg_{V'/W'_i}(g_im)=b_i g_i m.
\end{displaymath}
One readily verifies $(b_ig_i)_W=1_{(V'g_iW)/W}$, and so $(\sum_{i=1}^r b_i g_i)_W=1_{V/W}$. Thus one can take $a=\sum_{i=1}^r b_i g_i$, and so the result follows.
\end{proof}

In particular, we see that if $U \subset G$ is a fixed open subgroup then any averaging operator for $G$ can be expressed in terms of averaging operators involving subgroups of $U$. This suggests that $A(G)$ is generated by $A(U)$ and $k[G]$; in fact, one can show that $A(U) \cdot k[G]$ is dense in $A(G)$ (with respect to the inverse limit topology). The following corollary gives a concrete consequence of this:

\begin{corollary}
Let $N$ be a smooth $A$-module and let $f \colon M \to N$ be a $k$-linear map. Then $f$ is $A$-linear if and only if it is $G$-equivariant and $A(U)$-linear.
\end{corollary}

\subsection{Integrals with values in modules} \label{ss:modint}

Let $X$ be a $G$-set and let $M$ be a smooth $A$-module. Recall that a function $\phi \colon X \to M$ is smooth if it is $U$-equivariant for some open subgroup $U$. In this case, the support of $\phi$ is a $\hat{G}$-subset of $X$. We let $\cC(X,M)$ be the space of all smooth functions with finitary support. Our goal now is to define integration for functions in $\cC(X,M)$.

Given a point $x \in X$ with stabilizer $W$, an element $m \in M^W$, and an open subgroup $W \subset V$ there is a function $e^{V/W}_{x,m} \colon X \to M$ defined by $e^{V/W}_{x,m}(gx)=gm$ for $g \in V/W$, and $e^{V/W}_{x,m}(y)=0$ for $y \not\in Vx$. We call such functions \defn{elementary}. One easily sees that the elementary functions generate $\cC(X,M)$ as a $k$-module. The following proposition is the basis for our definition of integration.

\begin{proposition} \label{prop:elem2}
There is a unique $k$-linear map $\cI \colon \cC(X,M) \to M$ such that for an elementary function as above we have $\cI(e^{V/W}_{x,m})=\avg_{V/W}(m)$.
\end{proposition}

\begin{proof}
Let $V$ be an open subgroup of $G$, let $\{x_i\}_{i \in I}$ be representatives for the $V$-orbits on $X$, and let $W_i$ be the stabilizer of $x_i$ in $V$. Then we have an isomorphism
\begin{displaymath}
\cC(X, M)^V = \bigoplus_{i \in I} M^{W_i}.
\end{displaymath}
For $m \in M^{W_i}$ the corresponding element of $\cC(X,M)$ is the elementary function $e^{V/W_i}_{x_i,m}$. Define a $k$-linear map
\begin{displaymath}
\cI_V \colon \cC(X, M)^V \to M
\end{displaymath}
as follows: for $m \in M^{W_i}$, put $\cI_V(e^{V/W_i}_{x_i,m})=\avg_{V/W_i}(xm)$. One easily verifies that $\cI_V$ is a well-defined $k$-linear map, and is independent of the choice of orbit representatives.

Suppose now that $V' \subset V$ is a second open subgroup. Then we have $\cC(X,M)^V \subset \cC(X,M)^{V'}$, and we claim that $\cI_{V'}$ extends $\cI_V$. It suffices to consider a single $V$-orbit; thus fix $i \in I$. Write $V=\bigsqcup_{j=1}^r V' h_j W_i$ and let $W'_j=V' \cap h_j W_i h_j^{-1}$. Thus $Vx_i = \bigsqcup_{j=1}^r V'h_jx_i$, and $y_j$ has stabilizer $W'_j$ in $V'$. We have
\begin{displaymath}
e^{V/W_i}_{x_i,m} = \sum_{j=1}^r e^{V'/W'_j}_{h_jx_i, h_jm}.
\end{displaymath}
Indeed, as both sides are $V'$-equivariant and supported on $Vx_i$, it suffices to check the equality at the points $h_jx_i$ for $1 \le j \le r$, which is clear. Applying the definitions, we have
\begin{displaymath}
\cI_V(e^{V/W_i}_{x_i,m})=\avg_{V/W_i}(m), \qquad
\cI_{V'}(e^{V/W_i}_{x_i,m}) = \sum_{j=1}^r \avg_{V'/W'_j}(h_j m).
\end{displaymath}
These agree by Proposition~\ref{prop:avg-decomp}, and this establishes the claim.

It now follows that we have a well-defined $k$-linear map
\begin{displaymath}
\cI \colon \cC(X, M) \to M
\end{displaymath}
extending each $\cI_V$. It follows from the definition that $\cI(e^{V/W}_{x,m})=\avg_{V/W}(m)$. Since the elementary functions span $\cC(X,M)$ as a $k$-module, it follows that $\cI$ is unique.
\end{proof}

\begin{definition}
We define the \defn{integral} of a function $\phi$ in $ \cC(X,M)$ by
\begin{displaymath}
\int_X \phi(x) dx = \cI(\phi)
\end{displaymath}
where $\cI$ is the map constructed in Proposition~\ref{prop:elem2}.
\end{definition}

When $M=\bbone$ is the trivial module, we have $\cC(X,M)=\cC(X)$, and the above definition of integration agrees with the previous one: indeed, the elementary functions in the new sense match the elementary functions in the old sense, and the two integrals assign them the same value.

The action of $a \in A$ on $M$ can be expressed in a suggestive and useful way using integration, as the following proposition shows.

\begin{proposition}
Let $m \in M^V$ and $a \in A$. Then
\begin{displaymath}
am = \int_{G/V} a_V(g) \cdot gm\, dg.
\end{displaymath}
In particular, if $U$ is an open subgroup containing $V$ then
\begin{displaymath}
\avg_{U/V}(m) = \int_{U/V} gm\, dg.
\end{displaymath}
\end{proposition}

\begin{proof}
The second formula follows immediately from the definition of integration. The first follows from the second and Proposition~\ref{prop:A-avg}.
\end{proof}

Here is one more sample application of integration. Suppose $X$ is a transitive $G$-set and $x \in X$ has stabilizer $U$. Let $M$ be a smooth $A$-module. Then we have seen that the natural map $\Hom_A(\cC(X), M) \to M^U$ given by evaluation on $\delta_x$ is an isomorphism (Corollary~ \ref{cor:sinv}). We can succinctly write down the inverse using integration. Let $m \in M^U$ be given. We define a map $f \colon \cC(X) \to M$ by
\begin{displaymath}
f(\phi) = \int_{G/U} \phi(gx) gm\, dg.
\end{displaymath}
One readily verifies that $f$ is a map of $A$-modules, and that $f(\delta_x)=m$.

\begin{remark} \label{rmk:strict-invar}
There is one subtlety in the definition of integration that is worth pointing out. Suppose $m$ is a $U$-invariant element of $M$, let $\psi \in \cC(G/U)$, and let $\phi \in \cC(G/U, M)$ be the function $\phi(g)=\psi(g) gm$. Then the integral of $\phi$ is $am$, where $a$ is any element of $A$ with $a_U=\psi$, and this is really the only sensible definition. For this to be well-defined, however, we actually need $m$ to be \emph{strictly} $U$-invariant. In other words, the equality $M^U=M^{sU}$ established in Proposition~\ref{prop:sinv} is in fact a prerequisite to having a theory of integration for module-valued functions.
\end{remark}

\section{Tensor products} \label{s:genten}

\subsection{Overview}

In this section, we define a tensor structure on the category $\uRep(G)$. There is a formal approach to this: one can declare $\cC(X) \uotimes \cC(Y)=\cC(X \times Y)$ by fiat, and then extend $\uotimes$ to all of $\uRep(G)$ by choosing presentations by Schwartz spaces. While it is possible to carry out this plan, it is rather unenlightening since it does not attached any meaning to the tensor product of two modules.

We take a more organic approach: given two smooth $A$-modules $M$ and $N$, we define a tensor product to be a smooth $A$-module $T$ equipped with a universal ``strongly bilinear'' map $M \times N \to T$. Strongly bilinear means that in each variable the map is $k[G]$-linear and also commutes with integrals. The main result of this section (Theorem~\ref{thm:tensor}) asserts that the tensor products always exist, and that $\uRep(G)$ therefore has the structure of a tensor category. We give a similar treatment to internal $\Hom$, and show that the usual tensor-Hom adjunction holds.

We fix a first-countable pro-oligomorphic group $G$ and a $k$-valued normal measure $\mu$ for the duration of \S \ref{s:genten}.

\subsection{Strongly bilinear maps}

Our notion of tensor product is based on the following notion of bi-linearity for $A$-modules:

\begin{definition} \label{defn:strong-bi}
Let $M$, $N$, and $E$ be smooth $A$-modules, and consider a function
\begin{displaymath}
q \colon M \times N \to E.
\end{displaymath}
We say that $q$ is \defn{strongly bilinear} if it satisfies the following conditions:
\begin{enumerate}
\item $q$ is $k$-bilinear.
\item $q$ is $G$-equivariant, that is, $q(gm,gn)=gq(m,n)$ for $g \in G$, $m \in M$, and $n \in N$.
\item Given finitary $G$-sets $X$ and $Y$ and smooth functions $\phi \colon X \to M$ and $\psi \colon Y \to N$, we have
\begin{displaymath}
q \left( \int_X \phi(x) dx, \int_Y \psi(y) dy \right) = \int_{X \times Y} q(\phi(x), \psi(y)) d(x,y).
\end{displaymath}
Note that $(x,y) \mapsto q(\phi(x), \psi(y))$ is smooth by (b), and so the integral on the right side is well-defined. \qedhere
\end{enumerate}
\end{definition}

The most subtle aspect of the above definition is condition~(c). We now discuss it in some more detail. Fix a map $q$ as above, and assume that it satisfies (a) and (b). We say that $q$ is \defn{strongly linear} in the first parameter if for any smooth function $\phi \colon X \to M$ and any $n \in N$ we have
\begin{displaymath}
q \left( \int_X \phi(x) dx, n \right) = \int_X q(\phi(x), n) dx
\end{displaymath}
Strongly linear in the second variable is defined similarly. It is easy to see that $q$ satisfies~(c) if and only if it is strongly linear in each variable.

\begin{proposition} \label{prop:strong-lin}
With the above notation, the following conditions are equivalent:
\begin{enumerate}
\item $q$ is strongly linear in the first variable.
\item For every $n \in N$ there is an open subgroup $U$ of $G$ such that $n$ is $U$-invariant and $q(am,n)=aq(m,n)$ for all $a \in A(U)$ and $m \in M$.
\item For every $n \in N$ there is an open subgroup $U$ of $G$ such that $n$ is $U$-invariant and
\begin{displaymath}
q(\avg_{V/W}(m),n)=\avg_{V/W}(q(m,n))
\end{displaymath}
for all open subgroups $W \subset V \subset U$ and $m \in M^V$.
\end{enumerate}
\end{proposition}

\begin{proof}
First suppose (a) holds, and let us prove (b). Let $n \in N$ be given, and let $U$ be any open subgroup such that $n$ is $U$-invariant. We show that $q(-, n)$ is $A(U)$-linear. Thus let $m \in M$ and $a \in A(U)$ be given. Suppose that $m$ is $V$-invariant with $V \subset U$. Then
\begin{align*}
q(am, n)
&=q\left( \int_{U/V} a_V(g) gm dg, n \right)
=\int_{U/V} a_V(g) q(gm, n) dg \\
&=\int_{U/V} a_V(g) g q(m,n) dg = aq(m,n).
\end{align*}
In the first step, we used the expression of $am$ as an integral; in the second step, we used the strong linearity of $q$ in the first argument; in the third step, we used that $q$ is $G$-equivariant and that $gn=n$ for $g \in U$; in the final step, we used the expression for $aq(m,n)$ as an integral (note that $q(m,n)$ is $V$-invariant since both $m$ and $n$ are). We thus see that (b) holds.

It is clear that (b) implies (c), since the averaging operators are defined using the $A$-module structure.

Finally, suppose (c) holds and let us prove (a). Fix $n \in N$, and let $U$ be as in (c). Let $\phi \colon X \to M$ be a smooth function with $X$ a finitary $G$-set. Since the elementary functions generate $\cC(X,M)$ as a $k$-module, it suffices to treat the case where $\phi$ is elementary, say $e^{V/W}_{x,m}$. In this case, the identity in question becomes
\begin{displaymath}
q(\avg_{V/W}(m), n) = \int_{V/W} q(gm, n) dg.
\end{displaymath}
Let $V'=U \cap V$. Write $V=\bigsqcup_{i=1}^r V' h_i W$ and put $W_i=V' \cap h_iWh_i^{-1}$, so that $V/W \cong \bigsqcup_{i=1}^r V'/W_i$. Thus
\begin{align*}
q(\avg_{V/W}(m),n)
&= \sum_{i=1}^r q(\avg_{V'/W_i}(h_i m), n)
= \sum_{i=1}^r \avg_{V'/W_i}(q(h_im, n)) \\
&= \sum_{i=1}^r \int_{V'/W_i} gq(h_im, n) dg
= \sum_{i=1}^r \int_{V'/W_i} q(gh_im, n) dg \\
&= \int_{V/W} q(gm, n) dg.
\end{align*}
In the first step, we broke $\avg_{V/W}$ up over the $V'$ orbits on $V/W$ via Proposition~\ref{prop:avg-decomp}; in the second, we used that $q$ commutes with $\avg_{V'/W_i}$ since $V' \subset U$; in the third, we simply wrote the definition of $\avg_{V'/W_i}$; in the fourth, we used that $q$ is $G$-equivariant and that $gn=n$ since $g \in V' \subset U$; and in the final step, we simply merged the $r$ integrals into one. We thus see that (a) holds.
\end{proof}

\subsection{The definition of tensor product}

We now arrive at the notion of tensor product:

\begin{definition}
Let $M$ and $N$ be smooth $A$-modules. A \defn{tensor product} of $M$ and $N$ is a smooth $A$-module $T$ equipped with a strongly bilinear map $q \colon M \times N \to T$ that is universal, in the sense that if $q' \colon M \times N \to T'$ is any other strongly bilinear map to a smooth $A$-module then there is a unique map $f \colon T \to T'$ of $A$-modules such that $q'=f \circ q$.
\end{definition}

Suppose that a tensor product of $M$ and $N$ exists. By the universal property, any two tensor products are canonically isomorphic, so we can speak of \emph{the} tensor product of $M$ and $N$; we denote it by $M \uotimes N$. For $x \in M$ and $y \in N$, we write $x \uotimes y$ for the image of $(x,y) \in M \times N$ under the universal strongly bilinear map $q$. These elements are the \defn{pure tensors} in $M \uotimes N$. Given finitary $G$-sets $X$ and $Y$ and smooth functions $\phi \colon X \to M$ and $\psi \colon Y \to N$, we have the identity
\begin{displaymath}
\bigg( \int_X \phi(x) dx \bigg) \uotimes \bigg( \int_Y \psi(y) dy \bigg) = \int_{X \times Y} \big( \phi(x) \uotimes \psi(y) \big) d(x,y).
\end{displaymath}
This is simply a restatement of Definition~\ref{defn:strong-bi}(c).

We now establish a few simple properties of tensor products.

\begin{proposition} \label{prop:pure-gen}
Let $M$ and $N$ be smooth $A$-modules such that the tensor product $M \uotimes N$ exists. Then the pure tensors $x \uotimes y$ with $x \in M$ and $y \in N$ generate $M \uotimes N$ as an $A$-module.
\end{proposition}

\begin{proof}
Let $T=M \uotimes N$ and let $T_0$ be the $A$-submodule generated by the pure tensors. The universal property of $T$ implies that any $A$-linear map out of $T$ is determined by its restriction to $T_0$, and so $T=T_0$; indeed, the quotient map and the zero map $T \to T/T_0$ agree on $T_0$, and so are equal.
\end{proof}

\begin{proposition} \label{prop:ten-func}
Let $f \colon M \to M'$ and $g \colon N \to N'$ be maps of smooth $A$-modules. Suppose that the tensor products $M \uotimes N$ and $M' \uotimes N'$ exist. Then there is a unique map of $A$-modules
\begin{displaymath}
f \uotimes g \colon M \uotimes N \to M' \uotimes N'
\end{displaymath}
such that
\begin{displaymath}
(f \uotimes g)(x \uotimes y)=f(x) \uotimes g(y)
\end{displaymath}
holds for all $x \in M$ and $y \in N$.
\end{proposition}

\begin{proof}
One easily sees that the composition
\begin{displaymath}
M \times N \to M' \times N' \to M' \uotimes N'
\end{displaymath}
is strongly bilinear, and so the proposition follows from the universal property of the tensor product $M \uotimes N$.
\end{proof}

\subsection{Tensor product of Schwartz spaces}

Let $X$ and $Y$ be $G$-sets. We now examine the tensor product of the Schwartz spaces $\cC(X)$ and $\cC(Y)$. This will play a crucial role in our study of general tensor products. We have a natural map
\begin{displaymath}
q \colon \cC(X) \times \cC(Y) \to \cC(X \times Y)
\end{displaymath}
defined as follows: given $\phi \in \cC(X)$ and $\psi \in \cC(Y)$, we let $q(\phi, \psi)$ be the function on $X \times Y$ given by $(x,y) \mapsto \phi(x) \psi(y)$, which one readily verifies is a Schwartz function. As more might expect, this gives the tensor product in this case:

\begin{proposition} \label{prop:schwartz-tensor}
The map $q$ is a universal strongly bilinear map, and so:
\begin{displaymath}
\cC(X) \uotimes \cC(Y) = \cC(X \times Y).
\end{displaymath}
\end{proposition}

\begin{proof}
We break the proof into four steps.

\textit{Step 1: $q$ is strongly bilinear.} It is clear that $q$ is $k$-bilinear and $G$-equivariant. We now show that $q$ is strongly linear in the first argument. Let $\psi \in \cC(Y)$ be given and let $U$ be an open subgroup such that $\psi$ is left $U$-invariant. Suppose we have open subgroups $W \subset V \subset U$ and $\phi \in \cC(X)$ is left $W$-invariant. We then have
\begin{displaymath}
q(\avg_{V/W}(\phi), \psi)=\avg_{V/W}(q(\phi,\psi)).
\end{displaymath}
Indeed, evaluating the above functions at $(x,y)$, this simply becomes the identity
\begin{displaymath}
\int_{V/W} \phi(h^{-1}x) dh \cdot \psi(y) = \int_{V/W} \phi(h^{-1} x) \psi(h^{-1} y) dh,
\end{displaymath}
which is clear since $\psi(h^{-1} y)=\psi(y)$ for all $h \in V$. We thus see that $q$ is strongly linear in the first argument by Proposition~\ref{prop:strong-lin}. Strong linearity in the second argument follows by symmetry, and so $q$ is strongly bilinear.

To complete the proof, we must show that $q$ is universal. Thus suppose that $q' \colon \cC(X) \times \cC(Y) \to E$ is some other strongly bilinear map. Define $f \colon \cC(X \times Y) \to E$ by
\begin{displaymath}
f(\phi) = \int_{X \times Y} \phi(x,y) q'(\delta_x, \delta_y) d(x,y).
\end{displaymath}
Here $\delta_x \in \cC(X)$ is the point mass at $x$. Since $q'$ is $G$-equivariant, the function $X \times Y \to E$ given by $(x,y) \mapsto q'(\delta_x, \delta_y)$ is as well. Thus the integrand above is smooth, and so the integral is defined.

\textit{Step 2: $f$ is $A$-linear.} It is clear that $f$ is $k$-linear. Let $a_{\bullet} \in A$ be given, and suppose $\phi \in \cC(X \times Y)$ is left $U$-invariant. Then
\begin{align*}
f(a\phi)
&= \int_{X \times Y} \int_{G/U} a_U(g) \phi(g^{-1}x, g^{-1}y) q'(\delta_x, \delta_y) dg d(x,y) \\
&= \int_{G/U} \int_{X \times Y} a_U(g) \phi(x,y) g q'(\delta_x, \delta_y) d(x,y) dg \\
&= \int_{G/U} a_U(g) gf(\phi) dg = a f(\phi).
\end{align*}
In the second step, we changed the order of integration, made the change of variables $(x,y) \to (gx,gy)$, and used the identity $q'(\delta_{gx}, \delta_{gy})=gq'(\delta_x, \delta_y)$, which follows from the $G$-equivariance of $\delta$. Note that $f(\phi)$ is easily seen to be $U$-invariant, which is why $af(\phi)$ coincides with the final integral above. We have thus shown that $f$ is $A$-linear.

\textit{Step 3: $q'$ factors as $f \circ q$.} Let $\phi \in \cC(X)$ and $\psi \in \cC(Y)$ be given. Let $\phi' \colon X \to \cC(X)$ be the function $\phi'(x)=\phi(x) \delta_x$. Then $\phi = \int_X \phi'(x) dx$. A similar identity holds for $\psi$. We thus have
\begin{align*}
q'(\phi, \psi)
&= q'\left( \int_X \phi'(x) dx, \int_Y \psi'(y) dy \right) \\
&= \int_{X \times Y} q'(\phi'(x), \psi'(y)) d(x,y) \\
&= \int_{X \times Y} \phi(x) \psi(y) q'(\delta_x, \delta_y) d(x,y)
= f(q(\phi, \psi)).
\end{align*}
In the second step we used we used property Definition~\ref{defn:strong-bi}(c) of $q'$, and in the third step we used the order $k$-bilinearity of $q'$.

\textit{Step 4: the factorization of $q'$ is unique.} Suppose that $f' \colon \cC(X \times Y) \to E$ is another $A$-linear map satisfying $q'=f' \circ q$. We must show $f'=f$. However, this is clear as
\begin{displaymath}
f(\delta_{(x,y)})=f(q(\delta_x, \delta_y))=f'(q(\delta_x,\delta_y))=f'(\delta_{(x,y)})
\end{displaymath}
and the functions $\delta_{(x,y)}$, with $x \in X$ and $y \in Y$, generate $\cC(X \times Y)$ as an $A$-module (Proposition~\ref{prop:schwartz-generators}).
\end{proof}

\begin{remark}
Proposition~\ref{prop:schwartz-tensor} can be seen as an analog of the isomorphism
\begin{displaymath}
\cS(\bR^n) \mathbin{\hat{\otimes}} \cS(\bR^m) \cong \cS(\bR^{n+m}),
\end{displaymath}
where $\cS$ denotes the classical Schwartz space and $\hat{\otimes}$ is the (projective or injective) tensor product of nuclear spaces. See \cite[Theorem~5.16]{Treves}.
\end{remark}

\subsection{Right exactness of tensor products}

We now establish the following important property of tensor products:

\begin{proposition} \label{prop:ten-right-exact}
Let $N$ be a smooth $A$-module and let
\begin{displaymath}
\xymatrix@C=3em{
M_3 \ar[r]^{\beta} & M_2 \ar[r]^{\alpha} & M_1 \ar[r] & 0 }
\end{displaymath}
be an exact sequence of smooth $A$-modules. Suppose that $M_3 \uotimes N$ and $M_2 \uotimes N$ exist. Then $M_1 \uotimes N$ exists, and the sequence
\begin{displaymath}
\xymatrix@C=3em{
M_3 \uotimes N \ar[r]^{\beta \uotimes 1} & M_2 \uotimes N \ar[r]^{\alpha \uotimes 1} & M_1 \uotimes N \ar[r] & 0 }
\end{displaymath}
is exact.
\end{proposition}

\begin{proof}
We first note that the map $\beta \uotimes 1$ comes from Proposition~\ref{prop:ten-func}. Define $C=\coker(\beta \uotimes 1)$, and let $\pi \colon M_2 \uotimes N \to C$ be the quotient map. We show that $C$ is naturally the tensor product of $M_1$ and $N$. We proceed in three steps.

\textit{Step 1: existence of $q_1$.} Consider the diagram
\begin{displaymath}
\xymatrix@C=4em{
M_2 \uotimes N \ar[r]^-{\pi} & C \\
M_2 \times N \ar[u]^{q_2} \ar[r]^-{\alpha \times 1} & M_1 \times N \ar@{..>}[u]_{q_1} }
\end{displaymath}
Here $q_2$ is the universal strongly bilinear map. We claim that there is a unique function $q_1$ that makes the diagram commute. Uniqueness is clear, as $\alpha \times 1$ is surjective. As for existence, let $(x,y) \in M_1 \times N$ be given. We define $q_1(x,y)=\pi(q_2(x', y))$ where $x' \in M_2$ is a lift of $x$. To see that this is well-defined, supposed that $x'' \in M_2$ is a second lift. Then $x''-x'=\beta(w)$ for some $w \in M_3$. We have
\begin{displaymath}
q_2(x'',y)-q_2(x',y)=q_2(\beta(w),y)=(\beta \uotimes 1)(q_3(w,y)),
\end{displaymath}
where $q_3$ is the universal strongly bilinear map on $M_3 \uotimes N$. As the right term maps to~0 in $C$, we see that $\pi(q_2(x'',y))=\pi(q_2(x',y))$, and so $q_1$ is well-defined. This verifies the claim.

\textit{Step 2: $q_1$ is strongly bilinear.} It is clear that $q_1$ is $k$-bilinear and $G$-equivariant. Let $n \in N$ be given and let $U$ be an open subgroup such that $n \in N^U$ and $q_2(-,n)$ is $A(U)$-linear, which exists by Proposition~\ref{prop:strong-lin}. It is then clear that $q_1(-,n)$ is $A(U)$-linear, and so $q_1$ is strongly linear in the first variable by Proposition~\ref{prop:strong-lin}. Now let $m_1 \in M_1$ be given, and let $m_2 \in M_2$ be a lift of $m_1$. Let $V$ be an open subgroup such that $m_2$ is $V$-invariant and $q_2(m_2,-)$ is $A(V)$-linear, which exists by Proposition~\ref{prop:strong-lin}. Then it is clear that $q_1(m_1, -)$ is $A(V)$-linear, and so $q_1$ is strongly linear in the second variable by Proposition~\ref{prop:strong-lin}. Thus $q_1$ is strongly bilinear.

\textit{Step 3: $q_1$ is universal.} Let $q \colon M_1 \times N \to E$ be a strongly bilinear map. Then $q \circ (\alpha \times 1) \colon M_2 \times N \to E$ is also strongly bilinear, and so by the universal property of the tensor product, there is a unique $A$-linear map $\gamma \colon M_2 \uotimes N \to E$ such that $\gamma(q_2(x, y))=q(\alpha(x), y)$ for $x \in M_2$ and $y \in N$. For $x \in M_3$ and $y \in N$, we have
\begin{displaymath}
\gamma((\beta \uotimes 1)(q_3(x,y)))=\gamma(q_2(\beta(x),y))=q(\alpha(\beta(x)),y)=0,
\end{displaymath}
and so, by the universal property of $M_3 \uotimes N$, we have $\gamma \circ (\beta \uotimes 1)=0$. It follows that $\gamma$ factors through $C$, that is, there is a unique $A$-linear map $\delta \colon C \to E$ such that $\gamma=\delta \circ \pi$. Now let $x \in M_1$ and $y \in N$, and let $x' \in M_2$ lift $x$. Then
\begin{displaymath}
\delta(q_1(x,y))=\delta(\pi(q_2(x',y))=\gamma(q_2(x',y))=q(x,y).
\end{displaymath}
We have thus shown that $q$ factors through $q_1$. Suppose now that $\delta' \colon C \to E$ is a second $A$-linear map such that $\delta' \circ q_1=q$. Then $(\delta-\delta') \circ q_1=0$. However, the image of $q_1$ generates $C$ as an $A$-module, since the image of $q_2$ generates $M_2 \uotimes N$ (Proposition~\ref{prop:pure-gen}) and $\pi$ is surjective. Thus $\delta=\delta'$, and so $q_1$ is universal.

We have thus shown that $q_1 \colon M_1 \times N \to C$ is a tensor product of $M_1$ and $N$. The diagram in Step~1 shows that the map $\pi$ is identified with $\alpha \uotimes 1$, which yields the exact sequence in the statement of the lemma.
\end{proof}

\subsection{The main theorem on tensor products}

The following is our main theorem on the tensor product.

\begin{theorem} \label{thm:tensor}
We have the following:
\begin{enumerate}
\item The tensor product of any two smooth $A$-modules exists.
\item The tensor product $\uotimes$ endows $\uRep(G)$ with the structure of a tensor category.
\item The functor $\uotimes$ is co-continuous in each variable.
\item The tensor product of two finitely generated smooth $A$-modules is finitely generated.
\end{enumerate}
\end{theorem}

\begin{proof}
(a) We have already seen that tensor products exist for Schwartz spaces (Proposition~\ref{prop:schwartz-tensor}). The general case follows from choosing presentations by Schwartz spaces and appealing to Proposition~\ref{prop:ten-right-exact}.

(b) We now know that the tensor product of any two objects in $\uRep(G)$ exists. Proposition~\ref{prop:ten-func} thus shows that $\uotimes$ defines a functor. To show that we have a tensor structure, we must construct associative, symmetry, and unit constraints and verify the relevant axioms. This follows in a similar manner as for the usual tensor product. For instance, to get the associativity constraint, one observes that $(M_1 \uotimes M_2) \uotimes M_3$ and $M_1 \uotimes (M_2 \uotimes M_3)$ both receive a universal strongly trilinear map, and are therefore canonically isomorphic. The unit object for $\uotimes$ is the trivial $A$-module $\bbone$: indeed, one easily sees that the map $\bbone \times M \to M$ given by $(a,x) \mapsto ax$ realizes $M$ as the tensor product $\bbone \uotimes M$.

(c) Proposition~\ref{prop:ten-right-exact} shows that $\uotimes$ is right-exact in each argument. It remains to show that $\uotimes$ commutes with arbitrary direct sums. Let $\{M_i\}_{i \in I}$ be a family of smooth $A$-modules, and let $N$ be another smooth $A$-module. We must show that the natural map
\begin{displaymath}
\bigoplus_{i \in I} (M_i \uotimes N) \to \big( \bigoplus_{i \in I} M_i \big) \uotimes N
\end{displaymath}
is an isomorphism. If $M_i=\cC(X_i)$ and $N=\cC(Y)$, for $G$-sets $X_i$ and $Y$, then the result follows from Proposition~\ref{prop:schwartz-tensor} and the identification $\bigoplus_{i \in I} \cC(X_i) = \cC(\coprod_{i \in I} X_i)$. The general case follows from this by choosing presentations.

(d) Let $M$ and $N$ be finitely generated smooth $A$-modules. Choose surjections $\cC(X) \to M$ and $\cC(Y) \to M$ with $X$ and $Y$ finitary (Proposition~\ref{prop:schwartz-quot}). By Proposition~\ref{prop:schwartz-tensor} and the right-exactness of tensor products, we have a surjection $\cC(X \times Y) \to M \uotimes N$. Since $X \times Y$ is finitary, we conclude that $M \uotimes N$ is finitely generated (Proposition~\ref{prop:schwartz-quot} again).
\end{proof}

\subsection{More properties of tensor product}

We now establish a few additional results on our tensor product.

\begin{proposition}
The functor $\uPerm(G) \to \uRep(G)$ is naturally a tensor functor.
\end{proposition}

\begin{proof}
This follows from Proposition~\ref{prop:schwartz-tensor}.
\end{proof}

\begin{corollary} \label{cor:schwartz-rigid}
Let $X$ be a finitary $G$-set. Then $\cC(X)$ is a self-dual rigid object of $\uRep(G)$ of categorical dimension $\mu(X)$, and the functor $\cC(X) \uotimes -$ is exact.
\end{corollary}

\begin{proof}
We have seen that $\Vec_X$ is a self-dual rigid object of $\uPerm(G)$ of categorical dimension $\mu(X)$ (Propositions~\ref{prop:perm-rigid} and Corollary~\ref{cor:cat-dim}). Since a tensor functor preserves rigidity and categorical dimension, it follows that $\cC(X)$ is a self-dual rigid object of categorical dimension $\mu(X)$. If $M$ is any rigid object of a tensor category then $M \otimes -$ is exact, since $M^{\vee} \otimes -$ provides a left and right adjoint \cite[Proposition~2.10.8]{Etingof}. This in particular applies to $M=\cC(X)$.
\end{proof}

If $M$ and $N$ are $k[G]$-modules then the $k$-module underlying their tensor product $M \otimes_k N$ only depends on the $k$-modules underlying $M$ and $N$. For the tensor product $\uotimes$ of smooth $A$-modules, this is no longer the case. However, the next proposition shows that, in a sense, the tensor product only depends on the ``underlying $\hat{G}$-modules.''

\begin{proposition} \label{prop:res-tensor}
Let $U$ be an open subgroup of $G$. Then restriction $\uRep(G) \to \uRep(U)$ is naturally a tensor functor.
\end{proposition}

\begin{proof}
Write $\uotimes^G$ and $\uotimes^U$ for the tensor products on $\uRep(G)$ and $\uRep(U)$. Let $M$ and $N$ be smooth $A(G)$-modules. Then the universal map $M \times N \to M \uotimes^G N$ is clearly strictly bilinear after restricting to $U$, and so the universal property of $\uotimes^U$ furnishes us with a canonical map of $A(U)$-modules $M \uotimes^U N \to M \uotimes^G N$. This map is an isomorphism when $M$ and $N$ are Schwartz spaces by the explicit computation of tensor products in this case (Proposition~\ref{prop:schwartz-tensor}). The general case now follows by picking presentations of $M$ and $N$ by Schwartz spaces.
\end{proof}

\subsection{Internal Hom} \label{ss:gen-hom}

Let $M$ and $N$ be smooth $A$-modules. Define
\begin{displaymath}
\uHom_{A(G)}(M,N) = \bigcup_{U \subset G} \Hom_{A(U)}(M,N),
\end{displaymath}
where the union is taken over all open subgroups $U$ of $G$, and each $\Hom_{A(U)}(M,N)$ is regarded as a subset of $\Hom_k(M,N)$. Clearly, if $U$ is an open subgroup of $G$ then
\begin{displaymath}
\uHom_{A(G)}(M,N) = \uHom_{A(U)}(M,N).
\end{displaymath}
For this reason, we simply write $\uHom(M,N)$ when there is no danger of confusion.

\begin{proposition}
Let $M$ and $N$ be as above. Then $\uHom(M,N)$ carries the structure of a smooth $A$-module as follows: given $a \in A(G)$ and $f \in \Hom_{A(U)}(M,N)$, we have
\begin{displaymath}
(af)(m) = \int_{G/U} a_U(g) gf(g^{-1}m) dg.
\end{displaymath}
\end{proposition}

\begin{proof}
We break the proof into four steps.

\textit{Step 1: $af$ is well-defined.} We first show that the above formula for $(af)(m)$ is independent of the choice of $U$. Thus suppose that $V$ is an open subgroup of $U$. We have
\begin{align*}
\int_{G/V} a_V(g) gf(g^{-1} m) dg
&= \int_{G/U} \int_{U/V} a_V(gh) ghf(h^{-1}g^{-1} m) dh dg \\
&= \int_{G/U} \left[ \int_{U/V} a_V(gh) dh \right] gf(g^{-1} m) dg \\
&=\int_{G/U} a_U(g) gf(g^{-1} m) dg
\end{align*}
In the first step, we used the transitive of push-forward (Proposition~\ref{prop:push-trans}); in the second, we used that $f$ is $U$-equivariant; and in the third, we used that $a_U$ is the push-forward of $a_V$. This verifies independence of $U$. We now define $(af)(m)$ by the formula in the statement of the proposition.

\textit{Step 2: $af$ belongs to $\uHom$.} We have constructed a well-defined function $af \colon M \to N$. We now show that it actually belongs to $\uHom(M,N)$. It is clear that $af$ is $k$-linear. Let $V$ be an open subgroup of $U$ such that $a_U(g)$ is left $V$-invariant. We show that $f$ is $A(V)$-linear. Thus let $b \in A(V)$ and $m \in M$ be given. Let $W$ be an open subgroup of $V$ such that $m$ and $(af)(m)$ are $W$-invariant. For $g \in G$, we have
\begin{displaymath}
f(g^{-1} bm) = f \left( \int_{V/W} b_W(h) g^{-1}h m dh \right) = \int_{V/W} b_W(h) f(g^{-1} hm) dh,
\end{displaymath}
as $f$ is $A(U$)-linear and thus commutes with integrals over $\hat{U}$-sets (this is clear from the definition of integral). We thus have
\begin{align*}
(af)(bm)
&= \int_{G/U} \int_{V/W} a_U(g) b_W(h) gf(g^{-1} hm) dh dg \\
&= \int_{V/W} \int_{G/U} a_U(g) b_W(h) gf(g^{-1} hm) dg dh \\
&= \int_{V/W} \int_{G/U} a_U(g) b_W(h) hgf(g^{-1} m) dg dh \\
&= b \cdot (af)(m)
\end{align*}
where in the first second we changed the order of integration using Fubini's theorem (Corollary~\ref{cor:fubini}), and in the third step we made the change of variables $g \to hg$, and used the left $V$-invariance of $a_U$. We have thus shown that $f$ is $A(V)$-linear, and so $f$ belongs to $\uHom(M,N)$.

\textit{Step 3: module structure.} We now show that the multiplication $(a,f) \mapsto af$ defines an $A$-module structure. It is easy to see that $1 \cdot f =f$, and so it suffices to show that $b(af)=(ba)f$ for $a,b \in A$ and $f \in \uHom(M,N)$. Let $f$ be $A(U)$-linear, and let $V \subset U$ be such that $a_U$ is left $V$-invariant. Then we saw in Step~2 that $af$ is $A(V)$-linear. Thus
\begin{align*}
(b (af))(m)
&= \int_{G/V} b_V(h) h(af)(h^{-1} m) dh \\
&= \int_{G/V} \int_{G/U} b_V(h) a_U(g) hgf(g^{-1} h^{-1} m) dg dh \\
&= \int_{G/V} \int_{G/U} b_V(h) a_U(h^{-1} g) gf(g^{-1} m) dg dh \\
&= \int_{G/U} (ba)_U(g) gf(g^{-1} m) dg
= ((ba)f)(m).
\end{align*}
In the first two steps, we simply expanded the definition; in the third, we made the change of variables $g \to gh$; in the fourth, we changed the order of integration, and used the definition of $(ba)_U$; in the final step, we simply used the definition of $((ba)f)(m)$. We thus have $b(af)=(ba)f$, as required.

\textit{Step 4: strict smoothness.} Finally, we show that $\uHom(M,N)$ is smooth. It is clear from the definition of $af$ that any element of $\Hom_{A(U)}(M,N)$ is strictly $U$-invariant. Thus every element of $\uHom(M,N)$ is strictly $U$-invariant for some $U$, as required.
\end{proof}

\begin{example} \label{ex:uhom}
Let $X$ be a finitary $G$-set and let $M$ be a smooth $A$-module. By Corollary~\ref{cor:schwartz-adj}, $\Hom_{A(U)}(X, M)$ is naturally identified with the space of $U$-equivariant functions $X \to M$. We thus have a natural identification $\uHom(\cC(X), M)=\cC(X,M)$. In particular, $\cC(X,M)$ carries a natural $A$-module structure. It is not hard to see that the integration map $\cC(X,M) \to M$ is $A$-linear.
\end{example}

\subsection{Adjunction}

We next establish the usual adjunction between tensor and Hom.

\begin{proposition} \label{prop:ten-hom-unit}
Let $M$ and $N$ be smooth $A$-modules.
\begin{enumerate}
\item Then there is a unique map of $A$-modules
\begin{displaymath}
\epsilon \colon M \uotimes \uHom(M, N) \to N
\end{displaymath}
satisfying $\epsilon(m \uotimes f)=f(m)$ for $m \in M$ and $f \in \uHom(M,N)$.
\item Then there is a unique map of $A$-modules
\begin{displaymath}
\eta \colon N \to \uHom(M, M \uotimes N)
\end{displaymath}
satisfying $\eta(n)(m)=m \uotimes n$ for $m \in M$ and $n \in N$.
\end{enumerate}
\end{proposition}

\begin{proof}
(a) Consider the function
\begin{displaymath}
q \colon M \times \uHom(M,N) \to N, \qquad q(m,f)=f(m).
\end{displaymath}
It suffices to show that $q$ is strongly bilinear, for then the existence and uniqueness of $\epsilon$ will follow from the universal property of the tensor product. It is clear that $q$ is $k$-bilinear and $G$-equivariant. We show that $q$ is strongly linear in each argument. First, suppose that $f \in \uHom(M,N)$ is given, and let $U$ be an open subgroup such that $f$ is $A(U)$-linear. Then for any $a \in A(U)$ and $m \in M$, we have
\begin{displaymath}
q(am,f)=\phi(am)=a f(m)=aq(m,f).
\end{displaymath}
Thus $q$ is strongly linear in the first argument by Proposition~\ref{prop:strong-lin}. Next suppose that $m \in M$ is given, and let $U$ be an open subgroup such that $m$ is $U$-invariant. Let $a \in A(U)$ and $f \in \uHom(M,N)$, and let $V \subset U$ be an open subgroup such that $f$ is $A(V)$-linear. Then
\begin{align*}
q(m,af) &= (af)(m) = \int_{U/V} a_V(g) gf(g^{-1} m) dg \\
&= \int_{U/V} a_V(g) gf(m) dg = af(m)=aq(m,f).
\end{align*}
In the third step, we used the $U$-invariance of $m$. We thus see that $q$ is strongly linear in its second argument by Proposition~\ref{prop:strong-lin}.

(b) For $n \in N$, let $f_n \colon M \to M \uotimes N$ be the function given by $f_n(m)=m \uotimes n$. By Proposition~\ref{prop:strong-lin}, there is an open subgroup $U$, depending only on $n$, such that $f_n(am)=af_n(m)$ for $a \in A(U)$ and $m \in M$. We thus see that $f_n$ belongs to $\uHom(M,M \uotimes N)$. We thus have a well-defined function $\eta$ by $\eta(n)=f_n$.

We now show that $\eta$ is $A$-linear. We thus must show $f_{an}(m)=(af_n)(m)$ for $a \in A$, $m \in M$, and $n \in N$. Let $a$, $m$, and $n$ be given. Let $U$ be an open subgroup such that $f_n$ is $A(U)$-linear. Then
\begin{align*}
(af_n)(m)
&= \int_{G/U} a_U(g) gf_n(g^{-1} m) dg
= \int_{G/U} a_U(g) (m \uotimes gn) dg \\
&= m \uotimes \left( \int_{G/U} a_U(g) gn dg \right) = m \uotimes (an) = f_{an}(m).
\end{align*}
In the third step, we used the strong bilinearity of $\uotimes$. This completes the proof.
\end{proof}

\begin{proposition} \label{prop:ten-hom-adj}
Let $M$ be an $A$-module. Then the functors $M \uotimes -$ and $\uHom(M, -)$ are naturally an adjoint pair: that is, for smooth $A$-modules $X$ and $Y$ we have a natural isomorphism
\begin{displaymath}
\Hom_A(X, \uHom(M,Y)) = \Hom_A(X \uotimes M, Y).
\end{displaymath}
In fact, we have a natural isomorphism
\begin{displaymath}
\uHom(X, \uHom(M,Y)) = \uHom(X \uotimes M, Y).
\end{displaymath}
of $A$-modules.
\end{proposition}

\begin{proof}
The maps constructed in Proposition~\ref{prop:ten-hom-unit} provide the unit and co-unit for the adjunction. We leave the details to the reader. The second statement follows by passing to open subgroups.
\end{proof}

\begin{corollary} \label{cor:ten-hom-adj}
Let $M$ and $N$ be smooth $A$-modules, with $M$ rigid. Then we have a natural isomorphism
\begin{displaymath}
\uHom(M,N) = M^{\vee} \uotimes N
\end{displaymath}
\end{corollary}

\begin{proof}
For any smooth $A$-module $E$, we have isomorphisms
\begin{displaymath}
\Hom_A(E, \uHom(M,N))=\Hom_A(E \uotimes M, N) = \Hom_A(E, M^{\vee} \uotimes N).
\end{displaymath}
The first isomorphism comes from Proposition~\ref{prop:ten-hom-adj}, while the second is the adjunction one always has with a rigid object \cite[Proposition~2.10.8]{Etingof}. The result now follows from Yoneda's lemma.
\end{proof}

\begin{example} \label{ex:ten-schwartz}
Let $X$ be a finitary $G$-set and let $M$ be a smooth $A$-module. Since $\cC(X)$ is rigid and self-dual (Corollary~\ref{cor:schwartz-rigid}), we have
\begin{displaymath}
\uHom(\cC(X), M)=\cC(X) \uotimes M
\end{displaymath}
by adjunction (Corollary~\ref{cor:ten-hom-adj}). Thus, by Example~\ref{ex:uhom}, we have a natural isomorphism
\begin{displaymath}
\cC(X) \uotimes M=\cC(X,M).
\end{displaymath}
Let $I \colon \cC(X,M) \to \cC(M)$ and $J \colon \cC(X) \to k$ be the integration maps. One can show that, under the above identification, we have $I=J \uotimes \id$. Using this, one can easily transfer properties of integrals of $k$-valued functions to integrals of $M$-valued functions. In particular, we obtain a good theory of push-forwards for $M$-valued functions.
\end{example}

\section{(Quasi-)regular measures} \label{s:regular}

\subsection{Overview}

So far, assuming $\mu$ is normal, we have constructed an abelian tensor category $\uRep(G)$. In this section, we establish various finer results about $\uRep(G)$ assuming $\mu$ is quasi-regular or regular. In particular, we show that the category of finite length objects in $\uRep(G)$ forms a rigid tensor category under some hypotheses.

Throughout this section, $G$ denotes a first-countable pro-oligomorphic group and $\mu$ a $k$-valued measure. The measure $\mu$ will always be regular or quasi-regular (see Definition~\ref{ss:regular}), and $k$ will always be a field.

\subsection{The main theorem} \label{ss:regss}

Before stating our theorem, we introduce some terminology:

\begin{definition} \label{defn:pre-tan}
Let $\cC$ be a $k$-linear tensor category, with $k$ a field. We say that $\cC$ is \defn{pre-Tannakian} if the following conditions hold:
\begin{enumerate}
\item $\cC$ is abelian and all objects have finite length.
\item $\Hom_k(M,N)$ is finite dimensional for all $M,N \in \cC$.
\item The unit object $\bbone$ is simple, and $\End(\bbone)=k$.
\item $\cC$ is rigid, i.e., all objects are rigid.
\end{enumerate}
We say that $\cC$ is \defn{locally pre-Tannakian} if every object of $\cC$ is the sum of its finite length subobjects, and the category $\cC^{\rf}$ of finite length objects is pre-Tannakian.
\end{definition}

We now come to our the main result of \S \ref{s:regular}:

\begin{theorem} \label{thm:regss}
Suppose that $k$ is a field, and $\mu$ is quasi-regular and satisfies condition~(P) from Definition~\ref{defn:P}. Then $\uRep(G)$ is locally pre-Tannakian. Moreover, if $\mu$ is regular then it is semi-simple.
\end{theorem}

It is possible to weaken the hypotheses of the theorem in various ways. For example, (P) is only used to ensure nilpotent matrices have trace~0, so one could use this condition instead. We note that in the regular case, the theorem implies that $\uRep^{\rf}(G)$ is the Karoubian envelope of $\uPerm(G)$.

We assume $k$ is a field and $\mu$ is quasi-regular and satisfies~(P) for the remainder of \S \ref{ss:regss}. We break the proof into a number of lemmas.

\begin{lemma} \label{lem:regss-1}
Let $\cC$ be a $k$-linear Grothendieck abelian category such that every object is the union of its finitely generated subobjects. Let $\Sigma$ be class of finitely generated objects of $\cC$ such that
\begin{enumerate}
\item $\Sigma$ is closed under finite direct sums.
\item Every finitely generated object of $\cC$ is a quotient of an object in $\Sigma$.
\item $\End(M)$ is a finite dimensional semi-simple $k$-algebra for every $M \in \Sigma$.
\end{enumerate}
Then $\cC$ is a semi-simple category, and $\End(M)$ is finite dimensional over $k$ for all finite length objects $M$.
\end{lemma}

\begin{proof}
Let $E$ be an object of $\Sigma$. Put $R=\End(E)$, and consider a subobject of $E$ of the form $aE$ with $a \in R$. Since $R$ is semi-simple, the right ideal $aR$ is generated by an idempotent $e$. Thus we have $a=ex$ and $e=ay$ for $x,y \in R$. It follows that $eE \subset aE \subset eE$, and so $aE=eE$. In particular, we see that $aE$ is a summand of $E$.

Now suppose $M$ is a finitely generated subobject of $E$. Choose a surjection $f \colon F \to M$ with $F \in \Sigma$. Let $a \in \End(E \oplus F)$ be the endomorphism defined by $a(x)=f(x)$ for $x \in F$ and $a(y)=0$ for $y \in E$. Then $M=a (E \oplus F)$. By the previous paragraph, $M$ is a summand of $E \oplus F$, and thus of $E$. We have thus shown that every finitely generated subobject of $E$ is a summand.

Suppose now that $0 \subsetneq M_1 \subsetneq \cdots \subsetneq M_n \subseteq E$ is a chain of finitely generated subobjects. We can then write $E=N_1 \oplus \cdots \oplus N_{n+1}$ so that $M_i=N_1 \oplus \cdots \oplus N_i$; note that $N_1, \ldots, N_n$ are non-zero. By considering the projectors onto the $N_i$'s, we see that $\dim(\End(E)) \ge n$. Since $\End(E)$ is finite dimensional, it follows that $E$ has finite length. Thus every subobject is finitely generated, and therefore a summand, and so $E$ is semi-simple. Since every object of $\cC$ is a quotient of a sum of objects in $\Sigma$, it follows that $\cC$ is semi-simple.

As for the final statement, suppose that $M$ is a finite length object of $\cC$. Then, by assumption, $M$ is a quotient of some $E \in \Sigma$. By semi-simplicity, $M$ is actually a summand of $E$. Thus $\End(M)$ is a summand of $\End(E)$, and hence finite dimensional.
\end{proof}

\begin{lemma} \label{lem:regss-2}
If $\mu$ is regular then $\uRep(G)$ is semi-simple, and $\End(M)$ is finite dimensional if $M$ has finite length.
\end{lemma}

\begin{proof}
We apply Lemma~\ref{lem:regss-1} with $\Sigma$ the class of all modules of the form $\cC(X)$ with $X$ a finitary $G$-set. We must verify the three conditions of the lemma. Condition~(a) is clear, while (b) follows from Proposition~\ref{prop:schwartz-quot}. As for (c), we have $\End_A(\cC(X))=\Mat_X^G$ by Proposition~\ref{prop:normal-full}, and this algebra is semi-simple by Proposition~\ref{prop:semisimple-alg}. The result follows.
\end{proof}

\begin{lemma} \label{lem:regss-3}
$\uRep(G)$ is locally of finite length, and $\End(M)$ is finite dimensional if $M$ has finite length.
\end{lemma}

\begin{proof}
Let $U$ be an open subgroup such that $\mu \vert_U$ is regular; note that $\mu \vert_U$ still satisfies~(P) (see the discussion following Definition~\ref{defn:P}). If $M$ is a finitely generated smooth $A(G)$-module then it is finitely generated as an $A(U)$-module (this follows easily from Proposition~\ref{prop:schwartz-quot}), and thus finite length as an $A(U)$-module by Lemma~\ref{lem:regss-2}. Hence it is finite length as an $A(G)$-module. Moreover, $\End_{A(G)}(M)$ is a subalgebra of $\End_{A(U)}(M)$, and is thus finite dimensional by Lemma~\ref{lem:regss-2}.
\end{proof}

It now remains to prove that $\uRep^{\rf}(G)$ is pre-Tannakian. Condition~(a) of Definition~\ref{defn:pre-tan} holds tautologically for this category. Condition~(b) follows from the above lemma, since $\Hom(M,N)$ injects into $\End(M \oplus N)$. Also, (c) follows easily from the above results. If $U$ is an open subgroup for which $\mu \vert_U$ is regular then $\bbone \in \uRep(U)$ is semi-simple and has $\End_{A(U)}(\bbone)=k$, and is thus simple; hence $\bbone$ is simple in $\uRep(G)$ as well. It remains to prove condition~(d). We begin with the following lemma (which is actually implied by the theorem):

\begin{lemma}
The functors $\uotimes$ and $\uHom$ on $\uRep(G)$ are exact in both arguments.
\end{lemma}

\begin{proof}
Let $U$ be an open subgroup of $G$ such that $\mu \vert_U$ is regular. Then $\uRep(U)$ is semi-simple by Lemma~\ref{lem:regss-2}. It follows that any additive functor on $\uRep(U)$ is exact; in particular, $\uotimes$ and $\uHom$ are exact on $\uRep(U)$. Since formation of $\uotimes$ and $\uHom$ on $\uRep(G)$ are compatible with restriction to $\uRep(U)$ (for $\uotimes$ this is Proposition~\ref{prop:res-tensor}, and for $\uHom$ it is obvious) it follows that they are exact on $\uRep(G)$.
\end{proof}

For an object $M$ of $\uRep(G)$, put $\bD(M)=\uHom(M, \bbone)$. Thus $\bD$ defines an exact contravariant endofunctor of $\uRep(G)$. If $M$ is a rigid object of $\uRep(G)$ with dual $M^{\vee}$ then we have $\bD(M)=M^{\vee}$ by Corollary~\ref{cor:ten-hom-adj}. In particular, if $X$ is a finitary $G$-set then we have a natural isomorphism $\cC(X)=\bD(\cC(X))$ by Corollary~\ref{cor:schwartz-rigid}. We aim to show that whenever $M$ has finite length, $M$ is rigid and $M^{\vee}=\bD(M)$.

\begin{lemma} \label{lem:qr-rigid-2}
If $M$ has finite length then so does $\bD(M)$, and the natural map $M \to \bD(\bD(M))$ is an isomorphism.
\end{lemma}

\begin{proof}
Choose a presentation
\begin{displaymath}
\cC(Y) \to \cC(X) \to M \to 0
\end{displaymath}
with $X$ and $Y$ finitary $G$-sets; this is possible by Lemma~\ref{lem:regss-3} and Proposition~\ref{prop:schwartz-quot}. Then we have an injection $\bD(M) \to \bD(\cC(X))=\cC(X)$, and so $\bD(M)$ is finite length. Applying $\bD$ twice, we obtain a commutative diagram
\begin{displaymath}
\xymatrix{
\cC(Y) \ar[r] \ar[d] & \cC(X) \ar[r] \ar[d] & M \ar[r] \ar[d] & 0 \\
\bD(\bD(\cC(Y))) \ar[r] & \bD(\bD(\cC(X))) \ar[r] & \bD(\bD(M)) \ar[r] & 0 }
\end{displaymath}
The left two vertical maps are easily seen to be isomorphisms, and so the right vertical map is as well. This completes the proof.
\end{proof}

\begin{lemma} \label{lem:qr-rigid-3}
If $M$ and $N$ are finite length objects of $\uRep(G)$ then the natural map
\begin{displaymath}
\iota \colon \bD(M) \uotimes \bD(N) \to \bD(M \uotimes N)
\end{displaymath}
is an isomorphism.
\end{lemma}

\begin{proof}
Choose surjections $\cC(X) \to M$ and $\cC(Y) \to N$, with $X$ and $Y$ finitary $G$-sets. We then have a commutative diagram
\begin{displaymath}
\xymatrix{
\bD(M) \uotimes \bD(N) \ar[r]^{\iota} \ar[d] & \bD(M \uotimes N) \ar[d] \\
\bD(\cC(X)) \uotimes \bD(\cC(Y)) \ar[r] & \bD(\cC(X) \uotimes \cC(Y)) }
\end{displaymath}
The bottom map is easily seen to be an isomorphism. Since $\bD$ is exact and $\uotimes$ are exact, the left vertical map is an injection. Thus $\iota$ is injective.

Now choose surjections $\cC(X') \to \bD(M)$ and $\cC(Y') \to \bD(N)$, with $X'$ and $Y'$ finitary $G$-sets; this is possible since $\bD(M)$ and $\bD(N)$ are finite length by Lemma~\ref{lem:qr-rigid-2}. We thus have injections $M \to \cC(X')$ and $N \to \cC(Y')$, and obtain a commutative diagram
\begin{displaymath}
\xymatrix{
\bD(M) \uotimes \bD(N) \ar[r]^{\iota} & \bD(M \uotimes N) \\
\bD(\cC(X')) \uotimes \bD(\cC(Y')) \ar[r] \ar[u] & \bD(\cC(X') \uotimes \cC(Y')) \ar[u] }
\end{displaymath}
Again, the bottom map is an isomorphism and the right vertical map is surjective. Thus $\iota$ is surjective. This completes the proof.
\end{proof}

The following lemma completes the proof of the theorem:

\begin{lemma}
Let $M$ be a finite length object of $\uRep(G)$. Then $M$ is rigid and $\bD(M)=M^{\vee}$.
\end{lemma}

\begin{proof}
We have a natural map $\ev \colon \bD(M) \uotimes M \to \bbone$. Applying $\bD$, we obtain a map
\begin{displaymath}
\xymatrix@C=3em{
\bbone \ar[r]^-{\bD(\ev)} & \bD(\bD(M) \uotimes M) \ar@{=}[r] & \bD(\bD(M)) \uotimes \bD(M) \ar@{=}[r] & M \uotimes \bD(M) }
\end{displaymath}
where the first identification above is Lemma~\ref{lem:qr-rigid-3}, and the second is Lemma~\ref{lem:qr-rigid-2}. Define $\cv \colon \bbone \to \bD(M) \uotimes M$ to be the above composition. One then verifies that $\ev$ and $\cv$ satisfy the necessary conditions for $\bD(M)$ to be the dual of $M$.
\end{proof}

\begin{remark}
When $G$ is a finite group, Theorem~\ref{thm:regss} in the regular case recovers Maschke's theorem (see Example~\ref{ex:ordinary}). Our proof amounts to showing that the group algebra $k[G]$ is semi-simple by computing the discriminant of the trace pairing.
\end{remark}

\begin{remark}
The proof of Theorem~\ref{thm:regss} shows that the dual of a finite length object $M$ is given by $\uHom(M, \bbone)$, which is (essentially) the space of smooth vectors in the ordinary dual space $M^*$. This is very similar to the construction of the contragredient representation in the theory of admissible representations of $p$-adic groups; see, e.g., \cite[\S 4.2]{Bump}.
\end{remark}

\subsection{Abelian envelopes} \label{ss:abenv2}

We recall the notion of abelian envelope for tensor categories proposed in \cite[Definition~3.1.2]{CEAH}:

\begin{definition} \label{defn:abenv}
Let $\cP$ be a $k$-linear tensor category. An \defn{abelian envelope} of $\cP$ is a pair $(\cR, \Phi)$, where $\cR$ is a pre-Tannakian category and $\Phi$ is a tensor functor, such that for any pre-Tannakian category $\cT$ the functor
\begin{displaymath}
\Fun^{\rm ex}_{\otimes}(\cR, \cT) \to \Fun^{\rm faith}_{\otimes}(\cP, \cT), \qquad \Psi \mapsto \Psi \circ \Phi
\end{displaymath}
is an equivalence, where $\Fun^{\rm ex}_{\otimes}$ denotes the category of exact tensor functors and $\Fun^{\rm faith}_{\otimes}$ the category of faithful tensor functors.
\end{definition}

We note that an abelian envelope is unique up to canonical equivalence when it exists, so we can speak of ``the'' abelian envelope. With this definition, we can now formulate a precise result for when and how $\uRep(G)$ is an abelian envelope of $\uPerm(G)$.

\begin{theorem} \label{thm:abenv}
Suppose $k$ is a field, and $\mu$ is quasi-regular and satisfies condition~(P). Then $\uRep^{\rf}(G)$ is the abelian envelope of $\uPerm(G)$ in the sense of Definition~\ref{defn:abenv}.
\end{theorem}

\begin{proof}
Following \cite[Theorem~3.1.4]{CEAH}, it suffices to check:
\begin{enumerate}[(i)]
\item The functor $\Phi \colon \uPerm(G) \to \uRep^{\rf}(G)$ is fully faithful. 
\item Any $M \in \uRep^{\rf}(G)$ can be realized as the image of a map $\Phi(P) \to \Phi(Q)$ with $P,Q \in \uPerm(G)$.  
\item For any epimorphism $f \colon M \to N$ in $\uRep^{\rf}(G)$, there exists $T \in \uPerm(G)$ such that the map $\id \uotimes f \colon \Phi(T) \uotimes M \to \Phi(T) \uotimes N$ splits.
\end{enumerate}
(i) is Propositions~\ref{prop:perm-rep-func} and~\ref{prop:normal-full}. As for (ii), Proposition~\ref{prop:schwartz-quot} shows that for every $M \in \uRep^{\rf}(G)$ there is a surjection $\Phi(P) \to M$ for some $P \in \uPerm(G)$. Since $\uRep^{\rf}(G)$ is rigid and objects in $\uPerm(G)$ are self-dual, we also see that every $M$ admits an injection $M \to \Phi(Q)$ with $Q \in \uPerm(G)$. Composing these realizes $M$ as the image of a map $\Phi(P) \to \Phi(Q)$ as desired. 
 
We now prove (iii). Let $f \colon M \to N$ be given. Since $\mu$ is quasi-regular, there is an open subgroup $U$ such that $\mu \vert_U$ is regular, and so $\uRep^{\rf}(U)$ is semi-simple (Theorem~\ref{thm:regss}). We can therefore find an $A(U)$-linear splitting $s_0 \colon N \to M$ of $f$. Let $T=\Vec_{G/U}$, so that $\Phi(T)=\cC(G/U)$. We now define a map
\begin{displaymath}
s \colon \Phi(T) \uotimes N \to \Phi(T) \uotimes M.
\end{displaymath}
Identify $\Phi(T) \uotimes N$ with $\cC(G/U, N)$ (see Example~\ref{ex:ten-schwartz}), and similarly for the target. For $\phi \in \cC(G/U, N)$, we define $s(\phi) \in \cC(G/U, M)$ by
\begin{displaymath}
(s \phi)(g)=gs_0(g^{-1} \phi(g)).
\end{displaymath}
One readily verifies that $s$ is $A(G)$-linear and splits $\id \uotimes f$, which completes the proof. We note that in terms of the original tensor product, $s$ satisfies
\begin{displaymath}
s(\delta_g \uotimes n) = \delta_g \uotimes gs_0(g^{-1} n),
\end{displaymath}
and this uniquely determines it by $A(G)$-linearity.
\end{proof}

\subsection{The relative case} \label{ss:abrel}

Suppose that $\sE$ is a stabilizer class in $G$. The material in \S\S \ref{s:gpalg}--\ref{s:genten} goes through with minimal changes in the relative setting; essentially, one just restricts to $\sE$-smooth $G$-sets everywhere. Theorem~\ref{thm:regss}, on the other hand, does not extend to the relative case in general, as it relies upon Proposition~\ref{prop:semisimple-alg} (see \S \ref{ss:rel-matrix}). We give a counterexample in \S \ref{ss:sym-rel}. However, Theorem~\ref{thm:regss} (and our proof) holds if the condition from Remark~\ref{rmk:rel-binom} is verified.

\part{Examples} \label{part:ex}

\section{The symmetric group} \label{s:sym}

\subsection{Overview}

Let $\fS$ be the infinite symmetric group, which we take to be the group of all permutations of the set $\Omega=\{1,2,\ldots\}$. In \S \ref{s:sym}, we examine the theory developed in this paper in this case. We show that our work recover Deligne's interpolation category $\uRep(\fS_t)$ and some known results about it (e.g., the theorem of Comes--Ostrik \cite{ComesOstrik} on abelian envelopes). We also examine what happens in positive characteristic, and observe some pathological behavior.

\subsection{Subgroup structure}

For $n \in \bN$, let $\fS_n$ be the usual finite symmetric group, regarded as a subgroup of $\fS$, and let $\fS(n)$ be the subgroup of $\fS$ fixing the numbers $1, \ldots, n$. Thus $\fS_n \times \fS(n)$ is a Young subgroup of $\fS$, and the $\fS(n)$'s form a neighborhood basis of the identity in $\fS$. The following important proposition gives the structure of open subgroups of $\fS$. It is likely well-known (see, e.g., \cite[Proposition~A.4]{Sciarappa} for the analogous statement for finite symmetric groups), but we include a proof for the sake of completeness.

\begin{proposition} \label{prop:sym-open-subgroups}
Let $U$ be an open subgroup of $\fS$. Then there exists an integer $n \ge 0$ and a subgroup $H$ of $\fS_n$ such that $U$ is conjugate to $H \times \fS(n)$.
\end{proposition}

\begin{proof}
Let $n \ge 0$ be minimal such that $\fS(n)$ is contained in a conjugate of $U$. Replacing $U$ with a conjugate, we simply assume that $\fS(n) \subset U$.

We claim that $U$ is contained in $\fS_n \times \fS(n)$. Suppose not, and let $g$ be an element of $U$ not contained in $\fS_n \times \fS(n)$. Let $g=\sigma \tau_1 \cdots \tau_r$ be a decomposition of $g$ into disjoint cycles such that $\sigma$ does not belong to $\fS_n \times \fS(n)$. Write
\begin{displaymath}
\sigma=(a_1 \; a_2 \; \cdots \; a_k)
\end{displaymath}
where $a_1>n$ and $a_2 \le n$. Let $a_1'>n$ be an integer not appearing in any cycle in $g$. Since the transposition $(a_1\;a_1')$ belongs to $\fS(n)$, it belongs to $U$, and so the conjugate $g'$ of $g$ by this transposition also belong to $U$. We have $g'=\sigma' \tau_1 \cdots \tau_r$, where $\sigma'$ is like $\sigma$ but with $a_1$ changed to $a_1'$. We thus see that $U$ contains the 3-cycle
\begin{displaymath}
g' g^{-1} = (a_1 \; a'_1  \; a_2).
\end{displaymath}
Multiplying by the transposition $(a_1 \; a_1')$, we thus see that $U$ contains the transposition $(a_1\;a_2)$. Conjugating $U$ by an element of $\fS_n \times \fS(n)$, we may as well suppose that $U$ contains $(n\;n+1)$. But this transposition and $\fS(n)$ generate $\fS(n-1)$, and so $U$ contains $\fS(n-1)$, a contradiction. This proves the claim.

Since $U$ is contained in $\fS_n \times \fS(n)$ and contains $\fS(n)$, it necessarily has the form $H \times \fS(n)$ for some subgroup $H$ of $\fS_n$, which completes the proof.
\end{proof}

\begin{corollary} \label{cor:S-set}
Let $X$ be a transitive $\fS$-set. Then there exists a map of $\fS$-sets $X \to \Omega^{(n)}$ that is everywhere $d$-to-1, for unique integers $n \ge 0$ and $d \ge 1$.
\end{corollary}

\begin{proof}
Write $X=G/V$. By Proposition~\ref{prop:sym-open-subgroups}, we can assume $V=H \times \fS(n)$ for some $H \subset \fS_n$. Thus $V$ is contained in $U=\fS_n \times \fS(n)$ with finite index $d=[\fS_n:H]$. The map $G/V \to G/U=\Omega^{(n)}$ is everywhere $d$-to-1, as required.
\end{proof}

\subsection{Fixed point measures}

We now give a general method for constructing measures. This method will apply to the symmetric group and some linear groups (see \S \ref{s:finlin}). Let $G$ be a first-countable pro-oligomorphic group, and fix a chain $\cdots \subset U_2 \subset U_1$ of open subgroups that form a neighborhood basis of the identity. Let $\tilde{R}$ be the ring of functions $f \colon \bZ_{>0} \to \bZ$, let $I$ be the ideal of $\tilde{R}$ consisting of functions $f$ such that $f(n)=0$ for all $n \gg 0$, and let $R=\tilde{R}/I$.

For a finitary $\hat{G}$-set $X$, we let $f_X \in R$ be the function defined by $f_X(n)=\# X^{U_n}$, where $(-)^{U_n}$ denotes fixed points. Note that $U_n$ is a group of definition for $X$ for all $n \gg 0$, and so $X^{U_n}$ makes sense for all but finitely many $n$. Furthermore, $U_n$ has finitely many orbits on $X$, and in particular, finitely many fixed points. Thus $f_X$ is well-defined in $R$.

\begin{proposition} \label{prop:fixed-pt}
Suppose the following conditions hold:
\begin{enumerate}
\item Given $g \in G$ and $n>0$ there is some $N$ such that for every $m>N$ the group $U_m$ is normalized by some element of $U_ng$.
\item For every open subgroup $V$ of $G$ there is some $N$ such that $U_N \subset V$, and for all $n > N$ all $G$-conjugate copies of $U_n$ in $V$ are also $U$-conjugate. That is, if $gU_ng^{-1} \subseteq V$ for some $g \in G$ then $g U_n g^{-1} = hU_nh^{-1}$ for some $h \in V$.   
\end{enumerate}
Then $\mu(X)=f_X$ defines an $R$-valued measure on $G$.
\end{proposition}

\begin{proof}
We must check that $\mu$ satisfies conditions \dref{defn:measure}{a}--\dref{defn:measure}{e}. Conditions \dref{defn:measure}{a}--\dref{defn:measure}{c} are clear. We now verify \dref{defn:measure}{d}. We must show $\mu(X)=\mu(X^g)$ for a finitary $\hat{G}$-set $X$ and an element $g \in \fS$. Let $U_n$ be a group of definition of $X$. Let $N$ be as in (a), relative to $g$ and $n$. Let $m>N$ and suppose that $hg$ normalizes $U_m$, with $h \in U_n$. We have $X^h=X$ since $U_n$ is a group of definition for $X$, and so $X^g=X^{hg}$. Since $hg$ normalizes $U_m$, we have $(X^g)^{U_m}=(X^{hg})^{U_m}=X^{U_m}$. We thus see that $f_X(m)=f_{X^g}(m)$ for all $m>N$, and so $f_X=f_{X^g}$. Thus $\mu(X)=\mu(X^g)$, as required.

We now verify \dref{defn:measure}{e}. Thus let $X \to Y$ be a surjection of transitive $U$-sets, for some open subgroup $U$, and let $F=F_y$ be the fiber of $y \in Y$. Let $V \subset U$ be the stabilizer of $y$, so that $F$ is a $V$-set. We must show $\# X^{U_n} = (\# F^{U_n}) (\# Y^{U_n} )$ for all $n \gg 0$. We have
\begin{displaymath}
X^{U_n} = \bigsqcup_{y' \in Y^{U_n}} F_{y'}^{U_n}
\end{displaymath}
and so it suffices to show that $F_{y'}^{U_n}$ has the same cardinality as $F^{U_n}$ for all $y'$ as above, at least for $n \gg 0$.

Let $N$ be as in (b), let $n>N$, and suppose $y' \in Y^{U_n}$. Write $y'=gy$ with $g \in U$. The stabilizer of $y'$ is $gVg^{-1}$, and so $U_n \subset gVg^{-1}$, and so $g^{-1}U_ng \subset V$. By (b), we have $g^{-1}U_ng=h^{-1}U_nh$ for some $h \in V$. We now have
\begin{displaymath}
F_{y'}^{U_n}=F_{gy}^{U_n} \cong F_y^{g^{-1}U_ng} = F_y^{h^{-1}U_nh} \cong F_{hy}^{U_n} = F_y^{U_n}
\end{displaymath}
which shows $\# F^{U_n}_{y'}=\# F^{U_n}_y$, as required.
\end{proof}

\subsection{The ring \texorpdfstring{$\Theta(\fS)$}{\textTheta(S)}}

Recall (\S \ref{ss:intpoly}) that $\bZ\langle x \rangle$ denotes the ring of integer-valued polynomials. The elements $\lambda_n(x)=\binom{x}{n}$ form a $\bZ$-basis for $\bZ\langle x \rangle$. The following is one of our main results on the symmetric group:

\begin{theorem} \label{thm:Theta-sym}
We have a ring isomorphism $i \colon \bZ\langle x \rangle \to \Theta(\fS)$ given by $i(\lambda_n(x))=[\Omega^{(n)}]$.
\end{theorem}

We establish two lemmas before proving the theorem.

\begin{lemma} \label{lem:Theta-sym-1}
The classes $[\Omega^{(n)}]$ generate $\Theta(\fS)$ as a $\bZ$-module.
\end{lemma}

\begin{proof}
Let $\Omega_m=\{m+1,m+2,\ldots\}$, which is a $\hat{\fS}$-subset of $\Omega$, and let $c_{m,n}=[\Omega_m^{(n)}]$. We must show that the $c_{0,n}$'s generate $\Theta(G)$ as a $\bZ$-module. We proceed in two steps.

\textit{Step 1.} We first show that the $c_{m,n}$'s generate $\Theta(\fS)$. Thus let $X$ be a finitary $U$-set for some open subgroup $U$. We must realize $X$ as a $\bZ$-linear combination of the classes $[\Omega_m^{(n)}]$. We are free to shrink $U$, so we may suppose $U=\fS(m)$ for some $m$. Clearly, it suffices to treat the case where $U$ acts transitively on $X$. By Proposition~\ref{prop:sym-open-subgroups}, there is some $\ell \ge m$ such that $X \cong \fS(m)/(H \times \fS(\ell))$, where $H$ is a subgroup of $\fS_{\ell-m}$. Let $X'=\fS(m)/(\fS_{\ell-m} \times \fS(\ell))$. The natural map $X \to X'$ is $d$-to-1, where $d=[\fS_{\ell-m}:H]$, and so $[X]=d \cdot [X']$ by Corollary~\ref{cor:Theta-d-to-1}. As $X' \cong \Omega_m^{(\ell-m)}$, we see that $[X]=dc_{m,\ell-m}$, as required.

\textit{Step 2.} We now show that the $c_{m,\bullet}$'s can be generated by the $c_{0,\bullet}$'s. We have $\Omega=[m] \sqcup \Omega_m$. We thus find
\begin{displaymath}
\Omega^{(n)} = \bigsqcup_{i+j=n} \big( [m]^{(i)} \times \Omega_m^{(j)} \big),
\end{displaymath}
and so
\begin{displaymath}
c_{0,n} = \sum_{i+j=n} \binom{m}{i} c_{m,j}.
\end{displaymath}
Thus $c_{0,\bullet}$ is related to $c_{m,\bullet}$ by an upper-triangular matrix with 1's on the diagonal, and so the $c_{0,\bullet}$'s generate the $c_{m,\bullet}$'s.
\end{proof}

\begin{lemma} \label{lem:Theta-sym-2}
The hypotheses of Proposition~\ref{prop:fixed-pt} hold with $U_n=\fS(n)$. We thus have a fixed point measure for $\fS$.
\end{lemma}

\begin{proof}
(a) Let $g \in \fS$ and $n>0$ be given. We can then find $h \in gU_n$ that belongs to the finite symmetric group $\fS_{2n}$. This element centralizes $\fS(m)$ for each $m>2n$.

(b) Let $U$ be a given open subgroup of $G$. To verify this condition, we can replace $U$ by a conjugate. By Proposition~\ref{prop:sym-open-subgroups}, we can therefore assume $U=H \times \fS(n)$ for some finite subgroup $H$ of $\fS_n$. For $m>n$, the $\fS$-conjugates of $\fS(m)$ that are contained in $U$ are the subgroups $\fS(A)$, where $A$ is an $m$-element subset of $\Omega$ containing $\{1,\ldots,n\}$. It is clear that all such subgroups are $U$-conjugate, which verifies the condition.
\end{proof}

\begin{proof}[Proof of Theorem~\ref{thm:Theta-sym}]
By Proposition~\ref{prop:R-Theta}, there is a ring homomorphism $i \colon \bZ\langle x \rangle \to \Theta(G)$ satisfying $i(\lambda_n(x))=[\Omega^{(n)}]$. By Lemma~\ref{lem:Theta-sym-1}, $i$ is surjective. Let $j \colon \Theta(\fS) \to R$ be the fixed-point measure provided by Lemma~\ref{lem:Theta-sym-2}. Consider the composition
\begin{displaymath}
\xymatrix{
\bZ\langle x \rangle \ar[r]^-i & \Theta(\fS) \ar[r]^-j & R }
\end{displaymath}
Since $U_n$ has $n$ fixed points on $\Omega$, we see that $j(i(x))=j([\Omega]) \in R$ is the function given by $n \mapsto n$. Since $i$ and $j$ are ring homomorphisms, we see that $j(i(f))$ is the function $n \mapsto f(n)$. It follows that $j \circ i$ is injective, and so $i$ is injective. This completes the proof.
\end{proof}

The proof of the theorem shows that $j$ is actually the inverse to $i$. In other words, if $X$ is a finitary $\hat{\fS}$-set then there is a unique polynomial $p_X \in \bZ\langle x \rangle$ such that $p_X(n)=\# X^{\fS(n)}$ for all $n \gg 0$, and $p_X$ is the element of $\bZ\langle x \rangle$ corresponding to $[X] \in \Theta(\fS)$ under the isomorphism $i$. We now see that the measure $\mu_t$ for $\fS$ defined in Example~\ref{ex:sym-measure} is actually a well-defined measure. In fact, it is now an easy matter to classify all field-valued measures for $\fS$; see \S \ref{ss:sym-rep} and \S \ref{ss:sym-char-p} below.

\subsection{Property~(P)} \label{ss:symP}

To apply our results on $\uRep$, we must verify the technical condition~(P) from Definition~\ref{defn:P}. We now do this. We begin with a somewhat more general statement.

\begin{proposition} \label{prop:P}
Let $R$ be a subring of $\bQ[x]$ containing $\bZ[x]$. Suppose that the following conditions hold:
\begin{enumerate}
\item Let $K$ be a number field, let $t \in K$, and let $\phi_t \colon \bQ[x] \to K$ be the ring homomorphism given by $\phi_t(x)=t$. Then the ring $\phi_t(R)$ has infinitely many prime ideals.
\item If $\fp$ is a prime ideal of $R$ containing the prime number $p$ then $R/\fp$ is a finitely generated $\bF_p$-algebra.
\end{enumerate}
Then $R$ satisfies~(P).
\end{proposition}

\begin{proof}
Let $R'$ be a finitely generated reduced $R$-algebra and let $J=\bigcap \ker(\phi)$, where the intersection is over all homomorphisms $\phi$ from $R'$ to fields of positive characteristic. We must show $J=0$. One easily reduces to the case where $R'$ is integral. If $R'$ has characteristic~$p$ then by (b) it is finitely generated as an $\bF_p$-algebra, and the result is clear. We thus suppose that $\Frac(R')$ has characteristic~0.

Pick a finite generating set for $R'$ as an $R$-algebra, and let $R'_0$ be the $\bZ[x]$-subalgebra of $R'$ they generate. Note that $R'_0 \otimes \bQ = R' \otimes \bQ$. Let $N \ge 1$ be a positive integer, and consider an element $f$ of $R'_0[1/N] \cap J$. Let $M$ be a multiple of $N$ and consider a ring homomorphism $\phi \colon R'_0 \to \cO_K[1/M]$ for some number field $K$. Then $\phi$ extends uniquely to a homomorphism $\phi \colon R' \otimes \bQ \to K$, and $\phi(R')$ has infinitely many prime ideals by (a). Let $\fp$ be a prime of $\phi(R')$ not dividing $M$. Then we have a commutative diagram
\begin{displaymath}
\xymatrix{
R'_0[1/N] \ar[r] \ar[d]_{\phi} & R'[1/N] \ar[d] \\
\cO_K[1/M] \ar[r] & \phi(R')/\fp }
\end{displaymath}
Since $f \in J$, its image under the right map is zero. Thus $\phi(f) \in \cO_K[1/M]$ reduces to~0 modulo infinitely many primes, and is therefore~0. We thus see that $f$ vanishes on all $\cO_K[1/M]$-points of $R'_0[1/N]$, for all $K$ and $M$. Since $R'_0$ is a finitely generated reduced ring, it follows that $f=0$. Finally, since every element of $J$ belongs to $R'_0[1/N]$ for some $N$, we see that $J=0$.
\end{proof}

\begin{corollary} \label{cor:sym-P}
The ring $\bZ\langle x \rangle$ of integer-valued polynomials satisfies~(P). Therefore any measure for $\fS$ satisfies~(P).
\end{corollary}

\begin{proof}
We check the two conditions of Proposition~\ref{prop:P}.

(a) Let $\phi_t \colon \bQ[x] \to K$ be given. Let $S_1$ be the set of primes $\fp$ of $\cO_K$ such that $t$ is integral at $\fp$, i.e., $t \in \cO_{K,\fp}$. This set contains all but finitely many primes of $\cO_K$. Let $S_2$ be the set of primes $\fp$ such that $K_{\fp}=\bQ_p$. This set is infinite; in fact, by the Chebotarev density theorem, a positive density set of primes of $\bQ$ split completely in $K$. Let $\fp \in S_1 \cap S_2$. Then $t$ belongs to $\cO_{K,\fp} \cong \bZ_p$, and so $\binom{t}{n}$ does as well, for all $n \ge 0$. We thus see that $\phi_t(R) \subset \cO_{K,\fp}$, and so $\fp$ is (or rather, extends to) a prime of $\phi_t(R)$.

(b) If $\fp$ is a prime of $R$ containing a prime number $p$ then $R/\fp=\bF_p$ \cite[Theorem~13]{CahenChabert}.
\end{proof}

\subsection{Characteristic~0} \label{ss:sym-rep}

Let $k$ be a field of characteristic~0. Given $t \in k$, we have a homomorphism $\bZ\langle x \rangle \to k$ by evaluating a polynomial at $t$, and all homomorphisms $\bZ\langle x \rangle \to k$ have this form; the key point here is that $\bZ\langle x \rangle \otimes \bQ=\bQ[x]$. Let $\mu_t$ be the measure for $\fS$ corresponding to this homomorphism. Thus
\begin{displaymath}
\mu_t(\Omega^{(n)})=\binom{t}{n}.
\end{displaymath}
The $\mu_t$ account for all of the $k$-valued measures for $\fS$.

\begin{proposition} \label{prop:sym-qreg}
The measure $\mu_t$ is quasi-regular for any $t \in k$. It is regular if and only if $t \in k \setminus \bN$.
\end{proposition}

\begin{proof}
First suppose $t \not\in \bN$. By the above formula, we see that $\mu_t(\Omega^{(n)})$ is non-zero for all $n \ge 0$. Let $X$ be a transitive $\fS$-set. By Corollary~\ref{cor:S-set}, there is a $d$-to-1 map $X \to \Omega^{(n)}$ for some $n$ and $d$, and so $\mu(X)=d \cdot \mu(\Omega^{(n)})$ is non-zero. Thus $\mu_t$ is regular.

Now suppose $t=m$ is a natural number. Then $\mu_t(\Omega^{(m+1)})=0$, and so $\mu$ is not regular. The group $\fS(m+1)$ is isomorphic to $\fS$. One easily sees that the restriction of $\mu_t$ to $\fS(m+1)$ corresponds to the measure $\mu_{-1}$ on $\fS$ under this isomorphism, and is therefore regular. Thus $\mu_t$ is quasi-regular.
\end{proof}

Fix $t$ and let $\uRep(\fS)$ denote the category $\uRep_k(\fS; \mu_t)$ in what follows. Since $\mu_t$ is normal, $\uRep(\fS)$ has the structure of a tensor category. The following is our main result about it:

\begin{theorem} \label{thm:S-cat}
The category $\uRep(\fS)$ is locally pre-Tannakian, and semi-simple if $t \not\in \bN$. Moreover, $\uRep^{\rf}(\fS)$ is the abelian envelope of $\uPerm(\fS)$ in the sense of Definition~\ref{defn:abenv}.
\end{theorem}

\begin{proof}
As $\mu_t$ is (quasi-)regular (Proposition~\ref{prop:sym-qreg}) and satisfies condition~(P) (Corollary~\ref{cor:sym-P}), the first statement follows from Theorem~\ref{thm:regss}. The second statement follows from Theorem~\ref{thm:abenv}
\end{proof}

In \cite[\S 2]{Deligne3}, Deligne constructed his interpolation category $\uRep(\fS_t)$, and in \cite[Proposition~8.19]{Deligne3}, he constructed an abelian version $\uRep^{\rm ab}(\fS_t)$. In \cite[Conjecture~8.21]{Deligne3}, Deligne formulated a conjecture that implies $\uRep^{\rm ab}(\fS_t)$ is the abelian envelope of $\uRep(\fS_t)$ in the sense of Definition~\ref{defn:abenv}. This conjecture was proved by Comes--Ostrik \cite[Theorem~1.2(b)]{ComesOstrik}. It is not difficult to see that the Karoubian envelope our category $\uPerm(\fS; \mu_t)$ is equivalent to Deligne's $\uRep(\fS_t)$, and that $\uRep(\fS;\mu_t)$ is also the abelian envelope of this Karoubian envelope. Thus, by the uniqueness of abelian envelopes, we see that our $\uRep(\fS; \mu_t)$ is equivalent to Deligne's $\uRep^{\rm ab}(\fS_t)$.

\begin{remark}
If $t=m$ belongs to $\bN$ then Deligne's construction of $\uRep^{\rm ab}(\fS_t)$ takes place within the semi-simple category $\uRep(\fS_{-1})$. We also make use of $\fS_{-1}$, in a certain sense, as we restrict to $\fS(m+1)$ to identify $\mu_t$ with the regular measure $\mu_{-1}$ (see Proposition~\ref{prop:sym-qreg}).
\end{remark}

\subsection{Positive characteristic} \label{ss:sym-char-p}

Let $k$ be a field of positive characteristic $p$. Given a $p$-adic integer $t \in \bZ_p$, there is a ring homomorphism $\bZ\langle x \rangle \to \bF_p \subset k$ given by evaluating a polynomial at $t$ (which results in a $p$-adic integer) and reducing modulo $p$. In fact, these account for all homomorphisms $\bZ\langle x \rangle \to k$, see \cite[\S 4]{CahenChabert}. Let $\mu_t \colon \Theta(\fS) \to k$ be the homomorphism corresponding to $t$. Thus
\begin{displaymath}
\mu_t(\Omega^{(n)})=\binom{t}{n} \pmod{p}.
\end{displaymath}
The $\mu_t$ account for all of the $k$-valued measures for $\fS$.

We now give an important example that illustrates the somewhat pathological behavior of the measure $\mu_t$. Let $X=\Omega^{[p]}$, let $Y=\Omega^{(p)}$, and let $f \colon X \to Y$ be the natural map. Given $\phi \in \cC(X)$, there is $c \in k$ and $N \in \bN$ such that $\phi(a_1, \ldots, a_p)=c$ whenever $a_i>N$ for all $i$. We call $c$ the \defn{generic value} of $\phi$. Similar remarks apply to $\phi\in \cC(Y)$ (simply consider $f^*\phi$). Let $M \subset \cC(Y)$ be the space of all functions with generic value~0.

\begin{proposition}
The image of $f_* \colon \cC(X) \to \cC(Y)$ is $M$.
\end{proposition}

\begin{proof}
Let $\phi$ be an element of $\cC(X)$. Let $c$ be the generic value of $\phi$, and let $N$ be such that $\phi(a_1,\ldots,a_p)=c$ when $a_i>N$. We thus see that $(f_*\phi)(\{a_1,\ldots,a_p\})=p! \cdot c=0$ if $a_i>N$. Thus $f_* \phi$ has generic value~0, so it belongs to $M$.

For notational simplicity, we prove the reverse inclusion just for $p=2$. It is clear that $\im(f_*)$ contains all point masses. Now let $A \subset Y$ consist of all 2-element subsets of the form $\{1,n\}$, and let $B \subset X$ consist of all ordered pairs of the form $(1,n)$. Then $f$ induces a bijection $B \to A$, and so $f_*(1_B)=1_A$. It is clear that $\im(f_*)$ is stable under the action of the group $\fS$, and thus a $k[\fS]$-module. One easily sees that $M$ is generated as a $k[\fS]$-module by the point masses and $1_A$, and so $M \subset \im(f_*)$.
\end{proof}

\begin{corollary} \label{cor:char-p-bad}
We have the following:
\begin{enumerate}
\item The measure $\mu_t$ is not normal.
\item $M$ is a subobject of $\cC(Y)$ in $\uRep(\fS; \mu_t)$.
\item The functor $\Phi \colon \uPerm(\fS; \mu_t) \to \uRep(\fS; \mu_t)$ is not full.
\item The Schwartz space $\cC(Y)$ is not generated by point masses (as an $A(\fS)$-module). 
\end{enumerate}
\end{corollary}

\begin{proof}
(a) is clear, as $f$ is a surjection of finitary $\fS$-sets but $f_*$ is not surjective.

(b) follows since $f_*$ is $A(\fS)$-linear (Proposition~\ref{prop:A-push-pull}).

(c) Let $N \subset \cC(Y)$ be the space of constant functions. This is an $A(\fS)$-submodule as well, being the image of the natural map $\bbone \to \cC(Y)$. We have
\begin{displaymath}
\cC(Y) = M \oplus N.
\end{displaymath}
We thus see that there is a non-zero $A(\fS)$-linear endomorphism $a$ of $\cC(Y)$ that kills all point masses (they all belong to $M$). However, any matrix that kills all point masses is zero (see Proposition~\ref{prop:matrix-faithful}). It follows that $a$ is not induced by a matrix, and so $\Phi$ is not full.

(d) Since $M$ contains the point masses in $\cC(Y)$ they do not generate.
\end{proof}

Despite these negative results about $\uRep(\fS; \mu_t)$, our results about $\uPerm(\fS; \mu_t)$ are consistent with it embedding into a rigid abelian tensor category (see Remark~\ref{rmk:abenv}). In fact, it is known that $\uPerm(\fS; \mu_t)$ does embed into such a category: this category can be found inside the ultraproduct of the categories $\Rep_k(\fS_n)$, as explained in \cite{Harman} (following an idea of Deligne). The results of \cite{Harman} suggest that $\uPerm(\fS; \mu_t)$ is the ``correct'' category of permutation modules in this setting, as it reflects the limiting behavior of the representation theory of finite symmetric groups over $k$. (Similarly, these results suggest that the $\bZ\langle x \rangle$-linear category $\uPerm(\fS; \mu_{\rm univ})$ is the ``correct'' integral category to consider.)

\begin{remark}
In fact, \cite{Harman} actually works in the relative setting $(G, \sY)$, where $\sY$ is the stabilizer class of Young subgroups (see Example~\ref{ex:stab-young}). It seems that this relative case is pretty much the same as the absolute case, though; for instance, one can show that $\Theta(\fS)=\Theta(\fS;\sY)$.
\end{remark}

\subsection{A relative case} \label{ss:sym-rel}

Let $\sE$ be the stabilizer class generated by $\Omega$ (see Example~\ref{ex:stab-class}). We now examine how our theory works in the relative case $(\fS, \sE)$. We begin by computing the $\Theta$ ring:

\begin{proposition} \label{prop:sym-theta-rel}
We have an isomorphism $i \colon \bZ[x] \to \Theta(G; \sE)$ given by $i(x)=[\Omega]$. More generally, we have $i((x)_n)=[\Omega^{[n]}]$, where $(x)_n=x(x-1) \cdots (x-n+1)$.
\end{proposition}

\begin{proof}
Let $X$ be a transitive $\fS$-smooth $\fS(n)$-set for some $n$. By Proposition~\ref{prop:sym-open-subgroups}, we see that $X$ is isomorphic to $\fS(n)/\fS(m)$ for some $m>n$. In the notation of Lemma~\ref{lem:Theta-sym-1}, this is isomorphic to $\Omega_n^{[m-n]}$. We thus see that these the classes of such sets span $\Theta(G; \sE)$. Arguing as in Lemma~\ref{lem:Theta-sym-1}, one then shows that the classes of the sets $\Omega^{[n]}$ span. It follows that $i$ is surjective. Since the composition
\begin{displaymath}
\bZ[x] \to \Theta(G; \sE) \to \Theta(G) \cong \bZ\langle x \rangle
\end{displaymath}
is a ring homomorphism taking $x$ to $x$, it is injective. Thus $i$ is injective as well, which completes the proof.
\end{proof}

Let $k$ be a field and let $t \in k$. By the above proposition, we get a measure $\mu_t \colon \Theta(G; \sE) \to k$ by $[\Omega] \mapsto t$. Since the characteristic~0 case is not much different from the absolute case, we assume for simplicity that $k$ has characteristic $p>0$ in what follows. If $t \not\in \bF_p$ then $\mu_t$ is regular, while if $t \in \bF_p$ then $\mu_t$ is not even quasi-regular. One can show that $\mu_t$ is always normal (see \cite[\S 15.8]{arxiv}). Thus the category $\uRep(\fS, \sE; \mu_t)$ has a tensor structure, and the functor from $\uPerm(\fS, \sE; \mu_t)$ is full. It is also not difficult to show that $\uRep(\fS,\sE;\mu_t)$ is locally of finite length using properties of smooth $k[\fS]$-modules.

However, there is some pathological behavior. Suppose $t \not\in \bF_p$. Since the object $\cC(\Omega)$ has categorical dimension $t$, it follows that $\uRep(\fS,\sE;\mu_t)$ is not a rigid tensor category (see Remark~\ref{rmk:abenv}). Also, a standard construction produces a nilpotent matrix of non-zero trace. Indeed, let $X=\Omega^p$ and let $A \in \Mat_X$ be the matrix giving the action of the $p$-cycle $(1\;2\;\cdots\;p)$ on the tensor factors of $\Vec_X=\Vec_{\Omega}^{\otimes p}$. Then $A^p=I_X$, and so $B=A-I_X$ is nilpotent. A simple computation shows that $\tr(B)=t-t^p$, which is non-zero.

When $t \in \bF_p$, the above obstructions to rigidity disappear. However, we still suspect that $\uRep(\fS,\sE;\mu_t)$ is not rigid.

\section{Linear groups over finite fields} \label{s:finlin}

\subsection{Overview}

Fix a finite field $\bF$ with $q$ elements, and put
\begin{displaymath}
\bV_n=\bF^n, \quad \bV = \bigcup_{n \ge 1} \bV_n, \quad
G_n = \GL_n(\bF), \quad G = \bigcup_{n \ge 1} G_n.
\end{displaymath}
The group $G$ is oligomorphic with respect to its action on $\bV$. The main purpose of \S \ref{s:finlin} is to examine our theory in this case. We compute $\Theta(G)$ and establish some results about $\uRep(G)$. We also examine some variants in \S \ref{ss:levi} and \S \ref{ss:otherlinear}.

\subsection{Subgroup structure}

For a finite dimensional subspace $W$ of $\bV$, let $U(W)$ be the subgroup of $G$ fixing every element of $W$. Put $U_n=U(\bV_n)$, so that
\begin{displaymath}
U_n = 
\begin{bmatrix}
1_n & * \\
0 & *  \\
\end{bmatrix}
\end{displaymath}
By definition, a subgroup of $G$ is open if and only if it contains some $U_n$. We thus see that the $U(W)$ are open subgroups. We now study the structure of open subgroups.

\begin{proposition} \label{prop:stabilizerlemma}
Let $U$ be an open subgroup containing $U(W)$ and $U(W')$, for two finite dimensional subspaces $W$ and $W'$ of $\bV$. Then $U$ contains $U(W \cap W')$.
\end{proposition}

\begin{proof}
Suppose $\dim(W \cap W') = a$, $\dim(W) = a+b$, and $\dim(W) = a+c$ then by conjugating we may assume that $W = \text{span}(e_1,e_2, \dots e_{a+b})$ and $W = \text{span}(\{e_1,e_2, \dots e_{a}\} \cup \{e_{a+b+1},e_{a+b+2}, \dots e_{a+b+c} \})$. Explicitly this means we want to write an arbitrary element of $U_a$, the pointwise stabilizer of $W \cap W'$, as a product of $(a + b +c + \infty) \times (a + b +c + \infty)$-block matrices of the form:
\begin{displaymath}
\begin{bmatrix}
1 & * & 0 & * \\  
0 & * & 0 & * \\
0 & * & 1 & * \\
0 & * & 0 & *
\end{bmatrix}
\text{ and }
\begin{bmatrix}
1 & 0 & * & * \\  
0 & 1& * & * \\
0 & 0  & * & * \\
0 & 0 & * & *
\end{bmatrix}
\end{displaymath}
These two forms include all elementary matrices (i.e., those invertible matrices differing from the identity matrix in a single entry) in $U_a$, which are well-known to generate $U_a$. \end{proof}

\begin{proposition} \label{prop:subgroupstructure}
If $U$ is an open subgroup of $G$ then $U$ is conjugate to a subgroup of the form
\begin{displaymath}
H\cdot U_n = 
\begin{bmatrix}
H & * \\
0 & *  \\
\end{bmatrix}
\end{displaymath}
where $H$ is a subgroup of $G_n$.
\end{proposition}

\begin{proof}
Proposition~\ref{prop:stabilizerlemma} implies there is a unique minimal subspace $W$ for which $U$ contains the pointwise stabilizer of $W$, it then suffices to show that $U$ stabilizes $W$ setwise. Conjugating the pointwise stabilizer of $W$ by $g \in U$ we see that $U$ also contains the stabilizer of $gW$, and therefore $W \cap gW$. Since $W$ was minimal we see that $W \cap gW = W$, so $gW = W$.  Hence $W$ is stabilized by all $g \in U$.
\end{proof}

\begin{proposition} \label{prop:subgroupinclusion}
Suppose $U \subset U_n$ is an open subgroup, then $U$ is conjugate within $U_n$ to a subgroup of the form
\begin{displaymath}
(K \rtimes H)\cdot U_\ell =
\begin{bmatrix}
1_n & K & * \\
0 & H  & *  \\
0& 0 & *
\end{bmatrix}
\end{displaymath}
where $H$ is a subgroup of $G_{\ell-n}$, and $K$ is an $H$-stable subgroup of $\bF^{n \times (\ell-n)}$.
\end{proposition}

\begin{proof}
By Proposition~\ref{prop:subgroupstructure}, the maximal finite dimensional $U$-stable subspace $V$ coincides with the minimal subspace for which $U$ contains the pointwise stabilizer. Define $\ell = \dim(V)$ and note that  $U$ stabilizes $\bV_n$ since since $U \subseteq U_n$, so $\bV_n \subseteq V$. Since $U_n$ acts transitively on the $\ell$-dimensional subspaces of $\bV$ which contain $\bV_n$, we see that some $U_n$-conjugate of $U$ stabilizes $\bV_{\ell}$ and is of the desired form.
\end{proof}

\subsection{Quantum integer-valued polynomials}

Recall that the $q$-integer $[n]_q$ is defined by
\begin{displaymath}
[n]_q = 1+ q+q^2+ \dots q^{n-1}.
\end{displaymath}
In the present context, $q$ is a fixed prime power (as opposed to a variable), and so $[n]_q$ is an actual integer; in fact, it is the cardinality of the projective space $\bP(\bV_n)$. We also require the $q$-factorial
\begin{displaymath}
[n]_q! = [1]_q \cdot [2]_q\cdots [n]_q,
\end{displaymath}
which counts the number of complete flags in $\bV_n$, and the $q$-binomial coefficient
\begin{displaymath}
\qbinom{n}{d}_q = \frac{[n]_q!}{[d]_q! \cdot [n-d]_q!},
\end{displaymath}
which counts the number of $d$-dimensional subspaces of $\bV_n$ (and is in particular an integer).

Following \cite{HarmanHopkins}, we define $\cR_q$ to be the subring of $\bQ[x]$ consisting of all polynomials $p$ such that $p([n]_q)$ belongs to $\bZ[\tfrac{1}{q}]$ for all $n \in \bN$. This ring is a $q$-analog of the ring of integer-valued polynomials. Define the \defn{$q$-binomial coefficient polynomial} $\omega_{0,d}$ by
\begin{displaymath}
\omega_{0,d}(x) = \frac{x(x-[1]_q)\dots(x-[d-1]_q)}{q^{\binom{d}{2}} \cdot [d]_q!}
\end{displaymath}
for $d \ge 1$ and $\omega_{0,0,}=1$. This polynomial is denoted $\qbinom{x}{d}$ in \cite{HarmanHopkins}. These polynomials satisfy
\begin{displaymath}
\omega_{0,d}([n]_q) = \qbinom{n}{d}_q
\end{displaymath}
for $n \in \bN$, and therefore belong to $\cR_q$. In fact:

\begin{proposition} \label{prop:qbinomsspan}
The $\omega_{0,d}$, for $d \in \bN$, form a basis for $\cR_q$ as a $\bZ[\frac{1}{q}]$-module.
\end{proposition}

\begin{proof}
This is essentially \cite[Proposition~1.2]{HarmanHopkins}. Here we are taking $q$ to be a fixed prime power, as opposed to a formal variable, but the proof goes through unchanged.
\end{proof}

We also require a ``more integral'' version of $\cR_q$. Define $\cR_q^{\gg 0}$ to be the subring of $\cR_q$ consisting of all polynomials $p$ such that $p([n]_q) \in \bZ$ for all $n$ sufficiently large. This ring contains the $\omega_{0,d}$, and so, by Proposition~\ref{prop:qbinomsspan}, we have $\cR_q=\cR^{\gg 0}_q[\tfrac{1}{q}]$. The map $x \to qx+1$ takes $[n]_q$ to $[n+1]_q$, and induces an isomorphism $S \colon \cR_q \to \cR_q$ sending $p(x)$ to $p(qx+1)$. This is the shift operator studied in \cite[\S 5]{HarmanHopkins}. Define the \defn{shifted $q$-binomial coefficient polynomial} $\omega_{m,d}$ to be $S^{-m}\omega_d$. We have
\begin{displaymath}
\omega_{m,d}([n]_q)=\qbinom{n-m}{d}_q,
\end{displaymath}
for $n \ge m$, and so $\omega_{m,d}$ belongs to the subring $\cR^{\gg 0}_q$. In fact:

\begin{proposition} \label{prop:shiftedqbinomsspan}
The ring $\cR_q^{\gg 0}$ is generated as a $\bZ$-module by the $\omega_{m,d}$.
\end{proposition}

\begin{proof}
We proceed by induction on degree. For degree~$0$ polynomials this is clear: degree~$0$ polynomials in $\cR_q^{\gg 0}$ are just integers, and $\omega_{m,0}=1$. If $p \in \cR_q^{\gg 0}$ is of degree $d$, then in $\cR_q$ we can write $p = \frac{a}{q^m} \omega_{0,d} + (\text{lower degree terms}) $ for some integer $a$ by Proposition~\ref{prop:qbinomsspan}. Then $p - a \omega_{m,d} \in \cR_q^{\gg 0}$ is a polynomial of degree less than $d$, so by induction we can write it as a $\bZ$-linear combination of shifted $q$-binomial coefficient polynomials; rearranging expresses $p$ in the desired way.
\end{proof}

\begin{remark}
The ring $\cR_q^{\gg 0}$ is a subring of $\bQ[x]$, and therefore $\bZ$-torsion free. Suppose $f \in \cR_q^{\gg 0}$, and let $f'=\binom{f}{d}$. Then $f'([n]_q)=\binom{f([n]_q)}{d}$ for $n \in \bN$. Since $f([n]_q)$ is integral for $n \gg 0$, so is $f'([n]_q)$. It follows that $f' \in \cR_q^{\gg 0}$. We thus see that $\cR_q^{\gg 0}$ is a binomial ring.
\end{remark}

\subsection{The ring \texorpdfstring{$\Theta(G)$}{\textTheta(G)}}

We now determine the ring $\Theta(G)$. We first introduce some notation. Let $\Omega_{m,d}$ denote the set of all $(d+m)$-dimensional subspaces of $\bV$ that contain $\bV_m$; we identify this with the set of $d$-dimensional subspaces of $\bV/\bV_m$. We also write $\Gr(d)$ for $\Omega_{0,d}$, as it is the collection of all $d$-dimensional subspaces of $\bV$. The set $\Omega_{m,d}$ carries a natural smooth action of $U_m$, and hence is a $\hat{G}$-set. The following is our main result:

\begin{theorem} \label{thm:GL-Theta}
There exists a unique isomorphism of rings
\begin{displaymath}
\phi \colon \cR_q^{\gg 0} \to \Theta(G), \qquad \phi(\omega_{m,d}) = [\Omega_{m,d}].
\end{displaymath}
\end{theorem}

We require a number of lemmas before giving the proof.

\begin{lemma} \label{lem:GL-1}
The classes $[\Omega_{m,d}]$ span $\Theta(G)$.
\end{lemma}

\begin{proof}
Let $X$ be a finitary $U$-set for some open subgroup $U$ of $G$. We must realize $[X]$ as a $\bZ$-linear combination of the classes $[\Omega_{m,d}]$. We are free to shrink or conjugate $U$, so we may suppose $U=U_m$ for some $m$. It suffices to treat the case where $U$ acts transitively on $X$. By Proposition~\ref{prop:subgroupinclusion} we have $X = U_m / (K \rtimes H)\cdot U_\ell$, for some subgroup $K \rtimes H$ of $\bF^{m \times (\ell-m)} \rtimes G_{\ell-m}$. Then $X$ admits a $n$-to-$1$ map to $U_m /(\bF^{m \times (\ell-m)} \rtimes G_{\ell-m}) \cong \Omega_{m,\ell-m}$ where $n$ is the index of $K \rtimes H$ in $\bF^{m \times (\ell-m)} \rtimes G_{\ell-m}$. We therefore have $[X] = n \cdot [\Omega_{m,\ell-m}]$ in $\Theta(G)$ by Corollary~\ref{cor:Theta-d-to-1}.
\end{proof}

\begin{lemma} \label{lem:GL-2}
The classes $[\Omega_{0,d}] = [\Gr(d)]$ span $\Theta(G) \otimes \bZ[\frac{1}{q}]$ as a $\bZ[\frac{1}{q}]$-module.
\end{lemma}

\begin{proof}
We need to show that after inverting $q$ the $[\Omega_{m,\bullet}]$ classes can be expressed just in terms of the  $[\Omega_{0,\bullet}] = [\Gr(\bullet)]$ classes. To start, we decompose $\Gr(d)$ as an $U_m$ set. First let $\Gr(d)_i$ denote the set of those $k$ dimensional subspaces which intersect $\bV_m$ in an $i$ dimensional subspace. We have a decomposition:
\begin{displaymath}
\Gr(d) = \bigsqcup_{i\le d} \Gr(d)_i
\end{displaymath}
Let $\Gr(m,i)$ denote the set of $i$-dimensional subspaces of $\bV_m$. If $j = d-i$, we have a $U_m$-equivariant map $\Gr(d)_i \to \Gr(m,i) \times \Omega_{m,j}$ given by sending a subspace $V$ to the intersection $V \cap \bV_m$ in the first factor, and to the image of $V$ in the quotient $\bV/ \bV_m$ in the second factor. The fibers of this map have cardinality $q^{mj}$. Therefore, in $\Theta(G)$ we have
\begin{displaymath}
[\Gr(d)_i] =  q^{mj} \cdot \# \Gr(m,i) \cdot [\Omega_{m,j}]
\end{displaymath}
by Corollary~\ref{cor:Theta-d-to-1}. Combining this with the previous equation, plugging in $\# \Gr(m,i) = \qbinom{m}{i}_q$, we obtain
\begin{displaymath}
[\Gr(d)] = \sum_{i+j =d} q^{mj} \qbinom{m}{i}_q [\Omega_{m,j}]
\end{displaymath}
Thus the $[\Omega_{0,\bullet}]$'s are related to the $[\Omega_{m,\bullet}]$'s via an upper triangular matrix with powers of $q$ on the diagonal, so the $[\Omega_{0,\bullet}]$'s generate the $[\Omega_{m,\bullet}]$'s over $\bZ[\tfrac{1}{q}]$.
\end{proof}

Define a $\bZ[\tfrac{1}{q}]$-linear map
\begin{displaymath}
\phi_0 \colon \cR_q \to \Theta(G) \otimes \bZ[\tfrac{1}{q}], \qquad \phi_0(\omega_{0,k})=[\Gr(k)].
\end{displaymath}
There is a unique such map since the $\omega_{0,n}$'s form a basis for $\cR_q$ as a $\bZ[\tfrac{1}{q}]$-module (Proposition~\ref{prop:qbinomsspan}).

\begin{lemma} \label{lem:GL-3}
The map $\phi_0$ is a ring homomorphism.
\end{lemma}

\begin{proof}
We need to check this map is compatible with multiplication. According to \cite[Theorem 3.2]{HarmanHopkins}, multiplication in $\cR_q$ is given by
\begin{displaymath}
\omega_{0,i} \omega_{0,j} = \sum_{d = \max(i,j)}^{i+j}  \frac{q^{(d-i)(d-j)} \, [d]_q!}{[d-i]_q![d-j]_q![i+j-d]_q!} \omega_{0,d}.
\end{displaymath}
To see the corresponding identity in $\Theta(G)$, first write
\begin{displaymath} 
\Gr(i) \times \Gr(j) = \bigsqcup_{d = \max(i,j)}^{i+j} (\Gr(i) \times \Gr(j))_d
\end{displaymath}
where $(\Gr(i) \times \Gr(j))_d$ denotes those pairs of subspaces $(V, W) \in \Gr(i) \times \Gr(j)$ such that $V + W$ has dimension $d$. The set $(\Gr(i) \times \Gr(j))_d$ has a map to $\Gr(d)$ taking $(V,W)$ to $V+W$. This map is $n$-to-$1$ where
\begin{displaymath}
n = \frac{q^{(d-i)(d-j)} \, [d]_q!}{[d-i]_q![d-j]_q![i+j-d]_q!};
\end{displaymath}
one sees this by fixing a $d$-dimensional subspace of $\bV$ and counting the pairs $(V,W)$ of subspaces of appropriate dimension which together span the space. Appealing to Corollary~\ref{cor:Theta-d-to-1}, this implies that in $\Theta(G)$ we have
\begin{displaymath}
[\Gr(i)] [\Gr(j)] = \sum_{d = \max(i,j)}^{i+j} \frac{q^{(d-i)(d-j)} \, [d]_q!}{[d-i]_q![d-j]_q![i+j-d]_q!} [\Gr(d)].
\end{displaymath}
So indeed this map respects multiplication.
\end{proof}

\begin{lemma} \label{lem:GL-4}
We have $\phi_0(\omega_{m,d})=[\Omega_{m,d}]$ for all $m, d \in \bN$.
\end{lemma}

\begin{proof}
We will proceed by induction on $m$ and $d$. When $m = 0$, this is part of the definition of the homomorphism, and when $d=0$ we always have $\omega_{m,d}= 1$ and $[\Omega_{m,d}] = [\bone]$.  In $\cR_q^{\gg 0}$ we have the $q$-Pascal identity \cite[(5.4)]{HarmanHopkins}:
\begin{displaymath}
\omega_{m,d} = q^d \omega_{m+1,d} + \omega_{m+1,d-1}
\end{displaymath}
So by induction we see that it suffices to check the corresponding identity holds in $\Theta(G)$. To see this write $\Omega_{m,d} = A \sqcup B$ where $A$ consists of those $d$-dimensional subspaces of $\bV/\bV_m$ that do not contain the kernel of the map $\pi \colon \bV/\bV_m \to \bV/\bV_{m+1}$, and $B$ consists of those $d$-dimensional subspaces that do.  

The map $\pi$ induces a $q^d$-to-$1$ map from $A$ to $\Omega_{m+1,d}$, with $q^d$ counting the ways to lift a $d$-dimensional subspace. So $[A] = q^d \cdot [\Omega_{m+1,d}]$ in $\Theta(G)$ by Corollary~\ref{cor:Theta-d-to-1}. Similarly, $\pi$ induces a bijection between $B$ and $\Omega_{m+1, d-1}$, with the preimage of a $(d-1)$-dimensional subspace defining a unique element of $B$, so $[B] = [\Omega_{m+1,d-1}]$ in $\Theta(G)$. Combining these gives the desired formula:
\begin{displaymath}
[\Omega_{m,d}]= q^d \cdot [\Omega_{m+1,d}] + [\Omega_{m+1,d-1}]. \qedhere
\end{displaymath}
\end{proof}

\begin{lemma}
The subgroups $U_{\bullet}$ satisfy the conditions of Proposition~\ref{prop:fixed-pt}. In particular, $G$ admits a fixed point measure.
\end{lemma}

\begin{proof}
If $g \in G_d$ then $gU_ng^{-1} = U_n$ for all $n \ge d$, so condition~(a) is satisfied. Condition~(b) follows from the classification of subgroups of $U_k$ in Proposition~\ref{prop:subgroupinclusion}.
\end{proof}

\begin{lemma} \label{Grindep}
The classes $[\Omega_{0,d}] = [\Gr(d)]$ are linearly independent over $\bZ$.
\end{lemma}

\begin{proof}
The fixed point measure sends $[\Gr(d)]$  to the function $f(n) = \qbinom{n}{d}_q$ which grows in $n$ approximately like a constant times $q^{dn}$, so we see that the asymptotic growth of the fixed point measure evaluated on any linear combination $\sum a_i [\Gr(i)]$ is dominated by the leading non-zero term, and the only way a linear combination can get sent to $f(n) = 0$ is if all the $a_i$'s vanish.
\end{proof}

\begin{proof}[Proof of Theorem~\ref{thm:GL-Theta}]
By Lemma~\ref{lem:GL-1}, the classes $[\Omega_{m,d}]$ span $\Theta(G)$ as a $\bZ$-module. Thus Lemma~\ref{lem:GL-4} (and Proposition~\ref{prop:shiftedqbinomsspan}) shows that $\phi_0$ restricts to a surjection $\phi \colon \cR_q^{\gg0} \to \Theta(G)$. Lemma~\ref{Grindep} shows that the elements $\phi(\omega_{0,d})$ are $\bZ$-linearly independent in $\Theta(G)$. Since $\cR_q^{\gg 0}$ is torsion-free and the $\omega_{0,d}$ form a basis for $\cR_q^{\gg 0} \otimes \bQ$, this implies the map is injective.
\end{proof}

\subsection{Representations}

Let $k$ be a field of characteristic~0. Since $\Theta(G) \otimes \bQ=\bQ[x]$, the $k$-valued measures for $G$ are parametrized by elements of $k$. Given $t \in k$, let $\mu_t$ be the corresponding measure. By definition, this measure satisfies
\begin{displaymath}
\mu_t(\Omega_{m,d}) = \omega_{m,d}(t).
\end{displaymath}
We fix $t \in k$ in what follows. We also let $\bN_q$ denote the set of numbers of the form $[n]_q$ with $n \in \bN$.

\begin{proposition}
The measure $\mu_t$ is quasi-regular, and regular if and only if $t \not\in \bN_q$.
\end{proposition}

\begin{proof}
First suppose $t \not\in \bN_q$. Given an open subgroup $U$ of $G$, Proposition~\ref{prop:subgroupstructure} shows that there is an $e$-to-1 map $G/U \to \Gr(d)$ for some $e$ and $n$, and so $\mu_t(G/U)=e \mu_t(\Gr(d))=e \omega_{0,d}(t)$, which is non-zero. Thus $\mu_t$ is regular.

Now suppose that $t=[n]_q$ for some $n \in \bN$. Then $\mu_t(\Gr(n+1))=\omega_{0,n+1}([n]_q)=0$, and so $\mu_t$ is not regular. We claim that the restriction of $\mu_t$ to $U_{n+1}$ is regular. Let $V$ be an open subgroup of $U_{n+1}$. By Proposition~\ref{prop:subgroupinclusion}, there is an $e$-to-1 map $U_{n+1}/V \to \Omega_{n+1,d}$ for some $e$ and $d$. Thus $\mu_t(U_{n+1}/V)=e \omega_{n+1,d}(t) \ne 0$. This completes the proof.
\end{proof}

Let $\uRep(G)$ denote the representation category $\uRep_k(G; \mu_t)$.

\begin{theorem} \label{thm:GLcat}
The category $\uRep(G)$ is locally pre-Tannakian, and semi-simple if $t \not\in \bN_q$. Moreover, $\uRep^{\rf}(G)$ is the abelian envelope of $\uPerm(G)$ in the sense of Definition~\ref{defn:abenv}.
\end{theorem}

\begin{proof}
First suppose that $t \in \bN_q$, say $t=[n]_q$ with $n \in \bN$. Then $\omega_{0,d}(t)=\qbinom{n}{d}_q$. This is an integer for all $d$ (it vanishes for $d>n$). Since the $\omega_{0,d}$ span $\Theta(G)[\tfrac{1}{q}]=\cR_q$ over $\bZ[\tfrac{1}{q}]$ (Proposition~\ref{prop:qbinomsspan}), we see that the measure $\mu_t$ takes values in $\bZ[\tfrac{1}{q}] \subset k$. It follows that $\mu_t$ satisfies property~(P). The statements now follow from Theorems~\ref{thm:regss} and~\ref{thm:abenv}.

Now suppose that $t \not\in \bN_q$. In this case, property~(P) is more subtle, and we do not know an unconditional proof (see Remark~\ref{rmk:GL-P} below). However, \cite{Knop2} treats this case, and we can appeal to the results there to get around this issue. In the notation and terminology of \cite{Knop2}, let $\cA$ be the category of finite dimensional $\bF$-vector spaces (with all linear maps), and let $\delta$ be the degree function defined for $\cA$ by $\delta(e)=((t-1)/(q-1))^{\dim \ker{e}}$; see \cite[\S 3, Example~3]{Knop2}. (We note that our parameter $[n]_q$ corresponds to Knop's parameter $q^n$, which is why we use $(t-1)/(q-1)$ in the definition of $\delta$ instead of $t$.) Knop shows that his tensor category $\cT(\cA,\delta)$ is semi-simple in this case (see \cite[\S 8, Example~4]{Knop2}).

Now, one can show that the endomorphism algebra of $\cC(\bV^{\oplus n})$ in our category $\uPerm(G)$ is isomorphic to the endomorphism algebra of $\bF^n$ in Knop's category $\cT(\cA,\delta)$. In particular, we see that the endomorphism algebra of $\cC(\bV^{\oplus n})$ is semi-simple. This is enough for the proof of Theorem~\ref{thm:regss} to go through (e.g., to get semi-simplicity, apply Lemma~\ref{lem:regss-1} with $\Sigma$ being the $\cC(\bV^{\oplus n})$. It is also enough to get the proof of Theorem~\ref{thm:abenv} to go through (the most difficult part of that proof, condition~(iii), is trivial in the semi-simple case).
\end{proof}

Interpolation categories associated to $\GL_n(\bF)$ were first considered in detail in the work of Knop \cite{Knop,Knop2} used above. In particular, he constructed a pre-Tannakian interpolation category in the case $t \not\in \bN_q$. In forthcoming work \cite{EntovaAizenbudHeidersdorf}, Entova-Aizenbud and Heidersdorf establish additional properties of these interpolation categories, including a mapping property. As far as we know, the existence of an abelian envelope in the case $t \in \bN_q$ has not previously appeared in the literature. We have been informed that Entova-Aizenbud and Heidersdorf also have a construction of the envelope.

\begin{remark} \label{rmk:GL-P}
As stated in the proof of Theorem~\ref{thm:GLcat}, property~(P) for the measure $\mu_t$ can be subtle. Suppose $t$ belongs to a number field $K$ (which is the most difficult case), and put $t'=(q-1)t+1$. Then~(P) would be implied by the following:
\begin{itemize}
\item[$(\ast)$] There are infinitely many prime ideals $\fp$ of $\cO_K$ such that $t'$ belongs to the closure of the subgroup of $K_{\fp}^{\times}$ generated by $q$.
\end{itemize}
Indeed, if $\fp$ were such a prime then $\omega_{0,d}(t)$ would be integral at $\fp$ since $t$ is $\fp$-adically very close to some $q$-integer. Thus the image of $\mu_t$ in $K_{\fp}$ would be contained in the integers of $K_{\fp}$, and we would obtain property~(P) by arguing as in Proposition~\ref{prop:P}. We note that $(\ast)$ does not hold for $t=(1-q)^{-1}$ (which corresponds to $t'=0$), but for any other $t$ it seems reasonable to expect. We also note that if $K=\bQ$ then a result of P\'olya \cite{Polya} (see also \cite[\S 8.4.2]{Moree}) shows that there are infinitely many primes $p$ such that the reduction mod $p$ of $t$ belongs to the subgroup of $\bF_p^{\times}$ generated by $q$.
\end{remark}

\subsection{The Levi topology} \label{ss:levi}

Recall that $\bV_n=\bF^n$. The standard projection map $\bV_{n+1} \to \bV_n$ induces an inclusion of dual spaces $\bV_n^* \to \bV_{n+1}^*$. The \defn{restricted dual} of $\bV$ is $\bV_*=\bigcup_{n \ge 1} \bV_n^*$. The group $G$ acts on $\bV_*$, and the action of $G$ on $\bV \oplus \bV_*$ is oligomorphic. This induces a second pro-oligomorphic topology on $G$, which we call the \defn{Levi topology}. The subgroup
\begin{displaymath}
L_n = \begin{bmatrix}
1_n & 0 \\
0 & *  \\
\end{bmatrix}
\end{displaymath}
is exactly the subgroup of $G$ fixing each of $e_1, \ldots, e_n$ and $e_1^*, \ldots, e_n^*$. Thus $L_n$ is open in the Levi topology, and any open subgroup contains some $L_n$. Write $G^{\ell}$ (resp.\ $G^p$) for the group $G$ equipped with the Levi (resp.\ parabolic) topology. The ring $\Theta(G^p)$ was computed in Theorem~\ref{thm:GL-Theta}. We now compute $\Theta(G^{\ell})$.

Since every open subgroup in the parabolic topology is also open in the Levi topology, there is a continuous homomorphism $G^{\ell} \to G^p$. This induces a ring homomorphism
\begin{displaymath}
\phi \colon \Theta(G^p) \to \Theta(G^{\ell})
\end{displaymath}
by \S \ref{ss:Theta-prop}(a). The following is our main result on $G^{\ell}$:
  
\begin{theorem} \label{thm:parabolic-levi}
The map $\phi$ is an isomorphism.
\end{theorem}

We refer to \cite[\S 16.6]{arxiv} for the proof. We have not studied the representation categories associated to the group $G^{\ell}$. We expect $\uRep(G^{\ell}; \mu_t)=\uRep(G^p; \mu_t)$ for generic $t$, but that the categories may differ for special values of $t$. It would be interesting to investigate this in more detail.

\subsection{Other linear groups} \label{ss:otherlinear}

There are other oligomorphic linear groups, such as the infinite orthogonal, symplectic, and unitary groups. It would be interesting to compute the $\Theta$ rings in these cases. These cases are all smoothly approximable, and so $\Theta \ne 0$ by the results of \S \ref{ss:Happrox}. Deligne \cite{DeligneLetter2} has observed that the infinite orthogonal group admits two 1-parameter families of measures (essentially from approximating it by even or odd finite orthogonal groups), and it seems probable to us that $\Theta \otimes \bQ$ is isomorphic to $\bQ[x,y]/(xy)$ in this case.

\section{Homeomorphisms of the line and circle} \label{s:order}

\subsection{Overview}

Let $G=\Aut(\bR, <)$ be the group of all order-preserving self-bijections of the real line $\bR$ considered in Example~\ref{ex:oligo}(c). In \S \ref{s:order}, we determine the measures on $G$ and construct a pre-Tannakian category associated to $G$. We also consider a variant where the line is replaced with the circle. These categories are studied in much greater depth in the follow-up papers \cite{line,circle}.

Throughout \S \ref{s:order}, the word ``interval'' will mean ``open interval in $\bR$.'' A \defn{left interval} is one of the form $(-\infty,a)$ where $a<\infty$; \defn{right interval} is defined analogously. A \defn{middle interval} is one of the form $(a,b)$ with $a<b$ both finite. For an interval $I$, we let $G^I$ be the group of orientation-preserving homeomorphisms of $I$; this is isomorphic to $G$. Note that if $x \in I$, and we write $I=J \sqcup \{x\} \sqcup K$, with $J$ and $K$ intervals, then the stabilizer of $x$ in $G^I$ is $G^J \times G^K$. This recursive structure of stabilizer groups will play a prominent role.

\subsection{Subgroup structure}

We begin by classifying the open subgroups of $G$. For a finite subset $A \subset \bR$, recall that $G(A)$ is the subgroup of $G$ that fixes each element of $A$ (which, in the present case, is equivalent to fixing the set $A$).

\begin{proposition} \label{prop:H1-sub}
The open subgroups of $G$ are exactly the $G(A)$ for $A \subset \bR$ finite.
\end{proposition}

\begin{proof}
If suffices to show that a subgroup $U$ of $G$ properly containing $G(A)$ contains some $G(B)$ with $B$ a proper subset of $A$. First suppose that $A=\{x\}$ is a single point, and let $g \in U$ not fix $x$. Using $g$ and $g^{-1}$, we can move $x$ to both a smaller and larger number. Since $G_x$ acts transitively on the intervals $(x,\infty)$ and $(-\infty,x)$, it follows that for each $a \in \bR$ there exists some $g_a \in U$ with $g_a(x)=a$. Now, let $h \in G$ be given. Put $a=h(x)$ and $h'=g_a^{-1} h$. Then $h'(x)=x$, and so $h'$ belongs to $U$. Since $g_a$ also belongs to $U$, we find that $h \in U$, and so $U=G$, as desired.

We now treat the general case. Let $A=\{x_1,\ldots,x_n\}$ with $x_1<\cdots<x_n$, and let $g$ be an element of $U$ not belonging to $G(A)$. If $g(A)=A$ then $g$ would induce an order-preserving self-map of the finite totally ordered set $A$ and would therefore fix each element of $A$. Since this is not the case, there is some $x_i$ such that $g(x_i) \not\in A$. Let $J$ be an interval around $g(x_i)$ that is small enough so that it contains no other element of $g(A)$ and is contained in $\bR \setminus A$. Then $G^J \subset G(A)$. Let $h \in G^J$ not fix $x$, and put $g'=g^{-1}hg$. Then $g'$ belongs to $U$ and fixes each $x_j$ for $j \ne i$. Let $I=(x_{i-1},x_{i+1})$, with the convention that $x_0=-\infty$ and $x_{n+1}=\infty$. Composing $g'$ with an element of $G(A)$, we can assume that it is the identity outside of $I$, and thus belongs to $G^I \subset G$. By the previous paragraph, $G^I$ is generated by $G^I_{x_i}$ and $g'$, and so $U$ contains $G^I$. Since $G^I$ and $G(A)$ clearly generate $G(B)$, where $B=A \setminus \{x_i\}$, we see that $U$ contains $G(B)$, which completes the proof.
\end{proof}

\begin{corollary}
The transitive $G$-sets are exactly the $\bR^{(n)}$ with $n \ge 0$.
\end{corollary}

\begin{corollary} \label{cor:A-open-trans}
Let $A$ be an $r$-element subset of $\bR$ and let $I_0, \ldots, I_r$ be the components of $\bR \setminus A$. Then the transitive $G(A)$-sets are exactly the $I_0^{(n_0)} \times \cdots \times I_r^{(n_r)}$ where $n_0, \ldots, n_r$ are non-negative integers.
\end{corollary}

Note that any finitary $\hat{G}$-set is a finite union of transitive $G(A)$-sets for some $A$, and so the above corollary tells us what these look like.

\subsection{Burnside rings}

We now determine the Burnside rings of $G$, its open subgroups, and $\hat{G}$. This is one of the few examples where a complete answer is possible. Recall (\S \ref{ss:intpoly}) that $\bZ\langle x_i \rangle_{i \in I}$ denotes the ring of integer-valued polynomials in indeterminates $\{x_i\}_{i \in I}$, and we write $\lambda_n(x)$ for $\binom{x}{n}$.

\begin{proposition} \label{prop:A-burnside}
We have a ring isomorphism $\bZ\langle x \rangle \to \Omega(G)$ via $\lambda_n(x) \mapsto \lbb \bR^{(n)} \rbb$.
\end{proposition}

\begin{proof}
Put $X_n=\bR^{(n)}$ and $x_n=\lbb X_n \rbb$. We have an isomorphism of $G$-sets
\begin{displaymath}
X_1 \times X_n \cong X_{n+1}^{\amalg n+1} \amalg X_n^{\amalg n}
\end{displaymath}
Indeed, an element of $X_1 \times X_n$ is a pair $(a, b_1, \ldots, b_n)$ where $b_1<\cdots<b_n$. There are $n$ ways that $a$ could coincide with one of the $b$'s; this contributes $n$ copies of  $X_n$. Away from this locus, $a$ is in one of the $n+1$ intervals obtained by deleting the $b$'s, and this gives $n+1$ copies of $X_{n+1}$. As a consequence, we find $x_1 x_n=(n+1)x_{n+1}+nx_n$.

Let $a_n=\lambda_n(x)$. Since the $a_n$'s form a $\bZ$-basis of $\bZ\langle x \rangle$, we have a unique $\bZ$-linear map $f \colon \bZ\langle x \rangle \to \Omega(G)$ given by $f(a_n)=x_n$. We can verify that it is a ring homomorphism after extending scalars to $\bQ$, since the domain and target are $\bZ$-free. This follows easily from the above identity involving the $x$'s.
\end{proof}

The same argument yields the following result:

\begin{corollary} \label{cor:A-burnside}
Let $A$ be an $r$-element subset of $\bR$ and let $I_0, \ldots, I_r$ be the components of $\bR \setminus A$. Then we have an isomorphism of rings
\begin{displaymath}
\bZ\langle x_0, \ldots, x_r \rangle \to \Omega(G(A)), \qquad \lambda_n(x_i) \mapsto \lbb I_i^{(n)} \rbb
\end{displaymath}
\end{corollary}

We now introduce a ring that we will play a prominent role in our discussion: we let $R=\bZ\langle x_I \rangle$, where $I$ varies over all intervals.

\begin{proposition} \label{prop:Ahat-burnside}
Let $\fc$ be the $\lambda$-ideal of $R$ generated by the equations
\begin{displaymath}
x_{(a,c)}=x_{(a,b)}+1+x_{(b,c)}
\end{displaymath}
over all choices of $-\infty \le a<b<c \le \infty$. Explicitly, $\fc$ is generated by the elements
\begin{displaymath}
\binom{x_{(a,c)}}{n} - \sum_{i+j+k=n} \binom{x_{(a,b)}}{i} \binom{1}{j} \binom{x_{(b,c)}}{k}
\end{displaymath}
for $a$, $b$, and $c$ as above, and $n \ge 0$. Then we have a ring isomorphism $R/\fc \to \Omega(\hat{G})$ given by $\lambda_n(x_I) \mapsto \lbb I^{(n)} \rbb$.
\end{proposition}

\begin{proof}
Let $a \in \bR$ and let $I_1=(-\infty,a)$ and $I_2=(a,\infty)$. By Corollary~\ref{cor:A-burnside}, we have an isomorphism $\bZ\langle y_1,y_2 \rangle \to \Omega(G_a)$ given by $\lambda_n(y_1) \mapsto \lbb I_1^{(n)} \rbb$ and $\lambda_n(y_2) \mapsto \lbb I_2^{(n)} \rbb$. Write $\Omega(G)=\bZ\langle x \rangle$ with $x=\lbb \bR \rbb$. We now examine the restriction map $\Omega(G) \to \Omega(G_a)$. The group $G_a$ has three orbits on $\bR$, namely $I_1$, $\{a\}$, and $I_2$; thus $\bR=I_1 \sqcup \{a\} \sqcup I_2$. More generally, we have
\begin{displaymath}
\bR^{(n)} = \bigsqcup_{i+j+k=n} I_1^{(i)} \times \{a\}^{(j)} \times I_2^{(k)}.
\end{displaymath}
We thus see that the restriction map $\Omega(G) \to \Omega(G_a)$ corresponds to the map $\bZ\langle x \rangle \to \bZ\langle y_1,y_2 \rangle$ given by $\lambda_n(x) \mapsto \lambda_n(y_1+1+y_2)$.

The general situation is similar to the above. Given a finite subset $A$ of $\bR$, let $I_0, \ldots, I_r$ be the components of $\bR \setminus A$. Then $\Omega(G_A)$ is identified with $\bZ\langle x_{I_0}, \ldots, x_{I_r} \rangle$. Let $B$ be another finite subset of $\bR$, containing $A$, and let $J_0, \ldots, J_s$ be the components of $\bR \setminus B$. Then the restriction map $\Omega(G_A) \to \Omega(G_B)$ is described as follows. Write $I_i=J_k \sqcup \pt \sqcup \cdots \sqcup \pt \sqcup J_{k+\ell}$. Then $x_{I_i}$ is mapped to $x_{J_k}+\cdots+x_{J_{k+\ell}}+(\ell-1)$, and the map is one of $\lambda$-rings. The description of $\Omega(\hat{G})$ now follows.
\end{proof}

\subsection{Analysis of \texorpdfstring{$\Theta(G)$}{\textTheta(G)}}
 
The following theorem gives a complete description of $\Theta(G)$.

\begin{theorem} \label{thm:A-Theta}
For $\epsilon,\delta \in \{-1,0\}$ there exists a unique ring homomorphism
\begin{displaymath}
\mu_{\epsilon,\delta} \colon \Theta(G) \to \bZ
\end{displaymath}
given by
\begin{displaymath}
\mu_{\epsilon,\delta}([I^{(n)}]) = \begin{cases}
(-1)^n & \text{if $I$ is a middle interval} \\
\epsilon^n & \text{if $I$ is a left interval} \\
\delta^n & \text{if $I$ is a right interval}
\end{cases}
\end{displaymath}
The ring homomorphism $\Theta(G) \to \bZ^4$ furnished by the four $\mu$'s is an isomorphism.
\end{theorem}

\begin{remark}
By Corollary~\ref{cor:A-open-trans}, every $\hat{G}$-set admits a smooth manifold structure. One easily sees that $\mu_{-1,-1}(X)$ is the compact Euler characteristic of $X$. The other measures are also related to Euler characteristics. For example, $\mu_{\epsilon,\delta}(\bR)$ is the compact Euler characteristic of a partial compactification $\bR$, where a left endpoint is added if $\epsilon=0$ and a right one is added if $\delta=0$. We thank David Treumann for this observation.
\end{remark}

We require a lemma before giving the proof. Let $R$ and $\fc$ be as in Proposition~\ref{prop:Ahat-burnside}, so that $\Omega(\hat{G})=R/\fc$. We classify intervals into four types: left, right, middle, and $\bR$ itself.

\begin{lemma} \label{lem:A-Theta-1}
Let $\fd \subset R$ be the ideal generated by the following elements:
\begin{enumerate}
\item $\lambda_n(x_I)-\lambda_n(x_J)$ for $n \ge 0$, whenever $I$ and $J$ have the same type.
\item For $-\infty \le a<b<c \le \infty$ and $n=i+j+1$, the element
\begin{displaymath}
\lambda_n(x_{(a,c)})-x_{(a,c)} \lambda_i(x_{(a,b)}) \lambda_j(x_{(b,c)}).
\end{displaymath}
\end{enumerate}
Then we have an isomorphism $R/(\fc+\fd) \to \Theta(G)$ given by $\lambda_n(x_I) \mapsto [I^{(n)}]$.
\end{lemma}

\begin{proof}
Let $\fd_1$ be the ideal of $R$ generated by the elements in (a), and let $\fd_2$ be the ideal generated by the elements in (b), so that $\fd=\fd_1+\fd_2$. Recall (Proposition~\ref{prop:Burnside-Theta}) that $\Theta(G)$ is the quotient of $\Omega(\hat{G})$ by an ideal $\fa$, and $\fa$ is similarly expressed as $\fa=\fa_1+\fa_2$. Let $\pi \colon R \to \Omega(\hat{G})$ be the natural map. To prove the lemma, it suffices to show that $\pi(\fd)=\fa$. In fact, we show that $\pi(\fd_i)=\fa_i$ for $i=1,2$.

The ideal $\fa_1$ is generated by the elements $\lbb X \rbb-\lbb X^g \rbb$ where $X$ is a finitary $G$-set and $g \in G$. We have
\begin{displaymath}
\pi\big( \lambda_n(x_I)-\lambda_n(x_J) \big) = \lbb I^{(n)} \rbb-\lbb J^{(n)} \rbb.
\end{displaymath}
If $I$ and $J$ have the same type the $I=J^g$ for some $g \in G$, and so the above element belongs to $\fa_1$. Thus $\pi(\fd_1) \subset \fa_1$. In fact, we have equality. Indeed, $\fa_1$ is in fact generated by the elements $\lbb X \rbb-\lbb X^g \rbb$ where $X$ is a transitive $G(A)$-set, for variable $A$. Let $A$ be given, and let $I_1, \ldots, I_r$ be the components of $\bR \setminus A$. A transitive $G(A)$-set then has the form $I_1^{(n_1)} \times \cdots \times I_r^{(n_r)}$ (Corollary~\ref{cor:A-open-trans}). Given $g \in G$, let $J_i=g(I_i)$. Then
\begin{displaymath}
\lbb X \rbb-\lbb X^g \rbb = \lbb I_1^{(n_1)} \rbb \cdots \lbb I_r^{(n_r)} \rbb - \lbb J_1^{(n_1)} \rbb \cdots \lbb J_r^{(n_r)} \rbb.
\end{displaymath}
This is clearly a consequence of the relations in $\pi(\fd_1)$. We have thus shown that $\pi(\fd_1)=\fa_1$.

The ideal $\fa_2$ is generated by the elements
\begin{displaymath}
r_{A,B,C}=\lbb G(A)/G(C) \rbb-\lbb G(A)/G(B) \rbb \cdot \lbb G(B)/G(C) \rbb
\end{displaymath}
for finite subsets $A \subset B \subset C$ of $\bR$. One easily sees that it suffices to consider the case where $\#B=1+\# A$ (i.e., these elements generate $\fa_2$), so we restrict our attention to such elements. Let us first examine these relations more closely. Let $I_1, \ldots, I_r$ be the components of $\bR \setminus A$. Write $B=A \cup \{x\}$ with $x \in I_i$, and let $I_i=J \sqcup \{x\} \sqcup K$. Let $n_j$ be the number of points in $C$ contained in $I_j$. Write $n_i=p+q+1$, where $p=\# (C \cap J)$ and $q=\# (C \cap K)$. We then have
\begin{align*}
G(A)/G(C) &= I_1^{(n_1)} \times \cdots \times I_r^{(n_r)} \\
G(A)/G(B) &= I_i \\
G(B)/G(C) &= I_1^{(n_1)} \times \cdots \times I_{i-1}^{(n_i)} \times J^{(p)} \times K^{(q)} \times I_{i+1}^{(n_{i+1})} \times \cdots \times I_r^{(n_r)}
\end{align*}
We thus find that
\begin{displaymath}
r_{A,B,C}= \big( \prod_{j \ne i} \lbb I_j^{(n_j)} \rbb \big) \cdot \big( \lbb I_i^{(n_i)} \rbb-\lbb I_i \rbb \lbb J^{(p)} \rbb \lbb K^{(p)} \rbb \big).
\end{displaymath}
Let $C'$ be obtained from $C$ by removing all points in $I_j$ with $j \ne i$. Then the left factor above is $r_{A,B,C'}$, and thus already belongs to $\fa_2$.

The above analysis can be summarized as follows. Let $I$ be an interval, let $x \in I$, and write $I=J \sqcup \{x\} \sqcup K$. Also, let $n=p+1+q$ with $p,q \in \bN$. Put
\begin{displaymath}
r_{I,J,K}^{p,q} = \lbb I^{(n)} \rbb -\lbb I \rbb \lbb J^{(p)} \rbb \lbb K^{(q)} \rbb.
\end{displaymath}
Then $\fa_2$ is generated by the elements $r_{I,J,K}^{p,q}$.

It is clear that the generators of $\fd_2$ are exactly lifts of the elements $r_{I,J,K}^{p,q}$, and so $\pi(\fd_2)=\fa_2$. This completes the proof.
\end{proof}

\begin{proof}[Proof of Theorem~\ref{thm:A-Theta}]
We first observe that the maps $\mu_{\epsilon,\delta}$ in the statement of the theorem are well-defined: one simply has to verify that they kill the generators of $\fc$ and $\fd$, which is a straightforward computation.

For an interval $I$, let $\lambda_n(y_I)$ be the element $[I^{(n)}]$ of $\Theta(G)$. From the first generators for $\fd$, we see that $y_I$ only depends on the type of $I$. Let $y_{\ell}=y_{(-\infty,1)}$, $y_m=y_{(-1,1)}$, and $y_r=y_{(1,\infty)}$. Here $\ell$, $m$, and $r$ stand for left, middle, and right. Note that $y_{(-\infty,\infty)}=y_{\ell}+1+y_m+1+y_r$, and so is redundant. From $\fc$, we have the equation
\begin{displaymath}
y_m=y_{(-1,1)}=y_{(-1,0)}+1+y_{(0,1)}=2y_m+1,
\end{displaymath}
from which we conclude $y_m=-1$. A similar argument shows that $\lambda_n(y_m)=(-1)^n$ for each $n \ge 0$. The remainder of the equations defining $\fc$ yield no new information. For instance, we have
\begin{displaymath}
y_{\ell}=y_{(-\infty,-1)}=y_{(-\infty,-2)}+1+y_{(-1,-2)}=y_{\ell}+1+y_m,
\end{displaymath}
but since $y_m=-1$ this holds automatically.

We now look at the second generators of $\fd$. For any $n=i+j+1$, we find
\begin{displaymath}
\lambda_n(y_{\ell})=(-1)^j y_{\ell} \lambda_i(y_{\ell}),
\end{displaymath}
where we have used $\lambda_j(y_m)=(-1)^j$. Taking $j=0$ and $i=n-1$, these equations inductively show that $\lambda_n(y_{\ell})$ belongs to the subring generated by $y_{\ell}$. Taking $n=2$ and $i=0$ and then $i=1$, we find
\begin{displaymath}
\lambda_2(y_{\ell})=-y_{\ell}=y_{\ell}^2.
\end{displaymath}
Thus $y_{\ell}^2+y_{\ell}=0$. A similar analysis holds for $y_r$.

The above discussion can be summarized as follows: we have a surjective ring homomorphism
\begin{displaymath}
\bZ[u,v]/(u^2+u,v^2+v) \to \Theta(G), \qquad u \mapsto x_{\ell}, \quad v \mapsto x_r.
\end{displaymath}
Note that the domain above is isomorphic to $\bZ^4$. The $\mu$'s thus provide a map in the opposite direction, which is easily seen to be an inverse.
\end{proof}

\subsection{Representations} \label{ss:order-rep}

Put $\mu=\mu_{-1,-1}$ in what follows, and fix a field $k$. We write simply $\uRep(G)$ for the category $\uRep_k(G; \mu_k)$, where $\mu_k$ is the extension of scalars of $\mu$ to $k$. The following is our main result about this category:

\begin{theorem} \label{thm:ord-rigid}
The category $\uRep(G)$ is semi-simple and locally pre-Tannakian. Moreover, it is the abelian envelope of $\uPerm(G)$ in the sense of Definition~\ref{defn:abenv}.
\end{theorem}

\begin{proof}
We have $\mu(\bR^{(n)})=(-1)^n$, and so $\mu$ is a regular measure. Since $\mu$ takes values in $\bZ$, it satisfies property~(P) of Definition~\ref{defn:P}. The group $G$ is not first-countable; however, by Remark~\ref{rmk:not-1st-count} we see that Proposition~\ref{prop:A-surj}(b) holds, and this is the only place where the first-countable condition was really used in Part~\ref{part:genrep}. (Alternatively, one could work with $\Aut(\bQ,<)$ instead of $G$, which is first-countable.) Applying Theorems~\ref{thm:regss} and~\ref{thm:abenv}, we obtain the stated result.
\end{proof}

When $k$ is a field of positive characteristic, $\uRep^{\rf}(G)$ is the first known example of a semi-simple pre-Tannakian category of superexponential growth. (Note that superexponential growth follows from Remark~\ref{rmk:Bell}.) Results from \cite{line} show that this category has a number of other very interesting properties.

\subsection{Decomposition of the standard module}

We now examine the simple decomposition of the ``standard module'' $\cC(\bR)$. This discussion is intended as an extended example to illustrate how one can work with objects in $\uRep(G)$.

Define $A,B \in \Mat_{\bR}$ by
\begin{displaymath}
A(x,y) = \begin{cases} 1 & \text{if $x<y$} \\ 0 & \text{otherwise} \end{cases}
\qquad
B(x,y) = \begin{cases} 1 & \text{if $x>y$} \\ 0 & \text{otherwise} \end{cases}
\end{displaymath}
Then $A$, $B$, and $I_{\bR}$ form a basis for $\End(\cC(\bR))$. We have
\begin{displaymath}
(A^2)(x,z) = \int_{\bR} A(x,y) A(y,z) dy = \mu(\{y \in \bR \mid x<y<z \})=-A(x,z).
\end{displaymath}
Indeed, if $x<z$ then the set in $\mu$ is an open interval, which has volume $-1$; otherwise, the set is empty and has volume~0. We have
\begin{displaymath}
(AB)(x,z) = \int_{\bR} A(x,y) B(y,z) dy = \mu(\{y \in \bR \mid x,z<y \})=-1.
\end{displaymath}
Indeed, the set in $\mu$ is always an open interval, and thus has volume $-1$. Appealing to the symmetry between $A$ and $B$, we thus have
\begin{displaymath}
A^2=-A, \quad B^2=-B, \quad AB=BA=-I_{\bR}-A-B.
\end{displaymath}
Note that every entry of $I_{\bR}+A+B$ is equal to~1, which is why the rightmost equation above holds. We thus see that $\End(\cC(\bR))$ is isomorphic to $k^3$, as an algebra. The primitive idempotents are $A+1$, $B+1$ and $AB=-1-A-B$. Since $\End(\cC(\bR))$ is semi-simple, the images of these idempotents are non-isomorphic simple submodules.

Recall that $1_{(p,q)} \in \cC(\bR)$ denotes the indicator function of the interval $(p,q)$. We have
\begin{displaymath}
(A \cdot 1_{(p,q)})(x) = \int_{\bR} A(x,y) 1_{(p,q)}(y) dy = \mu((x,\infty) \cap (p,q))=-1_{(-\infty,q)}(x).
\end{displaymath}
We also have
\begin{displaymath}
(A \cdot \delta_q)(x) = \int_{\bR} A(x,y) \delta_q(y) dy = 1_{(-\infty,q)}(x).
\end{displaymath}
We thus have
\begin{displaymath}
(A+1) 1_{(p,q)} = 1_{(p,q)}-1_{(-\infty,q)} = 1_{(-\infty,p]}
\end{displaymath}
and
\begin{displaymath}
(A+1) \delta_q = 1_{(-\infty,q)}+\delta_q = 1_{(-\infty,q]}
\end{displaymath}
We conclude that the image of $A+1$ is the subspace $M$ of $\cC(\bR)$ spanned by functions of the form $1_{(-\infty,q]}$. In particular, $M$ is a simple $A(G)$-submodule of $\cC(\bR)$.

By symmetry, the image of $B+1$ is the subspace $M'$ of $\cC(\bR)$ spanned by functions of the form $1_{[p,\infty)}$. It is a simple submodule that is not isomorphic to $M$.

We have
\begin{displaymath}
(1+A+B)1_{(p,q)}=1_{p,q}-1_{(-\infty,q)}-1_{(p,\infty)}=-1
\end{displaymath}
and
\begin{displaymath}
(1+A+B)\delta_q = \delta_q + 1_{(-\infty,q)} + 1_{(q,\infty)} = 1.
\end{displaymath}
We thus see that the image of $1+A+B$ is the subspace $N$ of $\cC(\bR)$ consisting of constant functions. Of course, $N$ is isomorphic to the trivial representation $\bbone$.

To conclude, we see that
\begin{displaymath}
\cC(\bR) = M \oplus M' \oplus N
\end{displaymath}
is the decomposition of $\cC(\bR)$ into simple objects in the category $\uRep(G)$.

\subsection{Interpolation categories} \label{ss:not-interp}

We now show that $\uRep(G)$ cannot be obtained by interpolating finite groups. We first recall the general construction of interpolation, as put forth in \cite{Deligne3}. Let $I$ be an index set equipped with ultrafilter $\cF$. For each $i \in I$, let $k_i$ be a field and $\cC_i$ a $k_i$-linear rigid tensor category. Let $k^*$ be the ultraproduct of the $k_i$, and let $\cC^*$ be the ultraproduct of the $\cC_i$. By an interpolation of the $\cC_i$, we mean a $k^*$-linear rigid tensor subcategory $\cC$ of $\cC^*$.

\begin{theorem} \label{thm:not-interp}
The category $\uRep^{\rf}_{k^*}(G)$ is not an interpolation of representation categories of finite groups.
\end{theorem}

\begin{proof}
In the above notation, let $\cC_i=\Rep^{\rf}_{k_i}(\Gamma_i)$ where $\Gamma_i$ is a finite group. Suppose we have a tensor equivalence $\uRep^{\rf}_{k^*}(G) \cong \cC$ for some $\cC$. We will obtain a contradiction.

The standard object $\cC(\bR)$ of $\uRep_{k^*}(G)$ corresponds to an object $(V_i)_{i \in I}$ of $\cC$. Since $\cC(\bR)$ has the structure of a Frobenius algebra, so does each $V_i$ (after possibly shrinking $I$, i.e., passing to a subset in $\cF$). We can thus write $V_i=k[X_i]$, where $X_i$ is a finite $\Gamma_i$-set (the set of primitive idempotents of $V_i$).

One can detect how many orbits $G$ has on $\bR^n$ or $\bR^{[n]}$ or $\bR^{(n)}$ purely in terms of the Frobenius algebra structure and tensor operations. For example, $G$ has three orbits on $\bR^2$, which corresponds to the fact that $\Hom(\cC(\bR)^{\otimes 2}, \bbone)$ is 3-dimensional. A similar observation applies to the action of $\Gamma_i$ on $X_i$. We can thus transfer information about orbits of $G$ and orbits of $\Gamma_i$ back and forth.

We now come to the key point. Since $G$ is transitive on $\bR^{(5)}$, it follows that $\Gamma_i$ is transitive on $X_i^{(5)}$ for all $i$ (after possibly shrinking $I$). A theorem of Livingstone--Wagner \cite[Theorem~2(b)]{LivinstoneWagner} now shows that $\Gamma_i$ is transitive on $X_i^{[5]}$. Thus $G$ is transitive on $\bR^{[5]}$, which is a contradiction: $G$ is not even transitive on $\bR^{[2]}$. (We note that the case where $\#X_i$ is bounded is easily dealt with, so we can assume that $\# X_i \ge 10$ for all $i$, which is needed for \cite{LivinstoneWagner}.)
\end{proof}

The above proof illustrates a general principle: if $H$ is an oligomorphic group and $\uRep(H)$ can be obtained by interpolating finite groups, then $H$ inherits properties of finite permutation groups. Thus if $H$ does not satisfy some property of finite permutation groups, one should be able to show that $\uRep(H)$ cannot be obtained by interpolation.

\subsection{Homeomorphisms of the circle} \label{ss:circle}

Let $\bS$ denote the unit circle, fix a point $\infty \in \bS$, and identify $\bS \setminus \{\infty\}$ with the real line $\bR$. Let $H$ be the group of orientation-preserving self-homeomorphisms of $\bS$. The stabilizer in $H$ of the point $\infty$ is identified with $G$; by definition, this is an open subgroup of $H$.

\begin{theorem}
The measure $\mu_{-1,-1}$ for $G$ extends to $H$, and induces an isomorphism $\Theta(H)=\bZ$. We have $\mu_{-1,-1}(\bS)=0$.
\end{theorem}

\begin{proof}
Similar to what we saw with $G$, a finitary $H$-set admits a canonical smooth manifold structure, and compact Euler characteristic defines a measure. This shows that $\mu_{-1,-1}$ extends, and that $\mu_{-1,-1}(\bS)=0$. The other measures for $G$ do not extend, since $\infty$ plays an asymmetrical role for them. Since the map $\Theta(G) \to \Theta(H)$ is surjective (\S \ref{ss:Theta-prop}(a)) and $\Theta(H_3)$ is torsion-free (Theorem~\ref{thm:binom}), the result follows.
\end{proof}

In what follows, we let $\mu=\mu_{-1,-1}$ be the unique $\bZ$-valued measure for $H$. Fix a field $k$, and write simply $\uRep(H)$ for $\uRep_k(H; \mu_k)$.

\begin{theorem}
The category $\uRep(H)$ is locally pre-Tannakian.
\end{theorem}

\begin{proof}
The measure $\mu$ is quasi-regular since its restriction to $G$ is regular. Since $\mu$ takes values in $\bZ$, it satisfies~(P). The remarks about first-countability in the proof of Theorem~\ref{thm:ord-rigid} apply here as well. Thus the result follows from Theorem~\ref{thm:regss}.
\end{proof}

The category $\uRep(H)$ is not semi-simple. Indeed, since $\bS$ is a transitive $H$-set, the natural maps
\begin{displaymath}
\phi \colon \bbone \to \cC(\bS), \qquad \psi \colon \cC(\bS) \to \bbone
\end{displaymath}
span their respective mapping spaces. The composition $\psi \circ \phi$ is equal to $\mu(\bS) \cdot \id_{\bbone}=0$. Thus $\psi$ is a non-split surjection.

\subsection{Additional remarks}

The line and circle admit orientation-reversing involutions, which define involutions of the groups $G$ and $H$; let $G'=\bZ/2 \ltimes G$ and $H'=\bZ/2 \ltimes H$.  It follows from \S \ref{ss:Theta-prop}(e) that $\Theta(G')=\bZ^2$ and $\Theta(H')=\bZ$.

The groups discussed in this section are significant from the perspective of permutation group theory. A permutation group $(\Gamma, \Omega)$ is called \defn{highly homogeneous} if $\Gamma$ acts transitively on $\Omega^{(n)}$ for all $n$. The four groups $G$, $G'$, $H$, and $H'$ are all highly homogeneous, as is the infinite symmetric group. Cameron \cite{Cameron7} proved that these are the only examples, up to a certain notion of equivalence.

The groups $G$ and $H$ are symmetry groups of one-dimensional manifolds. It would be interesting if one could construct tensor categories associated to higher dimensional manifolds. There is some speculative discussion about this in \cite[\S 17.10]{arxiv}.

\section{Boron trees} \label{s:boron}

\subsection{Overview}

A \defn{boron tree} is a tree in which all internal vertices have valence three. The internal vertices are called \defn{boron atoms} and the leaves are called \defn{hydrogen atoms}, in a nomenclature inspired by chemistry. When drawing boron trees, we depict boron atoms as unfilled circles and hydrogen atoms as filled circles. See Figures~\ref{fig:boron8} and~\ref{fig:boron} for some examples. For general background on boron trees, see \cite{Cameron6} and \cite[\S 2.6]{Cameron}.

Let $T$ be a boron tree, and let $T_H$ be the set of hydrogen atoms in it. Given $w,x \in T_H$, there is a unique geodesic in $T$ through $w$ and $x$. We define a quaternary relation $R$ on $T_H$ by declaring $R(w,x;y,z)$ to be true if the geodesic joining $w$ and $x$ meets the one joining $y$ and $z$. One can show that the tree $T$ can be recovered from the structure $(T_H, R)$, so the two points of view are equivalent; however, the latter will be more convenient for the moment.

Let $\fA$ be the class of all finite structures on a single quaternary relation that are isomorphic to $(T_H, R)$ for some boron tree $T$. This is a Fra\"iss\'e class \cite[Exercise 2.6.4]{Cameron} (see \S \ref{ss:fraisse} for the definition). Let $\Omega$ be the Fra\"iss\'e limit and let $G$ be its automorphism group. The group $G$ is oligomorphic (with respect to its action on $\Omega$).

In \S \ref{s:boron}, we prove two main results: the first (Theorem~\ref{thm:boron-theta}) computes $\Theta(\fA)$; the second (Theorem~\ref{thm:boron}) constructs rigid tensor categories associated to $G$. We do not completely determine $\Theta(G)$, but one could by extending our methods. The discussion in this section illustrates how measures can be effectively studied from the model-theoretic perspective developed in \S \ref{s:model}.

\begin{figure}
\begin{displaymath}
\begin{tikzpicture}
\tikzset{leaf/.style={circle,fill=black,draw,minimum size=1mm,inner sep=0pt}}
\tikzset{boron/.style={circle,fill=white,draw,minimum size=1.3mm,inner sep=0pt}}
\node[boron] (A) at (-.5,0) {};
\node[boron] (B) at (.5,0) {};
\node[leaf] (C) at (-1.366, .5) {};
\node[leaf] (D) at (-1.366, -.5) {};
\node[leaf] (E) at (1.366, .5) {};
\node[leaf] (F) at (1.366, -.5) {};
\path[draw] (A)--(B);
\path[draw] (A)--(C);
\path[draw] (A)--(D);
\path[draw] (B)--(E);
\path[draw] (B)--(F);
\end{tikzpicture}
\hspace{60pt}
\begin{tikzpicture}
\tikzset{leaf/.style={circle,fill=black,draw,minimum size=1mm,inner sep=0pt}}
\tikzset{boron/.style={circle,fill=white,draw,minimum size=1.3mm,inner sep=0pt}}
\node[boron] (A) at (-1,0) {};
\node[boron] (B) at (1,0) {};
\node[boron] (G) at (0,0) {};
\node[leaf] (C) at (-1.866, .5) {};
\node[leaf] (D) at (-1.866, -.5) {};
\node[leaf] (E) at (1.866, .5) {};
\node[leaf] (F) at (1.866, -.5) {};
\node[leaf] (H) at (0,1) {};
\path[draw] (A)--(G);
\path[draw] (G)--(B);
\path[draw] (A)--(C);
\path[draw] (A)--(D);
\path[draw] (B)--(E);
\path[draw] (B)--(F);
\path[draw] (G)--(H);
\end{tikzpicture}
\end{displaymath}
\caption{The unique boron trees on four and five hydrogen atoms.}
\label{fig:boron}
\end{figure}
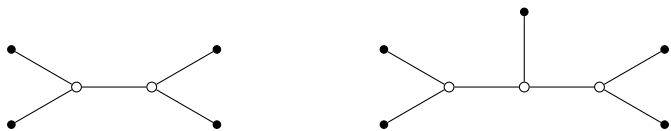

\subsection{Two measures} \label{ss:boron-meas}

We now define two $\bQ$-valued measures $\mu$ and $\nu$ on the class $\fA$, in the sense of Definition~\ref{defn:model-meas}. The measure $\mu$ is regular, and given on a tree $T$ by
\begin{displaymath}
\mu(T) = \begin{cases}
\tfrac{3}{2} & \text{if $n=1$} \\
3 \cdot (-\tfrac{1}{2})^n & \text{if $n \ge 2$} \end{cases}; \qquad n=\# T.
\end{displaymath}
Technically, the above formula is for the R-measure associated to $\mu$ (see Proposition~\ref{prop:R-meas}), but we drop this distinction. Here, and throughout this section, $\# T$ denotes the number of leaves (hydrogen atoms) in $T$. For a general inclusion of trees $T \subset T'$, we have $\mu(T \subset T')=\mu(T') \mu(T)^{-1}$. Note that if $T$ has at least two vertices then $\mu(T \subset T')=(-\tfrac{1}{2})^n$ where $n=\# T' - \# T$.

The measure $\nu$ is not regular, and a bit more complicated to describe. For $0 \le i \le 5$, let $\bT_i$ be the unique boron tree with $i$ hydrogen atoms. Note that for $0 \le i \le 3$ there is a unique inclusion $\bT_i \subset \bT_{i+1}$, up to isomorphism. We put
\begin{align*}
\nu(\bT_0 \subset \bT_1) &= 3 & \nu(\bT_2 \subset \bT_3) &= 1 \\
\nu(\bT_1 \subset \bT_2) &= 2 & \nu(\bT_1 \subset \bT_3) &= 2 \\
\nu(\bT_0 \subset \bT_2) &= 6 & \nu(\bT_0 \subset \bT_3) &= 6
\end{align*}
We also put $\nu(\bT_i \subset T)=0$ for $i \in \{0,1,2,3\}$ whenever $\# T \ge 4$. Now consider an inclusion $T \subset T'$ of boron trees where $\# T \ge 4$. We say that a hydrogen atom $x$ of $T'$ is \defn{paired} if there is another hydrogen atom $y$ such that $x$ and $y$ are connected to the same boron atom. We say that $T \subset T'$ is \defn{bad} if there is a paired vertex of $T'$ that does not belong to $T$, and \defn{good} otherwise. With this terminology in hand, we define
\begin{displaymath}
\nu(T \subset T')= \begin{cases}
0 & \text{if $T \subset T'$ is bad} \\
(-1)^n & \text{otherwise, with $n=\# T'-\# T$} \end{cases}
\end{displaymath}
The main result of \S \ref{ss:boron-meas} is the following theorem:

\begin{theorem} \label{thm:boron-meas}
The rules $\mu$ and $\nu$ are measures.
\end{theorem}

We must verify conditions (a)--(d) of Definition~\ref{defn:model-meas}. Conditions (a) and (b) are clear. We now turn to condition (c). This states that for boron trees $T_1 \subset T_2 \subset T_3$, we have
\begin{displaymath}
\theta(T_1 \subset T_3) = \theta(T_1 \subset T_2) \cdot \theta(T_2 \subset T_3).
\end{displaymath}
Here, and in what follows, we use $\theta$ to denote either $\mu$ or $\nu$.

\begin{lemma}
The rules $\mu$ and $\nu$ satisfy Definition~\dref{defn:model-meas}{b}.
\end{lemma}

\begin{proof}
This is immediate from the definition for $\mu$. We now handle $\nu$. Thus let $T_1 \subset T_2 \subset T_3$ be given. If $T_1$ has at most three leaves it is clear: indeed, if $T_3$ has $\le 3$ leaves then one just examines the various cases, and otherwise both sides of the identity vanish. Now suppose that $T_3$ has $\ge 4$ leaves. It suffices to prove that $T_1 \subset T_3$ is bad if and only if $T_1 \subset T_2$ or $T_2 \subset T_3$ is bad. Thus suppose $T_1 \subset T_3$ is bad, and let $x$ be a paired vertex of $T_3$ that does not belong to $T_1$; let $y$ be the vertex that $x$ is paired with. If either $x$ or $y$ does not belong to $T_2$ then $T_2 \subset T_3$ is bad. Thus suppose that $x$ and $y$ both belong to $T_2$. Then $x$ is a paired vertex of $T_2$, as it is paired to $y$, and so $T_1 \subset T_2$ is bad. This completes the proof.
\end{proof}

We now treat Definition~\dref{defn:model-meas}{c}, which we refer to as the \defn{amalgamation property}. We recall the statement of this condition. Let $T \subset T'$ and $T \subset U$ be inclusions of boron trees, and let $U'_1, \ldots, U'_n$ be the various amalgamations. The amalgamation property states 
\begin{equation} \label{eq:boron-amalg}
\theta(T \subset T') = \sum_{i=1}^n \theta(U \subset U_i').
\end{equation}
It is not difficult to see that one can reduce to the case where $\# U = \# T' = \# T+1$, and we assume this in what follows. Let $x$ be the leaf of $T'$ not contained in $T$, and let $y$ be the leaf of $U$ not contained in $T$. We break our analysis into three cases:
\begin{itemize}
\item Case 1: $T \subset T'$ and $T \subset U$ are isomorphic, and $\# T \ge 3$.
\item Case 2: $T \subset T'$ and $T \subset U$ are non-isomorphic, and $\# T \ge 3$.
\item Case 3: $\# T \le 2$.
\end{itemize}
We treat each case in turn.

\textit{Case~1.} We can identify $U$ with $T'$, but with $x$ relabeled to $y$. Draw $T'$ schematically as
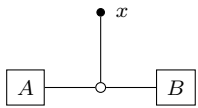
\begin{displaymath}
\begin{tikzpicture}
\tikzset{leaf/.style={circle,fill=black,draw,minimum size=1mm,inner sep=0pt}}
\tikzset{box/.style={rectangle,draw}}
\tikzset{boron/.style={circle,fill=white,draw,minimum size=1.3mm,inner sep=0pt}}
\node[boron] (A) at (0,0) {};
\node[leaf,label=right:{\tiny $x$}] (B) at (0, 1) {};
\node[box] (C) at (-1, 0) {\tiny $A$};
\node[box] (D) at (1, 0) {\tiny $B$};
\path[draw] (A)--(B);
\path[draw] (A)--(C);
\path[draw] (A)--(D);
\end{tikzpicture}
\end{displaymath}
Here $A$ and $B$ are subtrees, necessarily non-empty. Note that we must have this picture since $\# T \ge 3$. We have the following four amalgamations:
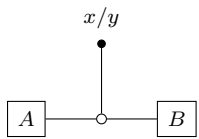
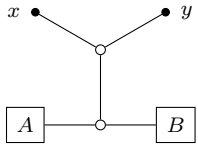
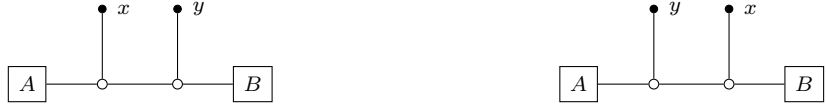
\begin{displaymath}
\begin{tikzpicture}
\tikzset{leaf/.style={circle,fill=black,draw,minimum size=1mm,inner sep=0pt}}
\tikzset{box/.style={rectangle,draw}}
\tikzset{boron/.style={circle,fill=white,draw,minimum size=1.3mm,inner sep=0pt}}
\node[boron] (A) at (0,0) {};
\node[leaf,label=above:{\tiny $x/y$}] (B) at (0, 1) {};
\node[box] (C) at (-1, 0) {\tiny $A$};
\node[box] (D) at (1, 0) {\tiny $B$};
\path[draw] (A)--(B);
\path[draw] (A)--(C);
\path[draw] (A)--(D);
\end{tikzpicture}
\hskip 1.5in
\begin{tikzpicture}
\tikzset{leaf/.style={circle,fill=black,draw,minimum size=1mm,inner sep=0pt}}
\tikzset{box/.style={rectangle,draw}}
\tikzset{boron/.style={circle,fill=white,draw,minimum size=1.3mm,inner sep=0pt}}
\node[boron] (A) at (0,0) {};
\node[boron] (B) at (0, 1) {};
\node[leaf,label=left:{\tiny $x$}] (E) at (-.866, 1.5) {};
\node[leaf,label=right:{\tiny $y$}] (F) at (.866, 1.5) {};
\node[box] (C) at (-1, 0) {\tiny $A$};
\node[box] (D) at (1, 0) {\tiny $B$};
\path[draw] (A)--(B);
\path[draw] (B)--(E);
\path[draw] (B)--(F);
\path[draw] (A)--(C);
\path[draw] (A)--(D);
\end{tikzpicture}
\end{displaymath}
\vskip 5pt
\begin{displaymath}
\begin{tikzpicture}
\tikzset{leaf/.style={circle,fill=black,draw,minimum size=1mm,inner sep=0pt}}
\tikzset{box/.style={rectangle,draw}}
\tikzset{boron/.style={circle,fill=white,draw,minimum size=1.3mm,inner sep=0pt}}
\node[boron] (A) at (-.5,0) {};
\node[boron] (B) at (.5,0) {};
\node[leaf,label=right:{\tiny $x$}] (C) at (-.5, 1) {};
\node[leaf,label=right:{\tiny $y$}] (D) at (.5, 1) {};
\node[box] (E) at (-1.5, 0) {\tiny $A$};
\node[box] (F) at (1.5, 0) {\tiny $B$};
\path[draw] (A)--(B);
\path[draw] (A)--(C);
\path[draw] (A)--(E);
\path[draw] (B)--(D);
\path[draw] (B)--(F);
\end{tikzpicture}
\hskip 1.5in
\begin{tikzpicture}
\tikzset{leaf/.style={circle,fill=black,draw,minimum size=1mm,inner sep=0pt}}
\tikzset{box/.style={rectangle,draw}}
\tikzset{boron/.style={circle,fill=white,draw,minimum size=1.3mm,inner sep=0pt}}
\node[boron] (A) at (-.5,0) {};
\node[boron] (B) at (.5,0) {};
\node[leaf,label=right:{\tiny $y$}] (C) at (-.5, 1) {};
\node[leaf,label=right:{\tiny $x$}] (D) at (.5, 1) {};
\node[box] (E) at (-1.5, 0) {\tiny $A$};
\node[box] (F) at (1.5, 0) {\tiny $B$};
\path[draw] (A)--(B);
\path[draw] (A)--(C);
\path[draw] (A)--(E);
\path[draw] (B)--(D);
\path[draw] (B)--(F);
\end{tikzpicture}
\end{displaymath}
It is not difficult to see that these are all the amalgamations.

\begin{lemma}
The rules $\mu$ and $\nu$ satisfy the amalgamation property in this case.
\end{lemma}

\begin{proof}
Let $U'_1$ and $U'_2$ be the first two trees above, and $U'_3$ and $U'_4$ the second two. Note that we actually have $U'_1=U$, and so $\theta(U \subset U'_1)=1$. In this case, \eqref{eq:boron-amalg} takes the form
\begin{displaymath}
\theta(T \subset T') = 1 + \sum_{i=2}^4 \theta(U \subset U'_i).
\end{displaymath}
The rule $\mu$ assigns value $-\tfrac{1}{2}$ to all of the above inclusions (recall $\# T \ge 3$), and so the equation takes the form $-\tfrac{1}{2} = 1 + 3 \cdot (-\tfrac{1}{2})$, which holds.

We now treat $\nu$. First suppose that $x$ is not a paired leaf in $T'$, or, equivalently, $A$ and $B$ each contain $\ge 2$ leaves (and thus $\# T \ge 4$). Then $x$ is not a paired vertex in either $U_3'$ or $U_4'$, but is a paired vertex in $U_2'$. We thus see that \eqref{eq:boron-amalg} becomes
\begin{displaymath}
-1 = 1 + 0 + (-1) + (-1)
\end{displaymath}
in this case, and thus holds. Now suppose that $x$ is a paired leaf in $T'$. Thus either $A$ or $B$ consists of a single leaf; without loss of generality, suppose $A$ does. Since $\# T \ge 3$, it follows that $B$ contains $\ge 2$ leaves. We thus see that $x$ is a paired vertex in $U_2'$ and $U_3'$, but not in $U'_4$. Therefore, \eqref{eq:boron-amalg} takes the form
\begin{displaymath}
0 = 1 + 0 + 0 + (-1),
\end{displaymath}
which again holds.
\end{proof}

\textit{Case~2.} Going back to the depiction of $T'$ above, choose a ray from the boron atom touching $x$ to a leaf $a$ in $A$, and also one to a leaf $b$ in $B$. We can then depict $T'$ as
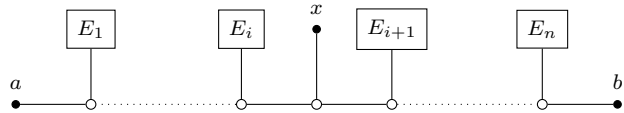
\begin{displaymath}
\begin{tikzpicture}
\tikzset{leaf/.style={circle,fill=black,draw,minimum size=1mm,inner sep=0pt}}
\tikzset{box/.style={rectangle,draw}}
\tikzset{boron/.style={circle,fill=white,draw,minimum size=1.3mm,inner sep=0pt}}
\node[boron] (A) at (0,0) {};
\node[leaf,label=above:{\tiny $x$}] (B) at (0, 1) {};
\node[boron] (C) at (-1, 0) {};
\node[box] (C1) at (-1, 1) {\tiny $E_i$};
\node[boron] (D) at (-3, 0) {};
\node[box] (D1) at (-3, 1) {\tiny $E_1$};
\node[leaf,label=above:{\tiny $a$}] (E) at (-4, 0) {};
\node[boron] (F) at (1,0) {};
\node[box] (F1) at (1,1) {\tiny $E_{i+1}$};
\node[boron] (G) at (3,0) {};
\node[box] (G1) at (3,1) {\tiny $E_n$};
\node[leaf,label=above:{\tiny $b$}] (H) at (4,0) {};
\path[draw] (A)--(B);
\path[draw] (A)--(C);
\draw[dotted] (C)--(D);
\path[draw] (D)--(E);
\path[draw] (C)--(C1);
\path[draw] (D)--(D1);
\path[draw] (A)--(F);
\draw[dotted] (F)--(G);
\path[draw] (F)--(F1);
\path[draw] (G)--(G1);
\path[draw] (G)--(H);
\end{tikzpicture}
\end{displaymath}
where each $E_i$ is some non-empty subtree. If $x$ is paired with $a$ then $i=0$, and if $x$ is paired with $b$ then $i=n$; otherwise, $x$ is unpaired. The tree $T$ is obtained by deleting $x$ from the above picture. Since we have assumed $\# T \ge 3$, it follows that $n \ge 1$, i.e., we have at least one $E$ subtree. The tree $U$ is obtained from $T$ by adding one new vertex $y$. There are two ways to do this:
\begin{enumerate}
\item The vertex $y$ is added to one of the $E$ subtrees, say $E_j$. Let $E_j'$ be the subtree thus obtained. Thus the picture for $U$ is just like the one for $T$, but with $E_j$ changed to $E_j'$. We obtain an amalgamation $U'$ by taking the picture for $T'$ and replacing $E_j$ with $E_j'$.
\item The vertex $y$ ``sprouts'' directly from the $a$--$b$ ray, as $x$ does in $T'$. In this case, $y$ occurs between $E_j$ and $E_{j+1}$ for some $0 \le j \le n$, with the $j=0$ and $j=n$ cases interpreted just like the $i=0$ and $i=n$ cases above. Since we have assumed $T \subset T'$ and $T \subset U$ are non-isomorphic, it follows that $i \ne j$. We obtain an amalgamation $U'$ by taking the picture for $T'$ and ``sprouting'' $y$ in between $E_j$ and $E_{j+1}$.
\end{enumerate}
In each case, we have constructed one amalgamation. Again, it is not difficult to see that these are all the amalgamations.

\begin{lemma}
The rules $\mu$ and $\nu$ satisfy the amalgamation property in this case.
\end{lemma}

\begin{proof}
The equation \eqref{eq:boron-amalg} takes the form $\theta(T \subset T') = \theta(U \subset U')$ in this case. It is clear that $\mu$ satisfies this, since it assigns value $-\tfrac{1}{2}$ to each containment. We now treat $\nu$. First suppose that $x$ is not paired in $T'$, which means $i \ne 0,n$; note that in this case, $n \ge 2$, and so $\# T \ge 4$. Then clearly $x$ is not paired in $U'$ either, and so $\nu$ assigns value $-1$ to both $T \subset T'$ and $U \subset U'$. Now suppose that $x$ is paired in $T'$; without loss of generality, suppose that $i=0$, so that it is paired to $a$. Then $x$ is also paired in $U'$ (to $a$) and so $\nu$ assigns value~0 to both $T \subset T'$ and $U \subset U'$. This completes the proof.
\end{proof}

\textit{Case~3.} Recall that for $0 \le i \le 5$ we denote by $\bT_i$ the unique boron tree with $i$ vertices. We examine amalgamation with $T=\bT_i$ and $U=T'=\bT_{i+1}$, for $0 \le i \le 2$.

First suppose $i=0$. Then there are two amalgamations: $U'_1=\bT_1$ and $U'_2=\bT_2$. Equation \eqref{eq:boron-amalg} becomes
\begin{displaymath}
\theta(\bT_0 \subset \bT_1) = 1 + \theta(\bT_1 \subset \bT_2).
\end{displaymath}
For $\mu$ this becomes $\tfrac{3}{2}=1+\tfrac{1}{2}$, while for $\nu$ it becomes $3=1+2$. Thus the equation holds in this case.

Now suppose $i=1$. There are again two amalgamations: $U'_1=\bT_2$ and $U'_2=\bT_3$. Equation \eqref{eq:boron-amalg} becomes
\begin{displaymath}
\theta(\bT_1 \subset \bT_2) = 1 + \theta(\bT_2 \subset \bT_3).
\end{displaymath}
For $\mu$ this becomes $\tfrac{1}{2}=1+(-\tfrac{1}{2})$, while for $\nu$ it becomes $2=1+1$. Thus it holds.

Finally suppose $i=2$. Let $a$ and $b$ be the two vertices of $T$, and recall that $x$ and $y$ are the new vertices of $T'$ and $U$. As in Case~1 above, there are four amalgamations: $U'_1=\bT_3$ with $x=y$, and $U'_i=\bT_4$ for $2 \le i \le 4$. (The final three amalgamations correspond to the three ways of labeling $\bT_4$ with $a$, $b$, $x$, and $y$.) The equation Equation \eqref{eq:boron-amalg} is thus
\begin{displaymath}
\theta(\bT_2 \subset \bT_3) = 1 + 3 \cdot \theta(\bT_3 \subset \bT_4).
\end{displaymath}
For $\mu$ this becomes $-\tfrac{1}{2}=1+3 \cdot (-\tfrac{1}{2})$, while for $\nu$ it becomes $1=1+3 \cdot 0$. This it holds.

We have now completed the proof that $\mu$ and $\nu$ satisfy the amalgamation property, which concludes the proof of Theorem~\ref{thm:boron-meas}.

\subsection{The $\Theta$ ring}

Recall the ring $\Theta(\fA)$ introduced in Definition~\ref{defn:model-theta}. Let $c \in \Theta(\fA)$ be the class $[\bT_3 \subset \bT_4]$. The following is the main theorem of \S \ref{s:boron}:

\begin{theorem} \label{thm:boron-theta}
We have a ring isomorphism
\begin{displaymath}
\bZ[x]/(2x^2+x) \to \Theta(\fA), \qquad x \mapsto c.
\end{displaymath}
This induces an isomorphism
\begin{displaymath}
\bZ[\tfrac{1}{6}] \cong \Theta^*(\fA).
\end{displaymath}
\end{theorem}

We require several lemmas before proving the theorem.

\begin{lemma}
The ring $\Theta(\fA)$ is generated by classes $[U \subset U']$ where $\#U'=1+\#U$ and $\#U \le 4$.
\end{lemma}

\begin{proof}
It follows from condition Definition~\dref{defn:model-meas}{b} that $\Theta(\fA)$ is generated by classes $[U \subset U']$ with $\#U'=1+\# U$. Suppose given such an inclusion with $\# U \ge 5$. We express the class $[U \subset U']$ in terms of classes involving smaller trees.

Let $x$ be the vertex in $U'$ not contained in $U$. Draw $U'$ as we drew $T'$ in Case~2 above: $x$ sprouts directly from a geodesic joining leaves $a$ and $b$, and there are various other subtrees $E_1, \ldots, E_n$. Suppose $\# E_j \ge 2$ for some $j$. Let $y$ be a vertex in $E_j$, and let $T$ and $T'$ be the trees obtained from $U$ and $U'$ by deleting $y$. Then as in Case~2, $U'$ is the unique amalgamation of $T \subset T'$ and $U \subset U'$, and so $[U \subset U']=[T \subset T']$ in $\Theta(\fA)$. This completes this case.

Now suppose that $\# E_i=1$ for all $i$. Then the number of leaves in $U$ is $2+n$, where the~2 counts $a$ and $b$, and the $n$ counts the unique leaf in each $E_i$ for $1 \le i \le n$. Since $\# U$ is assumed to have at least five leaves, we have $n \ge 3$. It follows that we can pick $j$ such that $E_j$ is not directly next to $x$; let $y$ be the leaf in $E_j$. Then defining $T$ and $T'$ as above, we again have that $U'$ is the unique amalgamation, and so $[U \subset U']=[T \subset T']$. This completes the proof.
\end{proof}

Recall that $\bT_5$ is the unique boron tree on five vertices. We name a few of its leaves:
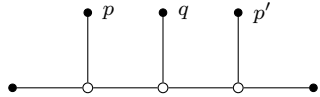
\begin{displaymath}
\begin{tikzpicture}
\tikzset{leaf/.style={circle,fill=black,draw,minimum size=1mm,inner sep=0pt}}
\tikzset{box/.style={rectangle,draw}}
\tikzset{boron/.style={circle,fill=white,draw,minimum size=1.3mm,inner sep=0pt}}
\node[boron] (A) at (-1,0) {};
\node[boron] (B) at (0,0) {};
\node[boron] (C) at (1,0) {};
\node[leaf,label=right:{\tiny $p$}] (A1) at (-1,1) {};
\node[leaf] (A2) at (-2,0) {};
\node[leaf,label=right:{\tiny $q$}] (B1) at (0,1) {};
\node[leaf,label=right:{\tiny $p'$}] (C1) at (1,1) {};
\node[leaf] (C2) at (2,0) {};
\path[draw] (A)--(B);
\path[draw] (A)--(A1);
\path[draw] (A)--(A2);
\path[draw] (B)--(C);
\path[draw] (B)--(B1);
\path[draw] (C)--(C1);
\path[draw] (C)--(C2);
\end{tikzpicture}
\end{displaymath}
We note that $p'$, and the two unlabeled leaves, are equivalent to $p$ under the automorphism group of the tree. Let $\bT_5^p$ denote the tree obtained from $\bT_5$ by deleting $p$, and similarly with the other leaves. Up to isomorphism, there are two inclusions $\bT_4 \subset \bT_5$, namely $\bT_5^p \subset \bT_5$ and $\bT_5^q \subset \bT_5$.

We now name some classes in $\Theta(\fA)$. For $1 \le i \le 4$, let $\alpha_i=[\bT_{i-1} \subset \bT_i]$.  Note that $\alpha_4$ is the element we have already named $c$. Also let $\alpha_5^p=[\bT_5^p \subset \bT_5]$ and $\alpha_5^q=[\bT_5^q \subset \bT_5]$. The above lemma shows that these six classes generate $\Theta(\fA)$.

\begin{lemma}
We have
\begin{align*}
\alpha_1=3c+3, \quad
\alpha_2=3c+2, \quad
\alpha_3=3c+1, \quad
\alpha_4=c, \quad
\alpha_5^p=c, \quad
\alpha_5^q=-1-c.
\end{align*}
In particular, $\Theta(\fA)$ is generated by $c$.
\end{lemma}

\begin{proof}
The analysis in Case~3 above gives
\begin{displaymath}
\alpha_1=1+\alpha_2, \qquad \alpha_2=1+\alpha_3, \qquad \alpha_3=1+3\alpha_4.
\end{displaymath}
As $\alpha_4=c$, this gives the stated formulas for $\alpha_1, \ldots, \alpha_4$.

We now look at $\alpha_5^p$. In fact, this is handled just as in the previous lemma: put
\begin{displaymath}
T=\bT^{p,p'}_5 \cong \bT_3, \qquad
U = \bT^{p} \cong \bT_4, \qquad
T'=\bT_5^{p'} \cong \bT_4, \qquad
U' = \bT_5
\end{displaymath}
Then $U'$ is the unique amalgamation of $T \subset T'$ and $T \subset U$, and so $[U \subset U']=[T \subset T']$, which gives $\alpha_5^p=\alpha_4$.

Finally we examine $\alpha_5^q$. Let $T'$ be the following tree
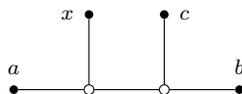
\begin{displaymath}
\begin{tikzpicture}
\tikzset{leaf/.style={circle,fill=black,draw,minimum size=1mm,inner sep=0pt}}
\tikzset{boron/.style={circle,fill=white,draw,minimum size=1.3mm,inner sep=0pt}}
\node[boron] (A) at (-.5,0) {};
\node[boron] (B) at (.5,0) {};
\node[leaf,label=above:{\tiny $a$}] (C) at (-1.5, 0) {};
\node[leaf,label=above:{\tiny $b$}] (D) at (1.5, 0) {};
\node[leaf,label=left:{\tiny $x$}] (A1) at (-.5, 1) {};
\node[leaf,label=right:{\tiny $c$}] (B1) at (.5, 1) {};
\path[draw] (A)--(B);
\path[draw] (A)--(C);
\path[draw] (A)--(A1);
\path[draw] (B)--(B1);
\path[draw] (B)--(D);
\end{tikzpicture}
\end{displaymath}
Let $T$ be the subtree obtained by deleting $x$, and let $U$ be a copy of $T'$ in which $x$ is relabeled to $y$. We have inclusions $T \subset T'$ and $T \subset U$. There are four amalgamations (this is just like Case~1 above):
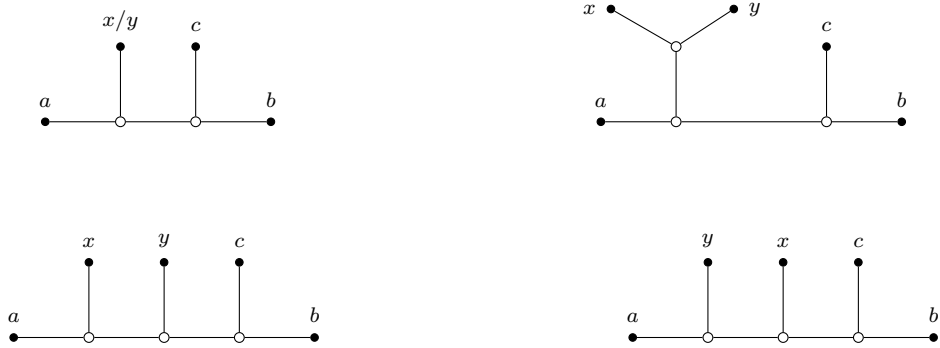
\begin{displaymath}
\begin{tikzpicture}
\tikzset{leaf/.style={circle,fill=black,draw,minimum size=1mm,inner sep=0pt}}
\tikzset{boron/.style={circle,fill=white,draw,minimum size=1.3mm,inner sep=0pt}}
\node[boron] (A) at (-.5,0) {};
\node[boron] (B) at (.5,0) {};
\node[leaf,label=above:{\tiny $a$}] (C) at (-1.5, 0) {};
\node[leaf,label=above:{\tiny $b$}] (D) at (1.5, 0) {};
\node[leaf,label=above:{\tiny $x/y$}] (A1) at (-.5, 1) {};
\node[leaf,label=above:{\tiny $c$}] (B1) at (.5, 1) {};
\path[draw] (A)--(B);
\path[draw] (A)--(C);
\path[draw] (A)--(A1);
\path[draw] (B)--(B1);
\path[draw] (B)--(D);
\end{tikzpicture}
\hskip 1.5in
\begin{tikzpicture}
\tikzset{leaf/.style={circle,fill=black,draw,minimum size=1mm,inner sep=0pt}}
\tikzset{boron/.style={circle,fill=white,draw,minimum size=1.3mm,inner sep=0pt}}
\node[boron] (A) at (-1,0) {};
\node[boron] (B) at (1,0) {};
\node[leaf,label=above:{\tiny $a$}] (C) at (-2, 0) {};
\node[leaf,label=above:{\tiny $b$}] (D) at (2, 0) {};
\node[boron] (A1) at (-1, 1) {};
\node[leaf,label=above:{\tiny $c$}] (B1) at (1, 1) {};
\node[leaf,label=left:{\tiny $x$}] (E) at (-1.866, 1.5) {};
\node[leaf,label=right:{\tiny $y$}] (F) at (-.233, 1.5) {};
\path[draw] (A)--(B);
\path[draw] (A)--(C);
\path[draw] (A)--(A1);
\path[draw] (B)--(B1);
\path[draw] (B)--(D);
\path[draw] (A1)--(E);
\path[draw] (A1)--(F);
\end{tikzpicture}
\end{displaymath}
\vskip 5pt
\begin{displaymath}
\begin{tikzpicture}
\tikzset{leaf/.style={circle,fill=black,draw,minimum size=1mm,inner sep=0pt}}
\tikzset{boron/.style={circle,fill=white,draw,minimum size=1.3mm,inner sep=0pt}}
\node[boron] (A) at (-1,0) {};
\node[boron] (B) at (0,0) {};
\node[boron] (C) at (1,0) {};
\node[leaf,label=above:{\tiny $a$}] (D) at (-2, 0) {};
\node[leaf,label=above:{\tiny $b$}] (E) at (2, 0) {};
\node[leaf,label=above:{\tiny $x$}] (A1) at (-1, 1) {};
\node[leaf,label=above:{\tiny $y$}] (B1) at (0, 1) {};
\node[leaf,label=above:{\tiny $c$}] (C1) at (1, 1) {};
\path[draw] (A)--(D);
\path[draw] (A)--(A1);
\path[draw] (A)--(B);
\path[draw] (B)--(B1);
\path[draw] (B)--(C);
\path[draw] (C)--(C1);
\path[draw] (C)--(E);
\end{tikzpicture}
\hskip 1.5in
\begin{tikzpicture}
\tikzset{leaf/.style={circle,fill=black,draw,minimum size=1mm,inner sep=0pt}}
\tikzset{boron/.style={circle,fill=white,draw,minimum size=1.3mm,inner sep=0pt}}
\node[boron] (A) at (-1,0) {};
\node[boron] (B) at (0,0) {};
\node[boron] (C) at (1,0) {};
\node[leaf,label=above:{\tiny $a$}] (D) at (-2, 0) {};
\node[leaf,label=above:{\tiny $b$}] (E) at (2, 0) {};
\node[leaf,label=above:{\tiny $y$}] (A1) at (-1, 1) {};
\node[leaf,label=above:{\tiny $x$}] (B1) at (0, 1) {};
\node[leaf,label=above:{\tiny $c$}] (C1) at (1, 1) {};
\path[draw] (A)--(D);
\path[draw] (A)--(A1);
\path[draw] (A)--(B);
\path[draw] (B)--(B1);
\path[draw] (B)--(C);
\path[draw] (C)--(C1);
\path[draw] (C)--(E);
\end{tikzpicture}
\end{displaymath}
Call the first two $U_1'$ and $U_2'$ and the second two $U_3'$ and $U_4'$. We thus have the equation
\begin{displaymath}
[T \subset T'] = \sum_{i=1}^4 [U \subset U_i']
\end{displaymath}
in $\Theta(\fA)$. We have
\begin{displaymath}
[T \subset T'] = \alpha_4, \quad [U \subset U'_1]=1, \quad [U \subset U_2']=\alpha_5^p, \quad
[U \subset U_3'] = \alpha_5^p, \quad [U \subset U_4'] = \alpha_5^q.
\end{displaymath}
We thus obtain
\begin{displaymath}
\alpha_4 = 1 + 2 \alpha_5^p + \alpha_5^q,
\end{displaymath}
and so $\alpha_5^q=-1-c$ as stated.
\end{proof}

\begin{lemma}
We have $c \cdot (2c+1)=0$.
\end{lemma}

\begin{proof}
Regard $\bT_3$ as the subtree of $\bT_5$ obtained by deleting $p$ and $q$. We thus have
\begin{displaymath}
[\bT_3 \subset \bT_5] = [\bT_3 \subset \bT_5^p] \cdot [\bT_5^p \subset \bT_5]
= [\bT_3 \subset \bT_5^q] \cdot [\bT_5^q \subset \bT_5].
\end{displaymath}
Note that $[\bT_3 \subset \bT_5^p]=[\bT_3 \subset \bT_5^q]=\alpha_4$ since $\bT_5^p$ and $\bT_5^q$ are just copies of $\bT_4$. Thus the above equation yields
\begin{displaymath}
\alpha_4 \cdot \alpha_5^p = \alpha_4 \cdot \alpha_5^q.
\end{displaymath}
Using the formulas for these quantities in terms of $c$ from the previous lemma yields the result.
\end{proof}

\begin{proof}[Proof of Theorem~\ref{thm:boron-theta}]
Let $R=\bZ[x]/(x(2x+1))$. The above lemmas show that we have a surjective ring homomorphism $i \colon R \to \Theta(\fA)$ given by $i(x)=c$. Now, $R$ is $\bZ$-torsion free and $R \otimes \bQ$ is isomorphic (as a ring) to $\bQ \times \bQ$. The existence of the measures $\mu,\nu \colon \Theta(\fA) \to \bQ$ thus shows that $i$ is injective. (Note that $\mu(c)=-\tfrac{1}{2}$ and $\nu(c)=0$.) Thus $i$ is an isomorphism.

We now turn to $\Theta^*(\fA)$. By the above, we have an isomorphism $R[1/x] \to \Theta(\fA)[1/c]$, and clearly $R[1/x] \cong \bZ[1/2]$.  If $i$ is an embedding in $\fA$ then $[i]=\pm 3/2^n$ in $\Theta(\fA)[1/c] \cong \bZ[1/2]$. Inverting these classes, we find $\Theta^*(\fA)=\bZ[1/6]$, as claimed.
\end{proof}

\subsection{Representation categories}

Recall that $\Omega$ is the Fra\"iss\'e limit of the class $\fA$, and $G$ is its automorphism group. For a finite subset $A$ of $\Omega$ recall that $G(A)$ is the subgroup of $G$ fixing each element of $A$. We let $\sE$ be the collection of all subgroups of the form $G(A)$. This is the stabilizer class generated by $\Omega$. By Theorem~\ref{thm:compare}, we have $\Theta^*(\fA)=\Theta^*(G;\sE)$. We let $\mu_1$ be the $\bZ[1/6]$-valued measure for $(G,\sE)$ corresponding to $\mu$ under this isomorphism. Fix a field $k$ of characteristic $\ne 2,3$; we also regard $\mu_1$ as $k$-valued in what follows. The following is our main result on tensor categories coming from boron trees:

\begin{theorem} \label{thm:boron}
The Karoubi envelope of $\uPerm_k(G,\sE;\mu_1)$ is a semi-simple pre-Tannakian category.
\end{theorem}

We require some lemmas before proving the theorem.

\begin{lemma} \label{lem:boron-1}
Suppose that $A$ and $B$ are finite subsets of $\Omega$ such that $G(B)$ is contained in $G(A)$ with finite index. Then $A=B$.
\end{lemma}

\begin{proof}
Replacing $B$ with $A \cup B$, we can assume that $A \subset B$. Suppose that this containment is proper and that $G(B)$ has finite index in $G(A)$; we will produce a contradiction. Recall that $\Omega^{[A]}$ denotes the set of embeddings $A \to \Omega$, and this is isomorphic to $G/G(A)$. Since $G(B)$ has finite index in $G(A)$, the fibers of $\Omega^{[B]} \to \Omega^{[A]}$ have finite cardinality, say $n$. This means that the standard embedding $A \to \Omega$ can be extended in only $n$ ways to $B$. However, it is easy to write down a finite boron tree $C$ and an embedding $A \to C$ that extends to $B$ in $>n$ ways. Since $C$ embeds into $\Omega$, this gives a contradiction.
\end{proof}

For a finite subset $A$ of $\Omega$, let $G[A]$ be the subgroup fixing $A$ as a set. By the above lemma, $G[A]$ is exactly the normalizer in $G$ of $G(A)$. We let $\sE^+$ be the set of all subgroups $U$ of $G$ such that $G(A) \subset U \subset G[A]$ for some $A$; it is easily seen to be a stabilizer class.

\begin{lemma} \label{lem:boron-2}
We have the following:
\begin{enumerate}
\item Suppose $V \subset U$ is a finite index containment with $V \in \sE^+$. Then $U \in \sE^+$.
\item If $X$ is an $\sE^+$-smooth $G$-set and $H$ is a subgroup of $\Aut(X)$ then $X/H$ is $\sE^+$-smooth.
\item If $X$ is an $\sE^+$-smooth set then so is $X^{(n)}$ for any $n \ge 0$.
\end{enumerate}
\end{lemma}

\begin{proof}
(a) Let $A$ be such that $G(A) \subset V \subset G[A]$. Thus $G(A)$ has finite index in $U$. The intersection of all conjugates of $G(A)$ in $U$ has the form $G(B)$, and is contained in $G(A)$ with finite index, and is thus equal to $G(A)$ by Lemma~\ref{lem:boron-1}. Thus $G(A)$ is normal in $U$, and so $U$ is contained in the normalizer of $G(A)$, which is $G[A]$. Thus $U \in \sE^+$, as required.

(b) It suffices to consider the case where $X$ is transitive. Write $X=G/U$ where $U \in \sE^+$. Then $\Aut(X)=\rN(U)/U$, and so $H$ corresponds to a subgroup $V$ between $U$ and $\rN(U)$. Since $H$ is finite (Proposition~\ref{prop:smooth}(d)), $V$ contains $U$ with finite index, and so $V \in \sE^+$ by (a). We have $X/H=G/V$, and so $X/H$ is $\sE^+$-smooth.

(c) This follows from (b) since $X^{(n)}=X^{[n]}/\fS_n$.
\end{proof}

\begin{lemma} \label{lem:boron-3}
Let $T$ be a boron tree. Then $\Aut(T)$ has order $2^r 3^s$ with $s \in \{0,1\}$.
\end{lemma}

\begin{proof}
If $\# T \le 3$ this is clear. Suppose now that $\# T \ge 4$. Let $T'$ be the boron tree obtained from $T$ by deleting all paired vertices, and converting the neighboring boron atoms to hydrogen. Any automorphism of $T$ restricts to an automorphism of $T'$, and we have an exact sequence
\begin{displaymath}
1 \to (\bZ/2\bZ)^r \to \Aut(T) \to \Aut(T')
\end{displaymath}
where $r$ is the number of pairs of paired vertices in $T$. (The $\bZ/2\bZ$ factors here are involutions that swap paired vertices.) The result is true for $T'$ by induction, and thus for $T$ by the above sequence.
\end{proof}

Every member of $\sE^+$ contains a member of $\sE$ with finite index, and so $\sE$ is ``large'' in $\sE^+$ in (a relative analog of) the sense of \S \ref{ss:Theta-prop}(d). Thus $\Theta^*(G;\sE^+) \otimes \bQ=\Theta^*(G;\sE) \otimes \bQ$, and so $\mu_1$ extends uniquely to a homomorphism $\mu_2 \colon \Theta^*(G;\sE^+) \to \bQ$.

\begin{lemma} \label{lem:boron-4}
$\mu_2$ takes values in $\bZ[1/6]$, and is regular if regarded as a $\bZ[1/6]$-valued measure for $(G,\sE^+)$.
\end{lemma}

\begin{proof}
Let $U \in \sE^+$, and let $A$ be a finite subset of $\Omega$ such that $G(A) \subset U \subset G[A]$. Then
\begin{displaymath}
\mu_2(G/U) = \frac{\mu_1(G/G(A))}{[U:G(A)]}.
\end{displaymath}
We know that the numerator has the form $\pm 2^{-r}3^s$. The denominator divides $[G[A]:G(A)]$, which has the form $2^r3^s$ by Lemma~\ref{lem:boron-3}. (We note that $G[A]/G(A)$ is the automorphism group of the boron tree structure on $A$ induced from that on $\Omega$.) Thus the result follows.
\end{proof}

\begin{proof}[Proof of Theorem~\ref{thm:boron}]
Regard $\mu_2$ as a $k$-valued measure. Since it takes values in the (image of) $\bZ[1/6]$, it satisfies condition~(P) from Definition~\ref{defn:P}. By Lemma~\ref{lem:boron-2}(c), $\sE^+$ satisfies condition $(\ast)$ from Remark~\ref{rmk:rel-binom}. By Theorem~\ref{thm:regss} in the relative case (see \S \ref{ss:abrel}), the category $\uRep_k(G,\sE^+;\mu_2)$ is semi-simple and locally pre-Tannakian. The functor
\begin{displaymath}
\uPerm_k(G,\sE;\mu_1) \to \uRep_k(G,\sE^+;\mu_2)
\end{displaymath}
is fully faithful, and every object in the target is a quotient of something in the image. Since the target category is semi-simple, it is therefore equivalent to the Karoubi envelope of the source. This completes the proof.
\end{proof}

\begin{remark}
Theorem~\ref{thm:boron} implies that $\mu_1$ extends to a regular measure for $G$. This follows from the main results of \cite{discrete}. See \cite[\S 18]{arxiv} for a proof of the weaker result that $\mu_1$ extends to a normal measure for $G$.
\end{remark}

\begin{remark}
Arguing as in \S \ref{ss:not-interp}, one can show that $\uRep_k^{\rf}(G;\mu')$ cannot be obtained by interpolating finite groups: again, $G$ acts transitively on $\Omega^{(5)}$ and not $\Omega^{[5]}$. (In this case, one could also use the fact that $G$ is 3-transitive and appeal to the classification of 3-transitive finite groups \cite[Theorem~5.2]{Cameron8}.)
\end{remark}

\begin{remark}
With Nekrasov, we investigate tree categories further in \cite{arboreal}.  For every $d \ge 3$, there are analogs of $\mu$ and $\nu$ for trees of valence at most $d$, as well as two one-parameter families of measures on unbounded valence trees which ``interpolate" these.
\end{remark}

\addtocontents{toc}{\bigskip}

\end{document}